\documentclass[12pt]{article}
\usepackage{graphicx,psfrag,epsf}
\usepackage{natbib}
\usepackage{url}

\usepackage[margin = 2.3cm]{geometry}

\usepackage{amssymb, amsthm, amsmath, ascmac, authblk, booktabs, comment, enumerate, latexsym, lscape, longtable, mathrsfs, mathtools, multirow, newtxtext, psfrag, setspace, thmtools}
\usepackage[stable]{footmisc}

\usepackage{graphicx}
\usepackage{booktabs}
\usepackage{rotating}

\usepackage{enumitem}

\usepackage{color}
\usepackage{pgfplots}
\pgfplotsset{compat=1.16}
\usepackage{here}
\usepackage{multirow}
\usepackage{caption}
\usepackage[subrefformat=parens]{subcaption}
\usepackage[colorlinks=true, citecolor=teal, linkcolor=teal, urlcolor=teal, backref=page]{hyperref}
\usepackage{threeparttable}
\usepackage{stackengine}
\usepackage{tikz}
\usetikzlibrary{arrows.meta}
\usetikzlibrary{bending}

\usepackage{lineno}

\usepackage[capitalize]{cleveref}

\numberwithin{equation}{section}

\newtheorem{theorem}{Theorem}[section]
\newtheorem{lemma}{Lemma}[section]

\newtheorem{corollary}{Corollary}[section]
\newtheorem{assumption}{Assumption}[section]

\theoremstyle{definition}
\newtheorem*{rmk*}{Remark}
\newtheorem{rmk}{Remark}[section]

\usepackage{etoolbox}

\usepackage{tocloft}
\usepackage{etoc}

\cftsetpnumwidth{4em}
\cftsetrmarg{5em}

\newcommand{\supplementtoc}[1]{%
    \begingroup
    \etocsetnexttocdepth{subsection}
    \etocsettagdepth{main}{none}
    \etocsettagdepth{suppI}{none}
    \etocsettagdepth{suppII}{none}
    \etocsettagdepth{#1}{subsection}
    \etocsettocstyle{\subsection*{Contents}}{}
    \tableofcontents
    \endgroup
}

\DeclareMathOperator{\Var}{Var}
\DeclareMathOperator{\Cov}{Cov}
\DeclareMathOperator{\E}{\mathbb{E}}
\DeclareMathOperator{\diag}{diag}

\newcommand{\ul}[1]{\underline{#1}}

\newcommand{\mcl}[1]{\mathcal{#1}}
\newcommand{\sgn}[1]{\text{sgn}(#1)}

\newcommand{\bra}[1]{\left(#1\right)}

\newcommand{\sbra}[1]{\left[#1\right]}
\newcommand{\abs}[1]{\left|#1\right|}

\newcommand{\eps}{\varepsilon}
\newcommand{\pr}{\mathbb{P}} 
\newcommand{\symg}{\mathbb S} 

\allowdisplaybreaks[4]

\begin{document}
\etocdepthtag.toc{main}
\baselineskip=18pt

\title{Gaussian Approximation for High-Dimensional Weighted and Two-Sample $U$-statistics with Size-Dependent Kernels}

\author[1]{Shunsuke Imai \thanks{This research was supported by JSPS KAKENHI Grant Number 24KJ1472. 
I thank the participants at SETA2026.
I am also grateful to  Yuta Koike for numerous invaluable comments on many aspects of this paper. All errors are entirely my own.
}}

\affil[1]{Graduate School of Economics, Kyoto University}

\maketitle

\begin{abstract}
    We develop Gaussian approximations for high‑dimensional vectors formed by second‑order $U$- and $V$‑statistics whose  kernels depend on sample size.
   To accommodate weighted and two-sample $U$- and $V$-statistics within a unified framework, we establish our results under independent but not identically distributed (i.n.i.d.) sampling.
    The resulting theorems apply to a broad range of problems, including estimation in linear regression models with many covariates or many-weak instruments, variance component analysis with growing ranks, smoothed two-sample testing, and inference for separately exchangeable arrays, among others.
    Our results do not require either component of the Hoeffding decomposition to dominate, thereby covering both non-degenerate and degenerate regimes.
    In addition, we develop maximal inequalities for high-dimensional $U$-statistics with size-dependent kernels under the i.n.i.d.~setting, in a form that remains sharp across a broad range of applications, which may be of independent interest.
    Finally, applications to estimation and testing with many weak instrumental variables demonstrate that, our results are sharp enough to recover the corresponding fixed-dimensional rate conditions up to logarithmic factors.
\vspace{2mm}

\end{abstract}

\clearpage


\newpage
\section{Introduction} \label{sec:intro_inid}

In modern empirical analysis, the number of target parameters of statistical inference can be large and one might wish to conduct simultaneous inference on such high-dimensional parameters, as noted by many authors (e.g. \citealp{belloni2018high,CCKK22,chernozhukov2023high}).
For example, such situations arise when there are many outcomes, groups or time points (or combination thereof), and parameters are estimated for each outcome, group, and time point \citep[Section 1.1]{belloni2018high}; when economic theory implies a large number of testable conditions \cite[e.g.][]{chernozhukov2019inference}; or when one seeks to conduct uniform inference over tuning parameters or model specifications for the purpose of adaptive inference \citep{horowitz2001adaptive}, sensitivity analysis or robustness check \citep{armstrong2018simple} or post-selection inference \citep{kuchibhotla2022post}. 

As a useful tool for such problems, \cite{CCK13} established Gaussian approximation for the sum of high-dimensional independent random vectors. 
Notably, their approximation is valid even when the dimension $p$ is much larger than the sample size $n$ and does not  restrict correlations of the coordinates of the random vectors. 
Since then, many authors have extended the remarkable result to several directions including $U$-statistic type statistics  \citep{chen2018gaussian,ChKa19,ChKa20,song2019approximating,song2023stratified,cheng2022gaussian,chiang2023inference,koike2023high,imai2025gaussian}. 

Among them, the most closely related work to ours is \cite{imai2025gaussian}.
In \cite{imai2025gaussian}, they provide Gaussian approximations for high‑dimensional non-weighted one-sample $U$-statistics with size-dependent kernels.
The point of their results is that their developed Gaussian approximations remain valid regardless of which term in the Hoeffding decomposition is dominant and even in the case such dominant component is absent.
Their results does not require non-degeneracy and apply, for example, to adaptive nonparametric specification tests based
on one-sample $U$-statistics \citep{hardle1993comparing,zheng1996consistent} and simultaneous inference for semiparametric estimators under  small-bandwidth asymptotics \citep{cattaneo2014bootstrapping,cattaneo2014small}.
Such applications are not generally covered by earlier Gaussian
approximation results that impose non-degeneracy requirement.
However, other important classes, such as weighted and two-sample
$U$- and $V$-statistics, also arise in applications of statistics and are not generally covered by their setting.

Weighted $U$- and $V$-statistics arise, for example, in many-covariate asymptotics for linear regression and series-based semiparametric estimators \citep{cattaneo2018alternative}, and in robust inference for linear regression with many-weak instruments \citep{chao2012asymptotic} and Anderson-Rubin test with many-weak instruments \citep{mikusheva2022inference}.
Weighted $U$-statistics also appear in analysis on variance components \citep{kline2020leave}, testing many linear restrictions \citep{anatolyev2023testing}, testing for peer effects without specifying the network structure \citep{jung2026testing},
and two-step estimation in factor models \citep{anatolyev2021limit}, among others.
In these examples, the quadratic component in the Hoeffding decomposition is dominant or at least is not always negligible, so Gaussian approximations based only on the linear projection may be insufficient.

Two-sample $U$- and $V$-statistics arise in smoothed tests of
distributional homogeneity
\citep{li2009nonparametric,gretton2012kernel,li2024optimality} and equivalence of two regression functions \citep{lavergne2001equality}.
They also appear in the analysis of separately exchangeable
arrays under sample-size dependent data generating processes  \citep{chiang2023using}, among others.
As in the cases of weighted $U$- and $V$-statistics, the quadratic component in these examples are also not always negligible.

The literature on these problems already provides several motivations for conducting uniform inference over large collections of parameters or test statistics.
Examples include series-based partially linear regression that remains valid after specification searches over the number of series terms \citep{kang2021inference};
tests for equality of distributions uniformly valid over diverging number of bandwidth candidates \citep{schrab2023mmd,li2024optimality}.
Other examples arise when comparing distributions across many populations \citep{kim2021comparing}, investigating many exposure--outcome relationships in Mendelian randomization
\citep{sanderson2022mendelian}, which is a design of causal inference where many weak instruments often arise \citep{davies2015many}.
In change-point analysis, the unknown location of a change-point leads researchers to examine collections of two-sample test statistics indexed by candidate break points \citep{wang2022inference}.
These examples are by no means exhaustive, and many other applications
may benefit from simultaneous inference based on weighted and two-sample
$U$- and $V$-statistics.

This paper contributes to the literature of high-dimensional Gaussian approximation and the applied theory of $U$- and $V$-statistics by establishing  Gaussian approximations for high-dimensional vectors of weighted and two-sample second-order $U$- and $V$-statistics with size-dependent kernels.
To treat these statistics within a unified framework, we allow
independent but not identically distributed (i.n.i.d.) sampling
and observation-index-dependent kernels.
As in \cite{imai2025gaussian}, our bounds on Gaussian approximation errors are analytical and explicitly express the dependence on the dimension $p$ and remain valid regardless of whether the first‑ or second‑order Hoeffding component dominates or such dominant component does not exist.
As special cases, we allow non-degenerate cases and degenerate cases.

We illustrate the applicability of our results by developing high-dimensional extensions of two foundational contributions to inference with many weak instruments: the instrumental variables regression framework of \citet{chao2012asymptotic} and the jackknife Anderson-Rubin test of \citet{mikusheva2022inference}.
We provide the extensions of these works because these studies are milestones which provide theoretical foundations for a broad range of related work, so their extensions to simultaneous inference offer useful templates for applied theorists working with more specialized estimators and test statistics.
For the framework of \citet{chao2012asymptotic}, we provides uniformly valid derivation of the influential function with the $U$-statistic form (\cref{lemma:chao-Gamma-convergence}) and uniformly valid Gaussian approximation of the derived influential function (\cref{thm:gauss-chao}).
For the jackknife AR test of \citet{mikusheva2022inference}, we establish uniform consistency of the variance estimator (\cref{thm:mikusheva-sun-variance}) and joint Gaussian approximations under both the null and alternative hypotheses (\cref{thm:mikusheva-sun-null,thm:mikusheva-sun-alt}).
We further propose a multiplier bootstrap and establish its validity under both hypotheses (\cref{thm:mikusheva-sun-bootstrap}).
We establish family-wise error rate control and power properties of the resulting simultaneous testing procedure (\cref{coro:mikusheva-sun-multitest}).
Monte Carlo simulations examine its finite-sample performance and illustrate the construction of simultaneous confidence sets by AR test inversion.
Importantly, the conditions we derive for Gaussian approximation and consistency of the variance estimators recover the weakest known sufficient conditions up to logarithmic factors in the corresponding fixed-dimensional settings of \cite{chao2012asymptotic} and \cite{mikusheva2022inference}; see Remarks \cref{rem:chao-sharp} and \cref{rem:smikusheva-sun-sharp}, respectively.
This demonstrates that our maximal inequalities and Gaussian approximation bounds are sufficiently enough sharp to accommodate the structure of these applications.

On the technical side, we build on several lemmas and parts of the proof of Gaussian approximations in \citet{imai2025gaussian} while making suitable modifications of proof and developing new auxiliary results to accommodate the i.n.i.d.~sampling. 
In particular, we directly utilize their high-dimensional central limit theorem via exchangeable pairs \citep[Theorem 6]{imai2025gaussian} as the starting point for our analysis of Gaussian approximation error.
Also, to evaluate some quantities in the normal approximation error, we directly use their maximal inequalities for high-dimensional martingales \citep[Lemma 1]{imai2025gaussian}, with due care in the construction of the filtration.
However, in evaluating the Gaussian approximation error and deriving the primitive conditions for applications, \citet{imai2025gaussian} make substantial use of auxiliary results by \citet{DoPe19} in conjunction with the original lemmas by \cite{imai2025gaussian}, many of which can be applied only in the i.i.d.~setting and/or when the kernel function of $U$-statistics do not depend on the index. 
Accordingly, we reconstruct parts of the argument and several auxiliary lemmas tailored to the i.n.i.d.~setting.
In the i.n.i.d.~case, the marginal laws vary across observations, so simplifications using the exchangeability available under i.i.d.~sampling no longer apply.
In particular, the lack of a common distribution prevents usual simplifications of multi-index summations, and arguments of conditional expectation require delicate choices of conditioning $\sigma$-fields and filtrations, for the analysis of $U$-statistics.
Among them, the most important technical feature of this paper is that, both in deriving the Gaussian approximation error bounds and in developing the maximal inequalities for general-order $U$-statistics, we keep explicit the summations over certain sample indices to preserve sharpness of bounds on  Gaussian approximation error and the concentration of $U$-statistics for a broad class of applications, including many-covariate asymptotics for linear regression and series-based semiparametric estimators and many-weak instrumental variable asymptotics, where certain properties of the sum of elements in projection matrices plays important roles to derive the convergence rates of some quantities. 
Replacing these sums by maxima over sample indices would enable us to mimic the proof by \cite{imai2025gaussian} and give simpler results, but such simplification should lead to too rough bounds for  applications. 
As an auxiliary result for the purpose, we extend the $L_1$-maximal inequality for high-dimensional non-negative adapted sequences used in \citet{imai2025gaussian} to general $L_q$ norms and use it to establish maximal inequalities for general-order $U$-statistics under i.n.i.d.~sampling.
See \cref{sec:main_maximal} for details.

\paragraph{Organization:}
The rest of the paper is organized as follows.
\cref{sec:setup-inid} introduces the framework and explains how weighted and two-sample $U$-statistics fit within it.
\cref{sec:main-inid} presents our Gaussian approximation results for high-dimensional second-order $U$- and $V$-statistics with size-dependent kernels.
\cref{sec:main_maximal} establishes maximal inequalities for high-dimensional non-negative adapted sequences and for general-order $U$-statistics under i.n.i.d.~sampling.
\cref{sec:application-weighted} applies these results to instrumental variables regression and jackknife Anderson--Rubin tests with many weak instruments, including a bootstrap procedure and Monte Carlo experiments for simultaneous AR testing.
\cref{subsec:two-sample-U} discusses applications to adaptive tests of distributional homogeneity and inference for separately exchangeable arrays.
\cref{sec:conclusion} concludes and discusses directions for future research.
Supplementary Materials~I collect auxiliary results and proofs of the Gaussian approximation and maximal inequality results, while Supplementary Materials~II contain the proofs for the applications.

\paragraph{Notations:} 
For a positive integer $m$, we write $[m]:=\{1,\dots,m\}$. We also set $[0]:=\emptyset$ by convention. 
Given a vector $x\in\mathbb R^p$, its $j$-th component is denoted by $x_j$. 
Also, we set $|x|:=\sqrt{\sum_{j=1}^px_j^2}$ and $\|x\|_\infty:=\max_{j\in[p]}|x_j|$. 
Given a $p\times q$ matrix $A$, its $(j,k)$-th entry is denoted by $A_{jk}$. Also, we set $\|A\|_\infty:=\max_{j\in[p],k\in[q]}|A_{jk}|$. 
$\mcl R_p$ denotes the set of all rectangles in $\mathbb R^p$. 
For a normed space $\mathfrak X$, its norm is denoted by $\|\cdot\|_{\mathfrak X}$. 
We denote the Frobenius norm and operator norm as $\|\cdot\|_F$ and $\|\cdot\|_{\text{op}}$, respectively.
We interpret $\max\emptyset$ as 0 unless otherwise stated. 
For two random variables $\xi$ and $\eta$, we write $\xi\lesssim\eta$ or $\eta\gtrsim\xi$ if there exists a \emph{universal} constant $C>0$ such that $\xi\leq C\eta$. 
Given parameters $\theta_1,\dots,\theta_m$, we use $C_{\theta_1,\dots,\theta_m}$ to denote positive constants, which depend only on $\theta_1,\dots,\theta_m$ and may be different in different expressions.

\section{Setup}  \label{sec:setup-inid}
\subsection{Notation and Setting}\label{sec:notation-inid}
Given a probability space $(\Omega,\mcl A,\pr)$, let $X_1,\dots,X_n$ be independent random variables taking values in respective measurable spaces $(S_1,\mcl S_1), \dots, (S_n,\mcl S_n)$. 
Write $P_i$ for the marginal distribution of $X_i$.
For $r\geq1$ and each $(i_1, \dots, i_r)$, we say that a function $\psi_{(i_1, \dots, i_r)}:\prod_{s=1}^r S_{i_s}\to\mathbb R$ is \emph{symmetric} if $\psi_{(i_1, \dots, i_r)}(x_{1},\dots,x_{r})=\psi_{(i_{\sigma(1)}, \dots, i_{\sigma(r)})}(x_{\sigma(1)},\dots,x_{\sigma(r)})$ for all $(x_1,\dots,x_r)\in \prod_{s=1}^r S_{i_s}$ and $\sigma\in\symg_r$, where $\symg_r$ is the symmetric group of degree $r$. 
To lighten notation we write $\psi(X_{i_1}, \dots, X_{i_r})$ or $\psi_{(i_1,\dots, i_r)}$ for $\psi_{(i_1,\dots, i_r)}(X_{i_1}, \dots, X_{i_r})$ whenever no confusion can arise.
Let $\boldsymbol\psi \coloneqq \{\psi_{(i_1,\dots, i_r)}\}_{_{(i_1,\dots, i_r)\in I_{n,r}}}$ be  a family  of  $\otimes_{s=1}^r\mathcal{S}_{i_s}$-measurable symmetric functions $\psi_{(i_1, \dots, i_r)}:\prod_{s=1}^r S_{i_s}\to\mathbb R$, we define $r$-th order $U$-statistics as follows
\begin{align*}
    J_{r}(\boldsymbol\psi) \coloneqq \sum_{1\le i_1 < \cdots < i_r \le n} \psi(X_{i_1}, \dots, X_{i_r}) = \frac{1}{r!} \sum_{(i_1,\dots, i_r) \in I_{n,r}} \psi(X_{i_1}, \dots, X_{i_r}),
\end{align*}
where $I_{n,r}:=\{(i_1,\dots,i_r):1\leq i_1,\dots,i_r\leq n,~i_s\neq i_t\text{ for all }s\neq t\}$.
We use $P_i$ both for the marginal distribution of $X_i$ and for the corresponding integration operator. We write $\delta_x$ for the Dirac probability measure at $x$. 
That is, for $1\le s\le r-1$, we define a function $ \prod_{l=1}^{r-s} P_{k_l} \psi : \prod_{t=1}^{s} S_{i_t} \to \mathbb{R}$ as
\begin{align*}
    \prod_{l=1}^{r-s} P_{k_l}\psi (X_{i_1}, \dots,X_{i_s}) \coloneqq \int \psi(X_{i_1}, \dots, X_{i_s}, x_{k_1}, \dots, x_{k_{r-s}}) \prod_{l=1}^{r-s} dP_{k_l}(x_{k_l}),
\end{align*}
and when $r=s$, we define $\prod_{l=1}^{0} P_{k_l}$ as the identity function.
For a kernel $\psi_{(i_1,\dots, i_r)}$ with $\prod_{l=1}^rP_{i_l}\psi_{(i_1,\dots, i_r)} = 0$, we say $\psi_{(i_1,\dots, i_r)}$ is degenerate if $P_{i_l}\psi_{(i_1,\dots, i_r)} = 0 \text{ for every $1\le l \le r$}$.
For each $(i_1, \dots, i_s)$, define $K_{n,r-s}(i_1, \dots, i_s) \coloneqq \{1 \le k_1, \dots, k_{r-s} \le n, k_l \neq k_m \text{ for all } l\neq m \text{ and } k_l \notin \{i_1,\dots, i_s\} \text{ for all } 1\le l \le r-s\}$. 
Then, for $1\le s\le r$, Hoeffding projections are defined as follows
\begin{align*}
    & \pi_{s, (i_1, \dots, i_s)}\psi(X_{i_1}, \dots, X_{i_s}) \\
    & \coloneqq \frac{1}{(r-s)!} \sum_{(k_1,\dots, k_{r-s})\in K_{n,r-s}(i_1, \dots, i_s)}\left( \prod_{l=1}^s(\delta_{X_{i_l}} - P_{i_l}) \prod_{m=1}^{r-s} P_{k_m} \right) \psi(X_{i_1}, \dots, X_{i_s}, X_{k_1}, \dots, X_{k_{r-s}}).
\end{align*}
The Hoeffding decomposition of $J_r(\boldsymbol\psi)$ is given by
\begin{align*}
    J_r(\boldsymbol\psi) - \E[J_r(\boldsymbol\psi)] = \sum_{s=1}^r \frac{1}{s!} \sum_{(i_1, \dots, i_s)\in I_{n,s}}\pi_{s, (i_1, \dots, i_s)}\psi(X_{i_1}, \dots, X_{i_s}).
\end{align*}

\paragraph{Notations for Second-Order $U$-statistics:}
In this paper, we focus on Gaussian approximation in the case of $r=2$, though some auxiliary results for higher-order $U$-statistics are required (see \cref{sec:main_maximal}).
For reader's convenience, we restate, in the case of $r=2$, some definitions introduced above.
First, the second order $U$-statistics is given by
\begin{align*}
    J_{2}(\boldsymbol\psi) \coloneqq \sum_{1\le i_1 < i_2 \le n} \psi(X_{i_1}, X_{i_2}) = \frac{1}{2} \sum_{(i_1,i_2) \in I_{n,2}} \psi(X_{i_1},X_{i_2}).
\end{align*}
Then, the first-order and second-order Hoeffding projections are defined as follows

\begin{align} \label{eq:Hoeffding-projections-2nd-inid}
\begin{split}
    & \pi_{1,i}\psi(X_i) = \sum_{m\neq i} (\delta_{X_i} - P_i)P_m \psi(X_i,X_m), \\
    & \pi_{2,im}\psi(X_i,X_m) =  (\delta_{X_i} - P_i)(\delta_{X_m} - P_m) \psi(X_i,X_m).
\end{split}
\end{align}
Finally, Hoeffding decomposition of $J_2(\boldsymbol\psi)$ is given by
\begin{align*}
    J_2(\boldsymbol\psi) - \E[J_2(\boldsymbol\psi)] = \sum_{i=1}^n \pi_{1,i}\psi(X_i) + \frac{1}{2}\sum_{(i,m)\in I_{n,2}}\pi_{2,im}\psi(X_i,X_m).
\end{align*}

\subsection{Weighted and Two-sample $U$-statistics} \label{subsec:motivation-inid}

As discussed in \cref{sec:intro_inid}, our framework accommodates weighted and two-sample  $U$-statistics by allowing observation-index-dependent kernels and heterogeneous marginal distributions. 
This subsection makes these representations explicit.

\subsubsection{Weighted $U$-statistics} \label{subsec:weighted-U-inid}

In the following remark, we confirm that $U$-statistics constructed from i.n.i.d.~observations include weighted $U$-statistics as a special case.
See also \cite[pp. 7]{DoPe17} for this observation. 

\begin{rmk}[Weighted $U$-statistics as $U$-statistics under i.n.i.d.~sampling] \label{rmk:weighted_as_inid}
Suppose we have an i.i.d.~sample $(X_i)_{i=1}^n$, and let $(w_{(i,m)})_{(i,m)\in I_{n,2}}$ be deterministic weights satisfying $w_{(i,m)}=w_{(m,i)}$.
Let $\varphi:S\times S\to\mathbb R$ be a measurable symmetric kernel,
and consider the second-order weighted $U$-statistic
$\sum_{1\le i<m\le n}w_{(i,m)}\varphi(X_i,X_m)$.
Define the index-dependent kernel family
$\boldsymbol\psi\coloneqq\{\psi_{(i,m)}\}_{(i,m)\in I_{n,2}}$ by
\[
    \psi_{(i,m)}(x,y)\coloneqq w_{(i,m)}\varphi(x,y),
    \qquad (x,y)\in S\times S.
\]
The symmetry of the weights and $\varphi$ implies
$\psi_{(i,m)}(x,y)=\psi_{(m,i)}(y,x)$.
Consequently,
\[
    \sum_{1\le i<m\le n}w_{(i,m)}\varphi(X_i,X_m)
    =
    \frac12\sum_{(i,m)\in I_{n,2}}\psi_{(i,m)}(X_i,X_m)
    =
    J_2(\boldsymbol\psi)
\]
in the notation of \cref{sec:notation-inid}.
Thus, weighted $U$-statistics are included in our framework by
incorporating the weights into observation-index-dependent kernels.
\end{rmk}

\begin{rmk}[Random weights]
Random weights can also be accommodated through conditioning.
Suppose that $(w_{(i,m)})_{(i,m)\in I_{n,2}}$ are random weights
satisfying $w_{(i,m)}=w_{(m,i)}$ almost surely.
Let $\mathcal G_n$ be a $\sigma$-field with respect to which all
weights are measurable, and assume that $X_1,\ldots,X_n$ are
conditionally independent given $\mathcal G_n$.
Conditional on $\mathcal G_n$, the weights are fixed, and our
Gaussian approximation results apply to
\[
    J_2(\boldsymbol\psi)
    - \mathbb E[J_2(\boldsymbol\psi)\mid\mathcal G_n],
\]
provided that their assumptions hold for the conditional
distributions of the observations.
The validity of the conditional Gaussian approximation can then
be established by showing that the resulting error bound
converges to zero in probability.
\end{rmk}

\subsubsection{Two-sample $U$-statistics}

The following remark shows how two-sample $U$-statistics can be
represented within our framework by pooling the two samples.

\begin{rmk}[Two-sample $U$-statistics as $U$-statistics under i.n.i.d.~sampling]
Let $(X_i)_{i=1}^n\stackrel{\mathrm{i.i.d.}}{\sim}P$ and
$(Y_j)_{j=1}^m\stackrel{\mathrm{i.i.d.}}{\sim}Q$ be two independent
samples.
Set $N\coloneqq n+m$, $\mathcal I\coloneqq\{1,\ldots,n\}$, and
$\mathcal J\coloneqq\{n+1,\ldots,N\}$.
Define the pooled sample $(Z_a)_{a=1}^N$ by
\[
    Z_a\coloneqq
    \begin{cases}
        X_a, & a\in\mathcal I,\\
        Y_{a-n}, & a\in\mathcal J.
    \end{cases}
\]
Then $Z_1,\ldots,Z_N$ are independent, with marginal distributions
$P_a=P$ for $a\in\mathcal I$ and $P_a=Q$ for $a\in\mathcal J$.
Let $\varphi:S\times S\to\mathbb R$ be a measurable symmetric
kernel, and consider the second-order two-sample $U$-statistic
\[
    \sum_{i=1}^n\sum_{j=1}^m\varphi(X_i,Y_j).
\]
Define the index-dependent kernel family
$\boldsymbol\psi\coloneqq\{\psi_{(a,b)}\}_{(a,b)\in I_{N,2}}$ by
\[
    \psi_{(a,b)}(z,z')
    \coloneqq
    \varphi(z,z')
    \bigl(
        \mathbf 1\{a\in\mathcal I,\ b\in\mathcal J\}
        +
        \mathbf 1\{a\in\mathcal J,\ b\in\mathcal I\}
    \bigr),
    \qquad (z,z')\in S\times S.
\]
The symmetry of $\varphi$ implies
$\psi_{(a,b)}(z,z')=\psi_{(b,a)}(z',z)$.
Consequently,
\[
    \sum_{i=1}^n\sum_{j=1}^m\varphi(X_i,Y_j)
    =
    \frac12\sum_{(a,b)\in I_{N,2}}\psi_{(a,b)}(Z_a,Z_b)
    =
    J_2(\boldsymbol\psi)
\]
in the notation of \cref{sec:notation-inid}.
\end{rmk}

\section{Main Results} \label{sec:main-inid}
In this section, we provide main results, that is, Gaussian Approximation for high-dimensional second-order $U$-statistics and $V$-statistics with size dependent kernels.
We prove theorems in this section in \cref{sec:Proof-main-results-inid}.

\subsection{Gaussian Approximation for High-dimensional $U$-statistics} \label{subsec:main-u-inid}
Let $p$ be a positive integer.
We assume $p\ge 3$ so that $\log p>1$.
For every $j\in[p]$, we consider a family of kernels 
$\boldsymbol{\psi}_j \coloneqq \{\psi_{j,(i,m)}:S_i\times S_m\to\mathbb R\}_{(i,m)\in I_{n,2}}$,
which are $\mathcal S_i\otimes\mathcal S_m$-measurable, symmetric,
and satisfy $\psi_{j,(i,m)}\in L^4(P_i\otimes P_m)$ for all $(i,m)\in I_{n,2}$.
We suppress $(i,m)$ in $\psi_{j,(i,m)}(x_i,x_m)$  when no confusion can arise and simply write $\psi_j(x_i,x_m)$.
Also, $\sigma_j \coloneqq \sqrt{\Var[J_2(\boldsymbol\psi_j)]}>0$.
Define
\begin{align*}
    W \coloneqq (J_2(\boldsymbol\psi_1) - \E[J_2(\boldsymbol\psi_1)], \dots, J_2(\boldsymbol\psi_p) - \E[J_2(\boldsymbol\psi_p)])^\top.
\end{align*}
Our main result is an explicit error bound on the normal approximation of $\pr(W\in A)$ uniformly over $A\in\mathcal{R}_p$.
To state the result concisely, we introduce some notation.
Set
\begin{align*}
     & \Delta_{1}^{(0)} \coloneqq \max_{(j,k)\in[p]^2} \frac{1}{\sigma_j^2\sigma_k^2} \sum_{m=1}^n\sum_{l\neq m} \left\| \sum_{i\neq m,l} \pi_{2,im}\psi_j \star_i^1 \pi_{2,il}\psi_k\right\|_{L^2(P_m\otimes P_l)}^2 , \\
     & \Delta_{1}^{(1)} \coloneqq \max_{(j,k)\in[p]^2}  \frac{1}{\sigma_j^2\sigma_k^2}\sum_{m=1}^n \left\|\sum_{i:i\neq m}\pi_{1,i}\psi_j \star_i^1 \pi_{2,im}\psi_k \right\|_{L^2(P_m)}^2,
\end{align*}
where we define $\pi_{2,im}\psi_j \star^1_i \pi_{2,il}\psi_k(X_m ,X_l) \coloneqq \int \pi_{2,im}\psi_j(x_i,X_m) \pi_{2,il}\psi_k(x_i,X_l)dP_i(x_i)$ and $\pi_{1,i}\psi_j \star^1_i \pi_{2,im}\psi_k(X_m) \coloneqq\int \pi_{1,i}\psi_j(x_i) \pi_{2,im}\psi_k(x_i,X_m)dP_i(x_i)$,
and
\begin{align*}
    & \Delta_{2,*}^{(1)}(1) \coloneqq \max_{j\in[p]} \frac{1}{\sigma_j^4} \sum_{i=1}^n \|\pi_{1,i}\psi_j\|_{L^4(P_i)}^4, ~~ \Delta_{2,q}^{(2)}(1) \coloneqq n^{4/q} \left\|  \max_{j\in[p]}\max_{i\in[n]} \frac{1}{\sigma_j}|\pi_{1,i}\psi_j|\right\|_{L^q(\pr)}^4 \log (np) 
\end{align*}
and
 \begin{align*}
    & \Delta_{2,*}^{(1)}(2) \coloneqq  \max_{j\in[p]} \frac{1}{\sigma_j^4} \sum_{i=1}^n\sum_{m\neq i}\|\pi_{2,im}\psi_j\|_{L^4(P_i\otimes P_m)}^4 \log^3 (np),\\
    & \Delta_{2,*}^{(2)}(2) \coloneqq \max_{j\in[p] } \frac{1}{\sigma_j^4} \sum_{i=1}^n \left\|\sum_{m\neq i}P_m(|\pi_{2,im}\psi_j|^2)\right\|^2_{L^2(P_i)}\log^2 (np), \\
    & 
    \Delta_{2,*}^{(3)}(2) \coloneqq \E\left[  \max_{j\in[p] }\max_{i\in[n]} \frac{1}{\sigma_j^4} \sum_{m\neq i} P_m(\pi_{2,im}\psi_j)^4(X_i) \right] \log^4 (np), \\
    & \Delta_{2,q}^{(4)}(2) = n^{4/q} \left\| \max_{j\in[p]} \max_{(i,m)\in I_{n,2}} \frac{1}{\sigma_j} |\pi_{2,im}\psi_j|\right\|_{L^q(\pr)}^4 \log^5(np),\\
    & \Delta_{2,q}^{(5)}(2) = n^{4/q}\left\| \max_{j\in[p]} \max_{i\in[n]} \frac{1}{\sigma_j^2} \sum_{m\neq i} P_m(|\pi_{2,im}\psi_j|^2) \right\|_{L^{q/2}(\pr)}^2 \log^3 (np).
 \end{align*}

Then, the error bound on the normal approximation is explicitly given as follows.
The following Gaussian approximation is an extension of Theorem 2 in \cite{imai2025gaussian} and is valid regardless of whether the first- or second-order Hoeffding component is dominant or neither is as long as the approximation error bound converges to $0$ as $n \to \infty$. 
The proof is in \cref{subsec:proof:thm:main-inid}.

\begin{theorem} \label{thm:main-inid}
Assume $\max_{j\in[p]}\max_{(i,m)\in I_{n,2}}\|\psi_{j,(i,m)}\|_{L^q(P_i\otimes P_m)} < \infty$ for some $q\in[4,\infty]$.
Then, there exists a universal constant $C$ such that
\begin{align}
    \sup_{A\in\mathcal{R}_p} |\pr(W\in A) - \pr(Z\in A)| \le C \left( \sqrt{\Delta_{1}'} + \left\{ (\Delta_{2,q}(1) + \Delta_{2,q}(2))\log^5 (np)\right\}^{1/4} \right), \label{eq:main-inid}
\end{align}
where $Z\sim N(0, \Cov(W))$ and
\begin{align*}
    & \Delta_{2,q}(1) \coloneqq \Delta_{2,*}^{(1)}(1) + \Delta_{2,q}^{(2)}(1), \quad \Delta_{2,q}(2) \coloneqq \sum_{\ell=1}^3 \Delta_{2,*}^{(\ell)}(2) + \sum_{\ell=4}^5 \Delta_{2,q}^{(\ell)}(2),\\
    & \Delta_1' \coloneqq  \sqrt{\Delta_1^{(0)}}\log^3 (np) +\sqrt{\Delta_1^{(1)}} \log^{5/2} (np) +  \left(\max_{j\in[p]}  \frac{1}{\sigma_j^2} \sum_{i=1}^n \|\pi_{1,i}\psi_j\|_{L^2(P_i)}^2 \right)^{1/2} \left( \Delta_{2,q}^{(5)}(2) \log^7 (np) \right)^{1/4}.
\end{align*}
\end{theorem}

\begin{rmk}[Verification of the convergence] \label{rmk:varification-U-main}
    The bound in \cref{thm:main-inid} is stated in the most general form, and therefore does not require the statistic to be either non-degenerate or degenerate. 
    This general form is useful when the dominant component of the Hoeffding decomposition is not known in advance or heterogeneous across coordinates and it is enough to verify the convergence of all quantities appearing in the theorem.  Specifically, the uniform convergence follows if
\[
    \Delta_1'\to0, \qquad\text{and}\qquad \{\Delta_{2,q}(1)+\Delta_{2,q}(2)\}\log^5(np)\to0.
\]
As an example of this case, we demonstrate the uniform Gaussian approximation of many-weak instrumental variable regression estimator by \cite{chao2012asymptotic} under many parameter setting.
See \cref{thm:gauss-chao} in \cref{subsubsec:application-chao}. 

On the other hand, there are applications where the dominant Hoeffding component is
known, at least asymptotically.  
In a non-degenerate regime, where the
quadratic component is $0$ or asymptotically negligible relative to the
linear term, the main approximation is driven by the linear term.  
In this case, the condition to be verified is
\[
    \Delta_{2,q}(1)\log^5(np)\to0.
\]
Similarly, in the degenerate case, the main approximation is driven by the quadratic term.  
Therefore it is sufficient to verify
\[
    \Delta_1^{(0)}\log^6(np)\to0,
    \qquad\text{and}\qquad
    \Delta_{2,q}(2)\log^5(np)\to0.  \tag*{\qed}
\]
\end{rmk}

As noted by \cite{imai2025gaussian}, it is often easy to evaluate the quantities in terms of kernel functions rather than those of Hoeffding projections, in application.
The following corollary is an extension of Corollary 2 in \cite{imai2025gaussian} and is useful for the purpose.
The proof is in \cref{subsec:proof:coro:main-kernel-inid}.
\begin{corollary} \label{coro:main-kernel-inid}
    Under the assumption of \cref{thm:main-inid}, there exists a universal constant $C$ such that
    \begin{align*}
        \sup_{A\in\mathcal{R}_p} |\pr(W\in A) - \pr(Z\in A)| \le C \left( \sqrt{\tilde\Delta_{1}'} + \left\{ (\tilde\Delta_{2,q}(1) + \tilde\Delta_{2,q}(2))\log^5 (np)\right\}^{1/4} \right),
    \end{align*}
    where 
    \begin{align*}
        & \tilde\Delta_{2,q}(1) \coloneqq \tilde\Delta_{2,*}^{(1)}(1) + \tilde\Delta_{2,q}^{(2)}(1), \quad \tilde\Delta_{2,q}(2) \coloneqq \sum_{\ell=1}^3 \tilde\Delta_{2,*}^{(\ell)}(2) + \sum_{\ell=4}^5 \tilde\Delta_{2,q}^{(\ell)}(2),\\
        & \tilde\Delta_1' \coloneqq  \sqrt{\tilde\Delta_1^{(0)}}\log^3 (np) + \sqrt{\tilde\Delta_1^{(1)} }\log^{5/2} (np) \\
        & \qquad\qquad+ \left(\max_{j\in[p]} \frac{1}{\sigma_j^2}  \sum_{i=1}^n \left\| \sum_{m\neq i}  P_m\psi_{j,(i,m)} \right\|_{L^2(P_i)}^2 \right)^{1/2} \left(\tilde\Delta_{2,q}^{(5)}(2) \log^7 (np)\right)^{1/4},
    \end{align*}
and
\begin{align*}
    & \tilde\Delta_{2,*}^{(1)}(1) \coloneqq \max_{j\in[p]} \frac{1}{\sigma_j^4} \sum_{i=1}^n  \left\|\sum_{m\neq i} P_m\psi_{j,(i,m)}\right\|_{L^4(P_i)}^4, \\
    & \tilde\Delta_{2,q}^{(2)}(1)\coloneqq n^{4/q} \left\| \max_{j\in[p]}  \max_{i\in[n]} \frac{1}{\sigma_j} \left|\sum_{m\neq i}P_m\psi_{j,(i,m)}\right|\right\|_{L^q(\pr)}^4 \log (np) ,
\end{align*}
and
 \begin{align*}
    & \tilde\Delta_{2,*}^{(1)}(2) \coloneqq  \max_{j\in[p]} \frac{1}{\sigma_j^4} \sum_{i=1}^n\sum_{m\neq i} \|\psi_{j,(i,m)}\|_{L^4(P_i\otimes P_m)}^4 \log^3 (np), \\
    & \tilde\Delta_{2,*}^{(2)}(2) \coloneqq \max_{j\in[p] } \frac{1}{\sigma_j^4}\sum_{i=1}^n \left\|\sum_{m\neq i}P_m\psi_{j,(i,m)}^2\right\|^2_{L^2(P_i)}\log^2 (np), \\
    & 
    \tilde\Delta_{2,*}^{(3)}(2) \coloneqq \E\left[  \max_{j\in[p] }\max_{i\in[n]} \frac{1}{\sigma_j^4}  \sum_{m\neq i}P_m\psi_{j,(i,m)}^4(X_i) \right] \log^4 (np), \\
    & \tilde\Delta_{2,q}^{(4)}(2) = n^{4/q} \left\| \max_{j\in[p]} \max_{(i,m)\in I_{n,2}} \frac{1}{\sigma_j} |\psi_{j,(i,m)}|\right\|_{L^q(\pr)}^4 \log^5(np),\\
    & \tilde\Delta_{2,q}^{(5)}(2) = n^{4/q}\left\| \max_{j\in[p]} \max_{i\in[n]} \frac{1}{\sigma_j^2} \sum_{m\neq i}P_m\psi_{j,(i,m)}^2  \right\|_{L^{q/2}(\pr)}^2 \log^3 (np).
 \end{align*}
 and
    \begin{align*}
     & \tilde\Delta_{1}^{(0)} \coloneqq  \max_{(j,k)\in[p]^2} \frac{1}{\sigma_j^2\sigma_k^2}\sum_{m=1}^n\sum_{l\neq m} \left\| \sum_{i\neq m,l}  \psi_{j,(i,m)}\star_i^1 \psi_{k,(i,l)}\right \|_{L^2(P_m\otimes P_l)}^2 +\max_{j\in[p]} \left( \frac{1}{\sigma_j^2} \sum_{i=1}^n\sum_{m\neq i}  \|P_m\psi_{j,(i,m)}\|_{L^2(P_i)}^2\right)^2  , \\
     & \tilde\Delta_{1}^{(1)} \coloneqq\left(\max_{j\in[p]} \frac{1}{\sigma_j^2}  \sum_{i=1}^n \left\| \sum_{m\neq i}  P_m\psi_{j,(i,m)} \right\|_{L^2(P_i)}^2 \right) \left( \max_{j\in[p]} \frac{1}{\sigma_j^2} \sum_{i=1}^n \left\|\sum_{m\neq i}P_m\psi_{j,(i,m)}^2\right\|_{L^2(P_i)}\right),
\end{align*}
where we define $\psi_{j,(i,m)} \star^1_i \psi_{k,(i,l)}(X_m ,X_l) \coloneqq \int \psi_j(x_i,X_m) \psi_k(x_i,X_l)dP_i(x_i)$.
\end{corollary}

\begin{rmk}[Verification of the convergence]
\label{rmk:verification-kernel-U-main}
The quantities in \cref{coro:main-kernel-inid} are useful when
the Hoeffding projections are cumbersome to compute explicitly.  In this
case, the uniform Gaussian approximation follows if
\[
    \widetilde\Delta_1'\to0,
    \qquad\text{and}\qquad
    \{\widetilde\Delta_{2,q}(1)+\widetilde\Delta_{2,q}(2)\}\log^5(np)\to0.
\]
If the kernel is degenerate, that is, $P_m\psi_{j,(i,m)}=0$ and $P_i\psi_{j,(i,m)}=0$ for all $(i,m)\in I_{n,2}$, it is sufficient to verify
\[
    \widetilde\Delta_1^{(0)}\log^6(np)\to0,
    \qquad\text{and}\qquad
    \widetilde\Delta_{2,q}(2)\log^5(np)\to0.
\]
As an example of this case, we demonstrate the uniform Gaussian approximation of the jackknife Anderson--Rubin statistic by \cite{mikusheva2022inference}. \qed
\end{rmk}

\subsection{Gaussian Approximation for High-dimensional $V$-statistics} \label{subsec:main-v-inid}
In applications, we often need to handle not only $U$-statistics but also $V$-statistics. 
Once Theorem~\ref{thm:main-inid} is established, the Gaussian approximation for the second-order high-dimensional $V$-statistic with size-dependent kernels follows immediately, because the only difference from the $U$-statistic case is the additional diagonal contribution to the first-order projection.
Nevertheless, to facilitate applications, we state it explicitly as a corollary.

Before stating the corollary, we introduce notations used below.
First, we define a family of kernels 
$\boldsymbol{\psi}_j \coloneqq \{\psi_{j,(i,m)}:S_i\times S_m\to\mathbb R\}_{(i,m)\in [n]^2}$ in this subsection.
We write the second order $ V$-statistics as $J_{2}^V(\boldsymbol\psi)$, that is,
\begin{align*}
    J_{2}^V(\boldsymbol\psi) \coloneqq \frac{1}{2}\sum_{i=1}^n\sum_{m=1}^n \psi(X_{i}, X_{m}),
\end{align*}
and define
\begin{align*}
    W^V \coloneqq (J_2^V(\boldsymbol\psi_1) - \E[J_2^V(\boldsymbol\psi_1)], \dots, J_2^V(\boldsymbol\psi_p) - \E[J_2^V(\boldsymbol\psi_p)])^\top.
\end{align*}
For $s=\{1,2\}$, Hoeffding projections are defined as follows
\begin{align} \label{eq:Hoeffding-projections-2nd-V-inid}
\begin{split}
    & \pi_{1,i}^V\psi(X_i) =  \frac{1}{2}(\delta_{X_i} - P_i)\psi(X_i,X_i) + \sum_{m\neq i} (\delta_{X_i} - P_i)P_m \psi(X_i,X_m), \\
    & \pi_{2,im}\psi(X_i,X_m) =  (\delta_{X_i} - P_i)(\delta_{X_m} - P_m) \psi(X_i,X_m).
\end{split}
\end{align}
Then, Hoeffding decomposition of $J_2^V(\boldsymbol\psi)$ is given by
\begin{align*}
    J_2^V(\boldsymbol\psi) - \E[J_2^V(\boldsymbol\psi)] = \sum_{i=1}^n \pi_{1,i}^V\psi(X_i) + \frac{1}{2}\sum_{(i,m)\in I_{n,2}}\pi_{2,im}\psi(X_i,X_m).
\end{align*}
We define $\sigma_j \coloneqq \sqrt{\Var[J_2^V(\boldsymbol{\psi}_j)]}$ in this subsection.

Then, the error bound on the normal approximation is explicitly given as follows.
As in the case of $U$-statistics, the following Gaussian approximations are valid regardless of whether the first- or second-order Hoeffding component is dominant or neither is, as long as the approximation error bound converges to $0$ as $n \to \infty$.

\begin{corollary} \label{coro:main-V-inid}
Assume $\max_{j\in[p]}\max_{i\in[n]}\|\psi_{j,(i,i)}\|_{L^q(P_i)} < \infty $ and $\max_{j\in[p]}\max_{(i,m)\in I_{n,2}}\|\psi_{j,(i,m)}\|_{L^q(P_i\otimes P_m)} < \infty$ for some $q\in[4,\infty]$.
There exists a universal constant $C$ such that
\begin{align}
    \sup_{A\in\mathcal{R}_p} |\pr(W^V\in A) - \pr(Z\in A)| \le C \left( \sqrt{(\Delta_{1}')^V} + \left\{ (\Delta_{2,q}^V(1) + \Delta_{2,q}(2))\log^5 (np)\right\}^{1/4} \right), \label{eq:main-V-inid}
\end{align}
where $Z \sim N(0, \Cov(W^V))$ and
\begin{align*}
    & \Delta_{2,q}^V(1) \coloneqq (\Delta_{2,*}^{(1)})^V(1) + (\Delta_{2,q}^{(2)})^V(1), \\
    & (\Delta_1')^V \coloneqq  \sqrt{\Delta_1^{(0)}}\log^3 (np) + \sqrt{(\Delta_1^{(1)})^V}\log^{5/2} (np) \\
    & \qquad\qquad\qquad+  \left(\max_{j\in[p]}  \frac{1}{\sigma_j^2} \sum_{i=1}^n \|\pi_{1,i}^V\psi_j\|_{L^2(P_i)}^2 \right)^{1/2} \left(\Delta_{2,q}^{(5)}(2) \log^7 (np)\right)^{1/4},
\end{align*}
with
\begin{align*}
    (\Delta_{1}^{(1)})^V \coloneqq  \max_{(j,k)\in[p]^2}  \frac{1}{\sigma_j^2\sigma_k^2}\sum_{m=1}^n \left\|\sum_{i:i\neq m}\pi_{1,i}^V\psi_j \star_i^1 \pi_{2,im}\psi_k \right\|_{L^2(P_m)}^2,
\end{align*}
and
\begin{align*}
    & (\Delta_{2,*}^{(1)})^V(1) \coloneqq \max_{j\in[p]} \frac{1}{\sigma_j^4} \sum_{i=1}^n \|\pi_{1,i}^V\psi_j\|_{L^4(P_i)}^4 ,\quad (\Delta_{2,q}^{(2)})^V(1) \coloneqq n^{4/q} \left\|  \max_{j\in[p]}\max_{i\in[n]} \frac{1}{\sigma_j}|\pi_{1,i}^V\psi_j|\right\|_{L^q(\pr)}^4 \log (np) . 
\end{align*}
\begin{proof}
    Replacing $\pi_{1}\psi_j$ in quantities in \cref{thm:main-inid} with $\pi_{1}^V\psi_j$ and straightforward evaluation complete the proof. 
\end{proof}
\end{corollary}

To facilitate applications, we formulate an analogue of our result expressed through quantities in terms of kernel functions, for the $V$-statistic case as well. 
The proof is in \cref{subsec:proof:coro:main-kernel-V-inid}.

\begin{corollary} \label{coro:main-kernel-V-inid}
Under the assumption of \cref{coro:main-V-inid}, there exists a universal constant $C$ such that
\begin{align*}
    \sup_{A\in\mathcal{R}_p} |\pr(W^V\in A) - \pr(Z\in A)| \le C \left( \sqrt{(\tilde\Delta_{1}')^V} + \left\{ (\tilde\Delta_{2,q}^V(1) + \tilde\Delta_{2,q}(2))\log^5 (np)\right\}^{1/4} \right), 
\end{align*}
where 
\begin{align*}
    & \tilde\Delta_{2,q}^V(1) \coloneqq (\tilde\Delta_{2,*}^{(1)})^V(1) + (\tilde\Delta_{2,q}^{(2)})^V(1), \\
    & (\tilde\Delta_1')^V \coloneqq  \sqrt{\tilde\Delta_1^{(0)}}\log^3 (np) + \sqrt{(\tilde\Delta_1^{(1)})^V}\log^{5/2} (np) \\
    & \quad +   \left(\max_{j\in[p]}  \frac{1}{\sigma_j^2} \sum_{i=1}^n \left[ \|\psi_{j,(i,i)}\|_{L^2(P_i)}^2 +  \left\| \sum_{m\neq i}  P_m\psi_{j,(i,m)} \right\|_{L^2(P_i)}^2 \right]\right)^{1/2}\left(\tilde\Delta_{2,q}^{(5)}(2) \log^7 (np)\right)^{1/4},
\end{align*}
with
\begin{align*}
    (\tilde\Delta_1^{(1)} )^V
         &\coloneqq \left(\max_{j\in[p]}   \frac{1}{\sigma_j^2}  \sum_{i=1}^n \|\psi_{j,(i,i)}\|_{L^2(P_i)}^2 \right) \left( \max_{j\in[p]} \frac{1}{\sigma_j^2}  \sum_{i=1}^n\left\|\sum_{m\neq i} P_m \psi_{j,(i,m)}^2\right\|_{L^2(P_i)} \right) \\
         & \quad +  \left( \max_{j\in[p]} \frac{1}{\sigma_j^2}  \sum_{i=1}^n \left\| \sum_{m\neq i}  P_m\psi_{j,(i,m)} \right\|_{L^2(P_i)}^2 \right) \left( \max_{j\in[p]}  \frac{1}{\sigma_j^2} \sum_{i=1}^n \left\|\sum_{m\neq i} P_m \psi_{j,(i,m)}^2 \right\|_{L^2(P_i)}\right),
\end{align*}
and
\begin{align*}
    & (\tilde\Delta_{2,*}^{(1)})^V(1) \coloneqq \max_{j\in[p]} \frac{1}{\sigma_j^4} \sum_{i=1}^n \left\|\sum_{m\neq i}P_m\psi_{j,(i,m)}\right\|_{L^4(P_i)}^4+ \max_{j\in[p]} \frac{1}{\sigma_j^4} \sum_{i=1}^n \|\psi_{j,(i,i)}\|_{L^4(P_i)}^4, \\
    &(\tilde\Delta_{2,q}^{(2)})^V(1) \coloneqq  n^{4/q} \left(\left\| \max_{j\in[p]}  \max_{i\in[n]} \frac{1}{\sigma_j} \left|\sum_{m\neq i}P_m\psi_{j,(i,m)}\right|\right\|_{L^q(\pr)}^4 +   \left\| \max_{j\in[p]}  \max_{i\in[n]} \frac{1}{\sigma_j}|\psi_{j,(i,i)}|\right\|_{L^q(\pr)}^4 \right)\log (np). 
\end{align*}
\end{corollary}

\section{Maximal Inequalities} \label{sec:main_maximal}
In this section, we provide maximal inequalities for the high-dimensional non-negative adapted sequence (\cref{lem:nonada-inid}) and for high-dimensional general order $U$-statistics based on i.n.i.d.~observations (\cref{thm:max-is-inid} and \cref{lem:max-is-inid}).

The first result is the maximal inequality for the high-dimensional non-negative adapted sequence.
This result is the extension of Lemma 2 in \cite{imai2025gaussian}. 
\cite{imai2025gaussian} have treated $L_1$ norm of the maximum of the high-dimensional non-negative adapted sequence, whereas we do the general $L_q$ norm. 
We use this inequality for the proof of \cref{lem:max-is-inid} but this inequality is useful for other problems involving the  high-dimensional non-negative adapted sequence.
\begin{lemma} \label{lem:nonada-inid}
Let $(\eta_i)_{i=1}^N$ be a sequence of random vectors in $\mathbb R^p$ adapted to a filtration $(\mathcal{F}_i)_{i=1}^N$.
Suppose that $\eta_{ij} \ge 0$ and $\eta_{ij} \in L^1(\mathbb P)$ for all $i\in [N]$ and $j\in[p]$.
Then, for any $q\ge 1$, there exists a universal constant $C$ such that
\begin{align*}
    \left\| \max_{j\in[p]} \sum_{i=1}^N \eta_{ij}\right\|_{L^q(\mathbb{P})} \le C\left( \left\| \max_{j\in[p]} \sum_{i=1}^N\mathbb{E}[ \eta_{i,j} \mid \mathcal{F}_{i-1}]\right\|_{L^q(\mathbb{P})} + (q+\log p) \left\|\max_{i\in[N]} \max_{j\in[p]} \eta_{ij}\right\|_{L^q(\mathbb{P})} \right),
\end{align*}
where we set $\mathcal{F}_0 \coloneqq \{\emptyset, \Omega\}$.
\end{lemma}

The following two theorems are maximal inequalities for high-dimensional general order $U$-statistics based on i.n.i.d.~observations.
These maximal inequalities are the extensions of Theorem 7 and 8 in \cite{imai2025gaussian} and are used in the proofs of the Gaussian approximation results in the previous section.
They are also useful for developing theoretical analysis in applications that involve high-dimensional $U$-statistics. 
As such, they may be of independent interest.
The proofs are in \cref{subsec:proof:maximal_inid}.

\begin{theorem} \label{thm:max-is-inid}
Let $q\ge 1$ and $\boldsymbol{\psi}_j \coloneqq  \{\psi_{j,(i_1,\dots, i_r)}\}_{{(i_1,\dots, i_r)\in I_{n,r}}}$. 
Assume $\psi_{j,(i_1,\dots, i_r)}\in L^{q\vee 2}(\otimes_{s=1}^r P_{i_s})$ be degenerate, symmetric kernels of order $r\ge 1$ for all $(i_1, \dots, i_r)\in I_{n,r}$. 
Then there exists a constant $C_r$ depending only on $r$ such that
\begin{align*}
    & \left\|\max_{j\in[p]} |J_r(\boldsymbol\psi_j)|\right\|_{L^q(\pr)}
    \le C_r \max_{0 \le s \le r} (q+\log (np))^{\frac{r+s}{2}} \\
    & \quad\times \left\| \max_{j\in[p]} \max_{(i_1,\dots,i_s)\in I_{n,s}} \sum_{(k_1,\dots k_{r-s})\in K_{n,r-s}(i_1, \dots, i_s)}\int \psi_j^2(X_{i_1}, \dots, X_{i_s}, x_{k_1}, \dots x_{k_{r-s}} )\prod_{l=1}^{r-s} P_{k_l}(dx_{k_l}) \right\|_{L^{1\vee q/2}(\pr)}^{1/2}.
\end{align*}
\end{theorem}

\begin{theorem} \label{lem:max-is-inid} 
Let $q\ge 1$ and $\boldsymbol{\psi}_j \coloneqq  \{\psi_{j,(i_1,\dots, i_r)}\}_{{(i_1,\dots, i_r)\in I_{n,r}}}$. 
Assume $\psi_{j,(i_1,\dots, i_r)}\in L^{q}(\otimes_{s=1}^r P_{i_s})$ be non-negative, symmetric kernels of order $r\ge 1$ for all $(i_1, \dots, i_r)\in I_{n,r}$. 
Then there exists a constant $c_r$ depending only on $r$ such that
\begin{align*}
& \left\|\max_{j\in[p]} J_r(\boldsymbol\psi_j)\right\|_{L^q(\pr)}  \le c_r \max_{0 \le s \le r}  (q+\log (np))^{s} \\
    & \quad\times\left\| \max_{j\in[p]} \max_{(i_1,\dots,i_s)\in I_{n,s}} \sum_{(k_1,\dots k_{r-s})\in K_{n,r-s}(i_1, \dots, i_s)}\int \psi_j(X_{i_1}, \dots, X_{i_s}, x_{k_1}, \dots x_{k_{r-s}}) \prod_{l=1}^{r-s} P_{k_l}(dx_{k_l}) \right\|_{L^{q}(\pr)}. 
\end{align*}
\end{theorem}

\begin{rmk}[Comparison with \cite{imai2025gaussian}]

    The maximal inequalities above differ from i.i.d. counterparts in \cite{imai2025gaussian} in several respects. 
   First, our results apply to $U$-statistics based on i.n.i.d. observations with index-dependent kernels. This extension is important to facilitate applications because such class of $U$-statistics include weighted $U$-statistics and two-sample $U$-statistics as special cases.
    Second and most importantly, the right-hand sides in \cref{thm:max-is-inid} and \cref{lem:max-is-inid} partially retain the summations over sample indices.
    In principle, one could obtain simpler bounds by replacing such sums with maxima over sample indices and multiplying by suitable powers of $n$. 
    Such simplification would enable us to build on the proof strategy of \cite{imai2025gaussian}. 
    However, such a replacement can give rough evaluation in applications.
    For instance, in weighted statistics arising from many-weak IV and many-covariates asymptotics, the sharp stochastic order is often governed by certain properties of sum of elements in projection matrices rather than by the largest single element. 
    Finally, as a technical issue, the logarithmic factors of $\log (np)$ in our maximal inequalities is larger than those of $\log p$ in the corresponding inequalities of \cite{imai2025gaussian}.
    This deterioration appears to come from the proof strategy and we do not claim that these logarithmic factors are optimal. 
    Improving the logarithmic factors while keeping the summations over sample indices would be an interesting technical refinement.
\end{rmk}

\section{Application of Weighted $U$-statistics} \label{sec:application-weighted}
In this section, we illustrate how our Gaussian approximations and maximal inequalities can be used to conduct uniform inference with many weak instruments.
We consider the instrumental variables regression framework of \citet{chao2012asymptotic} and the jackknife Anderson--Rubin test of \citet{mikusheva2022inference}.
These applications provide templates for extending inference to other settings involving weighted $U$- and $V$-statistics, including many-covariate regression and variance component analysis.

\subsection{Regression with Many-Weak Instrumental Variables}  \label{subsubsec:application-chao} \label{exm:weak_IV} 

Many and potentially weak instruments are common in empirical IV applications. Weak instruments remain a concern in applied research \citep{andrews2019weak,lee2022valid}, while models with many instruments have been studied extensively since \citet{kunitomo1980asymptotic,morimune1983approximate,bekker1994alternative}; see also \citet{anatolyev2019many} for a review. \citet{chao2012asymptotic} develop an asymptotic framework for jackknife IV estimators that accommodates a wide range of identification strengths. Their analysis exploits weighted $U$-statistic representations and establishes asymptotic normality across regimes in which the linear and quadratic Hoeffding components have different relative magnitudes.

An important application involving both many weak instruments and many target parameters is Mendelian randomization, where genetic variants serve as instruments and researchers investigate causal effects across numerous outcomes \citep{davies2015many,sanderson2022mendelian}. For example, \citet{millard2019searching} conduct $22,922$ outcome-level tests to investigate the phenotypic consequences of higher body mass index.

We extend the framework of \citet{chao2012asymptotic} to high-dimensional simultaneous inference using our Gaussian approximation (\cref{thm:main-inid}) and maximal inequality (\cref{thm:max-is-inid}). Since identification strength and the number of instruments may vary across coordinates, linear-dominant, balanced, and quadratic-dominant regimes may coexist within the same inference problem. Our results accommodate such heterogeneity and general cross-coordinate dependence without requiring prior classification of the asymptotic regimes or imposing structural restrictions on the dependence across coordinates. More broadly, our Gaussian approximation provides a high-dimensional counterpart to the degenerate quadratic-form CLT of \citet{chao2012asymptotic}, which has also been used in other many-weak IV problems.

\subsubsection{Overview}

For each $j\in[p]$, consider the following linear IV model:
\begin{align*}
    Y_{ij} &= X_{ij}^\top\theta_j+u_{ij},\\
    X_{ij} &= \gamma_{n,j}^\top Z_{ij}+\varepsilon_{ij}, \quad \mathbb{E}[u_{ij}\mid Z]=0, \quad \mathbb{E}[\varepsilon_{ij}\mid Z]=0,
\end{align*}
where $Y_{ij}$ is a scalar outcome, $X_{ij}$ is a $d_j$-dimensional vector of potentially endogenous regressors, $Z_{ij}$ is a $K_{n,j}$-dimensional vector of instruments, and $Z\coloneqq(Z_{ij})_{i\in[n],j\in[p]}$.
The structural parameter $\theta_j\in\mathbb{R}^{d_j}$ and the first-stage coefficient matrix $\gamma_{n,j}\in\mathbb{R}^{K_{n,j}\times d_j}$ are nonrandom.
We allow $p$ to diverge with $n$, while treating each $d_j$ as fixed, and allow $K_{n,j}$ to be fixed or diverge.
Following \citet{chao2012asymptotic}, we assume
\begin{align}
    \gamma_{n,j}\coloneqq n^{-1/2}\mu_{n,j}\pi_j, \label{eq:local_to_zero}
\end{align}
where $\mu_{n,j}$ is a scalar sequence and $\pi_j$ is a nonrandom $K_{n,j}\times d_j$ matrix.

We consider the JIVE2 estimator proposed by \citet{angrist1999jackknife}:
\begin{align*}
    \hat{\theta}_{n,j}=\left(\sum_{i=1}^n\sum_{m\neq i}\Pi_{im,j}X_{ij}X_{mj}^\top\right)^{-1}\sum_{i=1}^n\sum_{m\neq i}\Pi_{im,j}X_{ij}Y_{mj},
\end{align*}
where $\Pi_{im,j}$ denotes the $(i,m)$-th entry of the projection matrix $\Pi_j\coloneqq Z_j(Z_j^\top Z_j)^{-1}Z_j^\top$, with $Z_j\coloneqq(Z_{1j},\dots,Z_{nj})^\top$.
Substituting the structural equation yields
\begin{align*}
    \hat{\theta}_{n,j}-\theta_j=\hat{\Gamma}_{n,j}^{-1}S_{n,j},
\end{align*}
where
\begin{align*}
    \hat{\Gamma}_{n,j}&\coloneqq\mu_{n,j}^{-2}\sum_{i=1}^n\sum_{m\neq i}\Pi_{im,j}X_{ij}X_{mj}^\top, \quad 
    S_{n,j}\coloneqq\mu_{n,j}^{-2}\sum_{i=1}^n\sum_{m\neq i}\Pi_{im,j}X_{ij}u_{mj}.
\end{align*}
By \cref{lemma:chao-Gamma-convergence,lemma:chao-hoeffding} in the next subsection, the denominator $\hat{\Gamma}_{n,j}$ can be uniformly approximated by $\Gamma_{n,j}$, while $S_{n,j}$ admits a Hoeffding-type decomposition. Consequently,
\begin{align}
    \hat{\theta}_{n,j}-\theta_j
    &\approx\Gamma_{n,j}^{-1}\left(\frac{1}{\mu_{n,j}\sqrt n}\sum_{i=1}^n(1-\Pi_{ii,j})(\pi_j^\top Z_{ij})u_{ij}+\frac{1}{2\mu_{n,j}^2}\sum_{i=1}^n\sum_{m\neq i}\Pi_{im,j}(\varepsilon_{ij}u_{mj}+\varepsilon_{mj}u_{ij})\right), \label{eq:chao-Hoeffding}
\end{align}
where $\Gamma_{n,j}\coloneqq n^{-1}\sum_{i=1}^n(1-\Pi_{ii,j})(\pi_j^\top Z_{ij})(\pi_j^\top Z_{ij})^\top$.

The linear and quadratic components in \cref{eq:chao-Hoeffding} have stochastic orders $1/\mu_{n,j}$ and $\sqrt{K_{n,j}}/\mu_{n,j}^2$, respectively. Their relative magnitudes lead to three asymptotic regimes:
\begin{align*}
    &\text{Case (i): } K_{n,j}/\mu_{n,j}^2\to0 &&\text{(linear-dominant)},\\
    &\text{Case (ii): } K_{n,j}/\mu_{n,j}^2\to\kappa\in(0,\infty) && \text{(balanced)},\\
    &\text{Case (iii): } K_{n,j}/\mu_{n,j}^2\to\infty && \text{(quadratic-dominant)}.
\end{align*}
Thus, the dominant Hoeffding component may differ across coordinates, and no single asymptotic regime need characterize the entire vector of estimators.
The following results establish a simultaneous Gaussian approximation for $\hat{\theta}_n\coloneqq(\hat{\theta}_{n,j})_{j\in[p]}$ without requiring the same regime across coordinates.

We next establish a simultaneous Gaussian approximation for the vector $\hat{\theta}_n\coloneqq(\hat{\theta}_{n,j})_{j\in[p]}$, allowing different asymptotic regimes across coordinates.
The proofs of the following results are provided in \cref{subsec:proof-chao}.
\subsubsection{Theoretical Results}

\begin{assumption}\mbox{} \label{as:many-weak-iv}
In what follows, we define $\bar d \coloneqq \max_{j\in[p]} d_j$ and $ D_n \coloneqq (np\bar d^2)$ and assume
\begin{itemize} 
    \item[(i)] Conditional on $Z$, the random vectors $\bigl((Y_{ij},X_{ij}^\top,Z_{ij}^\top)_{j\in[p]}\bigr)_{i=1}^n$ are independent across $i$.
    \item[(ii)] For all $j\in[p]$, $K_{n,j}$ is fixed or $K_{n,j} \to \infty$, $\text{rank}(Z_j) = K_{n,j}$ and, for some $C<1$ and all $i\in[n]$, $\Pi_{ii,j} < C$ almost surely. 
    \item[(iii)]  For all $j\in[p]$, $\gamma_{n,j} = n^{-1/2}\mu_{n,j}\pi_j$,  and there exist constants $c$ and $C$ such that
     \begin{align*}
        c\le \lambda_{\min}\left(\pi_j^\top \left( n^{-1}\sum_{i=1}^n Z_{ij}Z_{ij}^\top \right)\pi_j\right) \le \lambda_{\max}\left(\pi_j^\top \left( n^{-1}\sum_{i=1}^n Z_{ij}Z_{ij}^\top \right)\pi_j\right) \le C,
    \end{align*}
    where $\lambda_{\min}$ and $\lambda_{\max}$ indicate minimum and maximum eigenvalues. 

    \item[(iv)] Either $\mu_{n,j} = \sqrt{n}$ or $\mu_{n,j}/\sqrt{n} \to 0$ and 
    \begin{align*}
        \frac{1}{\mu_{n,j}}\log^{5/2}D_n \to 0, \quad 
        \frac{\sqrt{K_{n,j}}}{\mu_{n,j}^2}\log^4D_n \to 0, \text{ uniformly in } j\in[p].
    \end{align*}
    
    \item[(v)] For all $j\in[p]$, $\mathbb{E}[u_{ij} \mid Z] = 0$ and $\mathbb{E}[\varepsilon_{ij} \mid Z ] = 0$. Also,  for all $j\in[p]$, $k\in[d_j]$ and $i\in[n]$, and  $\lambda_{\min}\{\Var[(u_{ij}, \varepsilon_{ij})^\top \mid Z]\} \ge c$ almost surely. 
    
    \item[(vi)] 
     For all $j\in[p]$, $k\in[d_j]$ and $i\in[n]$, there exists a constant $C$ such that,$|u_{ij}| \le C$, $|\varepsilon_{ij,k}| \le C$ almost surely. Also, it holds that
    \begin{align*}
        \frac{1}{n^2} \sum_{i=1}^n |\pi_j^\top Z_{ij}|^4 \log^6D_n \to 0 , \text{ uniformly in } j\in[p].
    \end{align*} 
\end{itemize}
\end{assumption}

\cref{as:many-weak-iv} is a high-dimensional version of the regularity conditions imposed in \cite{chao2012asymptotic}. The main difference is that these restrictions are imposed uniformly over the coordinates $j\in[p]$.
Assumption (ii) is the standard assumption of many-instrument and almost same as Assumption 1 in \cite{chao2012asymptotic}. 
Assumptions (iii) and (iv) correspond to the two main parts of Assumption 2 in \cite{chao2012asymptotic}. 
Assumption (iii) imposes the standard local-to-zero assumption and eigenvalue condition.
Assumption (iv) corresponds to their restrictions on the strength sequence. 
Assumption (v) imposes conditional mean-zero restrictions as in Assumption 3 in \cite{chao2012asymptotic}. 
Also the covariance lower bound provides a primitive sufficient condition for the nondegeneracy of the normalized quadratic-component variance, corresponding to the asymptotic variance restriction in Theorem 3 of \cite{chao2012asymptotic}.
Assumption (vi) is the counterpart of Assumption 5 in \cite{chao2012asymptotic}. 

Under this assumption, we demonstrate the uniform Gaussian approximation for the high-dimensional extension of \cite{chao2012asymptotic}. 
The proofs of all results are in \cref{subsec:proof-chao}.
The proofs also serve as a useful template for applying our general maximal inequality (\cref{thm:max-is-inid}) and Gaussian approximation (\cref{thm:main-inid}) to related many-weak IV problems that are not literal extensions of \cite{chao2012asymptotic} but share the same conditional weighted $U$-statistic structure.

The following two lemmas provide the uniform over $j\in[p]$ convergence of the 'denominator' part $\hat{\Gamma}_{n,j}$ to $\Gamma_{n,j}$ and Hoeffding decomposition of 'numerator' part $S_{n,j}$, respectively. 
These two lemmas enables us to confirm that the influential function of the centered estimator is defined as \cref{eq:chao-Hoeffding}.
As mentioned above, the proof of \cref{lemma:chao-Gamma-convergence} relies on our maximal inequality (\cref{thm:max-is-inid}).

\begin{lemma} \label{lemma:chao-Gamma-convergence}
Under \cref{as:many-weak-iv}, it holds that
\begin{align*}
    \mathbb{E}\left[ \max_{j\in[p]} \max_{(k_1,k_2)\in[d_j]^2}\left| e_{k_1}^\top\{\hat{\Gamma}_{n,j} - \Gamma_{n,j}\}e_{k_2} \right| \mid Z \right] = O_p\left(  \max_{j\in[p]}\frac{\sqrt{K_{n,j}}}{\mu_{n,j}^2} \log^{3/2}D_n +  \max_{j\in[p]}\frac{1}{\mu_{n,j}} \log D_n\right),
\end{align*}
where $e_k$ is $d_j$-dimensional vector whose $k$-th element is $1$ and other elements are $0$.
\end{lemma}
\begin{lemma} \label{lemma:chao-hoeffding}
For each $j\in[p]$, the Hoeffding-type decomposition of $S_{n,j}$ is given by $S_{n,j}  = L_{n,j} + Q_{n,j}$ and the definitions of $L_{n,j}$ and $Q_{n,j}$ are gives by as follows
\begin{align*}
    L_{n,j} \coloneqq  \frac{1}{\mu_{n,j}\sqrt{n}}\sum_{i=1}^n(1 - \Pi_{ii,j}) (\pi_j^\top Z_{ij})u_{ij}, \quad Q_{n,j} \coloneqq \frac{1}{2\mu_{n,j}^2} \sum_{i=1}^n\sum_{m\neq i}^n \Pi_{im,j} (\varepsilon_{ij}u_{mj} + \varepsilon_{mj}u_{ij}) .
\end{align*}
\end{lemma}

Given the results of \cref{lemma:chao-Gamma-convergence} and \ref{lemma:chao-hoeffding},  applying our \cref{thm:main-inid} gives the uniform Gaussian approximation.
Before the statement, we introduce some notations.
For each $j\in[p]$ and $k\in[d_j]$, define $\hat{\theta}_{n,jk} \coloneqq e_k^\top \hat{\theta}_{n,j}$ and $\theta_{jk} \coloneqq e_k^\top \theta_j$.
Also define $\sigma_{jk}^2(Z)\coloneqq \Var[e_k^\top\Gamma_{n,j}^{-1}(L_{n,j}+Q_{n,j})\mid Z]$ and $T^\circ_{jk}\coloneqq \sigma_{jk}^{-1}(Z)e_k^\top \Gamma_{n,j}^{-1}(L_{n,j}+Q_{n,j})$.
Then, set $T_{jk}\coloneqq (\hat\theta_{n,jk}-\theta_{jk})/\sigma_{jk}(Z)$.

\begin{theorem} \label{thm:gauss-chao}
Under \cref{as:many-weak-iv}, it holds that
\begin{align*}
    \sup_{t\in\mathbb{R}}\left|\mathbb{P}\left( \max_{j\in[p]}\max_{k\in[d_j]}|T_{jk}|\le t \mid Z \right) - \mathbb{P}\left( \max_{j\in[p]}\max_{k\in[d_j]}|G_{jk}| \le t  \mid Z\right) \right| \xrightarrow{p} 0,
\end{align*}
with $(G_{jk})_{j\in[p], k\in[d_j]}$ is a centered Gaussian random vector which satisfies $\Cov(G_{jk}, G_{j'k'} \mid Z) = \Cov(T^\circ_{jk}, T^\circ_{j'k'} \mid Z)$ for all $j\in [p]$, $j'\in[p]$, $k\in[d_j]$ and $k'\in[d_{j'}]$.
\end{theorem}

\begin{rmk}[Sharpness] \label{rem:chao-sharp}
It is worth emphasizing that the primitive rate conditions in \cref{as:many-weak-iv} almost recover the  weakest conditions of \cite{chao2012asymptotic} up to logarithmic factors.
In particular, \cite{chao2012asymptotic} assume $\mu_{n,j}^{-1} \to 0$, $\sqrt{K_{n,j}}\mu_{n,j}^{-2} \to 0$ and $n^{-2} \sum_{i=1}^n |\pi_j^\top Z_{ij}|^4  \to 0$ for their asymptotic normality, meanwhile we assume $\mu_{n,j}^{-1} \log^{5/2}D_n\to 0$, $\sqrt{K_{n,j}}\mu_{n,j}^{-2}\log^4D_n \to 0$ and $n^{-2} \sum_{i=1}^n |\pi_j^\top Z_{ij}|^4 \log^6D_n \to 0$ for the uniform Gaussian approximation (\cref{thm:gauss-chao}) via our \cref{thm:main-inid}. 
This indicates that our high-dimensional CLT is sharp enough to almost recover the weakest condition for the asymptotic framework of \cite{chao2012asymptotic} in the high-dimensinal setting.
\end{rmk}

\subsection{Anderson-Rubin Test with Many-Weak IVs} \label{subsubsec:application-ms}

We next consider test-based inference in many-weak IV models, complementing the regression-based approach of \citet{chao2012asymptotic} discussed in the preceding subsection.
The Anderson--Rubin (AR) test provides inference robust to weak identification and can be inverted to construct confidence sets.
\citet{mikusheva2022inference} develop a jackknife AR test for linear IV models with many instruments and heteroskedastic errors.
Building on the CLT for degenerate quadratic forms in \citet{chao2012asymptotic}, they establish asymptotic validity of the test, consistency of its variance estimator under both the null and alternative hypotheses, and its power properties.

We extend the framework of \citet{mikusheva2022inference} to high-dimensional simultaneous inference.
Our maximal inequality for degenerate $U$-statistics (\cref{thm:max-is-inid}) yields uniform consistency of the variance estimator, while our Gaussian approximation (\cref{coro:main-kernel-inid}) provides joint Gaussian approximations for the AR statistics under both the null and alternative hypotheses.
We further propose an implementation method with bootstrap and establish its validity.

\subsubsection{Overview}

In what follows, the index $j\in[p]$ may represent an outcome, subgroup, time period, instrument set, instrument interaction, or other specification choice, or any combination thereof.
For each $j\in[p]$, consider the following linear IV model:
\begin{align*}
    Y_{ij} &= X_{ij}\theta_j+u_{ij},\\
    X_{ij} &= \gamma_{n,j}^\top Z_{ij}+\varepsilon_{ij}, \quad \mathbb{E}[u_{ij}\mid Z]=0, \quad \mathbb{E}[\varepsilon_{ij}\mid Z]=0,
\end{align*}
where $Y_{ij}$ is a scalar outcome, $X_{ij}$ is a potentially endogenous scalar regressor, $Z_{ij}$ is a $K_{n,j}$-dimensional vector of instruments, and $Z\coloneqq(Z_{ij})_{i\in[n],j\in[p]}$.
While \citet{mikusheva2022inference} allow nonlinear reduced forms, we consider a linear specification for simplicity.
As in the preceding subsection, we assume that $\gamma_{n,j}\coloneqq n^{-1/2}\mu_{n,j}\pi_j$, where $\mu_{n,j}$ is a scalar sequence and $\pi_j$ is a nonrandom $K_{n,j}$-dimensional vector.
Let $Z_j\coloneqq(Z_{1j},\dots,Z_{nj})^\top$ and define the projection matrix $\Pi_j\coloneqq Z_j(Z_j^\top Z_j)^{-1}Z_j^\top$.

For each $j\in[p]$, consider the null hypothesis $H_{0,j}:\theta_j=\theta_{0,j}$ and define the corresponding residual $u_{ij}(\theta_{0,j})\coloneqq Y_{ij}-X_{ij}\theta_{0,j}$.
For notational simplicity, we suppress the argument $\theta_{0,j}$ when defining the test statistic.
The jackknife AR statistic of \citet{mikusheva2022inference} is given by
\begin{align*}
    \hat{AR}_j \coloneqq \frac{1}{\sqrt{K_{n,j}}\sqrt{\hat\Phi_j}}\sum_{i=1}^n\sum_{m\neq i}\Pi_{im,j}u_{ij}u_{mj},
\end{align*}
where
\begin{align*}
    \hat\Phi_j \coloneqq \frac{2}{K_{n,j}}\sum_{i=1}^n\sum_{m\neq i}\frac{\Pi_{im,j}^2(u_{ij}M_{i,j}u_j)(u_{mj}M_{m,j}u_j)}{M_{ii,j}M_{mm,j}+M_{im,j}^2},
\end{align*}
$u_j\coloneqq(u_{1j},\dots,u_{nj})^\top$, and $M_j\coloneqq I_n-\Pi_j$ is the annihilator matrix, with $M_{i,j}$ denoting its $i$-th row and $M_{im,j}$ its $(i,m)$-th entry.

In the next subsection, we establish high-dimensional simultaneous inference for the vector $\hat{AR}\coloneqq(\hat{AR}_j)_{j\in[p]}$.
The proofs of the following results are provided in \cref{subsec:proof-ms}.

\subsubsection{Theoretical Results} 

\begin{assumption} \label{as:mikusheva-sun} \mbox{}
\begin{itemize}
    \item[(i)] For all $j\in[p]$, $\Pi_j$ is $n\times n$ projection matrix of rank $K_{n,j}$ and $K_{n,j}/\log^{10}(np) \to \infty$ uniformly in $j\in[p]$ as $n\to\infty$, and there exists a constant $C$ such that $\Pi_{ii,j} \le C < 1$. 
    \item[(ii)]  The vector of errors $(v_{i,1}, \dots, v_{i,p})$ is independent across $i$ conditional on $Z \coloneqq (Z_{ij})_{1\le j \le p, 1\le i \le n}$. 
    Also, $\mathbb{E}[v_{i,j} \mid Z] = 0$, $\max_{j\in[p]}\max_{i\in[n]}|v_{i,j}| < C$ almost surely and there exist constants $c$ and $C$ such that
    \begin{align*}
        c \le  \min_j\min_i \min_x  \frac{x^\top \Var[v_{i,j}\mid Z] x}{x^\top x} \le \max_j\max_i\max_x \frac{x^\top \Var[v_{i,j}\mid Z] x}{x^\top x} \le C,
    \end{align*} 
    where $x$ is a non-zero vector.
\end{itemize}
\end{assumption}

\cref{as:mikusheva-sun}(i) is a high-dimensional analogue of Assumption 1 in \cite{mikusheva2022inference}, which strengthen $K_{n,j} \to \infty$ to $K_{n,j}/\log^{10}(np) \to \infty$ uniformly in $j\in[p]$. 
\cref{as:mikusheva-sun}(ii) corresponds to Assumption 2 in \cite{mikusheva2022inference}.
Also, note that the assumption in \cite{mikusheva2022inference} that $(Z_j\gamma_j)^\top M_j (Z_j\gamma_j) \lesssim K_{n,j}^{-1}(Z_j\gamma_j)^\top(Z_j\gamma_j)$ (in our notation) is automatically satisfied since we assume the first stage is exactly linear in $Z_j$ so it holds that $M_jZ_j\gamma_j = 0$.

The following results are high-dimensional extension of the results in section 4 of \cite{mikusheva2022inference}.
Specifically, \cref{thm:mikusheva-sun-variance}(i) and (ii) extend Theorem 2 and 3 in \cite{mikusheva2022inference} and establish the uniform consistency of the variance estimator under null and alternative hypothesis, respectively.
\cref{thm:mikusheva-sun-null} and  \cref{thm:mikusheva-sun-alt} corresponds to Lemma 1 and Theorem 4 in \cite{mikusheva2022inference}, which give the Gaussian approximation of the test statistic under  null and  hypothesis, respectively.

\begin{theorem} \label{thm:mikusheva-sun-variance}\mbox{}
\begin{itemize}
    \item[(i)] Suppose \cref{as:mikusheva-sun}(i) and (ii) for errors $u_i$ hold. If $\theta_j = \theta_{0,j}$ for each $j\in[p]$, it holds that 
    \begin{align*}
        \max_{j\in[p]} \left|\frac{\hat{\Phi}_j}{\Phi_j} - 1 \right| = O_p\left(  \max_{j\in[p]} \frac{\log^{4}(np)}{\sqrt{K_{n,j}}}\right), ~~ \text{with} ~~ \Phi_j \coloneqq \frac{2}{K_{n,j}} \sum_{i=1}^n\sum_{m\neq i} \Pi_{im,j}^2 \mathbb{E}[ u_{ij}^2\mid Z]\mathbb{E}[ u_{mj}^2\mid Z]
    \end{align*}
    \item[(ii)] Suppose \cref{as:mikusheva-sun}(i) and (ii) for errors $v_i = (u_i, \varepsilon_i)^\top$ hold. If $\theta_j =  \theta_{0,j} + \Delta_j$ with some bounded constant $\Delta_j$ such that $ \Delta_j^2 K_{n,j}^{-1}|Z_j\gamma_j|^2 \log^6(np) \to 0$  uniformly $j\in[p]$, then it holds that 
    \begin{align*}
        & \max_{j\in[p]}  \left|\frac{\hat{\Phi}_j}{\Phi_j} - 1 \right| = O_p\left(\max_{j\in[p]} \frac{|\Delta_jZ_j\gamma_{n,j}|}{\sqrt{K_{n,j}}}\log^3(np)  + \max_{j\in[p]}  \frac{1}{\sqrt{K_{n,j}}}\log^4(np)\right), \\
        & \text{with} \quad \Phi_j \coloneqq \frac{2}{K_{n,j}} \sum_{i=1}^n\sum_{m\neq i} \Pi_{im,j}^2 \mathbb{E}[ \eta_{i,j}^2\mid Z]\mathbb{E}[ \eta_{m,j}^2\mid Z], \text{ where } \eta_{i,j} \coloneqq u_{ij} + \Delta_j \varepsilon_{ij}.
    \end{align*}
\end{itemize}
\end{theorem}

\begin{theorem}  \label{thm:mikusheva-sun-null}
Suppose \cref{as:mikusheva-sun}(i) and (ii) for errors $u_i$ hold and $\theta_j = \theta_{0,j}$ for all $j\in[p]$. 
Let $\hat{AR}$ be the vector of test statistics $\hat{AR} \coloneqq  (\hat{AR}_j)_{j\in[p]}$.
Then it holds that
\begin{align*}
    \sup_{R\in\mathcal{R}_p} \left| \mathbb{P}\left( \hat{AR} \in R \mid Z \right) - \mathbb{P}\left(  G \in R \mid Z \right) \right| = o_p(1)
\end{align*}
with the centered Gaussian random vector $G \coloneqq (G_j)_{j\in[p]}$ whose conditional covariance is same as that of $(AR_j)_{j\in[p]}$ where
\begin{align*}
    {AR}_j \coloneqq \frac{1}{\sqrt{K_{n,j}}\sqrt{\Phi_j}} \sum_{i=1}^n\sum_{m\neq i} \Pi_{im,j} u_{ij}u_{mj}.
\end{align*}
\end{theorem}

\begin{theorem}  \label{thm:mikusheva-sun-alt}
Suppose \cref{as:mikusheva-sun}(i) and (ii) for errors $v_i = (u_i, \varepsilon_i)^\top$ hold and  
$\theta_j =  \theta_{0,j} + \Delta_j$.
Also, we assume $ \Delta_j^2 K_{n,j}^{-1} |Z_j\gamma_j|^2\log^8(np) \to 0$ and $|\Delta_j|^3 K_{n,j}^{-1} |Z_j\gamma_j|^3 \log^{7/2}(np) \to 0 $  uniformly in $j\in[p]$ and over $\{(\Delta_1, \dots, \Delta_p) \in \mathbb{R}^{p}  : \max_{j\in[p]} |\Delta_j|^2 \le c\}$, then, for any positive constant $c$, it holds that 
\begin{align*}
     \sup_{\substack{ (\Delta_1, \dots, \Delta_p) \in \mathbb{R}^{p} \\ \max_{j\in[p]} |\Delta_j|^2 \le c}}\sup_{R\in\mathcal{R}_p} \left| \mathbb{P}\left( \hat{AR} \in R \mid Z \right) - \mathbb{P}\left( G^\Delta + \lambda^\Delta \in R \mid Z \right) \right| = o_p(1),
\end{align*}
where $G^\Delta \coloneqq (G_j^\Delta)_{j\in[p]}$ is the centered Gaussian random vector whose conditional covariance is same as $(AR^\Delta_j)_{j\in[p]}$ with 
\begin{align*}
    AR^\Delta_j \coloneqq \frac{1}{\sqrt{K_{n,j}}\sqrt{\Phi_j}} \sum_{i=1}^n\sum_{m\neq i} \Pi_{im,j} \eta_{i,j}\eta_{m,j}.
\end{align*}
Also, $\lambda^\Delta$ is defined as $\lambda^\Delta \coloneqq (\lambda_j^\Delta)_{j\in[p]}$ with 
\begin{align*}
    \lambda_j^\Delta \coloneqq \frac{\Delta_j^2 \varsigma^2_j }{\sqrt{K_{n,j}}\sqrt{\Phi}_j}, ~\text{  where  }~ \varsigma^2_j \coloneqq \sum_{i=1}^n\sum_{m\neq i} \Pi_{im,j} (\gamma_{n,j}^\top Z_{ij})(\gamma_{n,j}^\top Z_{mj}).
\end{align*}
\end{theorem}

\begin{rmk}[Sharpness] \label{rem:smikusheva-sun-sharp}
The primitive rate conditions for
\cref{thm:mikusheva-sun-null,thm:mikusheva-sun-alt}
recover the corresponding rate conditions of
\citet{mikusheva2022inference} up to logarithmic factors.
In particular, their Assumption~1 in \cite{mikusheva2022inference} requires $K_{n,j}\to\infty$,
whereas \cref{as:mikusheva-sun}(i) imposes $
   {K_{n,j}}/\log^{10}(np) \to \infty$ uniformly in $j\in[p]$.
For their results under alternatives,
\citet{mikusheva2022inference} requires
$K_{n,j}^{-1}|Z_j\gamma_{n,j}|^2\to0$ in our notation,
whereas \cref{thm:mikusheva-sun-alt} imposes $K_{n,j}^{-1}|Z_j\gamma_{n,j}|^2\log^8(np)\to0$  uniformly in $j\in[p]$.
Thus, our general Gaussian approximation and
maximal inequalities are sharp enough to extend their asymptotic
framework to high-dimensional simultaneous inference while
preserving these rate conditions.
\end{rmk}

\begin{rmk}[Studentization under alternatives] 
The condition corresponding to $\Delta_j^3 K_{n,j}^{-1} |Z_j\gamma_j|^3\log^{7/2}(np) \to 0 $ is not imposed in \citet{mikusheva2022inference} but this rate condition should be required even in the univariate setting.
Their proof appears to rely on variance consistency and a central
limit theorem for a degenerate $U$-statistic by \cite{chao2012asymptotic} in conjunction with Slutsky's theorem.
However, when the noncentrality parameter is allowed to diverge,
consistency of a variance estimator and asymptotic normality of a centered term alone are not always sufficient for a noncentered Gaussian approximation of the studentized statistic, as noted by Theorem~2.1(ii) and Remark~2.1 in 
\cite{bentkus2007limiting}.
To explain this point in the univariate setting, fix a coordinate, suppress the index $j$ and ignore logarithmic factors.
Since $\widehat{AR}=(\Phi/\widehat{\Phi})^{1/2}AR$, we have
\[
\widehat{AR}-\lambda^\Delta
=
\left(\frac{\Phi}{\widehat{\Phi}}\right)^{1/2}
(AR-\lambda^\Delta)
+
\lambda^\Delta
\left\{
\left(\frac{\Phi}{\widehat{\Phi}}\right)^{1/2}-1
\right\}.
\]
Note that the Gaussian approximation of studentized test-statistic with non-centered parameter $\lambda^\Delta$ holds if the right hand side of this equation asymptotically follows Gaussian distribution with mean $0$ and unit variance.
If $AR-\lambda^\Delta \xrightarrow{d} N(0, 1)$ and
$\widehat{\Phi}/\Phi\to_p1$, the first term satisfies 
$AR-\lambda^\Delta+o_p(1) \xrightarrow{d} N(0,1)$.
However, even if $\widehat{\Phi}/\Phi\to_p1$, the second term requires separate control when
$\lambda^\Delta$ diverges.
A sufficient condition for this term to be
$o_p(1)$ is
\[
|\lambda^\Delta|
\left|\frac{\widehat{\Phi}}{\Phi}-1\right|= O_p\left( \frac{\Delta^2 |Z\gamma_n|^2}{\sqrt{K}} \cdot \frac{\Delta |Z\gamma_n|}{\sqrt{K}} \right) = O_p\left( \frac{\Delta^3  |Z\gamma_n|^3}{K} \right) 
=o_p(1).
\]
Thus, the convergence rate of the variance estimator must be assessed
relative to the magnitude of the noncentrality parameter even in the univariate setting.
\end{rmk}

The preceding results establish joint Gaussian approximations for the AR statistics under both the null and alternative hypotheses. 
In practice, however, critical values are often computed using bootstrap methods rather than directly simulating from the approximating Gaussian distribution. 
We therefore consider the following quadratic multiplier bootstrap.

Let $\xi_1, \dots, \xi_n$ be independent random variables following $N(0,1)$, independent of data.
For each $j\in[p]$, define
\begin{align*}
    \hat{AR}_j^* \coloneqq \frac{1}{\sqrt{K_{n,j}}\sqrt{\check{\Phi}_j}} \sum_{i=1}^n\sum_{m\neq i} \Pi_{im,j} \bar u_{ij}\bar u_{mj} \xi_{i} \xi_m,
\end{align*}
with $\bar u_{ij} \coloneqq \sgn{u_{ij}} \min\{|u_{ij}|, \sqrt{\log(np)}\}$ and
\begin{align*}
    \check{\Phi}_j \coloneqq \frac{2}{K_{n,j}} \sum_{i=1}^n \sum_{m\neq i} \Pi_{im,j}^2 \bar{u}_{ij}^2 \bar{u}_{mj}^2.
\end{align*}

\begin{rmk}[Self-normalization and truncation]
The bootstrap AR statistics consists of two modifications of self-normalization and truncation.
Since the studentized AR statistics admit Gaussian approximations, bootstrap validity can be established by showing that the bootstrap statistics also admit Gaussian approximations with asymptotically matching covariance matrices.
Under alternative hypotheses, the residuals contain the first-stage signal $\Delta_j\gamma_{n,j}^{\top}Z_{ij}$, as $u_{ij}(\theta_{0,j})=\eta_{i,j}+\Delta_j\gamma_{n,j}^{\top}Z_{ij}$. While the variance estimator of the original AR statistic remains consistent under the alternatives, a naive multiplier bootstrap also randomizes the first-stage signal and potentially introduce additional variation that prevents covariance matching.
Truncation bounds the magnitude of the residuals used in the bootstrap, allowing us to control such additional contributions and establish covariance matching under the stated assumptions. 
Self-normalization by $\check{\Phi}_j$, constructed from the truncated residuals, then ensures that the bootstrap statistic has unit conditional variance.
Neither modification affects the original AR statistic.
\end{rmk}

The following theorem establishes the asymptotic validity of the proposed bootstrap under both the null and alternative hypotheses.
\begin{theorem}  \mbox{}\label{thm:mikusheva-sun-bootstrap}
\begin{itemize}
    \item[(i)] Suppose the conditions of the Gaussian approximation theorem under the null hypothesis hold. Then,
    \begin{align*}
        \sup_{R\in\mathcal{R}_p} \left| \mathbb{P}\left( \hat{AR} \in\  R \mid Z \right) - \mathbb{P}^*\left(\hat{AR}^* \in R \right)\right| = o_p(1).
    \end{align*}
    \item[(ii)]Suppose the conditions of the preceding Gaussian approximation theorem under the alternative hypotheses hold. Then, for any fixed $c > 0$,
    \begin{align*}
        \sup_{\substack{(\Delta_1, \dots, \Delta_p) \in \mathbb{R}^p \\ \max_{j\in[p]} |\Delta_j|^2 \le c }}\sup_{R\in \mathcal{R}_p}\left| \mathbb{P}\left(\hat{AR} - \lambda^\Delta\in R \mid Z\right) - \mathbb{P}^*\left( \hat{AR}^* \in R  \right)\right| = o_p(1).
    \end{align*}
\end{itemize}
\end{theorem}

Fix $\alpha\in(0,1)$.
Let $q_{1-\alpha}^*$ denote the conditional $(1-\alpha)$-quantile of $\max_{j\in[p]}\hat{AR}_j^*$.
Consider the testing procedure that rejects $H_{0,j}$ whenever $\hat{AR}_j>q_{1-\alpha}^*$.
The following corollary establishes the family-wise error control and power properties of this procedure.

\begin{corollary} \label{coro:mikusheva-sun-multitest} \mbox{}
\begin{itemize}
    \item[(i)] Strong family-wise error control :
    Let $\mathcal H_0\subseteq[p]$ denote the set of true null hypotheses.
    Suppose the conditions of \cref{thm:mikusheva-sun-bootstrap}(i) hold for the coordinates in $\mathcal H_0$.
    Then,
    \begin{align*}
        \mathbb P\left(\max_{j\in\mathcal H_0}\hat{AR}_j>q_{1-\alpha}^*\mid Z\right)\le\alpha+o_p(1).
    \end{align*}

    \item[(ii)] Consistency :
    Suppose the conditions of
    \cref{thm:mikusheva-sun-bootstrap}(ii) hold.
    Let $q_{1-\alpha}^\Delta(Z)$ denote the
    conditional $(1-\alpha)$-quantile of
    $\max_{j\in[p]}G_j^\Delta$.
    Then, for each $j\in[p]$, if
    $\lambda_j^\Delta-q_{1-\alpha}^\Delta(Z)
    \xrightarrow{p}\infty$, we have $\mathbb P\left(\hat{AR}_j>q_{1-\alpha}^*\mid Z\right)\xrightarrow{p}1$.
\end{itemize}
\end{corollary}

\begin{rmk}[Multi-step Procedure]
    For more powerful testing procedure while retaining family-wise error control,
    the single-step procedure above can be refined into the multi-step procedure by removing the hypotheses rejected at each stage and recomputing the critical value as the conditional $(1-\alpha)$-quantile of the Gaussian maximum over the remaining coordinates; see \citet[Section 5]{CCK13} and \citet[Section 2.4]{belloni2018high} for detail. 
\end{rmk}

In the next subsection, we investigate the finite-sample performance of  the proposed multiplier bootstrap through Monte Carlo simulations.
In particular, we compare marginal Gaussian approximation with Bonferroni correction, marginal bootstrap approximation with Bonferroni correction, and joint bootstrap approximation in terms of size control and power.

\subsubsection{Numerical Experiments}

We examine the finite-sample performance of simultaneous AR tests using a high-dimensional extension of the design in Section 4.2 of \citet{mikusheva2022inference}.

\paragraph{Simulation Design :}
Following the balanced group-dummy design in Section 4.2 of \citet{mikusheva2022inference}, we partition the $n$ observations into $K$ groups of equal size $m=n/K$.

We generate data as follows;
\begin{align*}
    Y_{ij} &= \beta_j X_i+u_{ij},\\
    X_i &= 0.35 S_{g(i)}+0.3\operatorname{sign}(F_i)+\sqrt{1-0.3^2}\,V_i,\\
    u_{ij} &= \sigma_i\operatorname{sign}\left(\sqrt{c_\rho}F_i+\sqrt{1-c_\rho}\,E_{ij}\right),
\end{align*}
with
\begin{align*}
    c_\rho &= \sin\left(\frac{\pi\rho}{2}\right),\qquad \sigma_i=0.7+0.6\frac{((i-1)\bmod20)+0.5}{20},
\end{align*}
where $(S_g)_{g=1}^K$ and $(V_i)_{i=1}^n$ are mutually independent Rademacher variables satisfying $\mathbb P(S_g=\pm1)=\mathbb P(V_i=\pm1)=1/2$, $F_i$ and $E_{ij}$ are independent standard normal variables, and all four collections are mutually independent.
We use $Z_{ig}=S_g\mathbf{1}\{g(i)=g\}$, $g\in[K]$, as instruments. 
The group-specific component $S_{g(i)}$ generates variation in $X_i$ explained by the instruments. 
Since $F_i$ enters both $X_i$ and $u_{ij}$, it induces endogeneity when $\rho>0$. 
Also, the same $F_i$ is shared across all outcomes for a given observation $i$, generating cross-outcome dependence. 
The independent shocks $E_{ij}$ introduce observation- and outcome-specific variation, while the deterministic scale factors $\sigma_i$ introduce heteroskedasticity.
The sign transformation in $u_{ij}$ produces bounded, non-Gaussian errors.
The Gaussian variables inside the sign functions are marginally standard normal and have correlation $c_\rho$ across distinct outcomes due to the shared factor $F_i$.
Consequently, $\mathbb E[u_{ij}\mid Z]=0$ and $\operatorname{Var}(u_{ij}\mid Z)=\sigma_i^2$, while the Gaussian sign-correlation identity\footnote{For jointly standard normal variables $(X,Y)$ with correlation $r$, $\mathbb E[\operatorname{sign}(X)\operatorname{sign}(Y)]=(2/\pi)\arcsin(r)$.} gives $\operatorname{Corr}(u_{ij},u_{ik})=(2/\pi)\arcsin(c_\rho)=\rho$ for $j\ne k$.
For distinct true-null coordinates $j$ and $k$, let $Q_j=\sum_{i\ne\ell}\Pi_{i\ell}u_{ij}u_{\ell j}$ denote the AR numerator. Since the errors are centered and independent across observations, we have
\begin{align*}
    \operatorname{Var}(Q_j\mid Z)&=2\sum_{i\ne\ell}\Pi_{i\ell}^2\mathbb E[u_{ij}^2\mid Z]\mathbb E[u_{\ell j}^2\mid Z]=2\sum_{i\ne\ell}\Pi_{i\ell}^2\sigma_i^2\sigma_\ell^2,\\
    \operatorname{Cov}(Q_j,Q_k\mid Z)&=2\sum_{i\ne\ell}\Pi_{i\ell}^2\mathbb E[u_{ij}u_{ik}\mid Z]\mathbb E[u_{\ell j}u_{\ell k}\mid Z]=2\rho^2\sum_{i\ne\ell}\Pi_{i\ell}^2\sigma_i^2\sigma_\ell^2.
\end{align*}
Consequently, $\operatorname{Corr}(Q_j,Q_k\mid Z)=\rho^2$, which we use as the strength of correlation across coordinates.

The simulation settings are
\begin{align*}
    n&\in\{400,800,1600\},\quad K\in\{40,80,160\},\quad p\in\{16,64,128\},\quad \rho^2\in\{0,0.1,0.3,0.5,0.7,0.9\}.
\end{align*}
We exclude $(n,K)=(400,160)$ because $n/K$ is not an integer, yielding $144$ data-generating processes.
These settings allow us to examine dimension dependence at fixed $(n,K,\rho^2)$ and the effects of varying $n$ and $K$ while holding either $n$, $K$, or the group size $n/K$ fixed.

For each $j\in[p]$, we test $H_{0j}:\beta_j=0$ with $\alpha=0.05$. 
We consider alternatives in which one-eighth of the coefficients are nonzero, with values equally spaced between $1$ and $3$, and the remaining coefficients are zero.
We compare three procedures: Normal + Bonferroni uses standard normal critical values at level $\alpha/p$; Bootstrap + Bonferroni uses marginal bootstrap critical values at the same level; and Joint Bootstrap uses the bootstrap distribution of the maximum AR statistic. 
Both bootstrap procedures employ the self-normalized, truncated quadratic Gaussian multiplier statistic defined above, using the same observation-level multipliers across coordinates to reproduce cross-coordinate dependence. 
We use $1,000$ Monte Carlo replications and $B=10,239$ bootstrap draws.

\paragraph{Performance of Simultaneous AR Tests :}
We evaluate the finite-sample performance of the three procedures in terms of the FWER (the probability of rejecting at least one true null hypothesis)  and power (the rejection probability across false null hypotheses). 
\Cref{fig:ms-sparse-fwer-p16,fig:ms-sparse-fwer,fig:ms-sparse-fwer-p128} report the FWER, while \Cref{fig:ms-sparse-power-p16,fig:ms-sparse-power,fig:ms-sparse-power-p128} report power for $p=16,64,128$, respectively. 
Each panel corresponds to a pair $(n,K)$, with $\rho^2$ on the horizontal axis.
The dashed lines in the FWER figures indicate the nominal level $\alpha=0.05$.

The FWER results reveal substantial differences across the three procedures. Normal + Bonferroni exhibits considerable overrejection when $\rho^2$ is small, and this distortion becomes more pronounced as $p$ increases. 
This deterioration reflects the sensitivity of Bonferroni correction to errors in the marginal normal approximation, particularly in the extreme upper tails corresponding to significance levels $\alpha/p$.
Bootstrap + Bonferroni substantially reduces this overrejection but becomes increasingly conservative as cross-coordinate dependence strengthens. 
By contrast, Joint Bootstrap alleviates this conservativeness and exhibits substantially less deterioration as the dimension increases. 

The FWER of Joint Bootstrap also varies with the sample size $n$ and the number of instruments $K$.
First, holding $n$ fixed, increasing $K$ does not necessarily improve the FWER, as illustrated by the overrejection observed at $(n,K)=(400,80)$.
Second, holding $K$ fixed, increasing $n$ generally brings the FWER closer to the nominal level.
Third, when $n$ and $K$ increase together with the group size $m=n/K$ held fixed, the FWER also tends to approach the nominal level, although this pattern does not hold in every setting.

Next, we consider the results on the power.
The comparisons should be interpreted in light of the FWER results reported above, particularly the substantial overrejection of Normal + Bonferroni and the conservativeness of Bootstrap + Bonferroni.
Joint Bootstrap generally achieves higher power than Bootstrap + Bonferroni, particularly when cross-coordinate dependence is strong.
This difference also becomes more pronounced as the dimension increases.
For example, at $(n,K)=(400,40)$ and $\rho^2=0.9$, the power of Bootstrap + Bonferroni decreases from approximately $45\%$ for $p=16$ to $30\%$ for $p=128$, whereas Joint Bootstrap maintains power around $60\%$.
Together with the FWER results, these findings illustrate the practical benefits of joint bootstrap calibration, particularly in high-dimensional settings with strong cross-coordinate dependence.

\begin{figure}[p]
    \centering
    \setlength{\abovecaptionskip}{3pt}
    \setlength{\belowcaptionskip}{0pt}
    \includegraphics[height=.45\textheight,trim=5bp 5bp 5bp 5bp,clip]{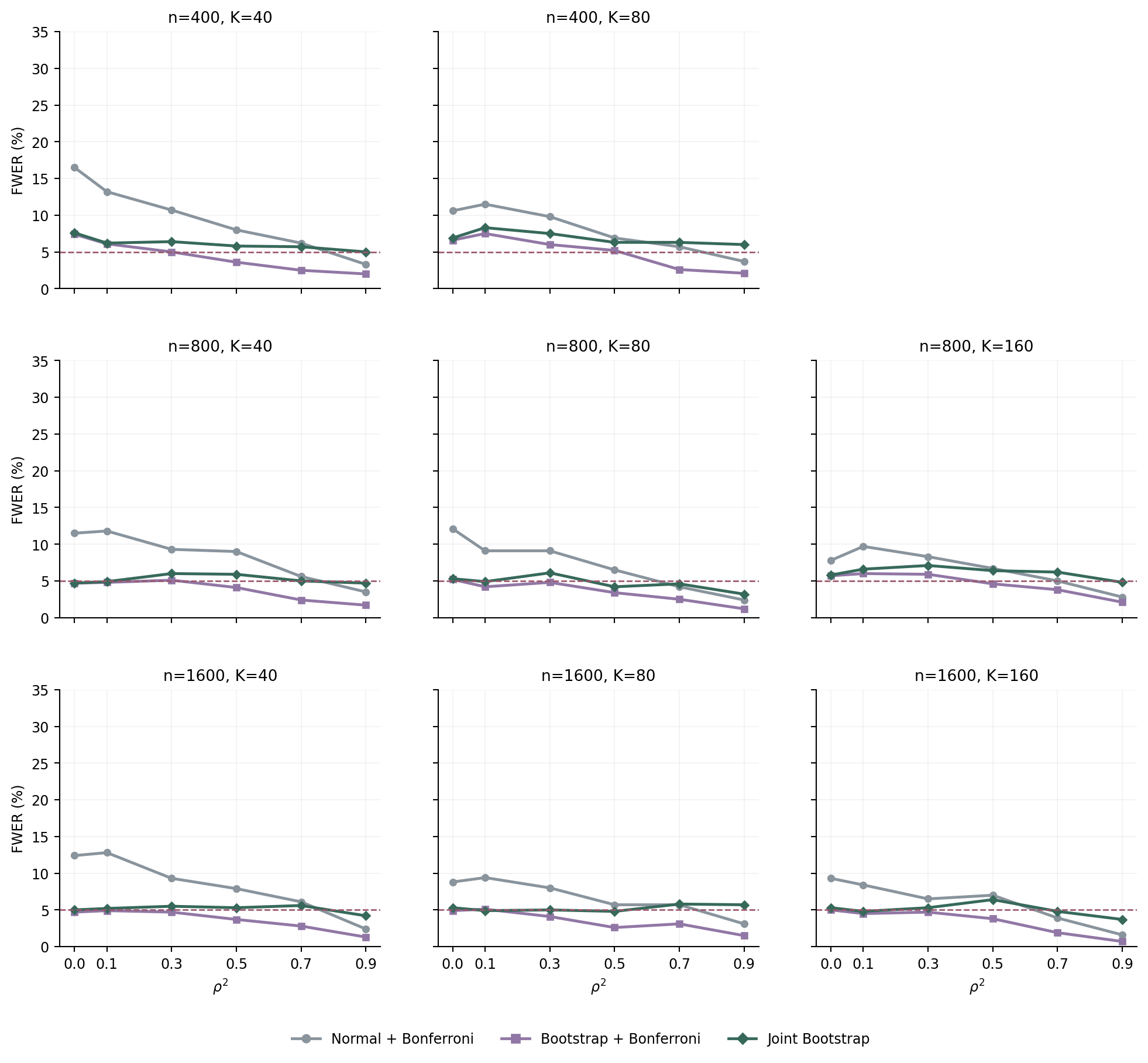}
    \caption{Family-wise error rates for $p=16$ at $\alpha=0.05$. Each panel corresponds to $(n,K)$, with $\rho^2$ on the horizontal axis.}
    \label{fig:ms-sparse-fwer-p16}

    \vspace{5pt}

    \includegraphics[height=.45\textheight,trim=5bp 5bp 5bp 5bp,clip]{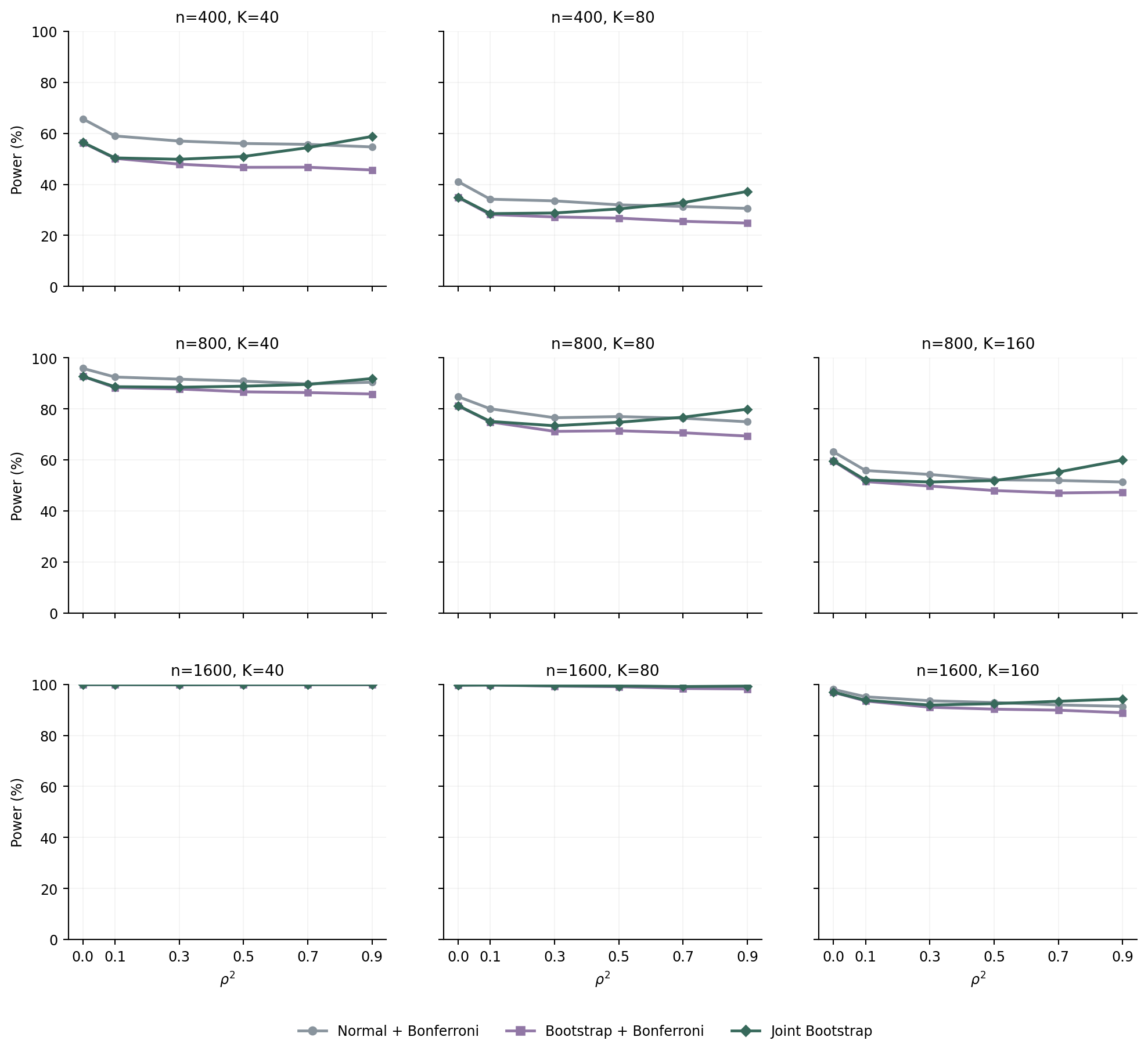}
    \caption{Power for $p=16$ at $\alpha=0.05$, averaged across false null hypotheses.}
    \label{fig:ms-sparse-power-p16}
\end{figure}

\begin{figure}[p]
    \centering
    \setlength{\abovecaptionskip}{3pt}
    \setlength{\belowcaptionskip}{0pt}
    \includegraphics[height=.45\textheight,trim=5bp 5bp 5bp 5bp,clip]{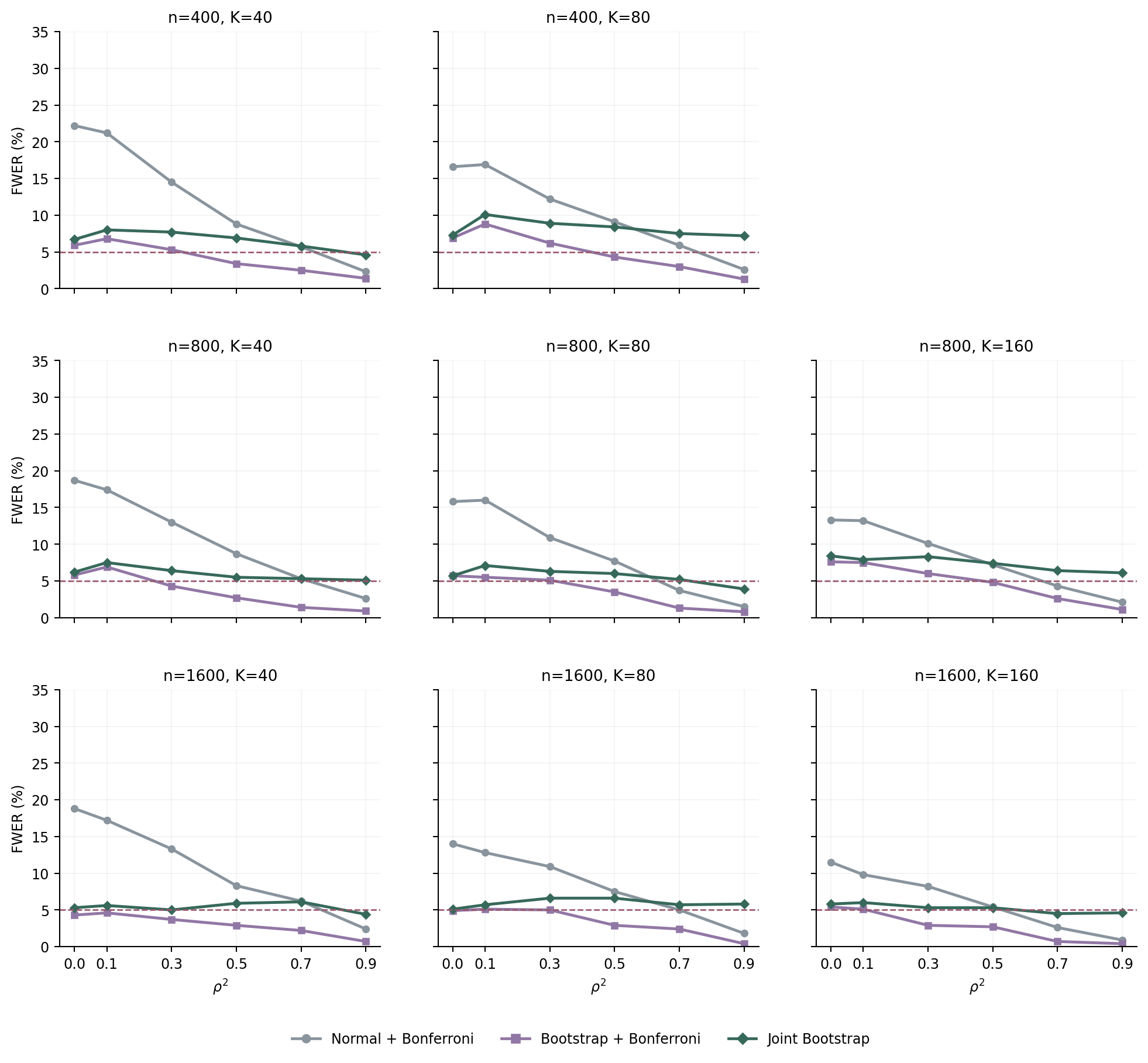}
    \caption{Family-wise error rates for $p=64$ at $\alpha=0.05$. Each panel corresponds to $(n,K)$, with $\rho^2$ on the horizontal axis.}
    \label{fig:ms-sparse-fwer}

    \vspace{5pt}

    \includegraphics[height=.45\textheight,trim=5bp 5bp 5bp 5bp,clip]{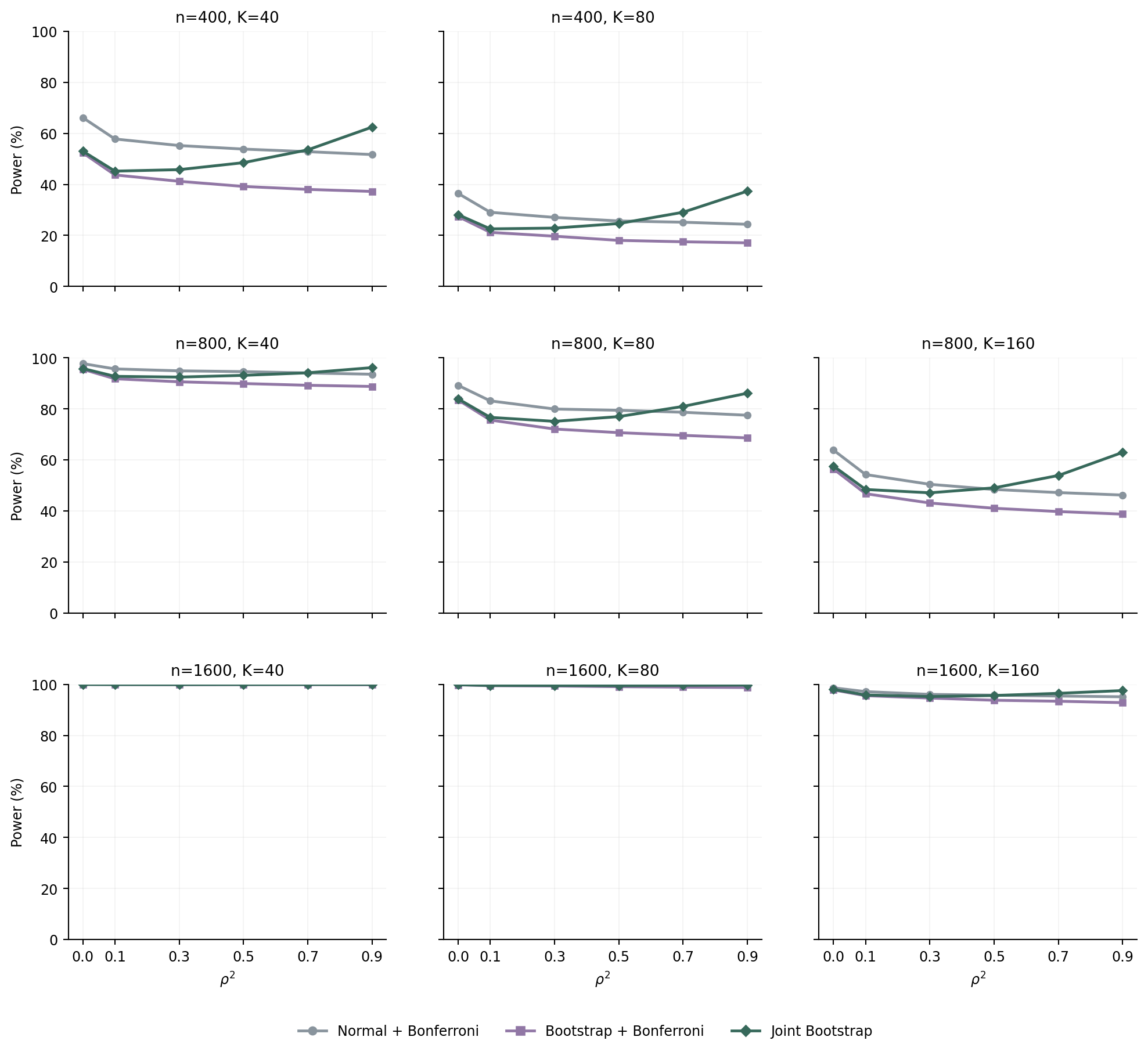}
    \caption{Power for $p=64$ at $\alpha=0.05$, averaged across false null hypotheses.}
    \label{fig:ms-sparse-power}
\end{figure}

\begin{figure}[p]
    \centering
    \setlength{\abovecaptionskip}{3pt}
    \setlength{\belowcaptionskip}{0pt}
    \includegraphics[height=.45\textheight,trim=5bp 5bp 5bp 5bp,clip]{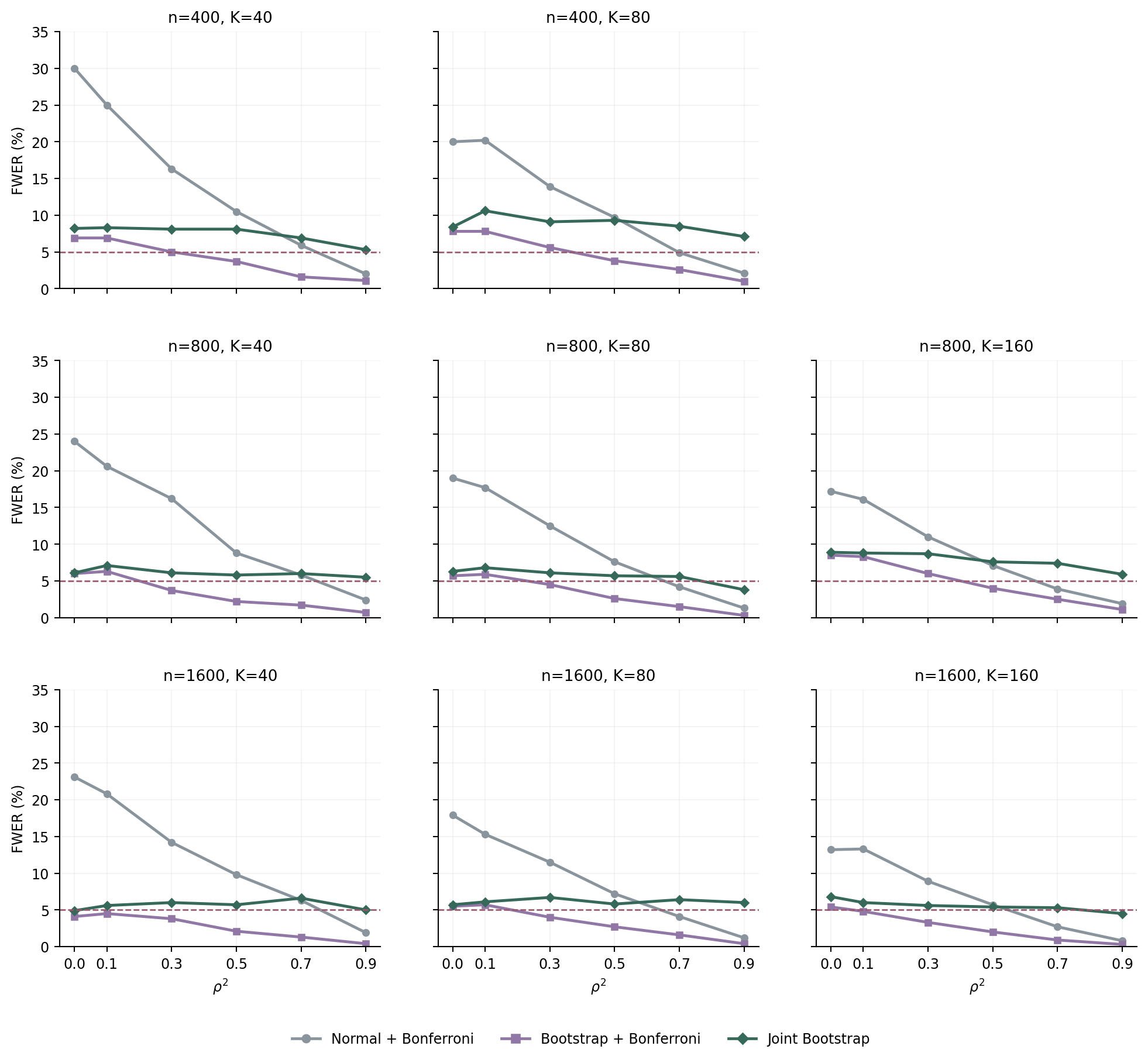}
    \caption{Family-wise error rates for $p=128$ at $\alpha=0.05$. Each panel corresponds to $(n,K)$, with $\rho^2$ on the horizontal axis.}
    \label{fig:ms-sparse-fwer-p128}

    \vspace{5pt}

    \includegraphics[height=.45\textheight,trim=5bp 5bp 5bp 5bp,clip]{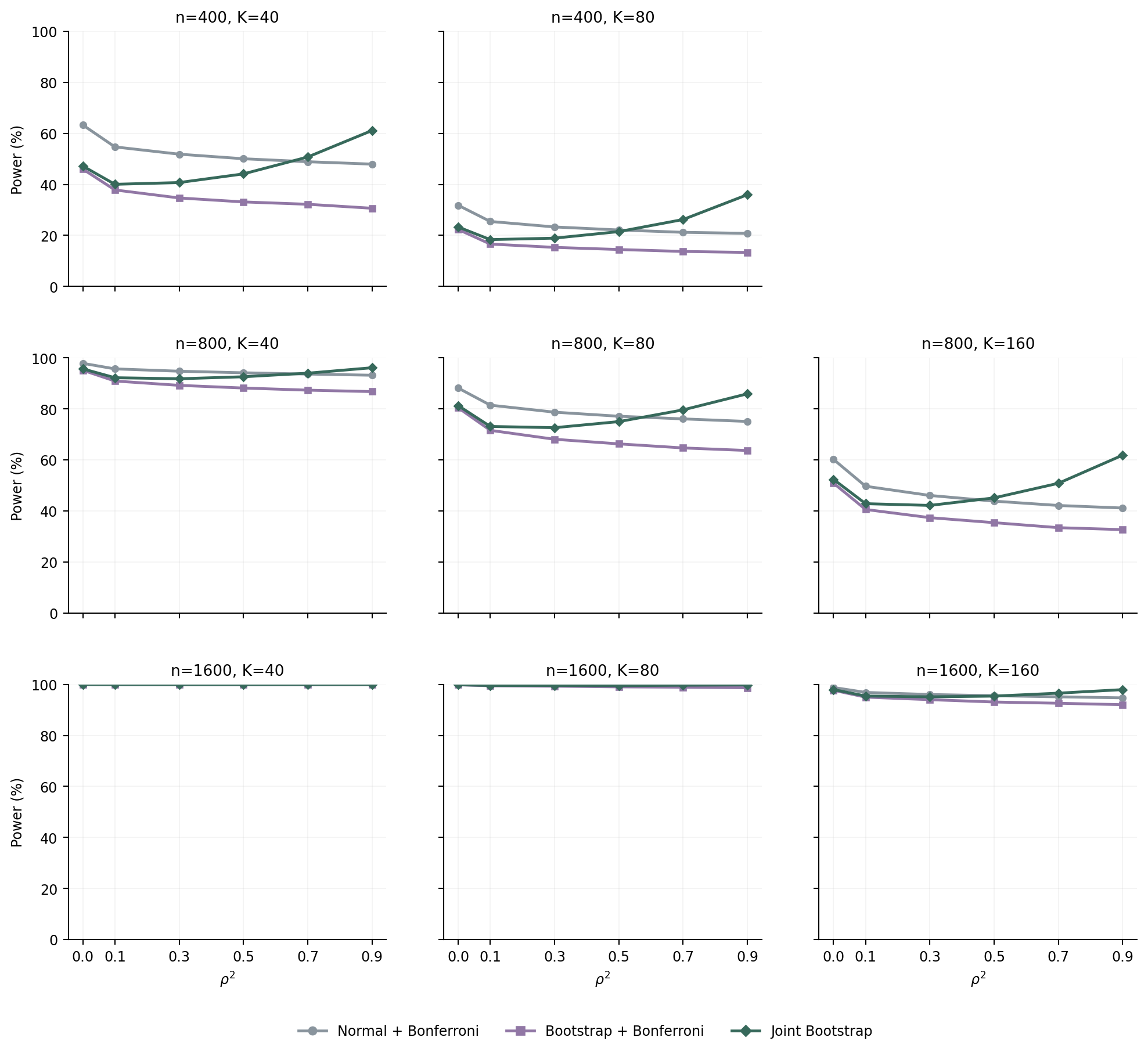}
    \caption{Power for $p=128$ at $\alpha=0.05$, averaged across false null hypotheses.}
    \label{fig:ms-sparse-power-p128}
\end{figure}

\paragraph{Illustration of Simultaneous Confidence Sets:}

An important advantage of AR tests is that they can be inverted to construct confidence sets even under weak identification.
We illustrate this feature under $(n,K,p,\rho^2)=(800,80,64,0.9)$.
For the two Bonferroni procedures, we invert each marginal AR test separately.
For Joint Bootstrap, we invert the joint maximum test and numerically approximate the coordinatewise projections of its acceptance region.
\Cref{fig:ms-sparse-interval} illustrates the average of the resulting coordinatewise confidence sets.
In this illustration, Joint Bootstrap yields shorter average confidence intervals than both Bonferroni procedures for every coordinate.

\begin{figure}[htb!]
    \centering
    \setlength{\abovecaptionskip}{4pt}
    \setlength{\belowcaptionskip}{0pt}

    \includegraphics[width=\textwidth,height=.85\textheight,keepaspectratio]{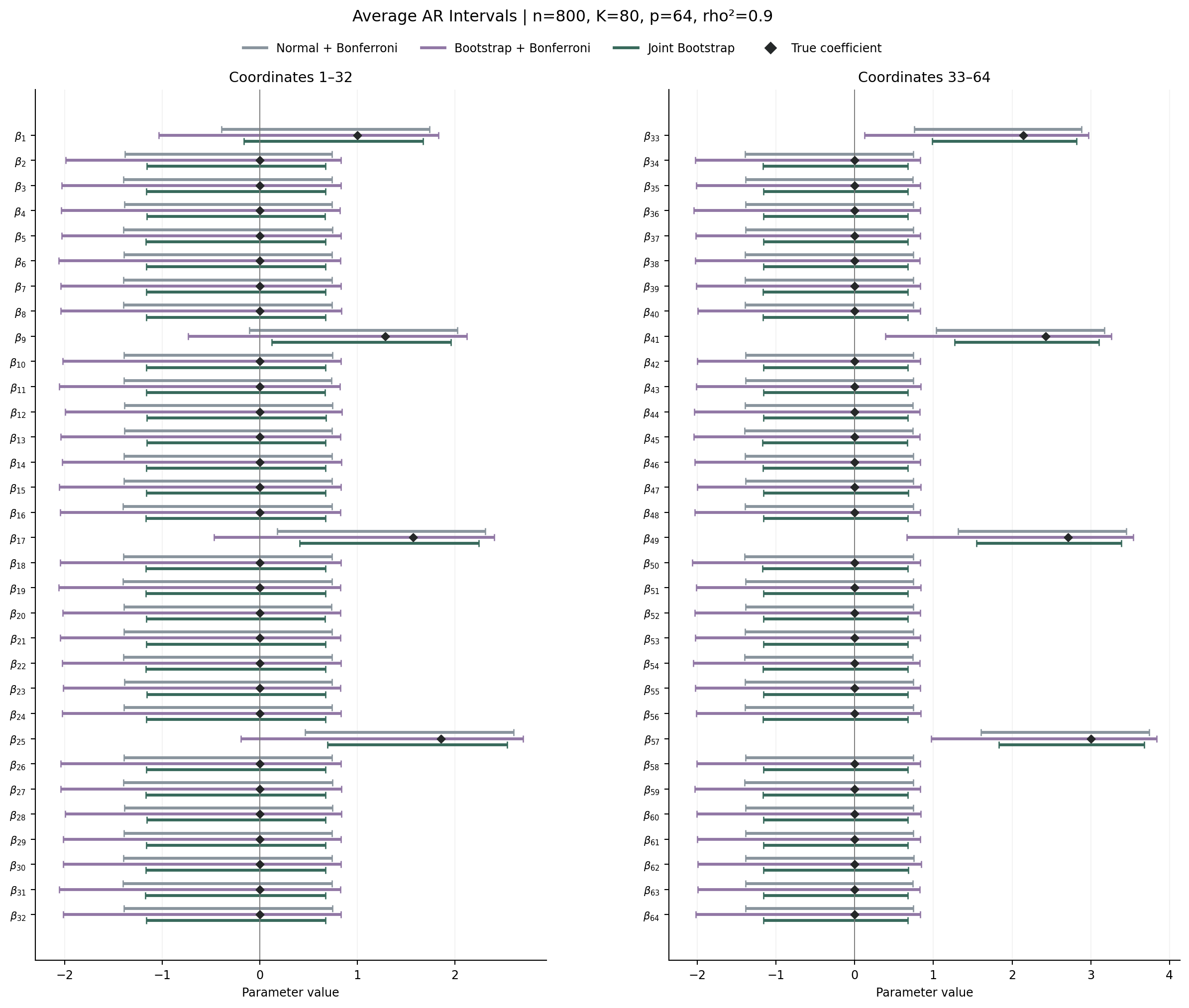}
    \caption{Average inverted AR acceptance sets for $(n,K,p,\rho^2)=(800,80,64,0.9)$ at $\alpha=0.05$. For each coordinate, averages are computed over replications in which all three procedures yield finite, nonempty computed sets.}
    \label{fig:ms-sparse-interval}
\end{figure}

\clearpage
\section{Application of Two-sample $U$-statistics} \label{subsec:two-sample-U}
In this section, we illustrate how our Gaussian approximations and maximal inequalities can be used to conduct uniform inference based on two-sample $U$- and $V$-statistics.
As an example, we consider tests of distributional homogeneity \citep{li2024optimality} and inference for separately exchangeable arrays under sample-size dependent data generating process \citep{chiang2023using}.

\subsection{Kernel Based Adaptive Distribution Homogeneity Test}\label{exm:homo_test}

Testing whether two populations have the same distribution is a standard problem in economics and statistics, with applications in a variety of empirical settings \citep{bera2013smooth,li2009nonparametric,zhang2024testing}. 
In kernel-based tests, the choice of bandwidth can substantially affect power, while the appropriate bandwidth generally depends on the unknown smoothness of the alternative. 
Adaptive testing addresses this issue by combining test statistics over multiple bandwidths \citep{horowitz2001adaptive,chetverikov2021adaptive}.

A prominent example is the maximum mean discrepancy (MMD), whose estimator can be expressed as a weighted two-sample $U$-statistic with a bandwidth-dependent kernel \citep{sriperumbudur2010hilbert,gretton2012kernel}. Considering multiple bandwidths therefore leads naturally to a high-dimensional vector of $U$-statistics with size-dependent kernels. \citet{li2024optimality} establish minimax optimality and adaptivity of Gaussian-kernel-based homogeneity tests. 
Our Gaussian approximation (\cref{coro:main-kernel-inid}) provides a tool for refining the theoretical analysis of such adaptive procedures.

\paragraph{Framework:}
$\mathcal I=\{1,\dots,n\}$ and $\mathcal J=\{n+1,\dots,n+m\}$ be index sets for two independent samples
$(X_i)_{i\in\mathcal I}\stackrel{\text{i.i.d.}}{\sim}P$ and $(Y_j)_{j\in\mathcal J}\stackrel{\text{i.i.d.}}{\sim}Q$ and define the densities of $P$ and $Q$ as $f$ and $g$, respectively.
We define $N \coloneqq n + m$ and  $Z:=(Z_a)_{a=1}^{N}$ by $Z_a=X_a$ for $1\le a\le n$ and $Z_a=Y_{a-n}$ for $n<a\le N$, and set the marginals $P_a=P$ for $a\le n$ and $P_a=Q$ for $a>n$, then $Z$ is i.n.i.d.observation.
We aim to test whether two independent sample $X$ and $Y$ come from a common population or not.
Namely, we consider the following hypothesis testing problem:
\begin{align*}
    H_0 : P=Q, \quad \text{vs} \quad H_1 : P\neq Q.
\end{align*}
Let $K:\mathbb{R}^d \to \mathbb{R}$ be a bounded positive definite function.
For every positive number $h>0$, write
\begin{align*}
    \varphi_{h}(x,y) \coloneqq K_{h}(x-y), \quad \text{with} \quad K_{h}(u) \coloneqq \frac{1}{h^d}K\left( \frac{u}{h} \right).
\end{align*}
Lemma 6 in \cite{gretton2012kernel} give the following distance between $P$ and $Q$;
\begin{align*}
     \text{MMD}^2(P||Q) \coloneqq \int_{\mathbb{R}^d\times\mathbb{R}^d} \varphi_{h}(x,y) \{f(x) - g(x)\}\{f(y)-g(y)\}dxdy,
\end{align*}
which can be seen as the squared maximum mean discrepancy (MMD) between $P$ and $Q$, based on the kernel $\varphi_{h_n}$ (cf. Eq(10) in \cite{sriperumbudur2010hilbert}).
In particular, $\text{MMD}^2(P||Q) = 0$ if and only if $f=g$ a.e, provided that $\varphi_{h_n}$ is a characteristic kernel in the sense of \citep[Definition 6]{sriperumbudur2010hilbert}.
This suggests rejecting the null hypothesis when an estimator for $\text{MMD}^2(P||Q)$ is large.
As a sample analogue of $\text{MMD}^2(P||Q) $, \cite{gretton2012kernel} propose;
\begin{align*}
    & \widehat{\text{MMD}}_{h}^2 \coloneqq \sum_{(i,j)\in I_{N,2}} w_{ij} \varphi_h(Z_i,Z_j), \\
    & ~~\text{with}~~
     w_{ij} \coloneqq \frac{1\{i,j\in\mathcal{I}\}}{n(n-1)} + \frac{1\{i,j\in\mathcal{J}\}}{m(m-1)} -\frac{ 1\{(i\in\mathcal{I}, j\in\mathcal{J}) \text{ or } (i\in\mathcal{J},j\in\mathcal{I})\}}{mn}.
\end{align*}
Letting $\boldsymbol\psi_h \coloneqq \{2w_{ij}\varphi_{h}\}_{(i,j)\in I_{N,2}}$, the test statistic satisfies $\widehat{\text{MMD}}_{h}^2 = J_2(\boldsymbol\psi_h)$.

Recently, \cite{li2024optimality} have shown that the test based on the normalized version of this estimator is minimax optimal against smooth alternatives if $K$ is Gaussian kernel and $h$ is chosen appropriately.
To be precise, denote by $\mathcal{P}_d$ the set of probability density functions on $\mathbb{R}^d$.
Fix a constant $R>0$.
Given a constant $\alpha>0$ and a sequence $\rho_n$ of positive numbers tending to $0$ as $n\to\infty$, we associate the sequence of alternatives as 
\begin{align*}
    H_1(\rho_n;\alpha) \coloneqq \{f,g \in\mathcal P_d, \|f\|_{H^\alpha} \vee \|g\|_{H^\alpha} \le R, \|f-g\|_{L^2(\mathbb{R}^d)} \ge \rho_n  \}.
\end{align*}
In Theorem 5(i) in \cite{li2024optimality}, if we choose $h \asymp n^{-2/(4\alpha + d)}$, the aforementioned test is consistent for the alternative $f,g \in H_1(\rho_n; \alpha)$ as long as $\rho_n/\rho_n^*(\alpha) \to \infty$, where $\rho_n^*(\alpha)\coloneqq n^{-2\alpha/(4\alpha+d)}$.
Moreover, Theorem 5(ii) in \cite{li2024optimality}, $\liminf_{n\to\infty} \rho_n/\rho_n^*(\alpha)<\infty$, there is no consistent test against $f,g \in H_1(\rho_n; \alpha)$ for some significant level.
To conduct the test without prior knowledge on $\alpha$,  \cite{li2024optimality} considered the maximum of $J_2(\boldsymbol\psi_h)$ over a range of $h$ and showed that this test is adaptive to $\alpha>d/4$ up to a logarithmic factor; see Theorem 10 in \cite{li2024optimality}.

We conjecture that we can refine the theory of \cite{li2024optimality} using our developed Gaussian approximation (\cref{coro:main-kernel-inid}), in the same way as Theorem 4 in \cite{imai2025gaussian}.
Specifically, by virtue of \cref{coro:main-kernel-inid}, we can establish an asymptotic theory which refines Theorem 10 in \cite{li2024optimality} in three directions: (i) It does not require $\alpha \ge d/4$. (ii) The kernel function $K$ is not necessarily Gaussian. (iii) It makes distinguishable separation rate $\rho_n(\alpha)$ smaller from  $(\log\log n/n)^{2\alpha/(4\alpha +d)}$  to the non-improvable rate $(\sqrt{\log\log n}/n)^{2\alpha/(4\alpha +d)}$. 
As noted in Remark 6 in \cite{imai2025gaussian}, \cite{schrab2023mmd} have addressed (i) and (ii) under the additional assumption that  the underlying densities are bounded.

\subsection{Inference for High-dimensional Separately Exchangeable  Array}\label{exm:separate}

Separately exchangeable arrays provide a natural framework for two-way clustered data, such as observations indexed by products and markets. 
This framework encompasses row-column exchangeable models, additive random effects, and multiway clustering \citep{mccullagh2000resampling,owen2007pigeonhole,owen2012bootstrapping,cameron2011robust}. 
\citet{menzel2021bootstrap} show that sample means under such dependence may have non-Gaussian limiting distributions, while \citet{chiang2023using} establish Gaussian asymptotic theory for two-way clustered triangular arrays under suitable conditions.
In the high-dimensional setting, \citet{chiang2023inference} establish Gaussian approximations and multiplier bootstrap validity under nondegeneracy conditions on the Hájek projection.

We illustrate how our Gaussian approximation for two-sample $V$-statistics (\cref{coro:main-kernel-V-inid}) can be combined with the high-dimensional CLT for independent random vectors in \citet{CCKK22} to analyze separately exchangeable arrays without requiring the first-order Hoeffding component to dominate. The key is to decompose the sample mean into a sum of conditionally independent random vectors and a two-sample $V$-statistic, and then combine their Gaussian approximations using \cref{lem:gluing_separable_array}.

\paragraph{Framework:} 
The following argument extends the approach of \cite{chiang2023using} to the high-dimensional setting.
For $n,m\in\mathbb{N}_+$, let $D_{it}\coloneqq(D_{it,j})_{j\in[p]}$ be a $p$-dimensional separately exchangeable array, meaning that
\begin{align*}
    (D_{it})_{i\in[n],t\in[m]}\overset{d}{=}(D_{\sigma_1(i)\sigma_2(t)})_{i\in[n],t\in[m]}
\end{align*}
for any permutations $\sigma_1$ and $\sigma_2$ of the row and column indices, respectively.
We focus on the dissociated case and assume an Aldous--Hoover--Kallenberg representation of the form \citep{hoover1979relations,aldous1981representations,kallenberg2005probabilistic}
\begin{align*}
    D_{it,j}=\psi_{nm,j}(\alpha_i,\gamma_t,\varepsilon_{it}),
\end{align*}
where $(\alpha_i)_{i\in[n]}$, $(\gamma_t)_{t\in[m]}$, and $(\varepsilon_{it})_{i\in[n],t\in[m]}$ are mutually independent collections of i.i.d.~random variables, shared across coordinates $j\in[p]$.
Also define
\begin{align*}
    \hat{\theta}_{nm,j} \coloneqq \frac{1}{nm}\sum_{i=1}^n\sum_{t=1}^m D_{it,j} = \frac{1}{nm}\sum_{i=1}^n\sum_{t=1}^m \psi_{nm,j}(\alpha_{i}, \gamma_{t}, \varepsilon_{it}) .
\end{align*}
For simplicity, we assume $\mathbb{E}[D_{it,j}] = 0$ for all $j\in[p]$ without loss of generality.
We can see that
\begin{align*}
    \hat{\theta}_{nm,j} = \frac{1}{nm}\sum_{i=1}^n\sum_{t=1}^m \left\{ D_{it,j}  - \mathbb{E}[D_{it,j} \mid \alpha_i, \gamma_t] \right\}  + \frac{1}{nm}\sum_{i=1}^n\sum_{t=1}^m \mathbb{E}[D_{it,j} \mid \alpha_i, \gamma_t] \eqqcolon I_j + II_j.
\end{align*}
Although the second term $II_j$ further admits a Hoeffding-type decomposition into linear and degenerate components (cf. Eq.(3.1) in \citealp{chiang2023using}), the decomposition above is sufficient for the application of our Gaussian approximation since our Gaussian approximation is valid regardless of which component is dominant in the underlying Hoeffding decomposition and require an evaluation on quantities only in terms of kernel functions rather than Hoeffding projections, by virtue of \cref{coro:main-kernel-V-inid}. 

Let $I\coloneqq(I_j)_{j\in[p]}$, $II\coloneqq(II_j)_{j\in[p]}$, and $\mathcal F\coloneqq\sigma(\alpha_1,\dots,\alpha_n,\gamma_1,\dots,\gamma_m)$.
Conditional on $\mathcal F$, $I$ is a sum of independent random vectors, whereas $II$ is a two-sample $V$-statistic with size-dependent kernels.
We can therefore apply the high-dimensional CLT of \citet{CCKK22} to $I$ conditionally on $\mathcal F$ and our Gaussian approximation (\cref{coro:main-kernel-V-inid}) to $II$.
The following lemma combines these approximations, accounting for the difference between the conditional and unconditional Gaussian distributions of $I$.
Its proof is provided in \cref{subsec:proof:lem:gluing_separable_array}.

\begin{lemma}
\label{lem:gluing_separable_array}
Let $I, II \in \mathbb{R}^p$ and let $\mathcal{F}$ be a $\sigma$-field such that $II$ is $\mathcal{F}$-measurable.
Assume there exist random vectors $Z_I(\mathcal{F}), Z_I, Z_{II}$ such that
\[
Z_I(\mathcal{F}) \mid \mathcal{F} \sim N\bigl(0,\Sigma_I(\mathcal{F})\bigr), 
~~ 
Z_I \sim N(0,\Sigma_I), 
~~
Z_{II} \sim N(0,\Sigma_{II}),
~~ Z_I \perp\!\!\!\perp (II, Z_{II}),
\]
where $\Sigma_I(\mathcal{F})$ is $\mathcal{F}$-measurable and $\Sigma_I,\Sigma_{II}$ are deterministic positive semidefinite matrices.
Suppose that, for the nonnegative random sequences $\delta_{1,n}(\mathcal{F})$, $\delta_{2,n}(\mathcal{F})$ and the deterministic sequence $\delta_{3,n}$, the following statements hold:
\begin{align*}
    & \sup_{A\in\mathcal{R}_p}\Bigl| \mathbb{P}( I \in A \mid \mathcal{F}) - \mathbb{P}( Z_I(\mathcal{F}) \in A \mid \mathcal{F}) \Bigr|
\le \delta_{1,n}(\mathcal{F}), \\
    & \sup_{A\in\mathcal{R}_p}\Bigl| \mathbb{P}( Z_I(\mathcal{F}) \in A \mid \mathcal{F}) - \mathbb{P}( Z_I \in A ) \Bigr|
\le \delta_{2,n}(\mathcal{F}), \quad  \sup_{A\in\mathcal{R}_p}\Bigl| \mathbb{P}( II \in A ) - \mathbb{P}( Z_{II} \in A ) \Bigr|
\le \delta_{3,n}.
\end{align*}
Let $Z:=Z_I+Z_{II}$ where $Z_I$ and $Z_{II}$ are taken independent, so that $Z\sim N(0,\Sigma_I+\Sigma_{II})$.
Then,
\[
\sup_{A\in\mathcal{R}_p}\Bigl| \mathbb{P}( I+II \in A ) - \mathbb{P}( Z \in A ) \Bigr|
\le \mathbb{E}[\delta_{1,n}(\mathcal{F})]+\mathbb{E}[\delta_{2,n}(\mathcal{F})]+\delta_{3,n}.
\]
\end{lemma}

Given \cref{lem:gluing_separable_array}, under suitable regularity conditions, one can establish Gaussian approximation of $\hat{\theta}_{nm,j}$ by combination of (i) \cref{lem:gluing_separable_array}, (ii) the high-dimensional CLT for sum of independent variables (Theorem 2.1 in \citealp{CCKK22}), (iii) Gaussian-to-Gaussian comparison inequality (Proposition 2.1 in \citealp{CCKK22}) and (iv) our Gaussian approximation result for $V$-statistics (\cref{coro:main-kernel-V-inid}).
In particular, the high-dimensional CLT for the sum of independent random vectors (Theorem~2.1 in \citealp{CCKK22}) to $I\mid\mathcal{F}$ gives $\delta_{1,n}(\mathcal{F})$, applying  Gaussian-to-Gaussian comparison bound (Proposition~2.1 in \citealp{CCKK22}) to $N(0,\Sigma_I(\mathcal{F}))$ and $N(0,\Sigma_I)$ gives $\delta_{2,n}(\mathcal{F})$, and applying our Gaussian approximation for the $V$-statistic with evaluation of quantities in terms of the kernel function (\cref{coro:main-kernel-V-inid}) to $II$ gives $\delta_{3,n}$.

\section{Concluding Remarks} \label{sec:conclusion}

This paper establishes Gaussian approximations for high-dimensional vectors of weighted and two-sample second-order $U$- and $V$-statistics with size-dependent kernels.
Our framework does not require either the linear or quadratic Hoeffding component to be asymptotically negligible.
A key ingredients are maximal inequalities that preserve summations over observation indices.

We apply our results to regression-based inference with many weak instruments \citep{chao2012asymptotic} and jackknife Anderson--Rubin tests \citep{mikusheva2022inference}, extending their asymptotic frameworks to high-dimensional simultaneous inference.
In these applications, we establish Gaussian approximations and uniform consistency of variance estimators under primitive conditions.
Notably, our sufficient conditions for Gaussian approximation and uniform consistency of variance estimators nearly recover, up to logarithmic factors, the corresponding rate conditions in the fixed-dimensional settings of \cite{chao2012asymptotic} and \cite{mikusheva2022inference}.
This suggests that our Gaussian approximation bounds and maximal inequalities are sufficiently sharp to preserve the essential rate conditions when extending these applications to high-dimensions.
Also, for jackknife AR tests, we further establish the validity of a quadratic multiplier bootstrap under both null and alternative hypotheses, derive family-wise error control and power properties, and investigate its finite-sample performance through Monte Carlo simulations.

We also discuss applications to adaptive tests of distributional homogeneity \citep{li2024optimality} and inference for separately exchangeable arrays \citep{chiang2023using}.
These examples illustrate how our general results can be used in conjunction with existing methods for problems involving two-sample $U$- and $V$-statistics.

Problems involving weighted and/or two-sample $U$- and $V$-statistics with size-dependent kernels extend well beyond the examples discussed in this paper. 
We expect our results to serve as a foundation for further methodological developments across a growing range of applications.

The present version remains a work in progress, and we plan several extensions in future revisions.
First, we will develop bootstrap procedures with theoretical guarantees for regression-based inference with many-weak IVs.
Second, we will provide detailed Gaussian approximation and bootstrap results for specific applications involving two-sample $U$- and $V$-statistics.
Third, we will conduct numerical experiments of bootstrap procedure for these applications.
Given the length of the paper, further methodological refinements motivated by empirical applications, together with empirical analyses, will be pursued in companion papers.

\newpage

\setcounter{page}{1}
\renewcommand{\thepage}{SI-\arabic{page}}

\setcounter{section}{0}
\renewcommand{\thesection}{SI-\arabic{section}}

\vspace{1cm}
\begin{center}
{\LARGE
{\bf Supplementary Materials I : Proof of Gaussian Approximations and Maximal Inequalities}
}
\end{center}

\etocdepthtag.toc{suppI}
\supplementtoc{suppI}

\section{Auxiliary Results} \label{sec:auxiliary-inid}

\subsection{High-dimensional CLT via generalized exchangeable pairs} \label{subsec:gexch-inid}
Since Theorem~6 in \cite{imai2025gaussian} does not assume identical distributions of observations, we can use the theorem even under i.n.i.d.\ setting.
For convenience, we restate the theorem below. 

\begin{lemma}[Theorem~6 in \cite{imai2025gaussian}]\label{lem:gexch-restate}
Let $(Y,Y')$ be an exchangeable pair of random variables taking values in a measurable space $(E,\mcl E)$. 
Let $\mathsf{W}:E\to\mathbb R^p$ be $\mcl E$-measurable and set $W:=\mathsf{W}(Y)$, $W':=\mathsf{W}(Y')$, and $D:=W'-W$. 
Suppose there exists an antisymmetric $\mcl E^{\otimes2}$-measurable function $\mathsf{G}:E^2\to\mathbb R^p$ (i.e.\ $\mathsf{G}(Y,Y')=-\mathsf{G}(Y',Y)$) such that $G:=\mathsf{G}(Y,Y')$ satisfies
\begin{equation}\label{eq:nlr}
\E[\,G\mid Y\,]=-(W+R)
\end{equation}
for some random vector $R\in\mathbb R^p$. 
Let $\Sigma$ be a $p\times p$ positive semidefinite symmetric matrix and write $\ul\sigma:=\min_{j\in[p]}\sqrt{\Sigma_{jj}}>0$. 
Then there exists a universal constant $C>0$ such that for any $\eps>0$,
\begin{align}\label{eq:gech}
& \sup_{A\in\mathcal R_p}\bigl|\pr(W\in A)-\pr(Z\in A)\bigr|\notag\\
& \;\le\; \frac{C}{\ul\sigma}\Bigl\{
\E\sbra{\|R^\eps\|_\infty}\sqrt{\log p}
+\eps^{-1}\E\sbra{\|V^\eps\|_\infty}(\log p)^{3/2}
+\eps^{-3}\E\sbra{\Gamma^\eps}(\log p)^{7/2}
+\eps\sqrt{\log p}\Bigr\},
\end{align}
where $Z\sim N(0,\Sigma)$ and $\beta:=\eps^{-1}\log p$, and we define
\[
R^\eps \coloneqq R+\E\!\bigl[\,G\,\mathbf 1_{\{\|D\|_\infty>\beta^{-1}\}}\mid Y\bigr],\qquad
V^\eps \coloneqq \frac{1}{2}\E\!\bigl[\,GD^\top\,\mathbf 1_{\{\|D\|_\infty\le \beta^{-1}\}}\mid Y\bigr]-\Sigma,
\]
\[
\Gamma^\eps \coloneqq\max_{j,k,l,m\in[p]}\E\!\bigl[\,|G_jD_kD_lD_m|\,\mathbf 1_{\{\|D\|_\infty\le \beta^{-1}\}}\mid Y\bigr].
\]
\end{lemma}

\subsection{Maximal Inequality} \label{subsec:inid:maximal-inequality}

The following inequalities are the extension of Lemma~3 in \cite{imai2025gaussian} to i.n.i.d. setting, which plays important role in evaluation of the Gaussian approximation error below.

\begin{lemma} \label{lem:max-rosenthal-inid}
For $j\in[p]$ and $(i,m)\in I_{n,2}$ , let $\psi_{j,(i,m)} \in L^4(P_i\otimes P_m)$ be degenerate symmetric kernels of order $2$.
There exists a universal constant $C$ such that
\begin{align}
      & \E\left[ \max_{j\in[p]} \sum_{i=1}^n \int \left| \sum_{m\in[n]:m>i} \psi_j(x_i,X_m) \right|^4 dP_i(x_i) \right] \nonumber  \\
      & \le C \Bigg(\max_{j\in[p]} \sum_{i=1}^n \left\|\sum_{m\neq i}P_m(\psi_j^2)\right\|_{L^2(\mathbb P)}^2  \log^2 p \nonumber \\
      & \quad\quad + \max_{j\in[p]} \sum_{i=1}^n\sum_{m\neq i} \|\psi_j\|_{L^4(P_i\otimes P_m)}^4 \log^3 p + \E\left[ \max_{j\in[p]} \max_{m\in[n]} \sum_{i\neq m} P_i\psi^4_j(X_m) \right] \log^4 p\Bigg) ,\label{eq:influence-mom-inid}
\end{align}
and
\begin{align}
      &  \E\left[ \max_{j\in[p]} \sum_{i=1}^n \int \left| \sum_{m\in[n]:m<i} \psi_j(x_i,X_m) \right|^4 dP_i(x_i) \right] \nonumber  \\
      & \le C \Bigg( \max_{j\in[p]} \sum_{i=1}^n \left\|\sum_{m\neq i}P_m(\psi_j^2)\right\|_{L^2(\mathbb P )}^2  \log^2 p \nonumber \\
      & \quad\quad + \max_{j\in[p]} \sum_{i=1}^n\sum_{m\neq i} \|\psi_j\|_{L^4(P_i\otimes P_m)}^4 \log^3 p + \E\left[ \max_{j\in[p]} \max_{m\in[n]} \sum_{i\neq m} P_i\psi^4_j(X_m) \right] \log^4 p\Bigg) ,\label{eq:influence-mom-inid-<}
\end{align}
and
\begin{align} 
    & \E\left[ \max_{j\in[p]} \sum_{i=1}^n \left| \sum_{m\in[n]:m\neq i} \psi_j(X_i,X_m) \right|^4  \right] \nonumber\\
    & \le C\Bigg( \max_{j\in[p]} \sum_{i=1}^n \left\|\sum_{m\neq i}P_m(\psi_j^2)\right\|_{L^2(\mathbb P)}^2  \log^2 p \nonumber \\
      & \quad + \max_{j\in[p]} \sum_{i=1}^n\sum_{m\neq i} \|\psi_j\|_{L^4(P_i\otimes P_m)}^4 \log^3 p + \E\left[ \max_{j\in[p]} \max_{m\in[n]} \sum_{i\neq m} P_i\psi^4_j(X_m) \right] \log^4 p \nonumber\\
      & \quad+ \left\|  \max_{j\in[p]} \max_{i\in[n]} \sum_{m\neq i}   P_m \psi_j^2 (X_i) \right\|_{L^2(\mathbb P )}^2\log^3(np) + \E\left[ \max_{j\in[p]} \max_{(i,m)\in I_{n,2}}   \psi_j^4(X_i,X_m) \right] \log^5(np)  \Bigg). \label{eq:influence-mom-inid-2}
\end{align}
\end{lemma}

\section{Proof of Main Result} \label{sec:Proof-main-results-inid}

\subsection{Proof of \cref{thm:main-inid}} \label{subsec:proof:thm:main-inid}

We may assume $\sigma_j=1$ for all $j\in[p]$ without loss of generality.  
For the rest of the proof, we proceed in 5 steps.
\paragraph{Step 1} 
Regarding $X=(X_i)_{i=1}^n$ as a random element taking values in the measurable space $\otimes_{i=1}^n(S_i,\mcl S_i)$, we are going to apply \cref{lem:gexch-restate} to 
\[
\mathsf{W}(X):=\bra{J_{2,X}(\boldsymbol\psi_{1})-\E[J_{2,X}(\boldsymbol\psi_{1})],\dots,J_{2,X}(\boldsymbol\psi_{p})-\E[J_{2,X}(\boldsymbol\psi_{p})]}^\top.
\]
For this purpose, we need to construct an appropriate exchangeable pair $(X,X')$ and an antisymmetric function $\mathsf{G}$. 
Let $X^*=(X^*_i)_{i=1}^n$ be an independent copy of $X=(X_i)_{i=1}^n$. 
Also, let $\alpha$ be a random index uniformly distributed on $[n]$ and such that $X,X^*$ and $\alpha$ are independent. 
Then, define $X'=(X'_i)_{i=1}^n$ as $X_i':=X_i^*$ if $i=\alpha$ and $X_i':=X_i$ otherwise.
It is well-known that $(X,X')$ is an exchangeable pair. 

In addition, define a random vector $G = \mathsf{G}(X,X')$ in $\mathbb{R}^p$ as $G_j \coloneqq n D_{j,1} + \frac{n}{2}D_{j,2}$ for $j = 1,\dots, p$, where
\begin{align} \label{eq:def-D}
\begin{split}
    & D_{j,1} \coloneqq \sum_{i=1}^n \{\pi_{1,i} \psi_j(X_i') - \pi_{1,i} \psi_j(X_i)\}, \\
    & D_{j,2} \coloneqq \sum_{1\le i<m \le n} \{\pi_{2,im} \psi_j(X_i',X_m') - \pi_{2,im} \psi_j(X_i,X_m)\}.
\end{split}
\end{align}
$ \mathsf{G}$ is antisymmetric by construction. 
From Lemma 3.2 in \cite{DoPe17}, it holds that
\begin{align*}
    \E[D_{j,1} \mid X] = - \frac{1}{n}\sum_{i=1}^n\pi_{1,i} \psi_j(X_i), \quad \E[D_{j,2} \mid X] = -\frac{2}{n} \sum_{1\le i<m \le n}\pi_{2,im} \psi_j(X_i,X_m),
\end{align*}
so we have $\E[G\mid X] = -W$.
Therefore,  applying \cref{lem:gexch-restate}, we obtain for any $\varepsilon>0$
\begin{align*}
    & \sup_{A\in\mathcal{R}_p} |\pr(W\in A) - \pr(Z\in A)| \\
    & \lesssim 
\E\sbra{\|R^\eps\|_\infty}\sqrt{\log p}
+\eps^{-1}\E\sbra{\|V^\eps\|_\infty}(\log p)^{3/2}
+\eps^{-3}\E\sbra{\Gamma^\eps}(\log p)^{7/2}
+\eps\sqrt{\log p},
\end{align*}
where $R^\eps$, $V^\eps$ and $\Gamma^\eps$ are defined in the same way as in \cref{lem:gexch-restate} with $R=0$ and $(Y,Y')$ replaced by $(X,X')$.
From the same argument as proof of Theorem~2 in \cite{imai2025gaussian}, 
\begin{align*}
    &  \sup_{A\in\mathcal{R}_p} |\pr(W\in A) - \pr(Z\in A)|\\
& \lesssim \frac{n\sqrt{\log p}}{\beta \wedge \beta^3} \pr(\|D\|_\infty > \beta^{-1}) + \varepsilon^{-1}\E[\|V\|_\infty](\log p)^{3/2} + \varepsilon^{-3}\E[\Gamma_1 + \Gamma_2](\log p)^{7/2} + \varepsilon\sqrt{\log p},
\end{align*}
where $V \coloneqq \frac{1}{2}\E[GD^\top \mid X] - \Sigma$ and $\Gamma_s \coloneqq n \max_{j\in[p]} \E[|D_{j,s}|^4 \mid X]$ for $s=1,2$. 
In the remaining proof, we will bound the quantities on the right-hand side and then choose $\varepsilon$ appropriately. 

\paragraph{Step 2}
In this step, we bound $\E[\|V\|_\infty](\log p)^{3/2}$.
First, we derive a $U$-statistic representation of $V_{jk} = \frac{1}{2}\E[G_jD_k\mid X] - \E[W_jW_k]$ for $j,k\in[p]$.
For $j,k\in[p]$, observe that
\begin{align*}
    G_j D_k = \left(n D_{j,1} + \frac{n}{2}D_{j,2} \right) \left( D_{k,1} + D_{k,2}\right).
\end{align*}
From \cref{eq:def-D}, we can see that
\begin{align*}
    & n\E[D_{j,1}D_{k,1} \mid X] \\
    &= n\E\left[ \sum_{i=1}^n\sum_{m=1}^n \{\pi_{1,i} \psi_j(X_i') - \pi_{1,i} \psi_j(X_i)\} \{\pi_{1,m} \psi_k(X_m') - \pi_{1,m} \psi_k(X_m)\} \mid X \right] \\
    &= n\E\left[ \sum_{i=1}^n\sum_{m=1}^n 1_{\{\alpha=i\}} 1_{\{\alpha=m\}} \{\pi_{1,i} \psi_j(X_i^*) - \pi_{1,i} \psi_j(X_i)\} \{\pi_{1,m} \psi_k(X_m^*) - \pi_{1,m} \psi_k(X_m)\} \mid X \right] \\
    &= n\sum_{i=1}^n \E\left[ 1_{\{\alpha=i\}} \{\pi_{1,i} \psi_j(X_i^*) - \pi_{1,i} \psi_j(X_i)\}\{\pi_{1,i} \psi_k(X_i^*) - \pi_{1,i} \psi_k(X_i)\} \mid X\right] \\
    &= n\sum_{i=1}^n \E[1_{\{\alpha=i\}}] \Big( \E[\pi_{1,i} \psi_j(X_i^*) \pi_{1,i} \psi_k(X_i^*)] + \pi_{1,i} \psi_j(X_i)\pi_{1,i} \psi_k(X_i)   \Big) \\
    &= \sum_{i=1}^n\Big( P_i\{\pi_{1,i} \psi_j(X_i^*) \pi_{1,i} \psi_k(X_i^*)\} + \pi_{1,i} \psi_j(X_i)\pi_{1,i} \psi_k(X_i)   \Big) \\
    &= \sum_{i=1}^n\Big( P_i\{\pi_{1,i} \psi_j(X_i) \pi_{1,i} \psi_k(X_i)\} + \pi_{1,i} \psi_j(X_i)\pi_{1,i} \psi_k(X_i)   \Big) ,
\end{align*}
where the second equality follows from the definition of $(X_i')_{i=1}^n$, the third equality holds because the event $\{\alpha=i\}$ and $\{\alpha=m\}$ are disjoint for $i\neq m$,  the fourth equality follows from the degeneracy of Hoeffding projections ($\E[\pi_{1,i}\psi_j(X_i^*)]=0$), the fifth equality holds since $\alpha$ is uniformly distributed on $[n]$, and the final equality follows from the fact that $(X_i^*)_{i=1}^n$ is an independent copy of $(X_i)_{i=1}^n$.

Similarly, from \cref{eq:def-D}, we can see that
\begin{align*}
    & n\E[D_{j,1}D_{k,2} \mid X] \\
    & = n \E\Bigg[ \left(   \sum_{i=1}^n \{\pi_{1,i} \psi_j(X_i') - \pi_{1,i} \psi_j(X_i)\}\right) \left( \sum_{1\le i<m \le n} \{\pi_{2,im} \psi_k(X_i',X_m') - \pi_{2,im} \psi_k(X_i,X_m)\}\right)   \mid X\Bigg] \\
    & = n \E\Bigg[  \left(   \sum_{i=1}^n 1_{\{\alpha=i\}}\{\pi_{1,i} \psi_j(X_i^*) - \pi_{1,i} \psi_j(X_i)\}\right) \Bigg( \sum_{1\le i<m \le n} 1_{\{\alpha=i\}}\{\pi_{2,im} \psi_k(X_i^*,X_m) - \pi_{2,im} \psi_k(X_i,X_m)\} \\
    &\qquad\qquad\qquad + \sum_{1\le i<m \le n} 1_{\{\alpha=m\}}\{\pi_{2,im} \psi_k(X_i,X_m^*) - \pi_{2,im} \psi_k(X_i,X_m)\} \Bigg) \mid X\Bigg] \\
    & = n \sum_{1\le i<m \le n} \E\left[ 1_{\{\alpha=i\}}\{\pi_{1,i} \psi_j(X_i^*) - \pi_{1,i} \psi_j(X_i)\}\{\pi_{2,im} \psi_k(X_i^*,X_m) - \pi_{2,im} \psi_k(X_i,X_m)\}  \mid  X\right] \\
    & \quad + n \sum_{1\le i<m \le n} \E\left[ 1_{\{\alpha=m\}}\{\pi_{1,m} \psi_j(X_m^*) - \pi_{1,m} \psi_j(X_m)\}\{\pi_{2,im} \psi_k(X_i,X_m^*) - \pi_{2,im} \psi_k(X_i,X_m)\}  \mid  X\right] \\
    & = \sum_{1\le i<m \le n} \Big( P_i\{\pi_{1,i}\psi_j(X_i^*)\pi_{2,im}\psi_k(X_i^*,X_m)\} + \pi_{1,i}\psi_j(X_i)\pi_{2,im}\psi_k(X_i,X_m)  \Big) \\
    & \quad + \sum_{1\le i<m \le n} \Big( P_m\{\pi_{1,m}\psi_j(X_m^*)\pi_{2,im}\psi_k(X_i,X_m^*)\} + \pi_{1,m}\psi_j(X_m)\pi_{2,im}\psi_k(X_i,X_m)  \Big) \\
    & = \sum_{i=1}^n\sum_{m\neq i} \Big( P_i\{\pi_{1,i}\psi_j(X_i)\pi_{2,im}\psi_k(X_i,X_m)\} + \pi_{1,i}\psi_j(X_i)\pi_{2,im}\psi_k(X_i,X_m)  \Big),
\end{align*}
where the final equality follows from the symmetric property of the Hoeffding projections ($\pi_{2,im}\psi_k(X_i,X_m) = \pi_{2,mi}\psi_k(X_m,X_i)$).
In the same way, it holds that
\begin{align*}
    & \frac{n}{2}\E[D_{j,2}D_{k,1} \mid X] \\
    &= \frac{1}{2} \sum_{i=1}^n\sum_{m\neq i} \Big( P_i\{\pi_{1,i}\psi_k(X_i)\pi_{2,im}\psi_j(X_i,X_m)\} + \pi_{1,i}\psi_k(X_i)\pi_{2,im}\psi_j(X_i,X_m) \Big).
\end{align*}
Also, from \cref{eq:def-D}, we can see that
\begin{align*}
    & \frac{n}{2}\E[D_{j,2}D_{k,2} \mid X] \\
    & = \frac{n}{2}\E\Bigg[ \left( \sum_{1\le i<m \le n} \{\pi_{2,im} \psi_j(X_i',X_m') - \pi_{2,im} \psi_j(X_i,X_m)\}\right) \\
    & \qquad\qquad\qquad\times \left( \sum_{1\le i<m \le n} \{\pi_{2,im} \psi_k(X_i',X_m') - \pi_{2,im} \psi_k(X_i,X_m)\}\right) \mid X  \Bigg] \\
    & = \frac{n}{2}\E\Bigg[ \bigg( \sum_{1\le i<m \le n} 1_{\{\alpha=i\}}\{\pi_{2,im} \psi_j(X_i^*,X_m) - \pi_{2,im} \psi_j(X_i,X_m)\} \\
    & \qquad\qquad\qquad\qquad\qquad+ 1_{\{\alpha=m\}}\{\pi_{2,im} \psi_j(X_i,X_m^*) - \pi_{2,im} \psi_j(X_i,X_m)\}\bigg) \\
    & \quad\quad\quad\quad\times \bigg( \sum_{1\le i<m \le n} 1_{\{\alpha=i\}}\{\pi_{2,im} \psi_k(X_i^*,X_m) - \pi_{2,im} \psi_k(X_i,X_m)\} \\
    & \qquad\qquad\qquad\qquad\qquad\qquad+ 1_{\{\alpha=m\}}\{\pi_{2,im} \psi_k(X_i,X_m^*) - \pi_{2,im} \psi_k(X_i,X_m)\}\bigg) \mid X \Bigg] \\
    & = \frac{n}{2} \sum_{i}\sum_{m:i<m}\sum_{l:i<l} \E\Bigg[ 1_{\{\alpha=i\}}\{\pi_{2,im} \psi_j(X_i^*,X_m) - \pi_{2,im} \psi_j(X_i,X_m)\}\{\pi_{2,il} \psi_k(X_i^*,X_l) - \pi_{2,il} \psi_k(X_i,X_l)\} \mid X  \Bigg]\\
    & ~ + \frac{n}{2} \sum_{i}\sum_{m:i<m}\sum_{l:l<i} \E\Bigg[ 1_{\{\alpha=i\}}\{\pi_{2,im} \psi_j(X_i^*,X_m) - \pi_{2,im} \psi_j(X_i,X_m)\}\{\pi_{2,li} \psi_k(X_l,X_i^*) - \pi_{2,li} \psi_k(X_l,X_i)\} \mid X  \Bigg] \\
    & ~ + \frac{n}{2} \sum_{i}\sum_{m:i<m} \sum_{l:m<l} \\
    &\qquad\qquad\qquad\E\Bigg[  1_{\{\alpha=m\}}\{\pi_{2,im} \psi_j(X_i,X_m^*) - \pi_{2,im} \psi_j(X_i,X_m)\} \{\pi_{2,ml} \psi_k(X_m^*,X_l) - \pi_{2,ml} \psi_k(X_m,X_l)\} \mid X\Bigg] \\
    & ~ + \frac{n}{2} \sum_{i}\sum_{m:i<m} \sum_{l:l<m} \\
    &\qquad\qquad\qquad\E\Bigg[  1_{\{\alpha=m\}}\{\pi_{2,im} \psi_j(X_i,X_m^*) - \pi_{2,im} \psi_j(X_i,X_m)\} \{\pi_{2,lm} \psi_k(X_l,X_m^*) - \pi_{2,lm} \psi_k(X_l,X_m)\} \mid X \Bigg] \\
    & = \frac{1}{2} \sum_{i=1}^n\sum_{m\neq i}\sum_{l\neq i} \Big( P_i\{\pi_{2,im} \psi_j(X_i^*,X_m)\pi_{2,il} \psi_k(X_i^*,X_l)\} + \pi_{2,im} \psi_j(X_i,X_m)\pi_{2,il} \psi_k(X_i,X_l) \Big). 
\end{align*}
Then, since $\E[W_jW_k] = \frac{1}{2}\E[G_jD_k]$, we have
\begin{align*}
    2V_{jk} = \sum_{i=1}^n  \chi^{(1,1)}_{i,jk}(X_i) + \sum_{i=1}^n\sum_{m\neq i} \{\chi^{(1,2)}_{im,jk}(X_i,X_m) + \chi^{(2,1)}_{im,jk}(X_i,X_m)\} +  \sum_{i=1}^n\sum_{m\neq i}\sum_{l\neq i} \chi_{iml,jk}^{(2,2)}(X_i,X_m,X_l),
\end{align*}
with
\begin{align*}
    & \chi^{(1,1)}_{i,jk}(X_i) \coloneqq  \left( \pi_{1,i}\psi_{j}(X_i)\pi_{1,i}\psi_{k}(X_i) - P_i\{ \pi_{1,i}\psi_{j}(X_i)\pi_{1,i}\psi_{k}(X_i)\}  \right), \\
    & \chi^{(1,2)}_{im,jk}(X_i,X_m) \coloneqq   \pi_{1,i}\psi_j(X_i)\pi_{2,im}\psi_k(X_i,X_m) + P_i\{ \pi_{1,i}\psi_j(X_i)\pi_{2,im}\psi_k(X_i,X_m)\}, \\
    & \chi^{(2,1)}_{im,jk}(X_i,X_m) \coloneqq \frac{1}{2}\left( \pi_{1,i}\psi_k(X_i)\pi_{2,im}\psi_j(X_i,X_m) + P_i\{ \pi_{1,i}\psi_k(X_i)\pi_{2,im}\psi_j(X_i,X_m)\} \right), \\
    & \chi^{(2,2)}_{iml,jk}(X_i,X_m,X_l)  \coloneqq \frac{1}{2}  \bigl( \pi_{2,im}\psi_j(X_i,X_m)\pi_{2,il}\psi_k(X_i,X_l) - P_iP_mP_l\{\pi_{2,im}\psi_j(X_i,X_m)\pi_{2,il}\psi_k(X_i,X_l)\}  \\
    & \qquad\qquad\qquad + P_i\{\pi_{2,im}\psi_j(X_i,X_m)\pi_{2,il}\psi_k(X_i,X_l)\} - P_lP_m[P_i\{\pi_{2,im}\psi_j(X_i,X_m)\pi_{2,il}\psi_k(X_i,X_l)\}]\bigr) .
\end{align*}
Observe that
\begin{align*}
    & \chi^{(1,2)}_{im,jk}(X_i,X_m) \\  &= \Big( \pi_{1,i}\psi_j(X_i)\pi_{2,im}\psi_k(X_i,X_m) - P_i\{ \pi_{1,i}\psi_j(X_i)\pi_{2,im}\psi_k(X_i,X_m)\} \Big)  + 2P_i\{ \pi_{1,i}\psi_j(X_i)\pi_{2,im}\psi_k(X_i,X_m)\} \\
    & \eqqcolon \bar{\chi}_{im,jk}^{(1,2)}(X_i,X_m) + 2P_i\{ \pi_{1,i}\psi_j(X_i)\pi_{2,im}\psi_k(X_i,X_m)\},
\end{align*}
and
\begin{align*}
    &\chi^{(2,1)}_{im,jk}(X_i,X_m) \\
    &= \frac{1}{2}\Big( \pi_{1,i}\psi_k(X_i)\pi_{2,im}\psi_j(X_i,X_m) - P_i\{ \pi_{1,i}\psi_k(X_i)\pi_{2,im}\psi_j(X_i,X_m)\} \Big)  + P_i\{ \pi_{1,i}\psi_k(X_i)\pi_{2,im}\psi_j(X_i,X_m)\} \\
    & \eqqcolon \bar{\chi}^{(2,1)}_{im,jk}(X_i,X_m) + P_i\{ \pi_{1,i}\psi_k(X_i)\pi_{2,im}\psi_j(X_i,X_m)\},
\end{align*}
and
\begin{align*}
    &\chi^{(2,2)}_{iml,jk}(X_i,X_m,X_l) \\
    &= \frac{1}{2}\Big(  \pi_{2,im}\psi_j(X_i,X_m)\pi_{2,il}\psi_k(X_i,X_l) - P_i\{\pi_{2,im}\psi_j(X_i,X_m)\pi_{2,il}\psi_k(X_i,X_l)\} \Big) \\
    & \quad + \Big( P_i\{\pi_{2,im}\psi_j(X_i,X_m)\pi_{2,il}\psi_k(X_i,X_l)\} - P_iP_mP_l\{\pi_{2,im}\psi_j(X_i,X_m)\pi_{2,il}\psi_k(X_i,X_l)\}  \Big) \\
    & \eqqcolon \bar\chi^{(2,2)}_{iml,jk}(X_i,X_m,X_l)  +  \Big( P_i\{\pi_{2,im}\psi_j(X_i,X_m)\pi_{2,il}\psi_k(X_i,X_l)\} - P_iP_mP_l\{\pi_{2,im}\psi_j(X_i,X_m)\pi_{2,il}\psi_k(X_i,X_l)\}  \Big).
\end{align*}
Therefore, letting
\begin{align*}
    & \varphi^{(1,2)}_{m,jk}(X_m) \coloneqq \sum_{i:m\neq i}2P_i\{ \pi_{1,i}\psi_j(X_i)\pi_{2,im}\psi_k(X_i,X_m)\}, \\
    & \varphi_{m,jk}^{(2,1)}(X_m) \coloneqq \sum_{i:m\neq i}P_i\{ \pi_{1,i}\psi_k(X_i)\pi_{2,im}\psi_j(X_i,X_m)\}, \\
    & \varphi^{(2,2)}_{ml,jk}(X_m,X_l) \coloneqq \sum_{i:m,l\neq i} \Big( P_i\{\pi_{2,im}\psi_j(X_i,X_m)\pi_{2,il}\psi_k(X_i,X_l)\} - P_iP_mP_l\{\pi_{2,im}\psi_j(X_i,X_m)\pi_{2,il}\psi_k(X_i,X_l)\}  \Big),
\end{align*}
we have
\begin{align*}
    2V_{jk} &= \sum_{i=1}^n  \chi^{(1,1)}_{i,jk}(X_i) + \sum_{i=1}^n\sum_{m\neq i} \{\bar\chi^{(1,2)}_{im,jk}(X_i,X_m) + \bar\chi^{(2,1)}_{im,jk}(X_i,X_m)\} +  \sum_{i=1}^n\sum_{m\neq i}\sum_{l\neq i} \bar\chi_{iml,jk}^{(2,2)}(X_i,X_m,X_l) \\
    & \quad + \sum_{m=1}^n \{\varphi^{(1,2)}_{m,jk}(X_m) + \varphi^{(2,1)}_{m,jk}(X_m)\} + \sum_{m=1}^n\sum_{l=1}^n \varphi^{(2,2)}_{ml,jk}(X_m,X_l).
\end{align*}
Also, define
\begin{align*}
    & \tilde{\chi}^{(1,2)}_{im,jk}(X_i,X_m) \coloneqq \frac{1}{2}\{\bar\chi_{im,jk}^{(1,2)}(X_i,X_m) + \bar\chi_{mi,jk}^{(1,2)}(X_m,X_i)\}, \\
    & \tilde{\chi}^{(2,1)}_{im,jk}(X_i,X_m) \coloneqq \frac{1}{2}\{\bar\chi_{im,jk}^{(2,1)}(X_i,X_m) + \bar\chi_{mi,jk}^{(2,1)}(X_m,X_i)\}, \\
    & \tilde{\chi}^{(2,2),\mathtt{diag}}_{im,jk}(X_i,X_m) \coloneqq \frac{1}{2}\big[ \pi_{2,im}\psi_j(X_i,X_m)\pi_{2,im}\psi_k(X_i,X_m) - P_i\{\pi_{2,im}\psi_j(X_i,X_m)\pi_{2,im}\psi_k(X_i,X_m)\} \\
    & \qquad\qquad\qquad - P_m\{\pi_{2,im}\psi_j(X_i,X_m)\pi_{2,im}\psi_k(X_i,X_m)\} + P_iP_m\{\pi_{2,im}\psi_j(X_i,X_m)\pi_{2,im}\psi_k(X_i,X_m)\}\big], \\
    & \tilde{\chi}^{(2,2)}_{iml,jk}(X_i,X_m,X_l) \coloneqq \frac{1}{6} \sum_{\sigma \in \mathbb{S}^3} \bar\chi^{(2,2)}_{\sigma(1)\sigma(2)\sigma(3),jk} (X_{\sigma(1)}, X_{\sigma(2)}, X_{\sigma(3)}), \\
    & \tilde{\varphi}_{ml,jk}^{(2,2)}(X_m,X_l) \coloneqq \frac{1}{2}\{\varphi^{(2,2)}_{ml,jk}(X_m,X_l) + \varphi^{(2,2)}_{lm,jk}(X_l,X_m)\}, \\
    & \varphi_{m,jk}^{(2,2),\mathtt{diag}}(X_m) \coloneqq  \sum_{i:m\neq i} \Big( P_i\{\pi_{2,im}\psi_j(X_i,X_m)\pi_{2,im}\psi_k(X_i,X_m)\} - P_iP_m\{\pi_{2,im}\psi_j(X_i,X_m)\pi_{2,i m}\psi_k(X_i,X_m)\}  \Big),
\end{align*}
where $\mathbb{S}^3$ is the symmetric group of degree $3$.
We use boldface to indicate families of kernels over their sample coordinates; for example $\boldsymbol{\tilde{\chi}}_{jk}^{(1,2)}\coloneqq\{\tilde{\chi}^{(1,2)}_{im,jk}(X_i,X_m)\}_{(i,m)\in I_{n,2}}$.
Noting that $\max_{(j,k)\in[p]^2} J_2(\boldsymbol{\tilde{\chi}}^{(2,1)}_{jk}) = \frac{1}{2} \max_{(j,k)\in[p]^2} J_2(\boldsymbol{\tilde{\chi}^{(1,2)}}_{jk})$ and  $\max_{(j,k)\in[p]^2} J_1(\boldsymbol{\varphi}_{jk}^{(2,1)})=\frac{1}{2}\max_{(j,k)\in[p]^2} J_1(\boldsymbol{\varphi}_{jk}^{(1,2)})$,
we can see that
\begin{align*}
    & \E[\|2V\|_\infty] \\
    & \lesssim \E\left[ \max_{(j,k)\in[p]^2} |J_1(\boldsymbol\chi_{jk}^{(1,1)})| \right] + \E\left[ \max_{(j,k)\in[p]^2} | J_2(\boldsymbol{\tilde{\chi}}_{jk}^{(1,2)})| \right]\\
    & \quad + \E\left[ \max_{(j,k)\in[p]^2} |J_2(\tilde{\boldsymbol\chi}^{(2,2), \mathtt{diag}}_{jk})| \right]  + \E\left[ \max_{(j,k)\in[p]^2} |J_3(\tilde{\boldsymbol\chi}_{jk}^{(2,2)})| \right] \\
    & \quad + \E\left[ \max_{(j,k)\in[p]^2} |J_1(\boldsymbol\varphi_{jk}^{(1,2)})| \right] +\E\left[ \max_{(j,k)\in[p]^2} |J_2(\tilde{\boldsymbol\varphi}^{(2,2)}_{jk})| \right] +  \E\left[ \max_{(j,k)\in[p]^2} |J_1(\boldsymbol\varphi^{(2,2), \mathtt{diag}}_{jk})| \right].
\end{align*}

In order to obtain a simple bound on $\E[\|V\|_\infty]$, we introduce the following lemma. 
The proof of this lemma is in \cref{subsec:proof-lem:inid:eval-step2}.
\begin{lemma} \label{lem:inid:eval-step2}
Under the assumptions for \cref{thm:main-inid}, there exists a universal constant $C$ such that
\begin{align}
    & \E\left[ \max_{(j,k)\in[p]^2} |J_1(\boldsymbol\chi_{jk}^{(1,1)})| \right] (\log p)^{3/2} \le  C  \sqrt{\left(\Delta_{2,*}^{(1)}(1) +\Delta_{2,*}^{(2)}(1)\right)\log^5 (np)}, \label{eq:inid:eval-step2-1}\\
    & \E\left[ \max_{(j,k)\in[p]^2} | J_2(\tilde{\boldsymbol\chi}_{jk}^{(1,2)})| \right](\log p)^{3/2} +\E\left[ \max_{(j,k)\in[p]^2} | J_1(\boldsymbol\varphi_{jk}^{(1,2)})| \right] (\log p)^{3/2} \nonumber\\
    & \quad \le C \Bigg(  \sqrt{\Delta_1^{(1)}} \log^2 (np) +  \left(\max_{j\in[p]} \sum_{i=1}^n \|\pi_{1,i}\psi_j\|_{L^2(P_i)}^2 \right)^{1/2} \left( \Delta_{2,*}^{(5)}(2) \log^7 (np) \right)^{1/4} \nonumber\\
    & \qquad\qquad\qquad\qquad\qquad\qquad\qquad\qquad\qquad\qquad\qquad + \sqrt{(\Delta_{2,*}(1) + \Delta_{2,*}(2)) \log^5 (np)}  \Bigg), \label{eq:inid:eval-step2-2}\\
    & \E\left[ \max_{(j,k)\in[p]^2} |J_2(\tilde{\boldsymbol\chi}^{(2,2), \mathtt{diag}}_{jk})| \right] (\log p)^{3/2} + \E\left[ \max_{(j,k)\in[p]^2} |J_3(\tilde{\boldsymbol\chi}^{(2,2)}_{jk})| \right] (\log p)^{3/2} \nonumber\\
    & \quad\quad\quad +\E\left[ \max_{(j,k)\in[p]^2} |J_2(\tilde{\boldsymbol\varphi}^{(2,2)}_{jk})| \right](\log p)^{3/2} +  \E\left[ \max_{(j,k)\in[p]^2} |J_1(\boldsymbol\varphi^{(2,2), \mathtt{diag}}_{jk})| \right](\log p)^{3/2}\nonumber \\
    & \quad  \le  C \Bigg( \sqrt{\Delta_{1}^{(0)}} \log^3 (np) + \sqrt{\Delta_{2,*}(2) \log^5 (np)} \Bigg). \label{eq:inid:eval-step2-3}
\end{align}
where $\Delta_{2,*}(1) \coloneqq \sum_{l=1}^2 \Delta_{2,*}^{(l)}(1)$ and $\Delta_{2,*}(2) \coloneqq \sum_{l=1}^5 \Delta_{2,*}^{(l)}(2)$ with
\begin{align*}
    & \Delta_{2,*}^{(2)}(1) \coloneqq \E\left[\max_{j\in[p]} \max_{i\in[n]} \frac{(\pi_{1,i}\psi_j)^4(X_i)}{\sigma_j^4}\right] \log (np), \\
    & \Delta_{2,*}^{(4)}(2) \coloneqq \E\left[  \max_{j\in[p] }\max_{(i,m)\in I_{n,2}} \frac{(\pi_{2,im}\psi_j)^4(X_i,X_m)}{\sigma_j^4} \right] \log^5 (np),\\
    & \Delta_{2,*}^{(5)}(2) \coloneqq \E\left[ \max_{j\in[p] }\max_{i\in[n]}\sum_{m\neq i}\frac{\left( P_m(\pi_{2,im}\psi_j)^2(X_i)\right)^2 }{\sigma_j^4} \right] \log^3 (np).
\end{align*}
\end{lemma}
From these evaluations, it follows that
\begin{align*}
     \E[\|V\|_\infty] (\log p)^{3/2} 
     &\le  C \Bigg( \Delta_1' + \sqrt{(\Delta_{2,*}(1) + \Delta_{2,*}(2)) \log^5 (np)} \Bigg).
\end{align*}

\paragraph{Step 3}
In this step, we bound $\E[\Gamma_1]$ and $\E[\Gamma_2]$.
Observe that
\begin{align*}
    \Gamma_1 
    &= n \max_{j\in[p]} \E[| D_{j,1}|^4 \mid X] \\
    &= n \max_{j\in[p]} \E\left[ \left|  \sum_{i=1}^n \{\pi_{1,i} \psi_j(X_i') - \pi_{1,i} \psi_j(X_i)\}  \right|^4 \mid X\right] \\
    &= n \max_{j\in[p]} \E\left[ \left|  \sum_{i=1}^n 1_{\{\alpha = i\}}\{\pi_{1,i} \psi_j(X_i^*) - \pi_{1,i} \psi_j(X_i)\}  \right|^4 \mid X\right] \\
    &= \max_{j\in[p]} \sum_{i=1}^n \E\left[ \left| \{\pi_{1,i} \psi_j(X_i^*) - \pi_{1,i} \psi_j(X_i)\}  \right|^4 \mid X\right] \\
    & \le 8 \max_{j\in[p]} \sum_{i=1}^n \left\{ P_i(\pi_{1,i}\psi_j)^4+  (\pi_{1,i}\psi_j)^4(X_i)\right\}.
\end{align*}
By Lemma~9 in \cite{CCK15}, 
\begin{align*}
     \E\left[ \max_{j\in[p]} \sum_{i=1}^n (\pi_{1,i}\psi_j)^4(X_i) \right] &\le \max_{j\in[p]} \E\left[ \sum_{i=1}^n (\pi_{1,i}\psi_j)^4(X_i)\right] + \E\left[ \max_{i\in[n]} \max_{j\in[p]} (\pi_{1,i}\psi_j)^4(X_i)  \right] \log p \\
     & \le  \max_{j\in[p]} \sum_{i=1}^n \|\pi_{1,i}\psi_j\|_{L^4(P_i)}^4 +  \left\| \max_{i\in[n]} \max_{j\in[p]} |\pi_{1,i}\psi_j| \right\|_{L^q(\pr)}^4 \log p \\
     & \le \Delta_{2,*}^{(1)}(1) + \Delta_{2,q}^{(2)}(1) = \Delta_{2,q}(1).
\end{align*}
Therefore
\begin{align*}
    \E[\Gamma_1] \lesssim \Delta_{2,q}(1).
\end{align*}
Next, observe that
\begin{align*}
    \Gamma_2 &= n \max_{j\in[p]} \E[|D_{j,2}|^4 \mid X] \\
    & = n \max_{j\in[p]}  \E\left[\left|\sum_{1\le i< m \le n}  \{\pi_{2,im} \psi_j(X_i',X_m') - \pi_{2,im} \psi_j(X_i,X_m)\}\right|^4 \mid X\right] \\
    & = n \max_{j\in[p]} \E \Bigg[\Bigg| \sum_{i=1}^n 1_{\{\alpha=i\}} \Bigg( \sum_{m\in[n]:i<m}  \{\pi_{2,im} \psi_j(X_i^*,X_m) - \pi_{2,im} \psi_j(X_i,X_m)\} \\
    &\qquad\qquad\qquad\qquad+\sum_{m\in[n]:m<i}\{\pi_{2,mi} \psi_j(X_m,X_i^*) - \pi_{2,mi} \psi_j(X_m,X_i)\} \Bigg)\Bigg|^4 \mid X\Bigg] \\
    & = \max_{j\in[p]} \sum_{i=1}^n \E\Bigg[ \Bigg| \sum_{m\in[n]:i<m} \{\pi_{2,im} \psi_j(X_i^*,X_m) - \pi_{2,im} \psi_j(X_i,X_m)\} \\
    & \qquad\qquad\qquad\qquad+\sum_{m\in[n]:m<i}\{\pi_{2,mi} \psi_j(X_m,X_i^*) - \pi_{2,mi} \psi_j(X_m,X_i)\} \Bigg|^4 \mid X\Bigg] \\
    & \le 8 \Bigg( \max_{j\in[p]}  \sum_{i=1}^n \E\Bigg[ \Bigg| \sum_{m\in[n]:m\neq i} \pi_{2,im} \psi_j(X_i^*,X_m)\Bigg|^4 \mid X \Bigg] + \max_{j\in[p]}  \sum_{i=1}^n \Bigg|\sum_{m\in[n]:m\neq i} \pi_{2,im} \psi_j(X_i,X_m)\Bigg|^4  \Bigg) \\
    & \coloneqq 8(\Gamma_{2,1} + \Gamma_{2,2}).
\end{align*}
Noting that $\{X_i^*\}_{i=1}^n$ is an i.n.i.d.~sequence and independent of $X$, from \cref{eq:influence-mom-inid} and \cref{eq:influence-mom-inid-<}, we have
\begin{align*}
    \E[\Gamma_{2,1}] &\lesssim \E\left[\max_{j\in[p]} \sum_{i=1}^n \int \left| \sum_{m\in[n]:m\neq i} \pi_{2,im}\psi_j(x_i,X_m) \right|^4 dP_i(x_i) \right] \\
    & \lesssim \max_{j\in[p]} \sum_{i=1}^n \left\|\sum_{m\neq i}P_m(\pi_{2,im}\psi_j)^2\right\|_{L^2(P_i)}^2  \log^2 p \nonumber \\
      & \quad\quad + \max_{j\in[p]} \sum_{i=1}^n\sum_{m\neq i} \|\pi_{2,im}\psi_j\|_{L^4(P_i\otimes P_m)}^4 \log^3 p + \E\left[ \max_{j\in[p]} \max_{m\in[n]} \sum_{i\neq m} P_i(\pi_{2,im}\psi_j)^4(X_m) \right] \log^4 p.
\end{align*}
Also, from \cref{eq:influence-mom-inid-2}, we have
\begin{align*}
    & \E[\Gamma_{2,2}] \\ &\lesssim \max_{j\in[p]} \sum_{i=1}^n \left\|\sum_{m\neq i}P_m(\pi_{2,im}\psi_j)^2\right\|_{L^2(P_i)}^2  \log^2 p \nonumber \\
      & \quad + \max_{j\in[p]} \sum_{i=1}^n\sum_{m\neq i} \|\pi_{2,im}\psi_j\|_{L^4(P_i\otimes P_m)}^4 \log^3 p + \E\left[ \max_{j\in[p]} \max_{m\in[n]} \sum_{i\neq m} P_i(\pi_{2,im}\psi_j)^4(X_m) \right] \log^4 p \nonumber\\
      & \quad+ \left\|  \max_{j\in[p]} \max_{i\in[n]} \sum_{m\neq i}   P_m (\pi_{2,im}\psi_j)^2 (X_i) \right\|_{L^2(\mathbb P )}^2\log^3(np) + \E\left[ \max_{j\in[p]} \max_{(i,m)\in I_{n,2}}   (\pi_{2,im}\psi_j)^4(X_i,X_m) \right]\log^5(np).
\end{align*}
Therefore
\begin{align*}
    \E[ \Gamma_2] \lesssim \Delta_{2,q}(2).
\end{align*}
Summing up
\begin{align*}
    \E[\Gamma_1 + \Gamma_2] \lesssim \Delta_{2,q}(1) +  \Delta_{2,q}(2)
\end{align*}

\paragraph{Step 4}
In this step, we bound $\pr(\|D\|_\infty > \beta^{-1})$.
By Markov's inequality
\begin{align*}
    \Pr\left( \|D\|_\infty > \beta^{-1} \right) \le \beta^q \E[\|D\|_\infty^q] \le (2\beta)^q \left( \E\left[ \max_{j\in[p]} |D_{j,1}|^q \right] + \E\left[ \max_{j\in[p]} |D_{j,2}|^q \right] \right).
\end{align*}
By definition, it holds that
\begin{align*}
    \E\left[ \max_{j\in[p]} |D_{j,1}|^q \right] &= \E\left[ \max_{j\in[p]} \left| \sum_{i=1}^n \{\pi_{1,i}\psi_j(X_i') - \pi_{1,i}\psi_j(X_i)\} \right|^q \right] \\
    &=\E\left[ \max_{j\in[p]} \left| \sum_{i=1}^n 1_{\{\alpha=i\}}\{\pi_{1,i}\psi_j(X_i^*) - \pi_{1,i}\psi_j(X_i)\} \right|^q \right] \\
    &= \frac{1}{n} \sum_{i=1}^n  \E\left[ \max_{j\in[p]} \left| \{\pi_{1,i}\psi_j(X_i^*) - \pi_{1,i}\psi_j(X_i)\} \right|^q \right] \\
    & \le \frac{2^q}{n} \sum_{i=1}^n \left\| \max_{j\in[p]} |\pi_{1,i}\psi_j|\right\|_{L^q(P_i)}^q \\
    & \le 2^q\left\|  \max_{i\in[n]} \max_{j\in[p]} |\pi_{1,i}\psi_j|\right\|_{L^q(\pr)}^q \le 2^q \left( \frac{\Delta_{2,q}^{(2)}(1)^{1/4}}{ n^{1/q} \{\log (np)\}^{1/4}}\right)^q.
 \end{align*} 

Also,
\begin{align*}
     \E\left[ \max_{j\in[p]} |D_{j,2}|^q \right] 
     &= \frac{1}{n} \sum_{i=1}^n \E\left[ \max_{j\in[p]} \left| \sum_{m:m\neq i} \{\pi_{2,im}\psi_{j}(X_i,X_m) - \pi_{2,im}\psi_{j}(X_i^*,X_m)\}\right|^q \right] \\
     & \le \max_{i\in[n]} \E\left[ \max_{j\in[p]} \left| \sum_{m:m\neq i} \{\pi_{2,im}\psi_{j}(X_i,X_m) - \pi_{2,im}\psi_{j}(X_i^*,X_m)\}\right|^q \right] \\
     & \le 2^q \max_{i\in[n]} \E\left[ \max_{j\in[p]} \left| \sum_{m:m\neq i} \pi_{2,im}\psi_{j}(X_i,X_m)\right|^q \right]. 
\end{align*}
Then, Lemma~1 in \cite{imai2025gaussian} yields
\begin{align*}
    &\max_{i\in[n]} \E\left[ \max_{j\in[p]} \left| \sum_{m:m\neq i} \pi_{2,im}\psi_{j}(X_i,X_m)\right|^q \right] \\
    & \quad \le C \max_{i\in[n]} \Bigg( \left\| \max_{j\in[p]} \sqrt{\sum_{m\neq i} P_m(|\pi_{2,im}\psi_j|^2)(X_i)} \right\|_{L^q(\pr)}^q \sqrt{\log^q (np)} \\
    & \qquad\qquad\qquad + \left\| \max_{m\neq i} \max_{j\in[p]} |\pi_{2,im}\psi_j(X_i,X_m)| \right\|_{L^q(\pr)}^q \log^q (np) \Bigg) \\
    & \quad \le  C \Bigg(  \left( \left\| \max_{j\in[p]} \max_{i\in[n]} \sum_{m\neq i} P_m(|\pi_{2,im}\psi_j|^2)(X_i) \right\|_{L^{q/2}(\pr)} \log (np) \right)^{q/2} \\
    & \qquad\qquad\qquad + \left( \left\| \max_{j\in[p]} \max_{(i,m)\in I_{n,2}}|\pi_{2,im}\psi_j(X_i,X_m)| \right\|_{L^q(\pr)} \log (np) \right)^q \Bigg) \\
    & \quad \le C \left(\frac{ \{\Delta_{2,q}^{(5)}(2) + \Delta_{2,q}^{(4)}(2)\}^{1/4}}{n^{1/q} \log^{1/4}(np)}\right)^q .
\end{align*}
Consequently, in conjunction with $\beta = \varepsilon^{-1} \log p$, there exists a universal constant $C>0$ such that
\begin{align*}
    n\pr\left( \|D\|_\infty > \beta^{-1} \right) \le \varepsilon^{-q} \{\log (np)\}^q \left( C \frac{\{\Delta_{2,q}(1) + \Delta_{2,q}(2)\}^{1/4}}{\log^{1/4}(np)}  \right)^q.
\end{align*}

\paragraph{Step 5}
In this step, we choose the value of $\eps$ appropriately and complete the proof.
Let 
\begin{align*}
    \varepsilon = \left(\Delta_1'  + \sqrt{(\Delta_{2,*}(1) + \Delta_{2,*}(2)) \log^5 (np)}\right)^{1/2} + C\left( (\Delta_{2,q}(1) + \Delta_{2,q}(2)) \log^3 (np)\right)^{1/4}.
\end{align*}
Then, from the same argument as Step~3 in the proof of Theorem~2 in \cite{imai2025gaussian}, we have the desired result.

\subsection{Proof of \cref{coro:main-kernel-inid}} \label{subsec:proof:coro:main-kernel-inid}

Before the Proof of \cref{coro:main-kernel-inid}, we provide the extensions of Lemma 4 and Lemma 6 in \cite{imai2025gaussian} to i.n.i.d.~setting. 

The first result is the extensions of Lemma 4  in \cite{imai2025gaussian} to i.n.i.d.~setting.

\begin{lemma}\label{lem:max-Jensen_inid}
Let $\boldsymbol{\psi}_j\coloneqq\{\psi_{j,(i_1,\dots,i_r)}\}_{(i_1,\dots,i_r)\in I_{n,r}}$ be a family of integrable, symmetric kernels of order $r\ge1$ for each $j\in[p]$. Then, for any $1\le l\le r\le n$,
\begin{align}
    &\E\left[\max_{j\in[p]}\max_{(i_1,\dots,i_{r-l})\in I_{n,r-l}}\max_{(k_1,\dots,k_l)\in K_{n,l}(i_1,\dots,i_{r-l})}\left|\left(\prod_{t=1}^lP_{k_t}\right)\psi_{j,(i_1,\dots,i_{r-l},k_1,\dots,k_l)}(X_{i_1},\dots,X_{i_{r-l}})\right|\right]\nonumber\\
    &\qquad\le\frac{r!}{(r-l)!}\E\left[\max_{j\in[p]}\max_{(i_1,\dots,i_r)\in I_{n,r}}|\psi_j(X_{i_1},\dots,X_{i_r})|\right].\label{eq:max-Jensen-inid}
\end{align}

\begin{proof}[Proof of \cref{lem:max-Jensen_inid}]
In line with the proof of Lemma 4 in \cite{imai2025gaussian}, the claim for general $l$ follows from repeated applications of the claim for $l=1$. Hence, it suffices to consider the case $l=1$.

With $\psi^*_{(i_1,\dots,i_r)}\coloneqq\max_{j\in[p]}|\psi_{j,(i_1,\dots,i_r)}|$, we have
\begin{align*}
    &\max_{j\in[p]}\max_{(i_1,\dots,i_{r-1})\in I_{n,r-1}}\max_{k\in K_{n,1}(i_1,\dots,i_{r-1})}|P_k\psi_j(X_{i_1},\dots,X_{i_{r-1}})|\\
    &\le\max_{(i_1,\dots,i_{r-1})\in I_{n,r-1}}\max_{k\in K_{n,1}(i_1,\dots,i_{r-1})}P_k\psi^*(X_{i_1},\dots,X_{i_{r-1}}),
\end{align*}
and
\begin{align*}
    &\max_{j\in[p]}\max_{(i_1,\dots,i_r)\in I_{n,r}}|\psi_j(X_{i_1},\dots,X_{i_r})|
    =\max_{(i_1,\dots,i_r)\in I_{n,r}}\psi^*(X_{i_1},\dots,X_{i_r}).
\end{align*}
Hence, we may assume $p=1$ and $\psi_1\ge0$ without loss of generality.

When $r=1$,
\begin{align*}
    \E\left[\max_{k\in[n]}P_k\psi_{1,(k)}\right]
    =\max_{k\in[n]}P_k\psi_{1,(k)}
    \le\E\left[\max_{k\in[n]}\psi_{1,(k)}(X_k)\right].
\end{align*}

Next, suppose $r>1$. Our goal is to show that
\begin{align*}
    &\E\left[\max_{(i_1,\dots,i_{r-1})\in I_{n,r-1}}\max_{k\in K_{n,1}(i_1,\dots,i_{r-1})}P_k\psi_1(X_{i_1},\dots,X_{i_{r-1}})\right]\le r\E\left[\max_{(i_1,\dots,i_r)\in I_{n,r}}\psi_1(X_{i_1},\dots,X_{i_r})\right].
\end{align*}
We first write
\begin{align*}
    &\E\left[\max_{(i_1,\dots,i_{r-1})\in I_{n,r-1}}\max_{k\in K_{n,1}(i_1,\dots,i_{r-1})}P_k\psi_1(X_{i_1},\dots,X_{i_{r-1}})\right]=\E\left[\max_{k\in[n]}\max_{\substack{(i_1,\dots,i_{r-1})\in I_{n,r-1}\\k\notin\{i_1,\dots,i_{r-1}\}}}P_k\psi_1(X_{i_1},\dots,X_{i_{r-1}})\right].
\end{align*}
For each $k\in[n]$, let $(i_1^{(k)},\dots,i_{r-1}^{(k)},m^{(k)})$ be the tuple attaining
\begin{align*}
    \max_{\substack{(i_1,\dots,i_{r-1},m)\in I_{n,r}\\k\notin\{i_1,\dots,i_{r-1}\}}}P_m\psi_1(X_{i_1},\dots,X_{i_{r-1}}).
\end{align*}
By the definition of $(i_1^{(k)},\dots,i_{r-1}^{(k)},m^{(k)})$, we have
\begin{align*}
    &\E\left[\max_{k\in[n]}\max_{\substack{(i_1,\dots,i_{r-1})\in I_{n,r-1}\\k\notin\{i_1,\dots,i_{r-1}\}}}P_k\psi_1(X_{i_1},\dots,X_{i_{r-1}})\right]\le\E\left[\sum_{k=1}^n\mathbf1_{\{m^{(k)}=k\}}P_k\psi_1(X_{i_1^{(k)}},\dots,X_{i_{r-1}^{(k)}})\right],
\end{align*}
where the inequality follows because the maximizer over $I_{n,r}$ is included in the sum and $\psi_1\ge0$.
Next, since $(i_1^{(k)},\dots,i_{r-1}^{(k)},m^{(k)})$ is measurable with respect to $\sigma(X_m:m\neq k)$, the independence of $X_k$ from $(X_m)_{m\neq k}$ gives
\begin{align*}
    &\E\left[\sum_{k=1}^n\mathbf1_{\{m^{(k)}=k\}}P_k\psi_1(X_{i_1^{(k)}},\dots,X_{i_{r-1}^{(k)}})\right]=\E\left[\sum_{k=1}^n\mathbf1_{\{m^{(k)}=k\}}\psi_1(X_{i_1^{(k)}},\dots,X_{i_{r-1}^{(k)}},X_k)\right],
\end{align*}
and since $\psi_1\ge0$,
\begin{align*}
    \E\left[\sum_{k=1}^n\mathbf1_{\{m^{(k)}=k\}}\psi_1(X_{i_1^{(k)}},\dots,X_{i_{r-1}^{(k)}},X_k)\right] \le \mathbb{E}\left[ \max_{(i_1,\dots, i_r) \in I_{n,r}}  \psi_1(X_{i_1},\dots,X_{i_r})\sum_{k=1}^n\mathbf1_{\{m^{(k)}=k\}} \right].
\end{align*}
Therefore, it remains to show that $\sum_{k=1}^n\mathbf1_{\{m^{(k)}=k\}} \le r$.
Let $(h_1,\dots,h_{r-1},m)$ be the maximizer of $P_m\psi_1(X_{i_1},\dots,X_{i_{r-1}})$ over $I_{n,r}$.
Since $m^{(k)}$ is the last index of the maximizing tuple obtained by excluding $k$ from the first $r-1$ indices, we have $m^{(k)}=m$ whenever $k\notin\{h_1,\dots,h_{r-1}\}$.
In particular, if $k\notin\{h_1,\dots,h_{r-1},m\}$, then $m^{(k)}=m\neq k$.
Thus, $m^{(k)}=k$ can hold only if $k\in\{h_1,\dots,h_{r-1},m\}$, which contains exactly $r$ indices. Therefore, $\sum_{k=1}^n\mathbf1_{\{m^{(k)}=k\}}\le r$.
\end{proof}

\end{lemma}

The second result is extension of Lemma 6 in \cite{imai2025gaussian}.
\begin{lemma} \label{lem:drop-pi-inid}
There exists a universal constant such that
\begin{align*}
    & \left\| \sum_{i\neq m,l} \pi_{2,im}\psi_j \star_i^1 \pi_{2,il}\psi_k\right\|_{L^2(P_m\otimes P_l)}^2 \\
    & \le C \left( \left\| \sum_{i\neq m,l}  \psi_{j,(i,m)}\star_i^1 \psi_{k,(i,l)}\right \|_{L^2(P_m\otimes P_l)}^2 +  \left(  \sum_{i\neq m}  \|P_i\psi_{j,(i,m)}\|_{L^2(P_m)}^2\right)\left(  \sum_{i\neq l}  \|P_i\psi_{k,(i,l)}\|_{L^2(P_l)}^2\right) \right).
\end{align*}
\begin{proof}
From the definition of $\pi_{2,im}\psi_j \star^1_i \pi_{2,il}\psi_k$, we have
\begin{align*}
    & \left\|\sum_{i\neq m,l} \pi_{2,im}\psi_j \star^1_i \pi_{2,il}\psi_k\right\|_{L^2(P_m\otimes P_l)}^2 \\
    & =  \left\|\sum_{i\neq m,l} \int \pi_{2,im}\psi_j(x_i, X_m)\pi_{2,il}\psi_k(x_i,X_l) dP_i(x_i) \right\|_{L^2(P_m\otimes P_l)}^2  .
\end{align*}
Recall that the second-order Hoeffding projection is defined in \cref{eq:Hoeffding-projections-2nd-inid} as
\begin{align*}
    & \pi_{2,im}\psi_j(X_i,X_m) =  (\delta_{X_i} - P_i)(\delta_{X_m} - P_m) \psi_j(X_i,X_m), \\
    & \pi_{2,il}\psi_k(X_i,X_l) =  (\delta_{X_i} - P_i)(\delta_{X_l} - P_l) \psi_k(X_i,X_m).
\end{align*}
Therefore
\begin{align*}
    & \left\|\sum_{i\neq m,l} \pi_{2,im}\psi_j \star^1_i \pi_{2,il}\psi_k\right\|_{L^2(P_m\otimes P_l)}^2 \\
    & =  \left\|(\delta_{X_m} - P_m)(\delta_{X_l} - P_l)  \sum_{i\neq m,l} \int \{\psi_j(x_i,x_m) - P_i\psi_j(x_i,X_m)\} \{\psi_k(x_i,X_l) - P_i\psi_k(x_i,X_l)\} dP_i(x_i)\right\|_{L^2(P_m\otimes P_l)}^2 \\
    & = \left\|(\delta_{X_m} - P_m)(\delta_{X_l} - P_l)  \sum_{i\neq m,l} \{\psi_{j,(i,m)}\star_i^1\psi_{k,(i,l)}(X_m, X_l) - P_i\psi_{j,(i,m)}(X_m)P_i\psi_{k,(i,l)}(X_l)\}\right\|_{L^2(P_m\otimes P_l)}^2.
\end{align*}
Since $(\delta_{X_m} - P_m)(\delta_{X_l} - P_l)$ is a contraction on $L^2(P_m\otimes P_l)$, 
\begin{align*}
    & \left\|\sum_{i\neq m,l} \pi_{2,im}\psi_j \star^1_i \pi_{2,il}\psi_k\right\|_{L^2(P_m\otimes P_l)}^2  \\
    & \le  \left\| \sum_{i\neq m,l} \{\psi_{j,(i,m)}\star_i^1\psi_{k,(i,l)}(X_m, X_l) - P_i\psi_{j,(i,m)}(X_m)P_i\psi_{k,(i,l)}(X_l)\}\right\|_{L^2(P_m\otimes P_l)}^2\\
    & \le 2\left\| \sum_{i\neq m,l}  \psi_{j,(i,m)}\star_i^1 \psi_{k,(i,l)}\right \|_{L^2(P_m\otimes P_l)}^2 + 2\left\|\sum_{i\neq m,l} P_i\psi_{j,(i,m)}(X_m) P_i\psi_{k,(i,l)}(X_l)\right\|_{L^2(P_m\otimes P_l)}^2 .
\end{align*}
Then, applying Schwarz inequality to the second term gives
\begin{align*}
    & \left\|\sum_{i\neq m,l} \pi_{2,im}\psi_j \star^1_i \pi_{2,il}\psi_k\right\|_{L^2(P_m\otimes P_l)}^2 \\
    & \le 2\left\| \sum_{i\neq m,l}  \psi_{j,(i,m)}\star_i^1 \psi_{k,(i,l)}\right \|_{L^2(P_m\otimes P_l)}^2 + 2 \left(  \sum_{i\neq m}  \|P_i\psi_{j,(i,m)}\|_{L^2(P_m)}^2\right)\left(  \sum_{i\neq l}  \|P_i\psi_{k,(i,l)}\|_{L^2(P_l)}^2\right).
\end{align*}
\end{proof}
\end{lemma}

\begin{proof}[Proof of \cref{coro:main-kernel-inid}]

    From \cref{eq:Hoeffding-projections-2nd-inid}, \cref{lem:max-Jensen_inid} and Jensen's inequality, we have $\Delta_{2,q}(1) \lesssim \tilde \Delta_{2,q}(1)$, $\Delta_{2,q}(2) \lesssim \tilde \Delta_{2,q}(2)$ and $\Delta_{2,q}^{(5)}(2) \lesssim \tilde \Delta_{2,q}^{(5)}(2)$.
    
    Next, from \cref{lem:drop-pi-inid}, we have
    \begin{align*}
        \Delta_1^{(0)}\log^3 (np) & = \max_{(j,k)\in[p]^2} \sum_{m=1}^n\sum_{l\neq m} \left\| \sum_{i\neq m,l} \pi_{2,im}\psi_j \star_i^1 \pi_{2,il}\psi_k\right\|_{L^2(P_m\otimes P_l)}^2 \log^3(np)\\
        & \lesssim  \max_{(j,k)\in[p]^2} \sum_{m=1}^n\sum_{l\neq m} \left\| \sum_{i\neq m,l}  \psi_{j,(i,m)}\star_i^1 \psi_{k,(i,l)}\right \|_{L^2(P_m\otimes P_l)}^2 \log^3 (np)\\
        & \quad + \max_{(j,k)\in[p]^2} \sum_{m=1}^n\sum_{l\neq m} \left(  \sum_{i\neq m}  \|P_i\psi_{j,(i,m)}\|_{L^2(P_m)}^2\right)\left(  \sum_{i\neq l}  \|P_i\psi_{k,(i,l)}\|_{L^2(P_l)}^2\right)  \log^3 (np).
    \end{align*}
    In terms of the second term, AM-GM inequality gives
    \begin{align*}
        & \max_{(j,k)\in[p]^2} \sum_{m=1}^n\sum_{l\neq m} \left(  \sum_{i\neq m}  \|P_i\psi_{j,(i,m)}\|_{L^2(P_m)}^2\right)\left(  \sum_{i\neq l}  \|P_i\psi_{k,(i,l)}\|_{L^2(P_l)}^2\right)  \log^3 (np) \\
        & \le \max_{(j,k)\in[p]^2}  \left(  \sum_{m=1}^n\sum_{i\neq m}  \|P_i\psi_{j,(i,m)}\|_{L^2(P_m)}^2\right)\left(  \sum_{l=1}^n \sum_{i\neq l}  \|P_i\psi_{k,(i,l)}\|_{L^2(P_l)}^2\right)  \log^3 (np) \\
        & \le \max_{(j,k)\in[p]^2} \left[ \frac{1}{2}\left(  \sum_{m=1}^n\sum_{i\neq m}  \|P_i\psi_{j,(i,m)}\|_{L^2(P_m)}^2\right)^2 + \frac{1}{2}\left(  \sum_{l=1}^n \sum_{i\neq l}  \|P_i\psi_{k,(i,l)}\|_{L^2(P_l)}^2\right)^2\right]  \log^3 (np) \\
        & \le \max_{j\in[p]} \left(  \sum_{i=1}^n\sum_{m\neq i}  \|P_m\psi_{j,(i,m)}\|_{L^2(P_i)}^2\right)^2 \log^3 (np).
    \end{align*}
    Thus
    \begin{align*}
         \Delta_1^{(0)}\log^3 (np)
         & \le \max_{(j,k)\in[p]^2} \sum_{m=1}^n\sum_{l\neq m} \left\| \sum_{i\neq m,l}  \psi_{j,(i,m)}\star_i^1 \psi_{k,(i,l)}\right \|_{L^2(P_m\otimes P_l)}^2 \log^3 (np) \\
         & \quad\quad + \max_{j\in[p]} \left(  \sum_{i=1}^n\sum_{m\neq i}  \|P_m\psi_{j,(i,m)}\|_{L^2(P_i)}^2\right)^2 \log^3 (np).
    \end{align*}
    
    Finally, we evaluate $ \Delta_1^{(1)}$.
    From $\int \pi_{1,i}\psi_{j}(x_i) dP_i(x_i) = 0$ and Fubini's theorem, it holds that
    \begin{align*}
        & \pi_{1,i}\psi_j \star^1_i \pi_{2,im}\psi_k(X_m) \\
        & = \int  \pi_{1,i}\psi_{j}(x_i)  (\delta_{x_i} - P_i)(\delta_{X_m} - P_m) \psi_k(x_i,X_m) dP_i(x_i) \\
        & = (\delta_{X_m} - P_m) \int  \pi_{1,i}\psi_{j}(x_i)  \psi_k(x_i,X_m) dP_i(x_i).
    \end{align*}
    Since $(\delta_{X_m} - P_m)$ is a contraction on $L^2(P_m)$, 
    \begin{align*}
         \left\|\sum_{i\neq m}\pi_{1,i}\psi_j \star_i^1 \pi_{2,im}\psi_k \right\|_{L^2(P_m)}^2
        &\le   \left\|\sum_{i\neq m} \int  \pi_{1,i}\psi_{j}(x_i) \psi_k(x_i,X_m) dP_i(x_i)\right\|_{L^2(P_m)}^2.
    \end{align*}
    Then, Schwarz inequality gives
    \begin{align*}
         \left\|\sum_{i\neq m}\pi_{1,i}\psi_j \star_i^1 \pi_{2,im}\psi_k \right\|_{L^2(P_m)}^2
        &\le \left( \sum_{i\neq m} \|\pi_{1,i}\psi_{j} \|_{L^2(P_i)}^2 \right) \left\|\sum_{i\neq m} P_i \psi_{k,(i,m)}^2(X_m) \right\|_{L^2(P_m)} \\
        & \le \left(  \sum_{i\neq m} \left\| \sum_{l\neq i}  P_l\psi_{j,(i,l)}(X_i) \right\|_{L^2(P_i)}^2 \right) \left\|\sum_{i\neq m} P_i \psi_{k,(i,m)}^2(X_m) \right\|_{L^2(P_m)} .
    \end{align*}
    Therefore
    \begin{align*}
        \Delta_1^{(1)} &= \max_{(j,k)\in[p]^2} \sum_{m=1}^n \left\|\sum_{i\neq m}\pi_{1,i}\psi_j \star_i^1 \pi_{2,im}\psi_k \right\|_{L^2(P_m)}^2  \\
        & \le  \max_{(j,k)\in[p]^2} \sum_{m=1}^n \left(  \sum_{i\neq m} \left\| \sum_{l\neq i}  P_l\psi_{j,(i,l)}(X_i) \right\|_{L^2(P_i)}^2 \right) \left\|\sum_{i\neq m} P_i \psi_{k,(i,m)}^2(X_m) \right\|_{L^2(P_m)}  \\
        & \le \left( \max_{j\in[p]} \sum_{i=1}^n \left\| \sum_{m\neq i}  P_m\psi_{j,(i,m)}(X_i) \right\|_{L^2(P_i)}^2 \right) \max_{k\in[p]} \sum_{i=1}^n \left\|\sum_{m\neq i} P_m \psi_{k,(i,m)}^2(X_i) \right\|_{L^2(P_i)} .
    \end{align*}
    Now, the desired result follows by inserting obtained bounds into \cref{eq:main-inid}.
\end{proof}

\subsection{Proof of \cref{coro:main-kernel-V-inid}} \label{subsec:proof:coro:main-kernel-V-inid}
\begin{proof}[Proof of \cref{coro:main-kernel-V-inid}]
As shown in the proof of \cref{coro:main-kernel-inid}, it holds that
\begin{align*}
    \Delta_{2,q}(2) \lesssim \tilde \Delta_{2,q}(2), \quad \Delta_{2,q}^{(5)}(2) \lesssim \tilde \Delta_{2,q}^{(5)}(2)
\end{align*}
and
\begin{align*}
        \Delta_1^{(0)}\log^3 p & \le \max_{(j,k)\in[p]^2} \sum_{m=1}^n\sum_{l\neq m} \left\| \sum_{i\neq m,l}  \psi_{j,(i,m)}\star_i^1 \psi_{k,(i,l)}\right \|_{L^2(P_m\otimes P_l)}^2 \log^3 (np) \\
         & \quad\quad + \max_{j\in[p]} \left(  \sum_{i=1}^n\sum_{m\neq i}  \|P_m\psi_{j,(i,m)}\|_{L^2(P_i)}^2\right)^2 \log^3 (np).
    \end{align*}
Also, from \cref{eq:Hoeffding-projections-2nd-V-inid}, \cref{lem:max-Jensen_inid} and Jensen's inequality, we have $\Delta^V_{2,q}(1) \lesssim \tilde \Delta^V_{2,q}(1)$.
Finally, we evaluate $(\Delta_1^{(1)})^V$.
In the same way as the evaluation $\Delta_1^{(1)}$, we have
    \begin{align*}
        (\Delta_1^{(1)})^V \le \max_{(j,k)\in[p]^2} \sum_{m=1}^n  \left( \sum_{i\neq m} \|\pi_{1,i}^V\psi_{j} \|_{L^2(P_i)}^2 \right) \left\|\sum_{i\neq m} P_i \psi_{k,(i,m)}^2(X_m) \right\|_{L^2(P_m)}.
    \end{align*}
    Observe that 
    \begin{align*}
         \|\pi_{1,i}^V\psi_{j}\|_{L^2(P_i)}^2
         & \le  \left\| \frac{1}{2}(\delta_{X_i} - P_i)\psi(X_i,X_i) + \sum_{l\neq i} (\delta_{X_i} - P_i)P_l \psi(X_i,X_l)\right\|_{L^2(P_i)}^2 \\
         & \lesssim  \|\psi_j(X_i,X_i)\|_{L^2(P_i)}^2  +  \left\| \sum_{l\neq i}  P_l\psi_{j,(i,l)}(X_i) \right\|_{L^2(P_i)}^2.
    \end{align*}
    Thus
    \begin{align*}
         \Delta_1^{(1)} 
         &\lesssim \max_{(j,k)\in[p]^2} \sum_{m=1}^n  \left( \sum_{i\neq m} \|\psi_j(X_i,X_i)\|_{L^2(P_i)}^2 \right) \left\|\sum_{i\neq m} P_i \psi_{k,(i,m)}^2(X_m) \right\|_{L^2(P_m)} \\
         & \quad + \max_{(j,k)\in[p]^2} \sum_{m=1}^n  \left(  \sum_{i\neq m} \left\| \sum_{l\neq i}  P_l\psi_{j,(i,l)}(X_i) \right\|_{L^2(P_i)}^2  \right) \left\|\sum_{i\neq m} P_i \psi_{k,(i,m)}^2(X_m) \right\|_{L^2(P_m)} \\
         & \le  \left(\max_{j\in[p]}   \sum_{i=1}^n \|\psi_j(X_i,X_i)\|_{L^2(P_i)}^2 \right) \left( \max_{k\in[p]} \sum_{i=1}^n\left\|\sum_{m\neq i} P_m \psi_{k,(i,m)}^2(X_m) \right\|_{L^2(P_i)} \right) \\
         & \quad +  \left( \max_{j\in[p]} \sum_{i=1}^n \left\| \sum_{m\neq i}  P_m\psi_{j,(i,m)}(X_i) \right\|_{L^2(P_i)}^2 \right) \max_{k\in[p]} \sum_{i=1}^n \left\|\sum_{m\neq i} P_m \psi_{k,(i,m)}^2(X_i) \right\|_{L^2(P_i)}.
    \end{align*}
    Now, the desired result follows by inserting obtained bounds into \cref{eq:main-V-inid}.
\end{proof}

\section{Proof of Auxiliary Result} \label{sec:proof-auxiliary-inid}

\subsection{Proof of \cref{lem:nonada-inid}}

\begin{proof}
The proof is almost same as that of Lemma 2 in \cite{imai2025gaussian}.
Let $\eta_N^*:=\max_{i\in[N]}\|\eta_i\|_\infty$ for every $N\geq0$. 
Define $\eta_i':=\eta_i1_{\{\eta_i^*\leq2\eta_{i-1}^*\}}$ and $\eta_i'':=\eta_i-\eta_i'$. 
From the proof of Lemma 1 in \cite{imai2025gaussian}, we can see that $\|\eta_i''\|_\infty=\|\eta_i\|_\infty1_{\{\eta_i^*>2\eta_{i-1}^*\}}\leq2(\eta_i^*-\eta_{i-1}^*)$. 
Hence
\begin{align*}
    \left\| \max_{j\in[p]}\sum_{i=1}^N\eta''_{ij} \right\|_{L^q(\mathbb P)} \le 2 \left\|\eta_N^* \right\|_{L^q(\mathbb P)}.
\end{align*}
Therefore, we complete the proof once we show
\begin{align} \label{nonneg-aim-inid}
    \left\| \max_{j\in[p]}\sum_{i=1}^N\eta'_{ij} \right\|_{L^q(\mathbb P)} \lesssim \left\| \max_{j\in[p]}\sum_{i=1}^N \mathbb{E}[\eta_{ij} \mid \mathcal{F}_{i-1} ]\right\|_{L^q(\mathbb P)} + (q + \log p)\left\|\eta_N^* \right\|_{L^q(\mathbb P)}.
\end{align}
With $\xi'_{i}:=\eta'_{i}-\E[\eta'_{i}\mid\mcl F_{i-1}]$ for every $i\in[N]$, we can bound the left hand side of \eqref{nonneg-aim-inid} as
\begin{align} \label{nonneg-eq1-inid}
     & \left\| \max_{j\in[p]}\sum_{i=1}^N\eta'_{ij} \right\|_{L^q(\mathbb P)} \nonumber\\
     & \le \left\| \max_{j\in[p]}\sum_{i=1}^N \mathbb{E}[\eta_{ij}' \mid \mathcal{F}_{i-1} ]\right\|_{L^q(\mathbb P)} + \left\|\max_{j\in[p]}\abs{\sum_{i=1}^N\xi_{ij}} \right\|_{L^q(\mathbb P)} \eqqcolon I + II.
\end{align}
By definition,
\begin{align}
    I\leq \left\|\max_{j\in[p]}\sum_{i=1}^N\E[\eta_{ij}\mid\mcl F_{i-1}] \right\|_{L^q(\mathbb P)}. \label{nonneg-eq2-inid}
\end{align}
Meanwhile, since $(\xi_i)_{i=1}^N$ is a martingale difference sequence in $\mathbb R^p$ with respect to $(\mcl F_i)_{i=0}^N$ by construction, we have by Lemma 1 in \cite{imai2025gaussian}
\begin{align*}
    II \lesssim \left\|\max_{j\in[p]} \sqrt{\sum_{i=1}^N\E[\xi_{ij}^2\mid\mcl F_{i-1}]}\right\|_{L^q(\mathbb P)} \sqrt{q + \log p} + \left\|\max_{i\in[N]} \|\xi_i'\|_\infty\right\|_{L^q(\mathbb P)} (q + \log p).
\end{align*}
Since $\|\xi_{i}\|_\infty\leq4\eta_{i-1}^*\leq4\eta_N^*$ by construction,
\begin{align*}
     II \lesssim \left\| \max_{j\in[p]}\sqrt{\eta_N^*\sum_{i=1}^N\E[|\xi_{ij}|\mid\mcl F_{i-1}]} \right\|_{L^q(\mathbb P)}\sqrt{q + \log p}+\|\eta_N^*\|_{L^q(\mathbb P)}( q +\log p).
\end{align*}
From Hölder's inequality and Young's inequality,
\begin{align*}
    & \left\| \max_{j\in[p]}\sqrt{\eta_N^*\sum_{i=1}^N\E[|\xi_{ij}|\mid\mcl F_{i-1}]} \right\|_{L^q(\mathbb P)}\sqrt{q + \log p} \\
    & \le \frac{1}{2}\left( \|\eta_N^*\|_{L^q(\mathbb P)}(q + \log p)  + \left\|\max_{j\in[p]}\sum_{i=1}^N\E[|\xi_{ij}|\mid\mcl F_{i-1}] \right\|_{L^q(\mathbb P)} \right).
\end{align*}
Since $\E[|\xi_{ij}|\mid\mcl F_{i-1}]\leq2\E[\eta_{ij}'\mid\mcl F_{i-1}]\leq2\E[\eta_{ij}\mid\mcl F_{i-1}]$, we conclude
\begin{align}
    II \lesssim \left\|\max_{j\in[p]}\sum_{i=1}^N\E[|\eta_{ij}|\mid\mcl F_{i-1}] \right\|_{L^q(\mathbb P)} + \|\eta_N^*\|_{L^q(\mathbb P)}( q +\log p) .\label{nonneg-eq3-inid}
\end{align}
Combining \eqref{nonneg-eq1-inid}--\eqref{nonneg-eq3-inid} gives \eqref{nonneg-aim-inid}. 
\end{proof}

\subsection{Proof of \cref{thm:max-is-inid} and \cref{lem:max-is-inid}} \label{subsec:proof:maximal_inid}

\begin{proof}[Proof of \cref{lem:max-is-inid}]
We prove the claim by induction on $r$.
It is trivial when $r=0$.
Next, suppose $r \ge 1$ and that the claim holds for all non-negative integers less than $r$.
We are going to show that there exists a constant $c_r \ge 1$ depending only on $r$ such that the statement of \cref{lem:max-is-inid} holds.
\begin{align*}
    J_r(\boldsymbol{\psi}_j) & = \sum_{1\le i_1 < \dots < i_r \le n}\psi_j(X_{i_1}, \dots, X_{i_r}) \\
    & = \sum_{i=1}^n \sum_{1\le h_1 < \dots < h_{r-1} < i} \psi_j(X_{h_1}, \dots, X_{h_{r-1}}, X_i)\\
    & = \frac{1}{(r-1)!} \sum_{i=1}^n \sum_{(h_1, \dots, h_{r-1})\in I_{i-1,r-1}} \psi_j(X_{h_1}, \dots, X_{h_{r-1}}, X_i) = \sum_{i=1}^n \eta_{ij},
\end{align*}
with $\eta_{ij} \coloneqq \{(r-1)!\}^{-1} \sum_{(h_1,\dots, h_{r-1})\in I_{i-1, r-1}} \psi_j(X_{h_1}, \dots, X_{h_{r-1}}, X_i)$.
Then, since  $(\eta_{ij})_{i=1}^n$ is a non-negative adapted sequence with respect to $\mathcal{F}_i \coloneqq \sigma(X_1, \dots, X_i)$, \cref{lem:nonada-inid} gives
\begin{align*}
    & \left\| \max_{j\in [p]} J_r(\boldsymbol{\psi}_j) \right\|_{L^q(\mathbb P)} \\
    & \lesssim \left\| \max_{j\in[p]} \sum_{i=1}^n \mathbb{E}[\eta_{ij} \mid \mathcal{F}_{i-1}] \right\|_{L^q(\mathbb P )} + \left\| \max_{i\in[n]} \max_{j\in[p]} \eta_{ij}\right\|_{L^q(\mathbb P)} (q + \log p) \eqqcolon I + II.
\end{align*}

In terms of $I$, observe that 
\begin{align*}
     \sum_{i=1}^n\mathbb{E}[\eta_{ij} \mid \mathcal{F}_{i-1}] 
     &= \frac{1}{(r-1)!}  \sum_{i=1}^n\sum_{(h_1,\dots, h_{r-1})\in I_{i-1, r-1}} P_i\psi_j(X_{h_1}, \dots, X_{h_{r-1}}) \\
     & =   \frac{1}{(r-1)!} \sum_{(h_1,\dots, h_{r-1})\in I_{n, r-1}}  \sum_{\substack{i\in K_{n,1}(h_1,\dots, h_{r-1}) \\ i> \max\{h_1,\dots, h_{r-1}\}}} P_i\psi_j(X_{h_1}, \dots, X_{h_{r-1}})\\
     & \le \frac{1}{(r-1)!} \sum_{(h_1,\dots, h_{r-1})\in I_{n, r-1}}  \sum_{\substack{i\in K_{n,1}(h_1,\dots, h_{r-1})}} P_i\psi_j(X_{h_1}, \dots, X_{h_{r-1}}) \\
     & = J_{r-1}\left( \sum_{\substack{i\in K_{n,1}(h_1,\dots, h_{r-1})}} P_i\psi_j \right).
\end{align*}
Therefore, the assumption of the induction gives
\begin{align*}
    I & \le c_r \max_{0 \le s \le r-1} (q + \log(np))^s   \left\| \max_{j\in[p]} \max_{(i_1,\dots, i_s) \in I_{n,s}} \right.  \\
    & \left. \sum_{(k_1,\dots, k_{r-1-s})\in K_{n,r-1-s}(i_1,\dots, i_s)} \sum_{a\in K_{n,1}(i_1,\dots, i_s, k_1,\dots, k_{r-1-s})} \left(\prod_{t=1}^{r-1-s} P_{k_t} P_a {\psi}_j\right) (X_{i_1}, \dots, X_{i_s}) \right\|_{L^q(\mathbb P )} \\
    & \le c_r \max_{0 \le s \le r-1} (q + \log(np))^s   \left\| \max_{j\in[p]} \max_{(i_1,\dots, i_s) \in I_{n,s}} \sum_{(k_1,\dots, k_{r-s})\in K_{n,r-s}(i_1,\dots, i_s)} \left(\prod_{t=1}^{r-s} P_{k_t} {\psi}_j\right) (X_{i_1}, \dots, X_{i_s}) \right\|_{L^q(\mathbb P )}.
\end{align*}

In terms of $II$, since we can regard $\eta_{ij}$ as $(r-1)$-th order $U$-statistics, the assumption of the induction and non-negativity of $\psi_j$ give
\begin{align*}
     & \left\| \max_{i\in[n]} \max_{j\in[p]} \eta_{ij}\right\|_{L^q(\mathbb P)} \\
     & \le c_r \max_{0 \le t \le r-1} (q + \log(np))^t  \left\| \max_{i\in[n]} \max_{j\in[p]} \max_{(i_1,\dots, i_t)\in I_{i-1, t}} \right. \\
    & \quad\quad\quad \left. \sum_{(k_1,\dots, k_{r-1-t})\in K_{i-1, r-1-t}(i_1,\dots, i_t)} \left( \prod_{l=1}^{r-1-t} P_{k_l} \psi_j\right) (X_{i_1}, \dots, X_{i_{r-1}}, X_i)\right\|_{L^q(\mathbb P)} \\
    & \le c_r \max_{0 \le t \le r-1} (q + \log(np))^t  \left\| \max_{j\in[p]} \max_{(i_1,\dots, i_t, i)\in I_{n, t+1}} \right. \\
    & \quad\quad\quad \left. \sum_{(k_1,\dots, k_{r-1-t})\in K_{n, r-1-t}(i_1,\dots, i_t, i)} \left( \prod_{l=1}^{r-1-t} P_{k_l} \psi_j\right) (X_{i_1}, \dots, X_{i_{r-1}}, X_i)\right\|_{L^q(\mathbb P)},
\end{align*}
where  in applying the induction hypothesis, we regard the pair $(i,j)$ as the coordinate index; hence the logarithmic factor becomes $q + \log(np)$.
Setting $s = t + 1$, we have
\begin{align*}
    II & \le c_r \max_{1 \le s \le r} (q + \log p)(q + \log(np))^{s-1}  \left\|  \max_{j\in[p]} \max_{(i_1,\dots, i_s)\in I_{n, s}} \right. \\
    & \quad\quad\quad \left. \sum_{(k_1,\dots, k_{r-s})\in K_{n, r-s}(i_1,\dots, i_s)} \left( \prod_{l=1}^{r-s} P_{k_l} \psi_j\right) (X_{i_1}, \dots, X_{i_{s}})\right\|_{L^q(\mathbb P)} \\
    & \le c_r \max_{0 \le s \le r} (q + \log(np))^{s}  \left\|  \max_{j\in[p]} \max_{(i_1,\dots, i_s)\in I_{n, s}} \sum_{(k_1,\dots, k_{r-s})\in K_{n, r-s}(i_1,\dots, i_s)} \left( \prod_{l=1}^{r-s} P_{k_l} \psi_j\right) (X_{i_1}, \dots, X_{i_{s}})\right\|_{L^q(\mathbb P)} .
\end{align*}
Summing up, we have
\begin{align*}
    & \left\| \max_{j\in [p]} J_r(\boldsymbol{\psi}_j) \right\|_{L^q(\mathbb P)} \\
    & \le c_r \max_{0 \le s \le r} (q + \log(np))^{s}  \left\|  \max_{j\in[p]} \max_{(i_1,\dots, i_s)\in I_{n, s}} \sum_{(k_1,\dots, k_{r-s})\in K_{n, r-s}(i_1,\dots, i_s)} \left( \prod_{l=1}^{r-s} P_{k_l} \psi_j\right) (X_{i_1}, \dots, X_{i_{s}})\right\|_{L^q(\mathbb P)} .
\end{align*}
\end{proof}

Before the proof of \cref{lem:max-is-inid}, we state the extension of Lemma 13 in \cite{imai2025gaussian} to the i.n.i.d. setting.
\begin{lemma} \label{lem:u-nemirovski-inid}
Let $q \ge 1$, $\boldsymbol{\psi}_j = \{\psi_{j,(i_1,\dots, i_r)}\}_{(i_1,\dots, i_r)\in I_{n,r}}$ and $\boldsymbol{\psi}^2_j = \{\psi^2_{j,(i_1,\dots, i_r)}\}_{(i_1,\dots, i_r)\in I_{n,r}}$.
Assume $\psi_j\in L^q(\otimes_{s=1}^r P_{i_s})$ ($j=1,\dots, p$) be degenerate, symmetric kernels of order $r\ge 1$ for all $(i_1, \dots, i_r)\in I_{n,r}$.
Then, there exists a constant $C_r$ depending only on $r$ such that
\begin{align*}
    \left\|  \max_{j\in[p]} |J_r(\boldsymbol\psi_j)| \right\|_{L^q(\pr)} \le C_r (q+\log p)^{r/2}\left\| \max_{j\in[p]}\sqrt{J_r(\boldsymbol\psi_j^2)}\right\|_{L^q(\pr)}.
\end{align*}
\begin{proof}
The proof of Lemma~13 in \cite{imai2025gaussian} relies only on
(i) the randomization theorem for $U$-statistics (Theorem~3.5.3 in \citealp{de1999decoupling}), 
(ii) hypercontractivity of Rademacher chaos (Theorem~3.5.2 in \citealp{de1999decoupling}),
and (iii) the identity $\E_\varepsilon[|J_r^\varepsilon(\boldsymbol\psi_j)|^2\mid X]=J_r(\boldsymbol\psi_j^2)$,
none of which requires the assumption of identical distributions.
Thus the same inequality holds even under i.n.i.d.~setting. 
\end{proof}
\end{lemma}

\begin{rmk}
Although \cite{de1999decoupling} state the hypercontractivity of Rademacher chaos and the randomization theorem for $U$-statistics with index‑independent kernels, we can regard $(i, X_i)$ as a single coordinate so we can utilize both theorems in our setting.
\end{rmk}

\begin{proof}[Proof of \cref{thm:max-is-inid}]
From \cref{lem:u-nemirovski-inid}, we have
\begin{align*}
    \left\|  \max_{j\in[p]} |J_r(\boldsymbol\psi_j)| \right\|_{L^q(\pr)} \le C_r (q+\log p)^{r/2}\left\| \max_{j\in[p]}\sqrt{J_r(\boldsymbol\psi_j^2)}\right\|_{L^q(\pr)}.
\end{align*}
Also, Lyapunov’s inequality gives
\begin{align*}
    \left\| \max_{j\in[p]}\sqrt{J_r(\boldsymbol\psi_j^2)}\right\|_{L^q(\pr)} \le  \left\| \max_{j\in[p]}J_r(\boldsymbol\psi_j^2)\right\|_{L^{1 \vee q/2}(\pr)}^{1/2}.
\end{align*}
Applying \cref{lem:max-is-inid} to the right hand side gives the desired result.
\end{proof}

\subsection{Proof of \cref{lem:max-rosenthal-inid}}
Before the proof of \cref{lem:max-rosenthal-inid}, we restate Lemma 16 in \cite{imai2025gaussian} adapted to our setting. 
\begin{lemma} \label{lem:nemirovski-Lp-inid}
Let $t \ge 1$. Assume $\psi_{j,(i,m)} \in L^t(P_i\otimes P_m)$ for all $j\in[p]$ and $(i,m)\in I_{n,2}$.
Then, there exists a universal constant $C$ such that
\begin{align}
    & \E\left[ \max_{j\in[p]} \sum_{i=1}^n \int \left| \sum_{m\in[n]:m\neq i} \left\{\psi_j(X_m,x_i) - (P_m\psi_j)(x_i)\right\} \right|^t dP_i(x_i) \right] \nonumber\\
    & \le C (t + \log p)^{t/2} \E\left[ \max_{j\in[p]} \sum_{i=1}^n \int \left( \sum_{m\in[n]:m\neq i} \psi_j^2(x_i,X_m) \right)^{t/2} dP_i(x_i)   \right]. \label{eq:nemirovski-Lp-inid}
\end{align}
\begin{proof}
    The proof is almost identical to that of Lemma 16 in \cite{imai2025gaussian}; we only replace the law of common distribution in \cite{imai2025gaussian} with the product measure.
\end{proof}
\end{lemma}

\begin{proof}[Proof of \cref{lem:max-rosenthal-inid}] 
First, we prove \cref{eq:influence-mom-inid}. 
By \cref{lem:nemirovski-Lp-inid},
\begin{align*}
    I 
    & \coloneqq   \E\left[ \max_{j\in[p]} \sum_{i=1}^n \int \left| \sum_{m\in[n]:m>i} \psi_j(x_i,X_m) \right|^4 dP_i(x_i) \right] \\
    & = \E\left[ \max_{j\in[p]} \sum_{i=1}^n \int \left| \sum_{m\in[n]:m\neq i} 1\{m>i\}\psi_j(x_i,X_m) \right|^4 dP_i(x_i) \right] \\
    & \lesssim  \E\left[  \max_{j\in[p]} \sum_{i=1}^n \int \left( \sum_{m\neq i} 1\{m>i\}\psi_j^2(x_i,X_m)\right)^2 dP_i(x_i)\right] (\log p)^2\\
    & \le \E\left[  \max_{j\in[p]} \sum_{i=1}^n \int \left( \sum_{m\neq i} \psi_j^2(x_i,X_m)\right)^2 dP_i(x_i)\right] (\log p)^2\\
    & = \E\left[ \max_{j\in[p]}  \sum_{i=1}^n  \int \left( \sum_{m\neq i} (P_m\psi_j^2)(x_i)+ \psi_j^2(x_i,X_m) - (P_m\psi_j^2)(x_i)\right)^2 dP_i(x_i)\right](\log p)^2 \\
    & \lesssim   \max_{j\in[p]} \sum_{i=1}^n \int \left( \sum_{m\neq i} (P_m\psi_j^2)(x_i) \right)^2dP_i(x_i)  (\log p)^2\\
    & \quad + \E\left[ \max_{j\in[p]} \sum_{i=1}^n\int \left|\sum_{m\neq i}\{\psi_j^2(x_i,X_m) - (P_m\psi_j^2)(x_i)\}\right|^2dP_i(x_i) \right]  (\log p)^2\\
    & \eqqcolon I_1 + I_2.
\end{align*}
By definition,
\begin{align*}
    I_1 \le  \max_{j\in[p]} \sum_{i=1}^n \left\|\sum_{m\neq i}P_m(\psi_j^2)\right\|_{L^2(P_i)}^2  \log^2 p.
\end{align*}
Applying \cref{lem:nemirovski-Lp-inid} to functions $(y,x) \to \psi_j(y,x)^2$, we have
\begin{align*}
    I_2 \le  \E\left[  \max_{j\in[p]} \sum_{i=1}^n \int \sum_{m\neq i} \psi^4_j(x_i,X_m) dP_i(x_i)  \right](\log p)^3 =  \E\left[  \max_{j\in[p]}  \sum_{i=1}^n \sum_{m\neq i} P_i\psi^4_j(X_m)  \right](\log p)^3.
\end{align*}
By Lemma~9 in \cite{CCK15}, 
\begin{align*}
   & \E\left[  \max_{j\in[p]}  \sum_{i=1}^n\sum_{m\neq i} P_i\psi^4_j(X_m)  \right] \\
    &\lesssim \max_{j\in[p]} \sum_{i=1}^n \sum_{m\neq i} P_mP_i\psi^4_j + \E\left[ \max_{j\in[p]} \max_{m\in[n]}  \sum_{i=1}^n P_i\psi^4_j(X_m) \right] \log p \\
    & \le \max_{j\in[p]}\sum_{i=1}^n \sum_{m\neq i} \|\psi_j\|_{L^4(P_i\otimes P_m)}^4 + \E\left[ \max_{j\in[p]} \max_{m\in[n]}  \sum_{i\neq m} P_i\psi^4_j(X_m) \right] \log p.
\end{align*}
Consequently, we have \cref{eq:influence-mom-inid}.
The proof of \cref{eq:influence-mom-inid-<} is almost same as that of \cref{eq:influence-mom-inid}.

Next, we prove \cref{eq:influence-mom-inid-2}.
\begin{align} 
     II &\coloneqq  \E\left[ \max_{j\in[p]} \sum_{i=1}^n \left| \sum_{m\in[n]:m\neq i} \psi_j(X_i,X_m) \right|^4  \right] \nonumber \\
     & \le 8  \left( \E\left[ \max_{j\in[p]} \sum_{i=1}^n  \left| \sum_{m\in[n]:m<i} \psi_j(X_i,X_m) \right|^4  \right] + \E\left[ \max_{j\in[p]} \sum_{i=1}^n  \left| \sum_{m\in[n]:m>i} \psi_j(X_i,X_m) \right|^4  \right] \right) \nonumber\\
     & \eqqcolon 8(II_1 + II_2).
\end{align}
We can construct a bound on $II_1$ likewise the proof of Lemma~3 in \cite{imai2025gaussian}, which uses Lemma~2 in \cite{imai2025gaussian}. 
Define a filtration $(\mathcal{G}_i)_{i=1}^n$ as $\mathcal{G}_i \coloneqq \sigma(X_1, \dots, X_i)$ for $i\in[n]$.
Also, for every $i\in[n]$, define a random vector $\eta_i^<= (\eta_{i1}^<, \dots, \eta_{ip}^<)^\top$ as
\begin{align*}
    \eta_{ij}^< \coloneqq \left| \sum_{m\in[n]:m<i} \psi_j(X_i,X_m) \right|^4, \quad j = 1,\dots, p.
\end{align*}
Then, $(\eta_i^<)_{i=1}^n$ is adapted to the filtration $(\mathcal{G}_i)_{i=1}^n$.
Hence \cref{lem:nonada-inid} or Lemma~2 in \cite{imai2025gaussian} give
\begin{align*}
    II_1 \lesssim  \E\left[ \max_{j\in[p]} \sum_{i=1}^n \E[\eta_{ij}^< \mid \mathcal{G}_{i-1}] \right] + \E\left[  \max_{j\in[p]} \max_{i\in[n]} \eta_{ij}^< \right] \log p  \eqqcolon III_1 + IV_1 \log p,
\end{align*}
where we set $\mathcal{G}_0 \coloneqq \{\emptyset,\Omega\}$. 
Since $X_{i}$ is independent of $\mathcal{G}_{i-1}$ for every $i$, by the same argument as that of the proof of Lemma~3 in \cite{imai2025gaussian}, \cref{eq:influence-mom-inid-<} gives
\begin{align*}
    III_1 &\lesssim  \max_{j\in[p]} \sum_{i=1}^n \left\|\sum_{m\neq i}P_m(\psi_j^2)\right\|_{L^2(\mathbb P)}^2  \log^2 p \nonumber \\
      & \quad\quad + \max_{j\in[p]} \sum_{i=1}^n\sum_{m\neq i} \|\psi_j\|_{L^4(P_i\otimes P_m)}^4 \log^3 p + \E\left[ \max_{j\in[p]} \max_{m\in[n]} \sum_{i\neq m} P_i\psi^4_j(X_m) \right] \log^4 p.
\end{align*}
To bound $IV_1$, observe that
\begin{align*}
    IV_1 = \E\left[  \max_{j\in[p]} \max_{i\in[n]} \left| \sum_{m\in[n]:m<i} \psi_j(X_i,X_m) \right|^4  \right] = \E\left[  \max_{j\in[p]} \max_{i\in[n]} \left|\sum_{m=1}^n Y_{m,(i,j)}^{<} \right|^4 \right],
\end{align*}
where $Y_{m,(i,j)}^{<} = \psi_j(X_i,X_m)1\{m<i\}$. 
Unlike \cite{imai2025gaussian}, we can not use time-reversal symmetry, so we define $\tilde{\mathcal{G}}^{(i)}_m \coloneqq \sigma(X_1, \dots, X_m, X_i)$.
Then, $(Y_{m,(i,j)}^<)_{m=1}^n$ is a martingale difference sequence with respect to $(\tilde{\mathcal{G}}_m^{(i)})_{m=0}^n$ for all $i\in[n]$ and $j\in[p]$. 
Hence Lemma~1 in \cite{imai2025gaussian} gives 
\begin{align*}
    IV_1 & \lesssim \E\left[  \max_{j\in[p]} \max_{i\in[n]}\left( \sum_{m=1}^n \E\left[ (Y_{m,(i,j)}^<)^2 \mid \tilde{\mathcal{G}}^{(i)}_{m-1} \right] \right)^2 \right] \log^2(np) + \E\left[ \max_{j\in[p]}\max_{(i,m)\in I_{n,2}}  |Y_{m,(i,j)}^<|^4 \right] \log^4(np) \\
    & \le  \E\left[  \max_{j\in[p]} \max_{i\in[n]} \left( \sum_{m\neq i} P_m \psi_j^2 (X_i) \right)^2 \right] \log^2(np) + \E\left[ \max_{j\in[p]} \max_{(i,m)\in I_{n,2}}   \psi_j^4(X_i,X_m) \right] \log^4(np) \\
    & = \left\|  \max_{j\in[p]} \max_{i\in[n]} \sum_{m\neq i}   P_m \psi_j^2 (X_i) \right\|_{L^2(\mathbb P )}^2 \log^2(np) + \E\left[ \max_{j\in[p]} \max_{(i,m)\in I_{n,2}}   \psi_j^4(X_i,X_m) \right] \log^4(np).
\end{align*}

Define a filtration $(\mathcal{G}_i^>)_{i=1}^n$ as $\mathcal{G}_i^> \coloneqq \sigma(X_i, X_{i+1}, \dots, X_n)$.
Also, for every $i\in[n]$, define a random vector $\eta^>_i = (\eta_{i1}^>,\dots, \eta_{ip}^>)^\top$ as 
\begin{align*}
    \eta_{ij}^> \coloneqq \left|  \sum_{m\in[n]:m>i} \psi_j(X_i,X_m) \right|^4, \quad j=1,\dots, p.
\end{align*}
Then, $(\eta_i^>)_{i=1}^n$ is adapted to the filtration $(\mathcal{G}_i^>)_{i=1}^n$.
Hence \cref{lem:nonada-inid} or Lemma~2 in \cite{imai2025gaussian} gives
\begin{align*}
    II_2 \lesssim \E\left[ \max_{j\in[p]} \sum_{i=1}^n \E[\eta^>_{ij} \mid \mathcal{G}_{i+1}^>] \right] + \E\left[\max_{j\in[p]} \max_{i\in[n]} \eta^>_{ij} \right] \log p \eqqcolon III_2 + IV_2\log p,
\end{align*}
where we set $\mathcal{G}^>_{n+1} \coloneqq\{\emptyset, \Omega\}$. 
Since $X_i$ is independent of $\mathcal{G}_{i+1}^>$, for every $i$, by the same argument as that of the proof of Lemma~3 in \cite{imai2025gaussian}, \cref{eq:influence-mom-inid} gives
\begin{align*}
    III_2 &\lesssim \max_{j\in[p]} \sum_{i=1}^n \left\|\sum_{m\neq i}P_m(\psi_j^2)\right\|_{L^2(\mathbb P)}^2  \log^2 p \nonumber \\
      & \quad\quad + \max_{j\in[p]} \sum_{i=1}^n\sum_{m\neq i} \|\psi_j\|_{L^4(P_i\otimes P_m)}^4 \log^3 p + \E\left[ \max_{j\in[p]} \max_{m\in[n]} \sum_{i\neq m} P_i\psi^4_j(X_m) \right] \log^4 p.
\end{align*}
To bound $IV_2$, observe that
\begin{align*}
    IV_2 = \E\left[ \max_{j\in[p]}\max_{i\in[n]}   \left| \sum_{m\in[n]:m>i} \psi_j(X_i,X_m) \right|^4  \right] = \E\left[ \max_{j\in[p]} \max_{i\in[n]} \left|\sum_{m=1}^n Y_{m,(i,j)}^{>} \right|^4 \right],
\end{align*}
where $Y_{m,(i,j)}^{>} = \psi_j(X_i,X_m)1\{m>i\}$. 
Observe that $(Y_{m,(i,j)}^{>})_{m=1}^n$ is martingale difference sequence with respect to $(\mathcal{G}_m)_{m=0}^n$ for all $i\in[n]$ and $j\in[p]$. 
Hence Lemma~1 in \cite{imai2025gaussian} gives
\begin{align*}
    IV_2
    & \lesssim \E\left[\max_{j\in[p]} \max_{i\in[n]}  \left( \sum_{m=1}^n \E[(Y^>_{m,(i,j)})^2 \mid \mathcal{G}_{m-1}] \right)^2\right] \log^2 (np)  + \E\left[  \max_{j\in[p]}\max_{(i,m)\in I_{n,2}} |Y^>_{m,(i,j)}|^4 \right] \log^4 (np) \\
    & \le  \E\left[  \max_{j\in[p]} \max_{i\in[n]} \left( \sum_{m\neq i} P_m \psi_j^2 (X_i) \right)^2 \right] \log^2(np) + \E\left[ \max_{j\in[p]} \max_{(i,m)\in I_{n,2}}   \psi_j^4(X_i,X_m) \right] \log^4(np). 
\end{align*}

Summing up,
\begin{align*}
    II &\lesssim \max_{j\in[p]} \sum_{i=1}^n \left\|\sum_{m\neq i}P_m(\psi_j^2)\right\|_{L^2(\mathbb P)}^2  \log^2 p \nonumber \\
      & \quad + \max_{j\in[p]} \sum_{i=1}^n\sum_{m\neq i} \|\psi_j\|_{L^4(P_i\otimes P_m)}^4 \log^3 p + \E\left[ \max_{j\in[p]} \max_{m\in[n]} \sum_{i\neq m} P_i\psi^4_j(X_m) \right] \log^4 p\\
      & \quad+  \left\|  \max_{j\in[p]} \max_{i\in[n]} \sum_{m\neq i}   P_m \psi_j^2 (X_i) \right\|_{L^2(\mathbb P )}^2\log^3(np) + \E\left[ \max_{j\in[p]} \max_{(i,m)\in I_{n,2}}   \psi_j^4(X_i,X_m) \right] \log^5(np). 
\end{align*}
\end{proof}

\subsection{Proof of \cref{lem:inid:eval-step2}} \label{subsec:proof-lem:inid:eval-step2}
Before the proof of \cref{lem:inid:eval-step2}, we introduce the following auxiliary results.
The proof of these results is in \cref{subsec:proof:lemlem5-inid}.

\begin{lemma} \label{lem:lem5-1-inid}
It holds that
\begin{align}
    &\max_{(j,k)\in[p]^2} \sum_{i=1}^n \int (\chi^{(1,1)}_{i,jk})^2(x_i)dP_i(x_i) \le   \max_{j\in[p]} \sum_{i=1}^n \|\pi_{1,i}\psi_j\|_{L^4(P_i)}^4, \label{eq:lem5-1-inid} \\
    & \left\| \max_{(j,k)\in[p]^2} \max_{i\in[n]} (\chi^{(1,1)}_{i,jk})^2(X_i) \right\|_{L^{1}(\pr)} \le  4\E\left[\max_{j\in[p]} \max_{i\in[n]} (\pi_{1,i}\psi_j)^4(X_i) \right].\label{eq:lem5-2-inid} 
\end{align}
\end{lemma}

\begin{lemma} \label{lem:lem5-2-inid}
It holds that
\begin{align}
    & \max_{(j,k)\in[p]^2} \sum_{i=1}^n\sum_{m\neq i} \int (\tilde{\chi}_{im,jk}^{(1,2)})^2(x_i,x_m) dP_i(x_i)dP_m(x_m) \nonumber \\
    &  \le  \left(\max_{j\in[p]}\sum_{i=1}^n \| \pi_{1,i}\psi_j(X_i) \|_{L^4{(P_i)}}^4 \right)^{1/2}\left(\max_{k\in[p]} \sum_{i=1}^n \left\|\sum_{m\neq i} P_m(|\pi_{2,im}\psi_k|^2) \right\|_{L^2(P_i)}^2\right)^{1/2}, \label{eq:lem5-3-inid}\\
    & \left\| \max_{(j,k)\in[p]^2} \max_{i\in[n]}\sum_{m\neq i} \int (\tilde{\chi}_{im,jk}^{(1,2)})^2(X_i,x_m)dP_m(x_m) \right\|_{L^1(\pr)} \nonumber\\
    &  \le 2\E\left[ \max_{j\in[p]} \max_{i\in[n]} (\pi_{1,i}\psi_j)^4(X_i) \right]^{1/2} \left\| \max_{k\in[p]} \max_{i\in[n]} \sum_{m\neq i} P_m(\pi_{2,im}\psi_k)^2(X_i) \right\|_{L^2(\pr)} \nonumber\\
    & \quad +  \frac{1}{2}\left(\max_{j\in[p]} \sum_{m=1}^n \|\pi_{1,m}\psi_j\|_{L^4(P_m)}^4 \right)^{1/2} \E\left[ \max_{k\in[p]} \max_{i\in[n]} \sum_{m\neq i}  P_m(\pi_{2,mi}\psi_k)^4(X_i) \right]^{1/2}, \label{eq:lem5-4-inid}\\
    &  \left\| \max_{(j,k)\in[p]^2} \max_{(i,m)\in I_{n,2}} (\tilde{\chi}_{im,jk}^{(1,2)})^2(X_i,X_m) \right\|_{L^1(\pr)} \nonumber \\
    & \quad\le 6\E\left[ \max_{j\in[p]} \max_{i\in[n]} (\pi_{1,i}\psi_j)^4(X_i) \right]^{1/2}  \left\| \max_{(j,k)\in[p]^2} \max_{(i,m)\in I_{n,2}} |\pi_{2,im}\psi_k| \right\|^{2}_{L^4(\pr)}. \label{eq:lem5-5-inid}
\end{align}
\end{lemma}

\begin{lemma}\label{lem:lem5-add1-inid}
It holds that
\begin{align}
    &\max_{(j,k)\in[p]^2} \sum_{i=1}^n \int (\varphi^{(1,2)}_{i,jk})^2(x_i)dP_i(x_i) \le 4  \max_{(j,k)\in[p]^2} \sum_{m=1}^n \left\|\sum_{i:i\neq m}\pi_{1,i}\psi_j \star_i^1 \pi_{2,im}\psi_k \right\|_{L^2(P_m)}^2,  \label{eq:lem5-add1-inid} \\
    & \left\| \max_{(j,k)\in[p]^2} \max_{i\in[n]} (\varphi^{(1,2)}_{i,jk})^2(X_i) \right\|_{L^{1}(\pr)} \nonumber \\
    & \quad \le  4 \left(\max_{j\in[p]} \sum_{i=1}^n \|\pi_{1,i}\psi_j\|_{L^2(P_i)}^2 \right) \left\|  \max_{j\in[p]}\max_{i\in[n]} \sum_{m\neq i}P_m(\pi_{2,im}\psi_j)^2(X_i) \right\|_{L^1(\pr)}. \label{eq:lem5-add2-inid} 
\end{align}
\end{lemma}

\begin{lemma} \label{lem:lem5-3-inid-1}
It holds that
\begin{align}
    & \max_{(j,k)\in[p]^2} \sum_{i=1}^n\sum_{m\neq i} \int (\tilde{\chi}_{im,jk}^{(2,2),\mathtt{diag}})^2(x_i,x_m) dP_i(x_i)dP_m(x_m)   \nonumber \\
    & \quad\qquad\qquad\qquad\qquad\qquad\le   \frac{1}{4}\max_{j\in[p]}\sum_{i=1}^n \sum_{m\neq i} \|\pi_{2,im} \psi_j\|_{L^4(P_i\otimes P_m)}^4 , \label{eq:lem5-6-inid}\\ 
    & \left\| \max_{(j,k)\in[p]^2} \max_{i\in[n]} \sum_{m\neq i} \int (\tilde{\chi}_{im,jk}^{(2,2), \mathtt{diag}})^2(X_i,x_m)dP_m(x_m) \right\|_{L^1(\pr)} \nonumber \\
    & \quad\qquad\qquad\qquad\qquad\qquad\le \E\left[ \max_{j\in[p]}  \max_{i\in[n]} \sum_{m\neq i} P_m(\pi_{2,im}\psi_j)^4(X_i) \right] ,\label{eq:lem5-7-inid}\\
    & \left\| \max_{(j,k)\in[p]^2} \max_{(i,m)\in I_{n,2}} (\tilde{\chi}_{im,jk}^{(2,2), \mathtt{diag}})^2(X_i,X_m) \right\|_{L^1(\pr)} \nonumber \\
    & \quad\qquad\qquad\qquad\qquad\qquad \le 6\left\| \max_{j\in[p]} \max_{(i,m)\in I_{n,2}} |\pi_{2,im}\psi_j| \right\|_{L^4(\mathbb P)}^4.\label{eq:lem5-8-inid}
\end{align}
\end{lemma}

\begin{lemma} \label{lem:lem5-3-inid-2}
It holds that
\begin{align}
    & \max_{(j,k)\in[p]^2} \sum_{i=1}^n\sum_{m\neq i}\sum_{l\neq i,m} \int (\tilde{\chi}^{(2,2)}_{iml,jk})^2(x_i,x_m,x_l) dP_i(x_i)dP_m(x_m)dP_l(x_l)\nonumber \\
    & \quad \le  \frac{1}{4}\max_{j\in[p]} \sum_{i=1}^n \left\| \sum_{m\neq i} P_m(\pi_{2,im}\psi_j)^2 \right\|_{L^2(P_i)}^2,  \label{eq:lem5-9-inid} \\
    & \left\| \max_{(j,k)\in[p]^2} \max_{i\in[n]} \sum_{m\neq i}\sum_{l\neq i,m}\int (\tilde{\chi}^{(2,2)}_{iml,jk})^2(X_i,x_m,x_l)dP_m(x_m)dP_l(x_l) \right\|_{L^1(\pr)} \nonumber \\
    & \quad\le \frac{1}{3} \left\| \max_{j\in[p]} \max_{i\in[n]} \sum_{m\neq i} P_m(\pi_{2,im}\psi_j)^2(X_i) \right\|_{L^2(\mathbb P )}^2 \nonumber\\
    & \quad + \frac{2}{3}\mathbb{E}\left[ \max_{j\in[p]} \max_{i\in[n]}  \sum_{m\neq i} P_m (\pi_{2,im}\psi_j)^4(X_i)\right]^{1/2} \left(\max_{j\in[p]} \sum_{i=1}^n \left\| \sum_{m\neq i} P_m(\pi_{2,im}\psi_j)^2 (X_i)\right\|_{L^2(P_i)}^2 \right)^{1/2},
    \label{eq:lem5-10-inid}\\
    &  \left\| \max_{(j,k)\in[p]^2} \max_{(i,m)\in I_{n,2}}\sum_{l\neq i,m} \int (\tilde{\chi}^{(2,2)}_{iml,jk})^2(X_i,X_m,x_l)dP_l(x_l)\right\|_{L^1(\pr)}  \nonumber \\
    & \quad\le \frac{1}{3} \left\| \max_{j\in[p]} \max_{(i,m)\in I_{n,2}} |\pi_{2,im}\psi_j| \right\|^{2}_{L^4(\mathbb P)} \left\| \max_{j\in[p]} \max_{i\in[n]} \sum_{m\neq i} P_m(\pi_{2,im}\psi_j)^2(X_i) \right\|_{L^2(\mathbb P)} \nonumber\\
    & \quad + \frac{1}{3} \mathbb{E}\left[  \max_{j\in[p]} \max_{i\in[n]} \sum_{m\neq i} P_m(\pi_{2,im}\psi_j)^4(X_i)\right]^{1/2} \left( \max_{j\in[p]} \sum_{i=1}^n \left\| \sum_{m\neq i } P_m(\pi_{2,im}\psi_j)^2(X_i)\right\|_{L^2(P_i)}^2 \right)^{1/2}  \nonumber\\
    & \quad + \frac{1}{3}  \E\left[ \max_{j\in[p]}  \max_{i\in [n]}   \sum_{m\neq i} P_m(\pi_{2,im}\psi_j)^4(X_i)\right], \label{eq:lem5-11-inid}\\
    & \left\| \max_{(j,k)\in[p]^2} \max_{(i,m,l)\in I_{n,3}} (\tilde{\chi}^{(2,2)}_{iml,jk})^2(X_i,X_m,X_l) \right\|_{L^1(\pr)} \le \E\left[ \max_{j\in[p]} \max_{(i,m)\in I_{n,2}} (\pi_{2,im}\psi_j)^4(X_i,X_m)\right]  . \label{eq:lem5-12-inid}
\end{align}
\end{lemma}

\begin{lemma}  \label{lem:lem5-4-inid}
It holds that
\begin{align}
    &  \max_{(j,k)\in[p]^2} \sum_{m=1}^n\sum_{l\neq m} \int (\tilde{\varphi}_{ml,jk}^{(2,2)})^2(x_m,x_l) dP_m(x_m)dP_l(x_l) \nonumber\\
    & \quad \le\max_{(j,k)\in[p]^2} \sum_{m=1}^n\sum_{l\neq m} \left\| \sum_{i\neq m,l} \pi_{2,im}\psi_j \star_i^1 \pi_{2,il}\psi_k\right\|_{L^2(P_m\otimes P_l)}^2,\label{eq:lem5-13-inid} \\
    &  \left\| \max_{(j,k)\in[p]^2} \max_{m\in[n]} \sum_{l\neq m} \int (\tilde{\varphi}_{ml,jk}^{(2,2)})^2(X_m,x_l)dP_l(x_l) \right\|_{L^1(\pr)} \nonumber \\
    & \quad \le 2\left( \max_{(j,k)\in[p]^2} \sum_{m=1}^n\sum_{l\neq m} \left\| \sum_{i\neq m,l} \pi_{2,im}\psi_j \star_i^1 \pi_{2,il}\psi_k\right\|_{L^2(P_m\otimes P_l)}^2\right)^{1/2} \left\| \max_{j\in[p]} \max_{i\in[n]} \sum_{m\neq i} P_m(\pi_{2,im}\psi_j)^2(X_i) \right\|_{L^2(\mathbb P)}, \nonumber \\
    & \quad\qquad +  \frac{3}{2}\left(\max_{j\in[p]} \sum_{i=1}^n  \left\|  \sum_{m\neq i} P_m(\pi_{2,im}\psi_j)^2 \right\|_{L^2(P_{i})}^2\right)^{1/2}\left\| \max_{j\in[p]} \max_{i\in[n]} \sum_{m\neq i} P_m(\pi_{2,im}\psi_j)^2(X_i) \right\|_{L^2(\mathbb P)}. \label{eq:lem5-14-inid}\\
    & \left\| \max_{(j,k)\in[p]^2} \max_{(m,l)\in I_{n,2}} (\tilde{\varphi}_{ml,jk}^{(2,2)})^2(X_m,X_l) \right\|_{L^1(\pr)} \le  \left\| \max_{j\in[p]} \max_{i\in[n]} \sum_{m\neq i} P_m(\pi_{2,im}\psi_j)^2(X_i) \right\|_{L^2(\mathbb P)}^2. \label{eq:lem5-15-inid}
\end{align}
\end{lemma}

\begin{lemma}  \label{lem:lem5-5-inid}
It holds that
\begin{align}
    & \max_{(j,k)\in[p]^2} \sum_{m=1}^n \int (\varphi^{(2,2), \mathtt{diag}}_{m,jk})^2(x_m)dP_m(x_m) \le 4  \max_{j\in[p]} \sum_{i=1}^n \left\| \sum_{m\neq i} P_m(\pi_{2,im}\psi_j)^2(X_i)\right\|_{L^2(P_i)}^2 , \label{eq:lem5-16-inid}\\
    & \left\| \max_{(j,k)\in[p]^2} \max_{m\in[n]}  (\varphi^{(2,2), \mathtt{diag}}_{m,jk})^2(X_m)\right\|_{L^1(\pr)}  \le 4 \left\| \max_{j\in[p]} \max_{m\in[n]} \sum_{i\neq m} P_i(\pi_{2,im}\psi_j)^2(X_m)\right\|_{L^2(\pr)}^2.  \label{eq:lem5-17-inid}
\end{align}
\end{lemma}

\paragraph{Proof of \cref{lem:inid:eval-step2}}
\subparagraph{Proof of \cref{eq:inid:eval-step2-1}:}
\cref{thm:max-is-inid} and \cref{lem:lem5-1-inid} give
\begin{align*}
    & \E\left[ \max_{(j,k)\in[p]^2} |J_1(\boldsymbol\chi_{jk}^{(1,1)})| \right](\log p)^{3/2} \\
    & \le C \left( \log^2 (np) \left(\max_{(j,k)\in[p]^2} \sum_{i=1}^n \int (\chi^{(1,1)}_{i,jk})^2(x_i)dP_i(x_i)   \right)^{1/2} \vee \log^{5/2} (np) \left\| \max_{(j,k)\in[p]^2} \max_{i\in[n]} (\chi^{(1,1)}_{i,jk})^2(X_i) \right\|_{L^{1}(\pr)}^{1/2}  \right) \\
    & \le C  \log^{2}(np)  \sqrt{\max_{j\in[p]}\sum_{i=1}^n \|\pi_{1,i}\psi_j\|_{L^4(P_i)}^4}  + \log^{5/2} (np) \E\left[\max_{j\in[p]} \max_{i\in[n]} (\pi_{1,i}\psi_j)^4(X_i) \right]^{1/2} \\
    & \le C  \sqrt{\left(\Delta_{2,*}^{(1)}(1) +\Delta_{2,*}^{(2)}(1)\right)\log^5 (np)}.
\end{align*}

\subparagraph{Proof of \cref{eq:inid:eval-step2-2}:}
\cref{thm:max-is-inid} and \cref{lem:lem5-2-inid} give
\begin{align*}
    & \E\left[ \max_{(j,k)\in[p]^2} |J_2(\tilde{\boldsymbol\chi}_{jk}^{(1,2)})| \right] (\log p)^{3/2} \\
    & \le C \left( \log^{5/2} (np) \left( \max_{(j,k)\in[p]^2} \sum_{i=1}^n\sum_{m\neq i} \int (\tilde{\chi}_{im,jk}^{(1,2)})^2(x_i,x_m) dP_i(x_i)dP_m(x_m)  \right)^{1/2} \right. \\
    & \left. \qquad\qquad \vee \log^3 (np) \left\| \max_{(j,k)\in[p]^2} \max_{i\in[n]} \sum_{m\neq i} \int (\tilde{\chi}_{im,jk}^{(1,2)})^2(X_i,x_m)dP_m(x_m) \right\|_{L^1(\pr)}^{1/2} \right. \\
    & \left. \qquad\qquad \vee \log^{7/2} (np) \left\| \max_{(j,k)\in[p]^2} \max_{(i,m)\in I_{n,2}} (\tilde{\chi}_{im,jk}^{(1,2)})^2(X_i,X_m) \right\|_{L^1(\pr)}^{1/2} \right) \\
    & \le C \log^{5/2} (np) \left(\max_{j\in[p]}\sum_{i=1}^n \| \pi_{1,i}\psi_j(X_i) \|_{L^4{(P_i)}}^4 \right)^{1/4}\left(\max_{k\in[p]} \sum_{i=1}^n \left\|\sum_{m\neq i} P_m(|\pi_{2,im}\psi_k|^2) \right\|_{L^2(P_i)}^2\right)^{1/4} \\
    & \quad + C  \log^{3}(np) \E\left[ \max_{j\in[p]} \max_{i\in[n]} (\pi_{1,i}\psi_j)^4(X_i) \right]^{1/4} \left\| \max_{k\in[p]} \max_{i\in[n]} \sum_{m\neq i} P_m(\pi_{2,im}\psi_k)^2(X_i) \right\|_{L^2(\pr)}^{1/2} \nonumber\\
    & \quad +  C  \log^{3}(np) \left(\max_{j\in[p]} \sum_{m=1}^n \|\pi_{1,m}\psi_j\|_{L^4(P_m)}^4 \right)^{1/4} \E\left[ \max_{k\in[p]} \max_{i\in[n]} \sum_{m\neq i}  P_m(\pi_{2,mi}\psi_k)^4(X_i) \right]^{1/4} \\
    & \quad + C \log^{7/2}(np)\E\left[ \max_{j\in[p]} \max_{i\in[n]} (\pi_{1,i}\psi_j)^4(X_i) \right]^{1/4}  \left\| \max_{(j,k)\in[p]^2} \max_{(i,m)\in I_{n,2}} |\pi_{2,im}\psi_k| \right\|_{L^4(\pr)} .
\end{align*}
From AM-GM inequality, we have
\begin{align*}
    & \E\left[ \max_{(j,k)\in[p]^2} |J_2(\tilde{\boldsymbol\chi}_{jk}^{(1,2)})| \right] (\log p)^{3/2} \\
    & \le C \log^{5/2}(np) \left(  \left(\max_{j\in[p]}\sum_{i=1}^n \| \pi_{1,i}\psi_j(X_i) \|_{L^4{(P_i)}}^4 \right)^{1/2} + \left(\max_{k\in[p]} \sum_{i=1}^n \left\|\sum_{m\neq i} P_m(|\pi_{2,im}\psi_k|^2) \right\|_{L^2(P_i)}^2\right)^{1/2} \right)  \\
    & \quad + C \log^2 (np)\left(\log^{1/2} (np) \E\left[ \max_{j\in[p]} \max_{i\in[n]} (\pi_{1,i}\psi_j)^4(X_i) \right]^{1/2} \right.\\
    & \left. \qquad\qquad\qquad\qquad + \log^{3/2}(np) \left\| \max_{k\in[p]} \max_{i\in[n]} \sum_{m\neq i} P_m(\pi_{2,im}\psi_k)^2(X_i) \right\|_{L^2(\pr)} \right) \\
    & \quad + C \log^2 (np) \left( \left(\max_{j\in[p]} \sum_{m=1}^n \|\pi_{1,m}\psi_j\|_{L^4(P_m)}^4 \right)^{1/2}  +   \log^2 (np)\E\left[ \max_{k\in[p]} \max_{i\in[n]} \sum_{m\neq i}  P_m(\pi_{2,mi}\psi_k)^4(X_i) \right]^{1/2} \right)\nonumber \\
    & \quad + C \log^{5/2} (np) \left( \log^{1/2} (np)\E\left[ \max_{j\in[p]} \max_{i\in[n]} (\pi_{1,i}\psi_j)^4 \right]^{1/2} +  \log^{3/2} (np)\left\| \max_{(j,k)\in[p]^2} \max_{(i,m)\in I_{n,2}} |\pi_{2,im}\psi_k| \right\|_{L^4(\pr)}^2 \right).
\end{align*}
Then, it holds that
\begin{align}
    & \E\left[ \max_{(j,k)\in[p]^2} |J_2(\tilde{\boldsymbol\chi}_{jk}^{(1,2)})| \right] (\log p)^{3/2} \nonumber \\
    & \le C \Bigg[ \log^{5/2} (np)\left( \sqrt{\Delta_{2,*}^{(1)}(1) } + \sqrt{\Delta_{2,*}^{(2)}(2) } \right) \nonumber\\
    & \quad + \log^2 (np) \left( \sqrt{\Delta_{2,*}^{(2)}(1)} + \sqrt{\Delta_{2,*}^{(5)}(2)}  \right) +   \log^2 (np)\left( \sqrt{\Delta_{2,*}^{(1)}(1)} + \sqrt{\Delta_{2,*}^{(3)}(2)} \right) \nonumber\\
    & \quad + \log^{5/2}(np) \left( \sqrt{\Delta_{2,*}^{(2)}(1)} + \sqrt{\Delta_{2,*}^{(4)}(2)}  \right) \Bigg]\nonumber \\
    & \le C \sqrt{(\Delta_{2,*}(1) + \Delta_{2,*}(2)) \log^5 (np)} . \label{eq:lem5-delta12-inid-1}
\end{align}

Also, \cref{thm:max-is-inid} and \cref{lem:lem5-add1-inid} give
\begin{align}
    & \E\left[ \max_{(j,k)\in[p]^2} | J_1(\boldsymbol\varphi_{jk}^{(1,2)})| \right] (\log p)^{3/2}  \nonumber \\
    & \le C\left( \log^2(np) \left(\max_{(j,k)\in[p]^2} \sum_{i=1}^n \int (\varphi^{(1,2)}_{i,jk})^2(x_i)dP_i(x_i) \right)^{1/2} \vee \log^{5/2}(np)\left\| \max_{(j,k)\in[p]^2} \max_{i\in[n]}  (\varphi^{(1,2)}_{i,jk})^2(X_i)\right\|_{L^1(\pr)}^{1/2}\right) \nonumber\\
    & \le C \Bigg(  \log^2 (np) \left( \max_{(j,k)\in[p]^2} \sum_{i=1}^n \left\|\sum_{m:m\neq i}\pi_{1,i}\psi_j \star_i^1 \pi_{2,im}\psi_k \right\|_{L^2(P_i)}^2\right)^{1/2} \nonumber\\
    & \qquad\qquad +  \log^{5/2}(np)\left(\max_{j\in[p]} \sum_{i=1}^n \|\pi_{1,i}\psi_j\|_{L^2(P_i)}^2 \right)^{1/2} \left\|  \max_{j\in[p]}\max_{i\in[n]} \sum_{m\neq i}P_m(\pi_{2,im}\psi_j)^2(X_i) \right\|_{L^1(\pr)}^{1/2}\Bigg) \nonumber\\
    & \le C \Bigg( \sqrt{\Delta_1^{(1)}} \log^2 (np) +  \left(\max_{j\in[p]} \sum_{i=1}^n \|\pi_{1,i}\psi_j\|_{L^2(P_i)}^2 \right)^{1/2} \left( \Delta_{2,*}^{(5)}(2) \log^7 (np) \right)^{1/4} \Bigg). \label{eq:lem5-delta12-inid-2}
\end{align}

From \cref{eq:lem5-delta12-inid-1} and \cref{eq:lem5-delta12-inid-2}, we have \cref{eq:inid:eval-step2-2}.

\subparagraph{Proof of \cref{eq:inid:eval-step2-3}:}
From \cref{thm:max-is-inid} and  \cref{lem:lem5-3-inid-1}, 
\begin{align}
    & \E\left[ \max_{(j,k)\in[p]^2} |J_2(\tilde{\boldsymbol\chi}_{jk}^{(2,2),\mathtt{diag}})| \right](\log p)^{3/2} \nonumber\\
    & \le C \left(  \log^{5/2} (np) \left( \max_{(j,k)\in[p]^2} \sum_{i=1}^n\sum_{m\neq i} \int (\tilde{\chi}_{im,jk}^{(2,2),\mathtt{diag}})^2(x_i,x_m) dP_i(x_i)dP_m(x_m)  \right)^{1/2} \right. \nonumber\\
    & \left. \qquad\qquad \vee \log^{3} (np) \left\| \max_{(j,k)\in[p]^2} \max_{i\in[n]}\sum_{m\neq i} \int (\tilde{\chi}_{im,jk}^{(2,2),\mathtt{diag}})^2(X_i,x_m)dP_m(x_m) \right\|_{L^1(\pr)}^{1/2} \right. \nonumber\\
    & \left. \qquad\qquad \vee \log^{7/2} (np) \left\| \max_{(j,k)\in[p]^2} \max_{(i,m)\in I_{n,2}} (\tilde{\chi}_{im,jk}^{(2,2),\mathtt{diag}})^2(X_i,X_m) \right\|_{L^1(\pr)}^{1/2} \right) \nonumber\\
    & \le C\log^{5/2} (np) \left( \max_{j\in[p]}\sum_{i=1}^n \sum_{m\neq i} \|\pi_{2,im} \psi_j\|_{L^4(P_i\otimes P_m)}^4 \right)^{1/2} \nonumber\\
    & \quad + C \log^3(np)\E\left[ \max_{j\in[p]}  \max_{i\in[n]} \sum_{m\neq i} P_m(\pi_{2,im}\psi_j)^4(X_i) \right]^{1/2} \nonumber \\
    & \quad + C\log^{7/2} (np) \left\| \max_{j\in[p]} \max_{(i,m)\in I_{n,2}} |\pi_{2,im}\psi_j| \right\|_{L^4(\mathbb P)}^2  \nonumber\\
    & \le C \sqrt{\Delta_{2,*}(2) \log^{5} (np)}. \label{eq:lem5-delta22-inid-1}
\end{align}
From \cref{thm:max-is-inid} and  \cref{lem:lem5-3-inid-2},
\begin{align*}
    & \E\left[ \max_{(j,k)\in[p]^2} |J_3(\tilde{\boldsymbol\chi}_{jk}^{(2,2)})| \right] (\log p)^{3/2} \\
    & \le C \left(  \log^3 (np) \left( \max_{(j,k)\in[p]^2} \sum_{i=1}^n\sum_{m\neq i}\sum_{l\neq i,m} \int (\tilde{\chi}^{(2,2)}_{iml,jk})^2(x_i,x_m,x_l) dP_i(x_i)dP_m(x_m)dP_l(x_l)  \right)^{1/2} \right. \\
    & \left. \qquad\qquad \vee \log^{7/2} (np) \left\| \max_{(j,k)\in[p]^2} \max_{i\in[n]}\sum_{m\neq i}\sum_{l\neq i,m} \int (\tilde{\chi}^{(2,2)}_{iml,jk})^2(X_i,x_m,x_l)dP_m(x_m)dP_l(x_l) \right\|_{L^1(\pr)}^{1/2} \right. \\
    & \left. \qquad\qquad \vee  \log^4 (np) \left\| \max_{(j,k)\in[p]^2} \max_{(i,m)\in I_{n,2}}\sum_{l\neq i,m} \int (\tilde{\chi}^{(2,2)}_{iml,jk})^2(X_i,X_m,x_l)dP_l(x_l)\right\|_{L^1(\pr)}^{1/2} \right. \\
    & \left. \qquad\qquad \vee \log^{9/2}(np) \left\| \max_{(j,k)\in[p]^2} \max_{(i,m,l)\in I_{n,3}} (\tilde{\chi}^{(2,2)}_{iml,jk})^2(X_i,X_m,X_l) \right\|_{L^1(\pr)}^{1/2} \right) \\
    & \le C  \log^{3} (np)\left(  \max_{j\in[p]} \sum_{i=1}^n \left\| \sum_{m\neq i} P_m(\pi_{2,im}\psi_j)^2 \right\|_{L^2(P_i)}^2 \right)^{1/2} \\
    & \quad + C\log^{7/2} (np) \Bigg(\left\| \max_{j\in[p]} \max_{i\in[n]} \sum_{m\neq i} P_m(\pi_{2,im}\psi_j)^2(X_i) \right\|_{L^2(\mathbb P )}  \\
    & \qquad +  \mathbb{E}\left[ \max_{j\in[p]} \max_{i\in[n]}  \sum_{m\neq i} P_m (\pi_{2,im}\psi_j)^4(X_i)\right]^{1/4} \left(\max_{j\in[p]} \sum_{i=1}^n \left\| \sum_{m\neq i} P_m(\pi_{2,im}\psi_j)^2 (X_i)\right\|_{L^2(P_i)}^2 \right)^{1/4} \Bigg) \\
    & \quad + C \log^{4} (np) \Bigg( \left\| \max_{j\in[p]} \max_{(i,m)\in I_{n,2}} |\pi_{2,im}\psi_j| \right\|_{L^4(\mathbb P)} \left\| \max_{j\in[p]} \max_{i\in[n]} \sum_{m\neq i} P_m(\pi_{2,im}\psi_j)^2(X_i) \right\|_{L^2(\mathbb P)}^{1/2} \nonumber\\
    & \qquad +  \mathbb{E}\left[  \max_{j\in[p]} \max_{i\in[n]} \sum_{m\neq i} P_m(\pi_{2,im}\psi_j)^4(X_i)\right]^{1/4} \left( \max_{j\in[p]} \sum_{i=1}^n \left\| \sum_{m\neq i } P_m(\pi_{2,im}\psi_j)^2(X_i)\right\|_{L^2(P_i)}^2 \right)^{1/4} \nonumber\\
    & \qquad +   \E\left[ \max_{j\in[p]}  \max_{i\in [n]}   \sum_{m\neq i} P_m(\pi_{2,im}\psi_j)^4(X_i)\right]^{1/2}  \Bigg) \\
    & \quad + C \log^{9/2}(np) \E\left[ \max_{j\in[p]} \max_{(i,m)\in I_{n,2}} (\pi_{2,im}\psi_j)^4(X_i,X_m) \right]^{1/2} .
\end{align*}
From AM-GM inequality, we have
\begin{align}
    & \E\left[ \max_{(j,k)\in[p]^2} |J_3(\tilde{\boldsymbol\chi}_{jk}^{(2,2)})| \right] (\log p)^{3/2} \nonumber\\
    & \le C \log^{3} (np)\left(  \max_{j\in[p]} \sum_{i=1}^n \left\| \sum_{m\neq i} P_m(\pi_{2,im}\psi_j)^2 \right\|_{L^2(P_i)}^2 \right)^{1/2}\nonumber \\
    & \quad + C\log^{7/2} (np)\Bigg( \left\| \max_{j\in[p]} \max_{i\in[n]} \sum_{m\neq i} P_m(\pi_{2,im}\psi_j)^2(X_i) \right\|_{L^2(\mathbb P )}\nonumber \\
    &\qquad +   \mathbb{E}\left[ \max_{j\in[p]} \max_{i\in[n]}  \sum_{m\neq i} P_m (\pi_{2,im}\psi_j)^4(X_i)\right]^{1/2}  +  \left(\max_{j\in[p]} \sum_{i=1}^n \left\| \sum_{m\neq i} P_m(\pi_{2,im}\psi_j)^2 (X_i)\right\|_{L^2(P_i)}^2 \right)^{1/2} \Bigg)   \nonumber\\
    & \quad +  C\log^{9/4} (np)\Bigg( \log^{3/2}(np)  \left\| \max_{j\in[p]} \max_{(i,m)\in I_{n,2}} |\pi_{2,im}\psi_j| \right\|_{L^4(\mathbb P)}\\
    & \qquad\qquad\qquad +  \log (np)  \left\| \max_{j\in[p]} \max_{i\in[n]} \sum_{m\neq i} P_m(\pi_{2,im}\psi_j)^2(X_i) \right\|_{L^2(\mathbb P)}\Bigg) \nonumber \\
    & \quad + C \log^{9/4} (np) \Bigg( \log^{5/2} (np) \mathbb{E}\left[  \max_{j\in[p]} \max_{i\in[n]} \sum_{m\neq i} P_m(\pi_{2,im}\psi_j)^4(X_i)\right]^{1/2}\nonumber \\
    & \qquad\qquad\qquad + \log (np) \left( \max_{j\in[p]} \sum_{i=1}^n \left\| \sum_{m\neq i } P_m(\pi_{2,im}\psi_j)^2(X_i)\right\|_{L^2(P_i)}^2 \right)^{1/2}\Bigg) \nonumber\\
    & \qquad + C \log^{4} (np) \E\left[ \max_{j\in[p]}  \max_{i\in [n]}   \sum_{m\neq i} P_m(\pi_{2,im}\psi_j)^4(X_i)\right]^{1/2} \nonumber\\
    & \quad + C \log^{9/2}(np) \E\left[ \max_{j\in[p]} \max_{(i,m)\in I_{n,2}} (\pi_{2,im}\psi_j)^4(X_i,X_m) \right]^{1/2} \nonumber \\
    & \le C  \sqrt{\Delta_{2,*}(2) \log^5 (np)} . \label{eq:lem5-delta22-inid-2}
\end{align}

Also, from \cref{thm:max-is-inid} and \cref{lem:lem5-4-inid}, we have
\begin{align*}
    & \E\left[ \max_{(j,k)\in[p]^2} |J_2(\tilde{\boldsymbol\varphi}^{(2,2)}_{jk})| \right](\log p)^{3/2} \\
    & \le C \left(   \log^{5/2} (np) \left( \max_{(j,k)\in[p]^2} \sum_{i=1}^n\sum_{m\neq i} \int (\tilde{\varphi}_{ml,jk}^{(2,2)})^2(x_m,x_l) dP_m(x_m)dP_l(x_l)  \right)^{1/2} \right. \\
    & \left. \qquad\qquad \vee \log^{3} (np) \left\| \max_{(j,k)\in[p]^2} \max_{m\in[n]}\sum_{l\neq m} \int (\tilde{\varphi}_{ml,jk}^{(2,2)})^2(X_m,x_l)dP_l(x_l) \right\|_{L^1(\pr)}^{1/2} \right. \\
    & \left. \qquad\qquad \vee \log^{7/2} (np) \left\| \max_{(j,k)\in[p]^2} \max_{(m,l)\in I_{n,2}} (\tilde{\varphi}_{ml,jk}^{(2,2)})^2(X_m,X_l) \right\|_{L^1(\pr)}^{1/2} \right) \\
    & \le C \Bigg(   \log^{5/2} (np) \left(\max_{(j,k)\in[p]^2} \sum_{m=1}^n\sum_{l\neq m} \left\| \sum_{i\neq m,l} \pi_{2,im}\psi_j \star_i^1 \pi_{2,il}\psi_k\right\|_{L^2(P_m\otimes P_l)}^2\right)^{1/2}\\
    & \quad  +  \log^3 (np) \left( \max_{(j,k)\in[p]^2} \sum_{m=1}^n\sum_{l\neq m} \left\| \sum_{i\neq m,l} \pi_{2,im}\psi_j \star_i^1 \pi_{2,il}\psi_k\right\|_{L^2(P_m\otimes P_l)}^2\right)^{1/4} \left\| \max_{j\in[p]} \max_{i\in[n]} \sum_{m\neq i} P_m(\pi_{2,im}\psi_j)^2(X_i) \right\|_{L^2(\mathbb P)}^{1/2}  \\
    & \quad +  \log^3 (np)\left(\max_{j\in[p]} \sum_{i=1}^n  \left\|  \sum_{m\neq i} P_m(\pi_{2,im}\psi_j)^2 \right\|_{L^2(P_{i})}^2 \right)^{1/4}\left\| \max_{j\in[p]} \max_{i\in[n]} \sum_{m\neq i} P_m(\pi_{2,im}\psi_j)^2(X_i) \right\|_{L^2(\mathbb P)}^{1/2} \\
    & \quad  + \log^{7/2}(np) \left\| \max_{j\in[p]} \max_{i\in[n]} \sum_{m\neq i} P_m(\pi_{2,im}\psi_j)^2(X_i) \right\|_{L^2(\mathbb P)}\Bigg).
\end{align*}
From AM-GM inequality, we have
\begin{align}
    & \E\left[ \max_{(j,k)\in[p]^2} |J_2(\tilde{\boldsymbol\varphi}^{(2,2)}_{jk})| \right](\log p)^{3/2} \nonumber\\
    & \le C \Bigg(   \log^{5/2} (np) \left(\max_{(j,k)\in[p]^2} \sum_{m=1}^n\sum_{l\neq m} \left\| \sum_{i\neq m,l} \pi_{2,im}\psi_j \star_i^1 \pi_{2,il}\psi_k\right\|_{L^2(P_m\otimes P_l)}^2\right)^{1/2}\nonumber\\
    & \quad  +  \log^3 (np) \left( \max_{(j,k)\in[p]^2} \sum_{m=1}^n\sum_{l\neq m} \left\| \sum_{i\neq m,l} \pi_{2,im}\psi_j \star_i^1 \pi_{2,il}\psi_k\right\|_{L^2(P_m\otimes P_l)}^2\right)^{1/2} \nonumber\\
    & \quad + \log^3 (np)\left(\max_{j\in[p]} \sum_{i=1}^n  \left\|  \sum_{m\neq i} P_m(\pi_{2,im}\psi_j)^2 \right\|_{L^2(P_{i})}^2 \right)^{1/2} \nonumber\\
    & \quad + \log^3 (np)\left\| \max_{j\in[p]} \max_{i\in[n]} \sum_{m\neq i} P_m(\pi_{2,im}\psi_j)^2(X_i) \right\|_{L^2(\mathbb P)}  \nonumber\\
    & \quad  + \log^{7/2}(np) \left\| \max_{j\in[p]} \max_{i\in[n]} \sum_{m\neq i} P_m(\pi_{2,im}\psi_j)^2(X_i) \right\|_{L^2(\mathbb P)}\Bigg) \nonumber\\
    & \le C \Bigg( \sqrt{\Delta_1^{(0)}} \log^{3} (np)   + \sqrt{\Delta_{2,*}(2) \log^5 (np) } \Bigg). \label{eq:lem5-delta22-inid-3}
\end{align}

Also, from \cref{thm:max-is-inid} and \cref{lem:lem5-5-inid}, we have
\begin{align}
    & \E\left[ \max_{(j,k)\in[p]^2} |J_1(\boldsymbol\varphi^{(2,2), \mathtt{diag}}_{jk})| \right](\log p)^{3/2} \nonumber\\
    & \le C\Bigg( \log^2(np) \left(\max_{(j,k)\in[p]^2}\sum_{m=1}^n \int (\varphi^{(2,2), \mathtt{diag}}_{m,jk})^2(x_m)dP_m(x_m) \right)^{1/2} \nonumber\\
    & \qquad\qquad \vee \log^{5/2}(np)\left\| \max_{(j,k)\in[p]^2} \max_{m\in[n]}  (\varphi^{(2,2), \mathtt{diag}}_{m,jk})^2(X_m)\right\|_{L^1(\pr)}^{1/2}\Bigg) \nonumber \\
    & \le C \Bigg( \log^2 (np) \left(\max_{j\in[p]} \sum_{i=1}^n \left\| \sum_{m\neq i} P_m(\pi_{2,im}\psi_j)^2(X_i)\right\|_{L^2(P_i)}^2 \right)^{1/2} \\
    & \qquad\qquad + \log^{5/2}(np) \left\| \max_{j\in[p]} \max_{m\in[n]} \sum_{i\neq m} P_i(\pi_{2,im}\psi_j)^2(X_m)\right\|_{L^2(\pr)}\Bigg) \nonumber \\
    & \le  C \sqrt{\Delta_{2,*}(2) \log^5 (np)}. \label{eq:lem5-delta22-inid-4}
\end{align}

From \cref{eq:lem5-delta22-inid-1},\cref{eq:lem5-delta22-inid-2}, \cref{eq:lem5-delta22-inid-3} and \cref{eq:lem5-delta22-inid-4}, we have \cref{eq:inid:eval-step2-3}.

\subsubsection{Proof of Lemma \ref{lem:lem5-1-inid}-\ref{lem:lem5-5-inid}} \label{subsec:proof:lemlem5-inid}

\begin{proof}[Proof of \cref{eq:lem5-1-inid}]
Recall that
\begin{align}
    (\chi_{i,jk}^{(1,1)})^2(X_i) = \left(\pi_{1,i}\psi_j(X_i)\pi_{1,i}\psi_k(X_i) - P_i \{\pi_{1,i}\psi_j(X_i)\pi_{1,i}\psi_k(X_i)\} \right)^2. \label{eq:chi11-cr}
\end{align}
then,  we have
\begin{align*}
    &\max_{(j,k)\in[p]^2} \sum_{i=1}^n \int (\chi^{(1,1)}_{i,jk})^2(x_i)dP_i(x_i) \\
    & \le \max_{(j,k)\in[p]^2}\sum_{i=1}^n \Big( P_i\{(\pi_{1,i}\psi_j)^2(X_i)(\pi_{1,i}\psi_k)^2(X_i)\}\Big) \\
    & \le  \max_{(j,k)\in[p]^2} \sum_{i=1}^n \left(  P_i\{(\pi_{1,i}\psi_j)^4(X_i) \}\right)^{1/2} \left(  P_i\{(\pi_{1,i}\psi_k)^4(X_i) \}\right)^{1/2} \\
    & \le \max_{j\in[p]} \sum_{i=1}^n \left(  P_i\{(\pi_{1,i}\psi_j)^4(X_i) \}\right) = \max_{j\in[p]} \sum_{i=1}^n \|\pi_{1,j}\psi_j\|_{L^4(P_i)}^4,
\end{align*}
where the first inequality follows from Jensen's inequality and the second inequality follows from Cauchy-Schwarz inequality.
\end{proof}

\begin{proof}[Proof of \cref{eq:lem5-2-inid}]
From \cref{eq:chi11-cr} and $c_r$ inequality,
\begin{align*}
    & \left\| \max_{(j,k)\in[p]^2} \max_{i\in[n]} (\chi^{(1,1)}_{i,jk})^2(X_i) \right\|_{L^{1}(\pr)} \\
    & \le 2\left\| \max_{(j,k)\in[p]^2}  \max_{i\in[n]} \{(\pi_{1,i}\psi_j)^2(X_i)(\pi_{1,i}\psi_k)^2(X_i)\} \right\|_{L^1(\pr)}   + 2\max_{(j,k)\in[p]^2}  \max_{i\in[n]}  \left( P_i \{\pi_{1,i}\psi_j(X_i)\pi_{1,i}\psi_k(X_i)\}\right)^2.
\end{align*}
In terms of the first term, it holds that
\begin{align*}
    \max_{(j,k)\in[p]^2} \max_{i\in[n]} (\pi_{1,i}\psi_j)^2(X_i)(\pi_{1,i}\psi_k)^2(X_i) =   \left( \max_{j\in[p]} \max_{i\in[n]} (\pi_{1,i}\psi_j)^2(X_i) \right)^2 = \max_{j\in[p]} \max_{i\in[n]} (\pi_{1,i}\psi_j)^4(X_i).
\end{align*}
In terms of the second term, 
\begin{align*}
    & \max_{(j,k)\in[p]^2}  \max_{i\in[n]}  \left( P_i \{\pi_{1,i}\psi_j(X_i)\pi_{1,i}\psi_k(X_i)\}\right)^2 \\
    & \le \max_{(j,k)\in[p]^2}  \max_{i\in[n]}  (P_i|\pi_{1,i}\psi_j|^4)^{1/2}(P_i|\pi_{1,i}\psi_|^4)^{1/2} \\
    & \le \max_{j\in[p]}\max_{i\in[n]}  P_i|\pi_{1,i}\psi_j|^4 \le \E\left[\max_{j\in[p]} \max_{i\in[n]} (\pi_{1,i}\psi_j)^4(X_i) \right]
\end{align*}
where the first inequality follows from Jensen's inequality and Cauchy-Schwarz inequality and the second inequality follows from AM-GM inequality.
Summing up, we have
\begin{align*}
    & \left\| \max_{(j,k)\in[p]^2} \max_{i\in[n]} (\chi^{(1,1)}_{i,jk})^2(X_i) \right\|_{L^{1}(\pr)}  \le  4\E\left[\max_{j\in[p]} \max_{i\in[n]} (\pi_{1,i}\psi_j)^4(X_i) \right].
\end{align*}
\end{proof}

\begin{proof}[Proof of \cref{eq:lem5-3-inid}]
Observe that
\begin{align*}
    \bar{\chi}_{im,jk}^{(1,2)} = (\delta_{X_i} - P_i)\{\pi_{1,i}\psi_j(X_i) \pi_{2,im}\psi_k(X_i,X_m)\}.
\end{align*}
Since $(\delta_{X_i} - P_i)$ is $L^2(P_i)$-contraction, from Eq.(5) in \cite{imai2025gaussian}, we can see that
\begin{align*}
    & \max_{(j,k)\in[p]^2} \sum_{i=1}^n\sum_{m\neq i} \int (\tilde{\chi}_{im,jk}^{(1,2)})^2(x_i,x_m) dP_i(x_i)dP_m(x_m) \\
    & \le \max_{(j,k)\in[p]^2} \sum_{i=1}^n\sum_{m\neq i} \Big( P_iP_m\{(\pi_{1,i}\psi_j)^2(X_i)(\pi_{2,im}\psi_k)^2(X_i,X_m)\} \Big) \\
    & =  \max_{(j,k)\in[p]^2}\sum_{i=1}^n\sum_{m\neq i}\Big( P_i\{ (\pi_{1,i}\psi_j)^2(X_i) P_m(\pi_{2,im}\psi_k)^2(X_i,X_m)\} \Big) .
\end{align*}
Then, Schwarz inequality gives
\begin{align*}
    & \max_{(j,k)\in[p]^2}\sum_{i=1}^n\sum_{m\neq i}\Big( P_i\{ (\pi_{1,i}\psi_j)^2(X_i) P_m(\pi_{2,im}\psi_k)^2(X_i,X_m)\} \Big) \\
    & \le \left( \max_{j\in[p]} \sum_{i=1}^n P_i(\pi_{1,i}\psi_{j})^4(X_i)  \right)^{1/2} \left( \max_{k\in[p]} \sum_{i=1}^n  P_i\left\{\sum_{m\neq i} P_m (\pi_{2,im}\psi_k)^2(X_i)\right\}^2 \right)^{1/2} \\
    & \le \left(\max_{j\in[p]}\sum_{i=1}^n \| \pi_{1,i}\psi_j(X_i) \|_{L^4{(P_i)}}^4 \right)^{1/2}\left(\max_{k\in[p]} \sum_{i=1}^n \left\|\sum_{m\neq i} P_m(|\pi_{2,im}\psi_k|^2) \right\|_{L^2(P_i)}^2\right)^{1/2}
\end{align*}
Therefore
\begin{align*}
    & \max_{(j,k)\in[p]^2} \sum_{i=1}^n\sum_{m\neq i} \int (\tilde{\chi}_{im,jk}^{(1,2)})^2(x_i,x_m) dP_i(x_i)dP_m(x_m)\\  
    & \le \left(\max_{j\in[p]}\sum_{i=1}^n \| \pi_{1,i}\psi_j(X_i) \|_{L^4{(P_i)}}^4 \right)^{1/2}\left(\max_{k\in[p]} \sum_{i=1}^n \left\|\sum_{m\neq i} P_m(|\pi_{2,im}\psi_k|^2) \right\|_{L^2(P_i)}^2\right)^{1/2}  .
\end{align*}
\end{proof}

\begin{proof}[Proof of \cref{eq:lem5-4-inid}]
From $c_r$ inequality, we have
\begin{align*}
    & \left\| \max_{(j,k)\in[p]^2} \max_{i\in[n]} \sum_{m\neq i} \int (\tilde{\chi}_{im,jk}^{(1,2)})^2(X_i,x_m)dP_m(x_m) \right\|_{L^1(\pr)} \\
    & \le \frac{1}{2}\left\| \max_{(j,k)\in[p]^2}  \max_{i\in[n]} \sum_{m\neq i}  \int ({\bar\chi}_{im,jk}^{(1,2)})^2(X_i,x_m)dP_m(x_m) \right\|_{L^1(\pr)} \\
    & \quad + \frac{1}{2}\left\| \max_{(j,k)\in[p]^2}  \max_{i\in[n]} \sum_{m\neq i}  \int ({\bar\chi}_{mi,jk}^{(1,2)})^2(x_m,X_i)dP_m(x_m) \right\|_{L^1(\pr)}.
\end{align*}

In terms of the first term, since 
\begin{align*}
    {\bar\chi}_{im,jk}^{(1,2)} = \pi_{1,i}\psi_j(X_i)\pi_{2,mi}\psi_k(X_i,X_m) - P_i\{\pi_{1,i}\psi_j(X_i)\pi_{2,im}\psi_k(X_i,X_m)\},
\end{align*}
it holds that
\begin{align*}
    & \frac{1}{2}\left\| \max_{(j,k)\in[p]^2}  \max_{i\in[n]} \sum_{m\neq i}  \int ({\bar\chi}_{im,jk}^{(1,2)})^2(X_i,x_m)dP_m(x_m) \right\|_{L^1(\pr)} \\
    & \le  \left\| \max_{(j,k)\in[p]^2}  \max_{i\in[n]} \sum_{m\neq i}  \int (\pi_{1,i}\psi_j)^2(X_i)(\pi_{2,mi}\psi_k)^2(X_i,x_m)dP_m(x_m) \right\|_{L^1(\pr)} \\
    & \quad +  \left\| \max_{(j,k)\in[p]^2}  \max_{i\in[n]} \sum_{m\neq i}  \int \{P_i(\pi_{1,i}\psi_j)(X_i)(\pi_{2,mi}\psi_k)(X_i,x_m)\}^2dP_m(x_m) \right\|_{L^1(\pr)} \\
    & \le  \left\| \max_{(j,k)\in[p]^2}  \max_{i\in[n]} \sum_{m\neq i}   (\pi_{1,i}\psi_j)^2(X_i)P_m (\pi_{2,mi}\psi_k)^2(X_i) \right\|_{L^1(\pr)} \\
    & \quad +  \max_{(j,k)\in[p]^2}  \max_{i\in[n]} \sum_{m\neq i}  P_i\{(\pi_{1,i}\psi_j)^2(X_i)P_m (\pi_{2,mi}\psi_k)^2(X_i)\} \\
    & \le 2 \mathbb{E}\left[  \max_{(j,k)\in[p]^2}  \max_{i\in[n]} \sum_{m\neq i}   (\pi_{1,i}\psi_j)^2(X_i)P_m (\pi_{2,mi}\psi_k)^2(X_i) \right].
\end{align*}
where the first inequality follows from $c_r$ inequality and the second inequality follows from Jensen's inequality.
Also, Cauchy-Schwarz inequality gives
\begin{align*}
    &  \mathbb{E}\left[  \max_{(j,k)\in[p]^2}  \max_{i\in[n]} \sum_{m\neq i}   (\pi_{1,i}\psi_j)^2(X_i)P_m (\pi_{2,mi}\psi_k)^2(X_i) \right] \\
    & \le \E\left[ \max_{j\in[p]} \max_{i\in[n]} (\pi_{1,i}\psi_j)^4(X_i) \right]^{1/2} \left\| \max_{k\in[p]} \max_{i\in[n]} \sum_{m\neq i} P_m(\pi_{2,im}\psi_k)^2(X_i) \right\|_{L^2(\pr)}.
\end{align*}
Then, we can see that 
\begin{align*}
     & \frac{1}{2}\left\| \max_{(j,k)\in[p]^2}  \max_{i\in[n]} \sum_{m\neq i}  \int ({\bar\chi}_{im,jk}^{(1,2)})^2(X_i,x_m)dP_m(x_m) \right\|_{L^1(\pr)} \\
     & \le 2 \E\left[ \max_{j\in[p]} \max_{i\in[n]} (\pi_{1,i}\psi_j)^4(X_i) \right]^{1/2} \left\| \max_{k\in[p]} \max_{i\in[n]} \sum_{m\neq i} P_m(\pi_{2,im}\psi_k)^2(X_i) \right\|_{L^2(\pr)}.
\end{align*}

In terms of the second term, since
\begin{align*}
     {\bar\chi}_{mi,jk}^{(1,2)} = (\delta_{X_m} - P_m)\pi_{1,m}\psi_j(X_m)\pi_{2,mi}\psi_k(X_m,X_i)
\end{align*}
and  $(\delta_{X_m} - P_m)$ is $L^2(P_m)$-contraction, 
\begin{align*}
    & \frac{1}{2}\left\| \max_{(j,k)\in[p]^2}  \max_{i\in[n]} \sum_{m\neq i}  \int ({\bar\chi}_{mi,jk}^{(1,2)})^2(x_m,X_i)dP_m(x_m) \right\|_{L^1(\pr)} \\
    & \le \frac{1}{2}\left\| \max_{(j,k)\in[p]^2}  \max_{i\in[n]} \sum_{m\neq i}  \int (\pi_{1,m}\psi_j)^2(x_m)(\pi_{2,mi}\psi_k)^2(x_m,X_i) dP_m(x_m) \right\|_{L^1(\pr)} \\
    & \le  \frac{1}{2}\left(  \max_{j\in[p]}\sum_{m=1}^n\int (\pi_{1,m}\psi_j)^4(x_m) dP_m(x_m)\right)^{1/2} \left\| \max_{k\in[p]}\max_{i\in[n]}  \left(\sum_{m\neq i}\int (\pi_{2,mi}\psi_k)^4(x_m,X_i) dP_m(x_m)\right)^{1/2}  \right\|_{L^1(\pr)} \\
    & \le \frac{1}{2}\left(\max_{j\in[p]} \sum_{m=1}^n \|\pi_{1,m}\psi_j\|_{L^4(P_m)}^4 \right)^{1/2} \E\left[ \max_{k\in[p]} \max_{i\in[n]} \sum_{m\neq i}  P_m(\pi_{2,mi}\psi_k)^4(X_i) \right]^{1/2},
\end{align*}
where the second follows from Schwarz inequality and the third inequality follows from Jensen's inequality.

Summing up
\begin{align*}
    & \left\| \max_{(j,k)\in[p]^2} \max_{(i,m)\in I_{n,2}} \int (\tilde{\chi}_{im,jk}^{(1,2)})^2(X_i,x_m)dP_m(x_m) \right\|_{L^1(\pr)} \\
    & \le2 \E\left[ \max_{j\in[p]} \max_{i\in[n]} (\pi_{1,i}\psi_j)^4(X_i) \right]^{1/2} \left\| \max_{k\in[p]} \max_{i\in[n]} \sum_{m\neq i} P_m(\pi_{2,im}\psi_k)^2(X_i) \right\|_{L^2(\pr)} \\
    & \quad +  \frac{1}{2}\left(\max_{j\in[p]} \sum_{m=1}^n \|\pi_{1,m}\psi_j\|_{L^4(P_m)}^4 \right)^{1/2} \E\left[ \max_{k\in[p]} \max_{i\in[n]} \sum_{m\neq i}  P_m(\pi_{2,mi}\psi_k)^4(X_i) \right]^{1/2}.
\end{align*}
\end{proof}

\begin{proof}[Proof of \cref{eq:lem5-5-inid}]
\begin{align*}
    & \left\| \max_{(j,k)\in[p]^2} \max_{(i,m)\in I_{n,2}} (\tilde{\chi}_{im,jk}^{(1,2)})^2(X_i,X_m) \right\|_{L^1(\pr)} \\
    & \le \left\| \max_{(j,k)\in[p]^2} \max_{(i,m)\in I_{n,2}} (\bar{\chi}_{im,jk}^{(1,2)})^2(X_i,X_m) \right\|_{L^1(\pr)} \\
    & \le 2\left\|  \max_{(j,k)\in[p]^2} \max_{(i,m)\in I_{n,2}} (\pi_{1,i}\psi_j)^2(X_i)(\pi_{2,im}\psi_k)^2(X_i,X_m)  \right\|_{L^1(\pr)} \\
    & \quad + 2\left\| \max_{(j,k)\in[p]^2} \max_{(i,m)\in I_{n,2}} ( P_i\{ \pi_{1,i}\psi_j(X_i)\pi_{2,im}\psi_k(X_i,X_m)\} )^2\right\|_{L^1(\pr)} \\
    & \le 2\left\|  \max_{(j,k)\in[p]^2} \max_{(i,m)\in I_{n,2}} (\pi_{1,i}\psi_j)^2(X_i)(\pi_{2,im}\psi_k)^2(X_i,X_m)\right\|_{L^1(\pr)} \\
    & \quad + 2\left\| \max_{(j,k)\in[p]^2} \max_{(i,m)\in I_{n,2}} P_i\{ (\pi_{1,i}\psi_j)^2(X_i)(\pi_{2,im}\psi_k)^2(X_i,X_m)\}\right\|_{L^1(\pr)} 
\end{align*}
where the first inequality follows from Eq.(5) in \cite{imai2025gaussian}, the second inequality follows from $c_r$ inequality and the final inequality follows from Jensen's inequality.
In terms of the second term, 
\begin{align*}
    & \left\| \max_{(j,k)\in[p]^2} \max_{(i,m)\in I_{n,2}} P_i\{ (\pi_{1,i}\psi_j)^2(X_i)(\pi_{2,im}\psi_k)^2(X_i,X_m)\}\right\|_{L^1(\pr)}  \\
    & \le 2 \left\|  \max_{(j,k)\in[p]^2} \max_{(i,m)\in I_{n,2}} (\pi_{1,i}\psi_j)^2(X_i)(\pi_{2,im}\psi_k)^2(X_i,X_m)\right\|_{L^1(\pr)}.
\end{align*}
Also, Schwarz inequality gives
\begin{align*}
    &\left\|  \max_{(j,k)\in[p]^2} \max_{(i,m)\in I_{n,2}} (\pi_{1,i}\psi_j)^2(X_i)(\pi_{2,im}\psi_k)^2(X_i,X_m)\right\|_{L^1(\pr)}\\
    & \le  \left\|  \max_{j\in[p]} \max_{i\in[n]} (\pi_{1,i}\psi_j)^2(X_i)  \max_{(j,k)\in[p]^2} \max_{(i,m)\in I_{n,2}} (\pi_{2,im}\psi_k)^2(X_i,X_m)  \right\|_{L^1(\pr)} \\
    & \le \left\| \max_{j\in[p]} \max_{i\in[n]} (\pi_{1,i}\psi_j)^4(X_i) \right\|^{1/2}_{L^1(\pr)} \left\| \max_{(j,k)\in[p]^2} \max_{(i,m)\in I_{n,2}} (\pi_{2,im}\psi_k)^4(X_i,X_m) \right\|^{1/2}_{L^1(\pr)} \\
    & = \E\left[ \max_{j\in[p]} \max_{i\in[n]} (\pi_{1,i}\psi_j)^4(X_i) \right]^{1/2}  \left\| \max_{(j,k)\in[p]^2} \max_{(i,m)\in I_{n,2}} |\pi_{2,im}\psi_k| \right\|^{2}_{L^4(\pr)}.
\end{align*}
Therefore
\begin{align*}
    & \left\| \max_{(j,k)\in[p]^2} \max_{(i,m)\in I_{n,2}} (\tilde{\chi}_{im,jk}^{(1,2)})^2(X_i,X_m) \right\|_{L^1(\pr)} \\
    & \le 6\E\left[ \max_{j\in[p]} \max_{i\in[n]} (\pi_{1,i}\psi_j)^4(X_i) \right]^{1/2}  \left\| \max_{(j,k)\in[p]^2} \max_{(i,m)\in I_{n,2}} |\pi_{2,im}\psi_k| \right\|^{2}_{L^4(\pr)}.
\end{align*}
\end{proof}

\begin{proof}[Proof of \cref{eq:lem5-add1-inid}]
Since $\varphi^{(1,2)}_{m,jk}(X_m) \coloneqq \sum_{i:m\neq i}2P_i\{ \pi_{1,i}\psi_j(X_i)\pi_{2,im}\psi_k(X_i,X_m)\}$, we have
\begin{align*}
    &\max_{(j,k)\in[p]^2} \sum_{m=1}^n \int (\varphi^{(1,2)}_{m,jk})^2(x_m)dP_m(x_m) \\
    & = 4  \max_{(j,k)\in[p]^2} \sum_{m=1}^n \int  \left( \sum_{i:m\neq i} P_i\{ \pi_{1,i}\psi_j(X_i)\pi_{2,im}\psi_k(X_i,x_m)\} \right)^2 dP_m(x_m) \\
    & = 4  \max_{(j,k)\in[p]^2} \sum_{m=1}^n \int \left( \sum_{i:m\neq i} \pi_{1,i}\psi_j \star_i^1 \pi_{2,im}\psi_k\right)^2(x_m)dP_m(x_m) \\
    & = 4  \max_{(j,k)\in[p]^2} \sum_{m=1}^n \left\|\sum_{i:m\neq i}\pi_{1,i}\psi_j \star_i^1 \pi_{2,im}\psi_k \right\|_{L^2(P_m)}^2.
\end{align*}
\end{proof}

\begin{proof}[Proof of \cref{eq:lem5-add2-inid}]
Since $\varphi^{(1,2)}_{m,jk}(X_m) \coloneqq \sum_{i:m\neq i}2P_i\{ \pi_{1,i}\psi_j(X_i)\pi_{2,im}\psi_k(X_i,X_m)\}$, we have
\begin{align*}
    & \left\| \max_{(j,k)\in[p]^2} \max_{m\in[n]} (\varphi^{(1,2)}_{m,jk})^2(X_m) \right\|_{L^{1}(\pr)} \\
    & =  \left\| \max_{(j,k)\in[p]^2} \max_{m\in[n]} \left( \sum_{i:m\neq i}2P_i\{ \pi_{1,i}\psi_j(X_i)\pi_{2,im}\psi_k(X_i,X_m)\} \right)^2 \right\|_{L^{1}(\pr)} \\
    & \le 4  \left\| \max_{(j,k)\in[p]^2} \max_{m\in[n]} \left( \sum_{i:m\neq i} \|\pi_{1,i}\psi_j\|_{L^2(P_i)}^2 \right) \left( \sum_{i:m\neq i}P_i(\pi_{2,im}\psi_k)^2(X_m) \right)\right\|_{L^{1}(\pr)} \\
    & \le 4 \left(\max_{j\in[p]} \sum_{i=1}^n \|\pi_{1,i}\psi_j\|_{L^2(P_i)}^2 \right) \left\|  \max_{j\in[p]}\max_{i\in[n]} \sum_{m\neq i}P_m(\pi_{2,im}\psi_j)^2(X_i) \right\|_{L^1(\pr)}.
\end{align*}
where the second line follows from Schwarz inequality.
\end{proof}

\begin{proof}[Proof of \cref{eq:lem5-6-inid}]
Observe that
\begin{align*} 
    \tilde{\chi}_{im,jk}^{(2,2), \mathtt{diag}} = \frac{1}{2}(\delta_{X_i} - P_i)(\delta_{X_m} - P_m) \{\pi_{2,im}\psi_j(X_i,X_m) \pi_{2,im}\psi_k(X_i,X_m)\}.
\end{align*}
Since $(\delta_{X_i} - P_i)(\delta_{X_m} - P_m)$ is contraction on $L^2(P_i\otimes P_m)$, we have 
\begin{align*}
   & \max_{(j,k)\in[p]^2} \sum_{i=1}^n \sum_{m\neq i} \int (\tilde{\chi}_{im,jk}^{(2,2), \mathtt{diag}})^2(x_i,x_m) dP_i(x_i)dP_m(x_m) \\
   & \le \frac{1}{4}\max_{(j,k)\in[p]^2} \sum_{i=1}^n \sum_{m\neq i}  \int \Big( \pi_{2,im}\psi_j(x_i,x_m)\pi_{2,im}\psi_k(x_i,x_m)  \Big)^2 dP_i(x_i)dP_m(x_m) .
\end{align*}
Schwarz inequality gives
\begin{align*}
    & \frac{1}{4}\max_{(j,k)\in[p]^2} \sum_{i=1}^n \sum_{m\neq i}   \int \Big( \pi_{2,im}\psi_j(x_i,x_m)\pi_{2,im}\psi_k(x_i,x_m) \Big)^2 dP_i(x_i)dP_m(x_m) \\
    & \le  \frac{1}{4}\max_{(j,k)\in[p]^2}\sum_{i=1}^n \sum_{m\neq i}  \|\pi_{2,im} \psi_j\|_{L^4(P_i\otimes P_m)}^2 \|\pi_{2,im} \psi_k\|_{L^4(P_i\otimes P_m)}^2 \\
    & \le \frac{1}{4}\max_{j\in[p]} \sum_{i=1}^n \sum_{m\neq i} \|\pi_{2,im} \psi_j\|_{L^4(P_i\otimes P_m)}^4 .
\end{align*}
Therefore, it holds that
\begin{align*}
    & \max_{(j,k)\in[p]^2} \sum_{i=1}^n \sum_{m\neq i} \int (\tilde{\chi}_{im,jk}^{(2,2), \mathtt{diag}})^2(x_i,x_m) dP_i(x_i)dP_m(x_m) \le \frac{1}{4}\max_{j\in[p]}\sum_{i=1}^n \sum_{m\neq i} \|\pi_{2,im} \psi_j\|_{L^4(P_i\otimes P_m)}^4  .
\end{align*}
\end{proof}

\begin{proof}[Proof of \cref{eq:lem5-7-inid}]
Observe that
\begin{align*} 
    \tilde{\chi}_{im,jk}^{(2,2), \mathtt{diag}} = \frac{1}{2}(\delta_{X_m} - P_m) \{(\pi_{2,im}\psi_j \pi_{2,im}\psi_k)(X_i,X_m) - P_i(\pi_{2,im}\psi_j \pi_{2,im}\psi_k)(X_m)\}.
\end{align*}
Since $(\delta_{X_m} - P_m)$ is contraction on $L^2(P_m)$,
\begin{align*}
    & \left\| \max_{(j,k)\in[p]^2} \max_{i\in[n]} \sum_{m\neq i} \int (\tilde{\chi}_{im,jk}^{(2,2), \mathtt{diag}})^2(X_i,x_m)dP_m(x_m) \right\|_{L^1(\pr)} \\
    & \le \frac{1}{4}\left\| \max_{(j,k)\in[p]^2}  \max_{i\in[n]} \sum_{m\neq i}  \int \Big( \pi_{2,im}\psi_j(X_i,x_m)\pi_{2,im}\psi_k(X_i,x_m) \right. \\
    & \left. \qquad\qquad\qquad\qquad\qquad\qquad - P_i\{\pi_{2,im}\psi_j(X_i,x_m)\pi_{2,im}\psi_k(X_i,x_m)\} \Big)^2 dP_m(x_m)  \right\|_{L^1(\pr)} \\
    & \le \frac{1}{2} \left\| \max_{(j,k)\in[p]^2}  \max_{i\in[n]} \sum_{m\neq i}  \int \Big( \pi_{2,im}\psi_j(X_i,x_m)\pi_{2,im}\psi_k(X_i,x_m) \Big)^2 dP_m(x_m)  \right\|_{L^1(\pr)} \\
    & \quad + \frac{1}{2}\max_{(j,k)\in[p]^2}  \max_{i\in[n]} \sum_{m\neq i} \int \Big( P_i\{\pi_{2,im}\psi_j(X_i,x_m)\pi_{2,im}\psi_k(X_i,x_m)\} \Big)^2 dP_m(x_m).
\end{align*}
where the second inequality follows from $c_r$-inequality.
In terms of the first term,
\begin{align*}
   & \frac{1}{2}\left\| \max_{(j,k)\in[p]^2}  \max_{i\in[n]} \sum_{m\neq i}  \int (\pi_{2,im}\psi_j)^2(X_i,x_m)(\pi_{2,im}\psi_k)^2(X_i,x_m) dP_m(x_m) \right\|_{L^1(\pr)}    \\
    & \le \frac{1}{2}\E\left[ \max_{(j,k)\in[p]^2}  \max_{i\in[n]} \sum_{m\neq i} \left(P_m(\pi_{2,im}\psi_j)^4(X_i)\right)^{1/2}\left(P_m(\pi_{2,im}\psi_k)^4(X_i)\right)^{1/2} \right] \\
    & \le \frac{1}{2}\E\left[ \max_{j\in[p]}  \max_{i\in[n]} \sum_{m\neq i} P_m(\pi_{2,im}\psi_j)^4(X_i) \right],
\end{align*}
where the first inequality follows from Schwarz inequality and the second inequality follows from AM-GM inequality. 
In terms of the second term, 
\begin{align*}
    & \frac{1}{2}\max_{(j,k)\in[p]^2}  \max_{i\in[n]} \sum_{m\neq i}\int  \Big( P_i\{\pi_{2,im}\psi_j(X_i,x_m)\pi_{2,im}\psi_k(X_i,x_m)\} \Big)^2 dP_m(x_m) \\
    & \le \frac{1}{2}\max_{(j,k)\in[p]^2}  \max_{i\in[n]} \sum_{m\neq i}  P_iP_m\Big(\{\pi_{2,im}\psi_j(X_i,X_m)\pi_{2,im}\psi_k(X_i,X_m)\} \Big)^2 \\
    & \le \frac{1}{2}\max_{(j,k)\in[p]^2}  \max_{i\in[n]} \sum_{m\neq i}  P_iP_m(\pi_{2,im}\psi_j)^4 \\
    & \le\frac{1}{2}\E\left[ \max_{j\in[p]}  \max_{i\in[n]} \sum_{m\neq i} P_m(\pi_{2,im}\psi_j)^4(X_i) \right],
\end{align*}
where the first inequality follows from Jensen's inequality and the second inequality follows from AM-GM inequality.
Summing up
\begin{align*}
    & \left\| \max_{(j,k)\in[p]^2}  \max_{i\in[n]} \sum_{m\neq i} \int (\tilde{\chi}_{im,jk}^{(2,2),\mathtt{diag}})^2(X_i,x_m)dP_m(x_m) \right\|_{L^1(\pr)}  \le \E\left[ \max_{j\in[p]}  \max_{i\in[n]} \sum_{m\neq i} P_m(\pi_{2,im}\psi_j)^4(X_i) \right] . 
\end{align*}

\end{proof}

\begin{proof}[Proof of \cref{eq:lem5-8-inid}]
Observe that
\begin{align*} 
    \tilde{\chi}_{im,jk}^{(2,2), \mathtt{diag}} & = \frac{1}{2} \{(\pi_{2,im}\psi_j \pi_{2,im}\psi_k) - P_i(\pi_{2,im}\psi_j \pi_{2,im}\psi_k)\\
    & \qquad\qquad- P_m(\pi_{2,im}\psi_j \pi_{2,im}\psi_k) + P_iP_m(\pi_{2,im}\psi_j \pi_{2,im}\psi_k)\}.
\end{align*}
Then, $c_r$ inequality gives 
\begin{align*}
    & \left\| \max_{(j,k)\in[p]^2} \max_{(i,m)\in I_{n,2}} (\tilde{\chi}_{im,jk}^{(2,2), \mathtt{diag}})^2(X_i,X_m) \right\|_{L^1(\pr)} \\
    & \le \left\| \max_{(j,k)\in[p]^2} \max_{(i,m)\in I_{n,2}} (\pi_{2,im}\psi_j)^2(X_i,X_m)(\pi_{2,im}\psi_k)^2(X_i,X_m)  \right\|_{L^1(\pr)} \\
    & \quad + \left\| \max_{(j,k)\in[p]^2} \max_{(i,m)\in I_{n,2}} (P_i\{\pi_{2,im}\psi_j\pi_{2,im}\psi_k\})^2(X_m)  \right\|_{L^1(\pr)} \\
    & \quad + \left\| \max_{(j,k)\in[p]^2} \max_{(i,m)\in I_{n,2}} (P_m\{\pi_{2,im}\psi_j\pi_{2,im}\psi_k\})^2(X_i)  \right\|_{L^1(\pr)} \\
    & \quad + \max_{(j,k)\in[p]^2} \max_{(i,m)\in I_{n,2}} (P_iP_m\{\pi_{2,im}\psi_j\pi_{2,im}\psi_k\})^2.
\end{align*}
In terms of the first term, we can see that
\begin{align*}
    & \left\| \max_{(j,k)\in[p]^2} \max_{(i,m)\in I_{n,2}} (\pi_{2,im}\psi_j)^2(X_i,X_m)(\pi_{2,im}\psi_k)^2(X_i,X_m)  \right\|_{L^1(\pr)}  =  \left\| \max_{j\in[p]} \max_{(i,m)\in I_{n,2}} |\pi_{2,im}\psi_j|\right\|_{L^4(\pr)}^4.
\end{align*}
In terms of the second term, 
\begin{align*}
    & \left\| \max_{(j,k)\in[p]^2} \max_{(i,m)\in I_{n,2}} (P_i\{\pi_{2,im}\psi_j\pi_{2,im}\psi_k\})^2(X_m)  \right\|_{L^1(\pr)}  \\
    & \le \left\| \max_{(j,k)\in[p]^2} \max_{(i,m)\in I_{n,2}} P_i(\pi_{2,im}\psi_j\pi_{2,im}\psi_k)^2(X_m)  \right\|_{L^1(\pr)} \\
    & \le \left\| \max_{j\in[p]} \max_{(i,m)\in I_{n,2}} P_i(\pi_{2,im}\psi_j)^4(X_m)  \right\|_{L^1(\pr)} \\
    & \le 2  \left\| \max_{j\in[p]} \max_{(i,m)\in I_{n,2}} |\pi_{2,im}\psi_j|\right\|_{L^4(\pr)}^4,
\end{align*}
where the first inequality follows from Jensen's inequality, the second inequality follows from AM-GM inequality and the final inequality follows from \cref{lem:max-Jensen_inid}.
Similarly, in terms of the third term,
\begin{align*}
    & \left\| \max_{(j,k)\in[p]^2} \max_{(i,m)\in I_{n,2}} (P_m\{\pi_{2,im}\psi_j\pi_{2,im}\psi_k\})^2(X_i)  \right\|_{L^1(\pr)} \le 2  \left\| \max_{j\in[p]} \max_{(i,m)\in I_{n,2}} |\pi_{2,im}\psi_j|\right\|_{L^4(\pr)}^4.
\end{align*}
In terms of the final term, it holds that
\begin{align*}
    &  \max_{(j,k)\in[p]^2} \max_{(i,m)\in I_{n,2}} (P_iP_m\{\pi_{2,im}\psi_j\pi_{2,im}\psi_k\})^2 \le \left\| \max_{j\in[p]} \max_{(i,m)\in I_{n,2}} |\pi_{2,im}\psi_j|\right\|_{L^4(\pr)}^4.
\end{align*}
Summing up, 
\begin{align*}
     \left\| \max_{(j,k)\in[p]^2} \max_{(i,m)\in I_{n,2}} (\tilde{\chi}_{im,jk}^{(2,2), \mathtt{diag}})^2(X_i,X_m) \right\|_{L^1(\pr)} \le 6 \left\| \max_{j\in[p]} \max_{(i,m)\in I_{n,2}} |\pi_{2,im}\psi_j|\right\|_{L^4(\pr)}^4.
\end{align*}
\end{proof}

\begin{proof}[Proof of \cref{eq:lem5-9-inid}]
Observe that
\begin{align}
    &\bar\chi^{(2,2)}_{iml,jk}(X_i,X_m,X_l)  \coloneqq \frac{1}{2} (\delta_{X_i} - P_i)  \pi_{2,im}\psi_j(X_i,X_m)\pi_{2,il}\psi_k(X_i,X_l). 
\end{align}
Then, 
since $(\delta_{X_i} - P_i)$ is $L^2(P_i)$-contraction, 
from Eq.(5) in \cite{imai2025gaussian}, we have
\begin{align*}
     & \max_{(j,k)\in[p]^2} \sum_{i=1}^n\sum_{m\neq i}\sum_{l\neq i,m} \int (\tilde{\chi}^{(2,2)}_{iml,jk})^2(x_i,x_m,x_l) dP_i(x_i)dP_m(x_m)dP_l(x_l) \\
     & \le \max_{(j,k)\in[p]^2}\sum_{i=1}^n\sum_{m\neq i}\sum_{l\neq i,m}  \int (\bar\chi^{(2,2)}_{iml,jk})^2(x_i,x_m,x_l) dP_i(x_i)dP_m(x_m)dP_l(x_l) \\
     & \le \frac{1}{4}\max_{(j,k)\in[p]^2} \sum_{i=1}^n\sum_{m\neq i}\sum_{l\neq i,m}   \int \Big( \pi_{2,im}\psi_j(x_i,x_m)\pi_{2,il}\psi_k(x_i,x_l) \Big)^2 dP_i(x_i)dP_m(x_m)dP_l(x_l)  .
\end{align*} 
Tower property of the conditionals expectations gives
\begin{align*}
    & \max_{(j,k)\in[p]^2} \sum_{i=1}^n\sum_{m\neq i}\sum_{l\neq i,m}   \int \Big( \pi_{2,im}\psi_j(x_i,x_m)\pi_{2,il}\psi_k(x_i,x_l) \Big)^2 dP_i(x_i)dP_m(x_m)dP_l(x_l) \\
    & = \max_{(j,k)\in[p]^2} \sum_{i=1}^n\sum_{m\neq i}\sum_{l\neq i,m}  \int P_m(\pi_{2,im}\psi_j)^2(x_i) P_l(\pi_{2,il}\psi_k)^2(x_i) dP_i(x_i) \\
    & \le \max_{(j,k)\in[p]^2} \sum_{i=1}^n\sum_{m\neq i}\sum_{l\neq i}  \int P_m(\pi_{2,im}\psi_j)^2(x_i) P_l(\pi_{2,il}\psi_k)^2(x_i) dP_i(x_i)\\
    & = \max_{(j,k)\in[p]^2} \sum_{i=1}^n  \int \left( \sum_{m\neq i} P_m(\pi_{2,im}\psi_j)^2(x_i) \right) \left( \sum_{l\neq i}P_l(\pi_{2,il}\psi_k)^2(x_i) \right) dP_i(x_i).
\end{align*}
Schwarz inequality gives
\begin{align*}
    & \max_{(j,k)\in[p]^2} \sum_{i=1}^n  \int \left( \sum_{m\neq i} P_m(\pi_{2,im}\psi_j)^2(x_i) \right) \left( \sum_{l\neq i}P_l(\pi_{2,il}\psi_k)^2(x_i) \right) dP_i(x_i) \\
    & \le \max_{(j,k)\in[p]^2} \left[\sum_{i=1}^n P_i \left( \sum_{m\neq i} P_m(\pi_{2,im}\psi_j)^2(x_i) \right)^2 \right]^{1/2} \left[  \sum_{i=1}^n P_i \left( \sum_{l\neq i}P_l(\pi_{2,il}\psi_k)^2(x_i) \right) ^2 \right]^{1/2} \\
    & \le \max_{j\in[p]} \sum_{i=1}^n \left\| \sum_{m\neq i} P_m(\pi_{2,im}\psi_j)^2 \right\|_{L^2(P_i)}^2.
\end{align*}
\end{proof}

\begin{proof}[Proof of \cref{eq:lem5-10-inid}]
From the $c_r$ inequality, we have,
\begin{align*}
     & \left\| \max_{(j,k)\in[p]^2} \max_{i\in[n]} \sum_{m\neq i}\sum_{l\neq i,m} \int (\tilde{\chi}^{(2,2)}_{iml,jk})^2(X_i,x_m,x_l)dP_m(x_m)dP_l(x_l) \right\|_{L^1(\pr)} \\
     & \le \frac{1}{3} \left\| \max_{(j,k)\in[p]^2} \max_{i\in[n]} \sum_{m\neq i}\sum_{l\neq i,m} \int (\bar\chi^{(2,2)}_{iml,jk})^2(X_i,x_m,x_l)dP_m(x_m)dP_l(x_l) \right\|_{L^1(\pr)}  \\
     & \quad + \frac{1}{3} \left\| \max_{(j,k)\in[p]^2}\max_{i\in[n]} \sum_{m\neq i}\sum_{l\neq i,m} \int (\bar\chi^{(2,2)}_{mil,jk})^2(x_m,X_i,x_l)dP_m(x_m)dP_l(x_l) \right\|_{L^1(\pr)}  \\
     & \quad + \frac{1}{3}  \left\| \max_{(j,k)\in[p]^2} \max_{i\in[n]} \sum_{m\neq i}\sum_{l\neq i,m} \int (\bar\chi^{(2,2)}_{mli,jk})^2(x_m,x_l,X_i)dP_m(x_m)dP_l(x_l) \right\|_{L^1(\pr)}.
\end{align*}
In terms of the first term, we can see that
\begin{align*}
    & \frac{1}{3} \left\| \max_{(j,k)\in[p]^2} \max_{i\in[n]} \sum_{m\neq i}\sum_{l\neq i,m} \int (\bar\chi^{(2,2)}_{iml,jk})^2(X_i,x_m,x_l)dP_m(x_m)dP_l(x_l) \right\|_{L^1(\pr)} \\
    & \le \frac{1}{6}  \left\| \max_{(j,k)\in[p]^2} \max_{i\in[n]} \sum_{m\neq i}\sum_{l\neq i,m} \int (\pi_{2,im}\psi_j)^2(X_i,x_m)(\pi_{2,il}\psi_k)^2(X_i,x_l)dP_m(x_m)dP_l(x_l) \right\|_{L^1(\pr)} \\
    & \quad + \frac{1}{6} \max_{(j,k)\in[p]^2} \max_{i\in[n]} \sum_{m\neq i}\sum_{l\neq i,m} \int \left( \int \pi_{2,im}\psi_j(x_i,x_m)\pi_{2,il}\psi_k(x_i,x_l)dP_i(x_i)\right)^2 dP_m(x_m)dP_l(x_l) \\
    & \le\frac{1}{6}  \left\| \max_{(j,k)\in[p]^2} \max_{i\in[n]} \sum_{m\neq i}\sum_{l\neq i,m} \int (\pi_{2,im}\psi_j)^2(X_i,x_m)(\pi_{2,il}\psi_k)^2(X_i,x_l)dP_m(x_m)dP_l(x_l) \right\|_{L^1(\pr)}  \\
    & \quad + \frac{1}{6}  \max_{(j,k)\in[p]^2} \max_{i\in[n]} \sum_{m\neq i}\sum_{l\neq i,m} P_i \int (\pi_{2,im}\psi_j)^2(X_i,x_m)(\pi_{2,il}\psi_k)^2(X_i,x_l)dP_m(x_m)dP_l(x_l) \\
    & \le \frac{1}{3} \E\left[ \max_{(j,k)\in[p]^2} \max_{i\in[n]} \left( \sum_{m\neq i} P_m(\pi_{2,im}\psi_j)^2(X_i) \right) \left( \sum_{l\neq i}P_l(\pi_{2,il}\psi_k)^2(X_i)  \right) \right]  \\
    & \le \frac{1}{3} \left\| \max_{j\in[p]} \max_{i\in[n]} \sum_{m\neq i} P_m(\pi_{2,im}\psi_j)^2(X_i) \right\|_{L^2(\mathbb P )}^2,
\end{align*}
where the first inequality follows from $c_r$ inequality, the second inequality follows from Jensen's inequality. 
In terms of the second term, we can see that
\begin{align*}
    & \frac{1}{3} \left\| \max_{(j,k)\in[p]^2} \max_{i\in[n]} \sum_{m\neq i}\sum_{l\neq i,m} \int (\bar\chi^{(2,2)}_{mil,jk})^2(x_m,X_i,x_l)dP_m(x_m)dP_l(x_l) \right\|_{L^1(\pr)} \\
    & \le \frac{1}{6} \left\| \max_{(j,k)\in[p]^2} \max_{i\in[n]} \sum_{m\neq i}\sum_{l\neq i,m}\int (\pi_{2,mi}\psi_j)^2(x_m,X_i)(\pi_{2,ml}\psi_k)^2 (x_m,x_l)dP_m(x_m)dP_l(x_l)\right\|_{L^1(\pr)} \\
    & \quad + \frac{1}{6} \left\| \max_{(j,k)\in[p]^2}   \max_{i\in[n]} \sum_{m\neq i}\sum_{l\neq i,m}\int \left( \int\pi_{2,mi}\psi_j(x_m,X_i)\pi_{2,ml}\psi_k(x_m,x_l) dP_m(x_m)\right)^2 dP_l(x_l) \right\|_{L^1(\pr)} \\
    & \le \frac{1}{3} \left\| \max_{(j,k)\in[p]^2} \max_{i\in[n]} \sum_{m\neq i}\sum_{l\neq m}\int (\pi_{2,mi}\psi_j)^2(x_m,X_i)(\pi_{2,ml}\psi_k)^2 (x_m,x_l)dP_m(x_m)dP_l(x_l)\right\|_{L^1(\pr)}\\
    & = \frac{1}{3} \left\| \max_{(j,k)\in[p]^2} \max_{i\in[n]} \sum_{m\neq i} P_m \left( (\pi_{2,im}\psi_j)^2(X_i,X_m) \sum_{l\neq m} P_l(\pi_{2,ml}\psi_k)^2(X_m) \right)  \right\|_{L^1(\pr)}\\
    & \le \frac{1}{3}\mathbb{E}\left[ \max_{j\in[p]} \max_{i\in[n]}  \sum_{m\neq i} P_m (\pi_{2,im}\psi_j)^4(X_i)\right]^{1/2} \left(\max_{j\in[p]} \sum_{i=1}^n \left\| \sum_{m\neq i} P_m(\pi_{2,im}\psi_j)^2 (X_i)\right\|_{L^2(P_i)}^2\right)^{1/2},
\end{align*}
where the final inequality follows from Schwarz inequality.
In terms of the third term, we can see that
\begin{align*}
    & \frac{1}{3}  \left\| \max_{(j,k)\in[p]^2} \max_{i\in[n]} \sum_{m\neq i}\sum_{l\neq i,m} \int (\bar\chi^{(2,2)}_{mli,jk})^2(x_m,x_l,X_i)dP_m(x_m)dP_l(x_l) \right\|_{L^1(\pr)} \\
    & \le \frac{1}{6} \left\|  \max_{(j,k)\in[p]^2} \max_{i\in[n]} \sum_{m\neq i}\sum_{l\neq i,m} \int (\pi_{2,ml}\psi_j)^2(x_m,x_l)(\pi_{2,mi}\psi_k)^2(x_m,X_i) dP_m(x_m)dP_l(x_l)\right\|_{L^1(\pr)} \\
    & \quad + \frac{1}{6} \left\|\max_{(j,k)\in[p]^2} \max_{i\in[n]} \sum_{m\neq i}\sum_{l\neq i,m} \int  \left( \int \pi_{2,ml}\psi_j(x_m,x_l)\pi_{2,mi}\psi_k(x_m,X_i) dP_m(x_m)\right)^2 dP_l(x_l) \right\|_{L^1(\pr)} \\
    & \le\frac{1}{3} \left\|  \max_{(j,k)\in[p]^2} \max_{i\in[n]} \sum_{m\neq i}\sum_{l\neq m} \int (\pi_{2,ml}\psi_j)^2(x_m,x_l)(\pi_{2,mi}\psi_k)^2(x_m,X_i) dP_m(x_m)dP_l(x_l)\right\|_{L^1(\pr)}  \\
    & = \frac{1}{3} \left\|  \max_{(j,k)\in[p]^2} \max_{i\in[n]} \sum_{m\neq i} P_m \left( (\pi_{2,mi}\psi_k)^2(X_m,X_i) \sum_{l\neq m} P_l(\pi_{2,ml}\psi_j)^2(X_m) \right)\right\|_{L^1(\pr)} \\
    & \le \frac{1}{3}\mathbb{E}\left[ \max_{j\in[p]} \max_{i\in[n]}  \sum_{m\neq i} P_m (\pi_{2,im}\psi_j)^4(X_i)\right]^{1/2} \left(\max_{j\in[p]} \sum_{i=1}^n \left\| \sum_{m\neq i} P_m(\pi_{2,im}\psi_j)^2 (X_i)\right\|_{L^2(P_i)}^2 \right)^{1/2}.
\end{align*}
Summing up
\begin{align*}
    &   \left\| \max_{(j,k)\in[p]^2} \max_{i\in[n]} \sum_{m\neq i}\sum_{l\neq i,m} \int (\tilde{\chi}^{(2,2)}_{iml,jk})^2(X_i,x_m,x_l)dP_m(x_m)dP_l(x_l) \right\|_{L^1(\pr)}  \\
    & \le \frac{1}{3} \left\| \max_{j\in[p]} \max_{i\in[n]} \sum_{m\neq i} P_m(\pi_{2,im}\psi_j)^2(X_i) \right\|_{L^2(\mathbb P )}^2 \\
    & \quad + \frac{2}{3}\mathbb{E}\left[ \max_{j\in[p]} \max_{i\in[n]}  \sum_{m\neq i} P_m (\pi_{2,im}\psi_j)^4(X_i)\right]^{1/2} \left(\max_{j\in[p]} \sum_{i=1}^n \left\| \sum_{m\neq i} P_m(\pi_{2,im}\psi_j)^2 (X_i)\right\|_{L^2(P_i)}^2 \right)^{1/2}.
\end{align*}

\end{proof}

\begin{proof}[Proof of \cref{eq:lem5-11-inid}]
From $c_r$ inequality, we have
\begin{align*}
    & \left\| \max_{(j,k)\in[p]^2} \max_{(i,m)\in I_{n,2}} \sum_{l\neq i,m} \int (\tilde{\chi}^{(2,2)}_{iml,jk})^2(X_i,X_m,x_l)dP_l(x_l)\right\|_{L^1(\pr)} \\
     &  \le \frac{1}{3} \left\| \max_{(j,k)\in[p]^2}\max_{(i,m)\in I_{n,2}} \sum_{l\neq i,m}  \int (\bar\chi^{(2,2)}_{iml,jk})^2(X_i,X_m,x_l)dP_l(x_l) \right\|_{L^1(\pr)}  \\
     & \quad + \frac{1}{3} \left\| \max_{(j,k)\in[p]^2} \max_{(i,m)\in I_{n,2}} \sum_{l\neq i,m}  \int (\bar\chi^{(2,2)}_{lim,jk})^2(x_l,X_i,X_m)dP_l(x_l) \right\|_{L^1(\pr)}  \\
     & \quad + \frac{1}{3}  \left\| \max_{(j,k)\in[p]^2} \max_{(i,m)\in I_{n,2}} \sum_{l\neq i,m}  \int (\bar\chi^{(2,2)}_{mli,jk})^2(X_m,x_l,X_i)dP_l(x_l) \right\|_{L^1(\pr)} .
\end{align*}
In terms of the first term, from $c_r$ inequality, we can see that
\begin{align*}
    &  \frac{1}{3} \left\| \max_{(j,k)\in[p]^2} \max_{(i,m)\in I_{n,2}} \sum_{l\neq i,m}  \int (\bar\chi^{(2,2)}_{iml,jk})^2(X_i,X_m,x_l)dP_l(x_l) \right\|_{L^1(\pr)} \\
    & \le \frac{1}{6} \left\| \max_{(j,k)\in[p]^2} \max_{(i,m)\in I_{n,2}}   (\pi_{2,im}\psi_j)^2(X_i,X_m) \sum_{l\neq i,m}\int (\pi_{2,il}\psi_k)^2(X_i,x_l) dP_l(x_l) \right\|_{L^1(\pr)} \\
    & \quad + \frac{1}{6} \left\| \max_{(j,k)\in[p]^2} \max_{(i,m)\in I_{n,2}} \sum_{l\neq i,m}  \int \left( \int \pi_{2,im}\psi_j(x_i,X_m)\pi_{2,il}\psi_k(x_i,x_l) dP_i(x_i) \right)^2 dP_l(x_l) \right\|_{L^1(\pr)} \\
    & \le \frac{1}{6}\E\left[ \max_{j\in[p]} \max_{(i,m)\in I_{n,2}} (\pi_{2,im}\psi_j)^4(X_i,X_m) \right]^{1/2} \left\| \max_{j\in[p]} \max_{i\in[n]}  \sum_{l\neq i}P_l(\pi_{2,il}\psi_j)^2(X_i) \right\|_{L^2(\mathbb P)} \\
    & \quad + \frac{1}{6} \E\left[ \max_{j\in[p]} \max_{(i,m)\in I_{n,2}} P_i(\pi_{2,im}\psi_j)^4(X_m) \right]^{1/2} \left( \max_{j\in[p]}\max_{i\in[n]} \left\|\sum_{l\neq i} P_l(\pi_{2,il}\psi_j)^2(X_i) \right\|_{L^2(P_i)}^2 \right)^{1/2} \\
    & \le \frac{1}{6}\left\| \max_{j\in[p]} \max_{(i,m)\in I_{n,2}} |\pi_{2,im}\psi_j| \right\|^{2}_{L^4(\mathbb P)} \left\| \max_{j\in[p]} \max_{i\in[n]}  \sum_{m\neq i}P_m(\pi_{2,im}\psi_j)^2(X_i) \right\|_{L^2(\mathbb P)} \\
    & \quad + \frac{1}{6} \E\left[ \max_{j\in[p]} \max_{i\in[n]} \sum_{m\neq i} P_m(\pi_{2,im}\psi_j)^4(X_i) \right]^{1/2} \left( \max_{j\in[p]}\sum_{i=1}^n \left\|\sum_{m\neq i} P_l(\pi_{2,il}\psi_j)^2(X_i) \right\|_{L^2(P_i)}^2 \right)^{1/2},
\end{align*}
where the second inequality follows from Schwarz inequality.
In terms of the second term, $c_r$ inequality and Jensen's inequality, we can see that
\begin{align*}
    & \frac{1}{3} \left\| \max_{(j,k)\in[p]^2}  \max_{(i,m)\in I_{n,2}} \sum_{l\neq i,m}  \int (\bar\chi^{(2,2)}_{lim,jk})^2(x_l,X_i,X_m)dP_l(x_l) \right\|_{L^1(\pr)} \\
    & \le \frac{1}{6} \left\| \max_{(j,k)\in[p]^2}  \max_{(i,m)\in I_{n,2}} \sum_{l\neq i,m}  \int (\pi_{2,li}\psi_j)^2(x_l,X_i)(\pi_{2,lm}\psi_k)^2(x_l,X_m)dP_l(x_l) \right\|_{L^1(\pr)} \\
    & \quad + \frac{1}{6} \left\| \max_{(j,k)\in[p]^2}  \max_{(i,m)\in I_{n,2}} \sum_{l\neq i,m}  \left(\int  \pi_{2,li}\psi_j(x_l,X_i)\pi_{2,lm}\psi_k(x_l,X_m )dP_l(x_l)\right)^2 \right\|_{L^1(\pr)} \\
    & \le \frac{1}{3} \left\| \max_{(j,k)\in[p]^2}  \max_{(i,m)\in I_{n,2}} \sum_{l\neq i,m}  \int (\pi_{2,li}\psi_j)^2(x_l,X_i)(\pi_{2,lm}\psi_k)^2(x_l,X_m)dP_l(x_l) \right\|_{L^1(\pr)}\\
    & \le \frac{1}{3}  \E\left[ \max_{(j,k)\in[p]^2}  \max_{(i,m)\in I_{n,2}}  \left( \sum_{l\neq i} P_l(\pi_{2,li}\psi_j)^4(X_i)\right)^{1/2} \left( \sum_{l\neq m} P_l(\pi_{2,lm}\psi_k)^4(X_m)\right)^{1/2}\right] \\
    & \le \frac{1}{3}  \E\left[ \max_{j\in[p]}  \max_{i\in [n]}   \sum_{m\neq i} P_m(\pi_{2,im}\psi_j)^4(X_i)\right].
\end{align*}
In terms of the third term,  we can see that
\begin{align*}
    & \frac{1}{3}  \left\| \max_{(j,k)\in[p]^2} \max_{(i,m)\in I_{n,2}} \sum_{l\neq i,m}  \int (\bar\chi^{(2,2)}_{mli,jk})^2(X_m,x_l,X_i)dP_l(x_l) \right\|_{L^1(\pr)} \\
    & \le \frac{1}{6}\left\| \max_{(j,k)\in[p]^2}  \max_{(i,m)\in I_{n,2}} (\pi_{2,mi}\psi_k)^2(X_m,X_i)  \sum_{l\neq i,m}\int (\pi_{2,ml}\psi_j)^2(X_m,x_l) dP_l(x_l) \right\|_{L^1(\pr)} \\
    & \quad + \frac{1}{6} \left\| \max_{(j,k)\in[p]^2}  \max_{(i,m)\in I_{n,2}} \sum_{l\neq i,m} \int \left( \int \pi_{2,ml}\psi_j(x_m,x_l)\pi_{2,mi}\psi_k(x_m,X_i)dP_m(x_m)\right)^2   dP_l(x_l) \right\|_{L^1(\pr)} \\
    & \le \frac{1}{6} \left\| \max_{j\in[p]} \max_{(i,m)\in I_{n,2}} |\pi_{2,im}\psi_j| \right\|^{2}_{L^4(\mathbb P)} \left\| \max_{j\in[p]} \max_{i\in[n]} \sum_{m\neq i} P_m(\pi_{2,im}\psi_j)^2(X_i) \right\|_{L^2(\mathbb P)} \\
    & \quad + \frac{1}{6} \mathbb{E}\left[  \max_{j\in[p]} \max_{i\in[n]} \sum_{m\neq i} P_m(\pi_{2,im}\psi_j)^4(X_i)\right]^{1/2} \left( \max_{j\in[p]} \sum_{i=1}^n \left\| \sum_{m\neq i } P_m(\pi_{2,im}\psi_j)^2(X_i)\right\|_{L^2(P_i)}^2 \right)^{1/2} .
\end{align*}
Summing up,
\begin{align*}
    & \left\| \max_{(j,k)\in[p]^2} \max_{(i,m)\in I_{n,2}} \sum_{l\neq i,m} \int (\tilde{\chi}^{(2,2)}_{iml,jk})^2(X_i,X_m,x_l)dP_l(x_l)\right\|_{L^1(\pr)} \\
    & \le \frac{1}{3} \left\| \max_{j\in[p]} \max_{(i,m)\in I_{n,2}} |\pi_{2,im}\psi_j| \right\|^{2}_{L^4(\mathbb P)} \left\| \max_{j\in[p]} \max_{i\in[n]} \sum_{m\neq i} P_m(\pi_{2,im}\psi_j)^2(X_i) \right\|_{L^2(\mathbb P)} \\
    & \quad + \frac{1}{3} \mathbb{E}\left[  \max_{j\in[p]} \max_{i\in[n]} \sum_{m\neq i} P_m(\pi_{2,im}\psi_j)^4(X_i)\right]^{1/2} \left( \max_{j\in[p]} \sum_{i=1}^n \left\| \sum_{m\neq i } P_m(\pi_{2,im}\psi_j)^2(X_i)\right\|_{L^2(P_i)}^2 \right)^{1/2}  \\
    & \quad + \frac{1}{3}  \E\left[ \max_{j\in[p]}  \max_{i\in [n]}   \sum_{m\neq i} P_m(\pi_{2,im}\psi_j)^4(X_i)\right].
\end{align*}
\end{proof}

\begin{proof}[Proof of \cref{eq:lem5-12-inid}]
From Eq.(5) in \cite{imai2025gaussian}, we have
\begin{align*}
    & \left\| \max_{(j,k)\in[p]^2} \max_{(i,m,l)\in I_{n,3}} (\tilde{\chi}^{(2,2)}_{iml,jk})^2(X_i,X_m,X_l) \right\|_{L^1(\pr)}  
    \le  \left\| \max_{(j,k)\in[p]^2} \max_{(i,m,l)\in I_{n,3}} (\bar\chi^{(2,2)}_{iml,jk})^2(X_i,X_m,X_l)  \right\|_{L^1(\pr)}.
\end{align*}
From Cauchy-Schwarz inequality, we can see that
\begin{align*}
    & \left\| \max_{(j,k)\in[p]^2} \max_{(i,m,l)\in I_{n,3}} (\chi^{(2,2)}_{iml,jk})^2(X_i,X_m,X_l)  \right\|_{L^1(\pr)} \\
    & \le \frac{1}{2} \left\| \max_{(j,k)\in[p]^2} \max_{(i,m,l)\in I_{n,3}} (\pi_{2,im}\psi_j)^2(X_i,X_m)(\pi_{2,il}\psi_k)^2(X_i,X_l)   \right\|_{L^1(\pr)} \\
    & \quad + \frac{1}{2} \left\| \max_{(j,k)\in[p]^2} \max_{(i,m,l)\in I_{n,3}} \left( \int \pi_{2,im}\psi_j(x_i,X_m)\pi_{2,il}\psi_k(x_i,X_l) dP_i(x_i)\right)^2  \right\|_{L^1(\pr)}  \\ 
    & \le \frac{1}{2} \E\left[ \max_{j\in[p]} \max_{(i,m)\in I_{n,2}} (\pi_{2,im}\psi_j)^4(X_i,X_m)\right]^{1/2} \E\left[ \max_{k\in[p]} \max_{(i,l)\in I_{n,2}} (\pi_{2,il}\psi_k)^4(X_i,X_l)\right]^{1/2} \\
    & \quad + \frac{1}{2}\E\left[ \max_{(j,k)\in[p]^2} \max_{(i,m,l)\in I_{n,3}} P_i(\pi_{2,im}\psi_j)^2(X_m) P_i(\pi_{2,il}\psi_k)^2(X_l) \right] \\
    & = \frac{1}{2} \E\left[ \max_{j\in[p]} \max_{(i,m)\in I_{n,2}} (\pi_{2,im}\psi_j)^4(X_i,X_m)\right]\\
    & \quad + \frac{1}{2}\E\left[ \max_{(j,k)\in[p]^2} \max_{(i,m,l)\in I_{n,3}} P_i(\pi_{2,im}\psi_j)^2(X_m) P_i(\pi_{2,il}\psi_k)^2(X_l) \right].
\end{align*}
In terms of the second term
\begin{align*}
    & \frac{1}{2}\E\left[ \max_{(j,k)\in[p]^2} \max_{(i,m,l)\in I_{n,3}} P_i(\pi_{2,im}\psi_j)^2(X_m) P_i(\pi_{2,il}\psi_k)^2(X_l) \right] \\
    & \le \frac{1}{4}\E\left[ \max_{(j,k)\in[p]^2} \max_{(i,m,l)\in I_{n,3}} \left\{ \left( P_i(\pi_{2,im}\psi_j)^2(X_m) \right)^2 + \left( P_i(\pi_{2,il}\psi_k)^2(X_l) \right)^2 \right\}\right] \\
    & \le \frac{1}{4}\E\left[ \max_{(j,k)\in[p]^2} \max_{(i,m,l)\in I_{n,3}} \left\{  P_i(\pi_{2,im}\psi_j)^4(X_m) +  P_i(\pi_{2,il}\psi_k)^4(X_l) \right\}\right] \\
    & \le \frac{1}{2} \E\left[ \max_{j\in[p]} \max_{(i,m)\in I_{n,2}} P_i(\pi_{2,im}\psi_j)^4(X_m) \right]\\
    & \le\frac{1}{2} \E\left[ \max_{j\in[p]} \max_{(i,m)\in I_{n,2}} (\pi_{2,im}\psi_j)^4(X_i, X_m) \right] ,
\end{align*}
where the first inequality follows from AM-GM inequality, the second inequality follows from Jensen's inequality and the final inequality follows from \cref{lem:max-Jensen_inid}.
Therefore
\begin{align*}
    & \left\| \max_{(j,k)\in[p]^2} \max_{(i,m,l)\in I_{n,3}} (\chi^{(2,2)}_{iml,jk})^2(X_i,X_m,X_l)  \right\|_{L^1(\pr)}  \le \E\left[ \max_{j\in[p]} \max_{(i,m)\in I_{n,2}} (\pi_{2,im}\psi_j)^4(X_i,X_m)\right]  .
\end{align*}
\end{proof}

\begin{proof}[Proof of \cref{eq:lem5-13-inid}]
First recall $\varphi^{(2,2)}_{ml,jk}(X_m,X_l)$ is defined as follows
\begin{align}
    & \varphi^{(2,2)}_{ml,jk}(X_m,X_l) \nonumber\\
    & \coloneqq \sum_{i:m,l\neq i} \Big( P_i\{\pi_{2,im}\psi_j(X_i,X_m)\pi_{2,il}\psi_k(X_i,X_l)\} - P_iP_mP_l\{\pi_{2,im}\psi_j(X_i,X_m)\pi_{2,il}\psi_k(X_i,X_l)\}  \Big). \label{eq:recall_def_varphi22}
\end{align}
From Eq.(5) in \cite{imai2025gaussian}, we have
\begin{align*}
     & \max_{(j,k)\in[p]^2} \sum_{m=1}^n\sum_{l\neq m} \int (\tilde{\varphi}_{ml,jk}^{(2,2)})^2(x_m,x_l) dP_m(x_m)dP_l(x_l) \\
     & \le \max_{(j,k)\in[p]^2}  \sum_{m=1}^n\sum_{l\neq m}  \int ({\varphi}_{ml,jk}^{(2,2)})^2(x_m,x_l) dP_m(x_m)dP_l(x_l) .
\end{align*}
Since we consider the case where $i\neq m\neq l$, it holds that $P_iP_mP_l\{\pi_{2,im}\psi_j(X_i,X_m)\pi_{2,il}\psi_k(X_i,X_l)\} = 0$.
Therefore, we can see that
\begin{align*}
    & \max_{(j,k)\in[p]^2}  \sum_{m=1}^n\sum_{l\neq m}  \int ({\varphi}_{ml,jk}^{(2,2)})^2(x_m,x_l) dP_m(x_m)dP_l(x_l) \\
    & = \max_{(j,k)\in[p]^2} \sum_{m=1}^n\sum_{l\neq m}  \int \left( \sum_{i:m,l\neq i}^n P_i\{\pi_{2,im}\psi_j(X_i,x_m)\pi_{2,il}\psi_k(X_i,x_l)\} \right)^2 dP_m(x_m)dP_l(x_l) \\
    & =  \max_{(j,k)\in[p]^2} \sum_{m=1}^n\sum_{l\neq m} \left\| \sum_{i\neq m,l} \pi_{2,im}\psi_j \star_i^1 \pi_{2,il}\psi_k\right\|_{L^2(P_m\otimes P_l)}^2.
\end{align*}
\end{proof}

\begin{proof}[Proof of \cref{eq:lem5-14-inid}]
From $c_r$ inequality, we have
\begin{align*}
    & \left\| \max_{(j,k)\in[p]^2} \max_{m\in [n]} \sum_{l\neq m} \int (\tilde{\varphi}_{ml,jk}^{(2,2)})^2(X_m,x_l)dP_l(x_l) \right\|_{L^1(\pr)}  \\
    & \le \frac{1}{2} \left\| \max_{(j,k)\in[p]^2}  \max_{m\in [n]} \sum_{l\neq m} \int \left(  \sum_{i:i\neq m,l} P_i\{\pi_{2,im}\psi_j(X_i,X_m)\pi_{2,il}\psi_k(X_i,x_l)\}\right)^2 dP_l(x_l) \right\|_{L^1(\pr)} \\
    & \quad + \frac{1}{2}\left\| \max_{(j,k)\in[p]^2}  \max_{m\in [n]} \sum_{l\neq m} \int \left( \sum_{i:i\neq m,l}   P_i\{\pi_{2,il}\psi_j(X_i,x_l)\pi_{2,im}\psi_k(X_i,X_m)\}\right)^2 dP_l(x_l) \right\|_{L^1(\pr)}.
\end{align*}
Likewise the proof of Lemma 17(c) in \cite{imai2025gaussian}, from Fubini's theorem, we have
\begin{align*}
    & \int \left(  \sum_{i:i\neq m,l} P_i\{\pi_{2,im}\psi_j(X_i,X_m)\pi_{2,il}\psi_k(X_i,x_l)\}\right)^2 dP_l(x_l) \\
    & = \int \left(\sum_{i:i\neq m,l}\int \pi_{2,im}\psi_j(x_i,X_m)\pi_{2,il}\psi_k(x_i,x_l) dP_i(x_i) \right)^2 dP_l(x_l) \\
    & = \int \sum_{i_1: i_1\neq m,l} \sum_{i_2:i_2\neq m,l} \int\int \pi_{2,i_1m}\psi_j(x_{i_1},X_m)\pi_{2,i_1l}\psi_k(x_{i_1},x_l) \\
    & \qquad\qquad\qquad\qquad\qquad\times\pi_{2,i_2m}\psi_j(x_{i_2},X_m)\pi_{2,i_2l}\psi_k(x_{i_2},x_l) dP_{i_1}(x_{i_1})dP_{i_2}(x_{i_2})dP_l(x_l) \\
    & = \sum_{i_1: i_1\neq m,l} \sum_{i_2:i_2\neq m,l}  \int\int (\pi_{2,i_1l}\psi_k \star^1_{l}\pi_{2,i_2l}\psi_k)(x_{i_1},x_{i_2})\pi_{2,i_1m}\psi_j(x_{i_1},X_m)\pi_{2,i_2m}\psi_j(x_{i_2},X_m)dP_{i_1}(x_{i_1})dP_{i_2}(x_{i_2}).
\end{align*}
This implies
\begin{align*}
    & \sum_{l\neq m,i_1,i_2}\int \left(  \sum_{i:i\neq m,l} P_i\{\pi_{2,im}\psi_j(X_i,X_m)\pi_{2,il}\psi_k(X_i,x_l)\}\right)^2 dP_l(x_l) \\
    & =  \sum_{i_1: i_1\neq m,l} \sum_{i_2:i_2\neq m,l} \int\int \pi_{2,i_1m}\psi_j(x_{i_1},X_m)\pi_{2,i_1m}\psi_j(x_{i_1},x_m) \\
    & \quad\quad\quad \times \left[ \sum_{l\neq m,i_1,i_2} (\pi_{2,i_1l}\psi_k \star^1_{l}\pi_{2,i_2l}\psi_k)(x_{i_1},x_{i_2})\right] dP_{i_1}(x_{i_1})dP_{i_2}(x_{i_2}).
\end{align*}
Then, Schwarz inequality gives
\begin{align*}
    &  \sum_{l\neq m,i_1,i_2}\int \left(  \sum_{i:i\neq m,l} P_i\{\pi_{2,im}\psi_j(X_i,X_m)\pi_{2,il}\psi_k(X_i,x_l)\}\right)^2 dP_l(x_l) \\
    & \le 
    \left( \sum_{i_1: i_1\neq m,l} \sum_{i_2:i_2\neq m,l} \left\|  \sum_{l\neq m,i_1,i_2}\pi_{2,i_1l}\psi_k \star^1_{l}\pi_{2,i_2l}\psi_k \right\|_{L^2(P_{i_1}\otimes P_{i_2})}^2  \right)^{1/2}\\
    & \quad \times 
    \left( \sum_{i_1: i_1\neq m,l} \sum_{i_2:i_2\neq m,l}  \left\{ P_{i_1} (\pi_{2,i_1m}\psi_j)^2(X_m)\right\}  \left\{ P_{i_2} (\pi_{2,i_2m}\psi_j)^2(X_m)\right\}\right)^{1/2} \\
    & = \left( \sum_{i_1: i_1\neq m,l} \sum_{i_2:i_2\neq m,l} \left\|  \sum_{l\neq m,i_1,i_2}\pi_{2,i_1l}\psi_k \star^1_{l}\pi_{2,i_2l}\psi_k \right\|_{L^2(P_{i_1}\otimes P_{i_2})}^2  \right)^{1/2}
   \sum_{i: i\neq m}   P_{i} (\pi_{2,im}\psi_j)^2(X_m)  
\end{align*}
Also, since 
\begin{align*}
    \sum_{l\neq m,i_1,i_2}\pi_{2,i_1l}\psi_k \star^1_{l}\pi_{2,i_2l} = \sum_{l\neq i_1,i_2}\pi_{2,i_1l}\psi_k \star^1_{l}\pi_{2,i_2l} - \pi_{2,i_1m}\psi_k \star^1_{m}\pi_{2,i_2m}\psi_k,
\end{align*}
we can see that
\begin{align*}
    & \left( \sum_{i_1: i_1\neq m,l} \sum_{i_2:i_2\neq m,l} \left\|  \sum_{l\neq m,i_1,i_2}\pi_{2,i_1l}\psi_k \star^1_{l}\pi_{2,i_2l}\psi_k \right\|_{L^2(P_{i_1}\otimes P_{i_2})}^2  \right)^{1/2}
   \sum_{i: i\neq m}   P_{i} (\pi_{2,im}\psi_j)^2(X_m)  \\
   & \le 2 \left( \underset{i_1\neq i_2}{\sum_{i_1: i_1\neq m} \sum_{i_2:i_2\neq m}}  \left\|  \sum_{l\neq i_1,i_2}\pi_{2,i_1l}\psi_k \star^1_{l}\pi_{2,i_2l}\psi_k \right\|_{L^2(P_{i_1}\otimes P_{i_2})}^2  \right)^{1/2}
   \sum_{i: i\neq m}   P_{i} (\pi_{2,im}\psi_j)^2(X_m) \\
   & \quad +2 \left( \underset{i_1\neq i_2}{\sum_{i_1: i_1\neq m} \sum_{i_2:i_2\neq m}}  \left\| \pi_{2,i_1m}\psi_k \star^1_{m}\pi_{2,i_2m}\psi_k \right\|_{L^2(P_{i_1}\otimes P_{i_2})}^2  \right)^{1/2}
   \sum_{i: i\neq m}   P_{i} (\pi_{2,im}\psi_j)^2(X_m) \\
   & \quad + \left( \sum_{i: i\neq l}  \left\|  \sum_{l\neq m,i}\pi_{2,il}\psi_k \star^1_{l}\pi_{2,il}\psi_k \right\|_{L^2(P_{i})}^2  \right)^{1/2}
   \sum_{i: i\neq m}   P_{i} (\pi_{2,im}\psi_j)^2(X_m) .
\end{align*}
Then, Schwarz inequality gives
\begin{align*}
    & \left( \sum_{i_1: i_1\neq m,l} \sum_{i_2:i_2\neq m,l} \left\|  \sum_{l\neq m,i_1,i_2}\pi_{2,i_1l}\psi_k \star^1_{l}\pi_{2,i_2l}\psi_k \right\|_{L^2(P_{i_1}\otimes P_{i_2})}^2  \right)^{1/2}
   \sum_{i: i\neq m,l}   P_{i} (\pi_{2,im}\psi_j)^2(X_m)  \\ 
   & \le 2 \left( \underset{i_1\neq i_2}{\sum_{i_1: i_1\neq m} \sum_{i_2:i_2\neq m}}  \left\|  \sum_{l\neq i_1,i_2}\pi_{2,i_1l}\psi_k \star^1_{l}\pi_{2,i_2l}\psi_k \right\|_{L^2(P_{i_1}\otimes P_{i_2})}^2  \right)^{1/2}
   \sum_{i: i\neq m,l}   P_{i} (\pi_{2,im}\psi_j)^2(X_m) \\
   & \quad +2   \left\|  \sum_{m\neq i} P_i(\pi_{2,im}\psi_k)^2 \right\|_{L^2(P_{m})}
   \sum_{i: i\neq m,l}   P_{i} (\pi_{2,im}\psi_j)^2(X_m) \\
   & \quad + \left(\sum_{i: i\neq l}  \left\|  \sum_{l\neq i} P_l(\pi_{2,il}\psi_k)^2 \right\|_{L^2(P_{i})}^2 \right)^{1/2}
   \sum_{i: i\neq m,l}   P_{i} (\pi_{2,im}\psi_j)^2(X_m) .
\end{align*}

Therefore
\begin{align*}
    & \max_{(j,k)\in[p]^2}  \max_{m\in [n]} \sum_{l\neq m} \int \left(  \sum_{i:i\neq m,l} P_i\{\pi_{2,im}\psi_j(X_i,X_m)\pi_{2,il}\psi_k(X_i,x_l)\}\right)^2 dP_l(x_l) \\
    & \le \left( \max_{(j,k)\in[p]^2} \sum_{m=1}^n\sum_{l\neq m} \left\| \sum_{i\neq m,l} \pi_{2,im}\psi_j \star_i^1 \pi_{2,il}\psi_k\right\|_{L^2(P_m\otimes P_l)}^2\right)^{1/2} \left\| \max_{j\in[p]} \max_{i\in[n]} \sum_{m\neq i} P_m(\pi_{2,im}\psi_j)^2(X_i) \right\|_{L^2(\mathbb P)} \\
    & \quad + \frac{3}{2} \left(\max_{j\in[p]} \sum_{i=1}^n  \left\|  \sum_{m\neq i} P_m(\pi_{2,im}\psi_j)^2 \right\|_{L^2(P_{i})}^2 \right)^{1/2}\left\| \max_{j\in[p]} \max_{i\in[n]} \sum_{m\neq i} P_m(\pi_{2,im}\psi_j)^2(X_i) \right\|_{L^2(\mathbb P)}.
\end{align*}
\end{proof}

\begin{proof}[Proof of  \cref{eq:lem5-15-inid}]
From $c_r$ inequality, we have
\begin{align*}
    & \left\| \max_{(j,k)\in[p]^2} \max_{(m,l)\in I_{n,2}} (\tilde{\varphi}_{ml,jk}^{(2,2)})^2(X_m,X_l) \right\|_{L^1(\pr)} \\
    & \le \frac{1}{2} \left\| \max_{(j,k)\in[p]^2} \max_{(m,l)\in I_{n,2}} \left(  \sum_{i:i\neq m,l}  P_i\{\pi_{2,im}\psi_j(X_i,X_m)\pi_{2,il}\psi_k(X_i,X_l)\}\right)^2  \right\|_{L^1(\pr)} \\
    & \quad +  \frac{1}{2} \left\| \max_{(j,k)\in[p]^2} \max_{(m,l)\in I_{n,2}} \left(  \sum_{i:i\neq m,l} P_i\{\pi_{2,il}\psi_j(X_i,X_l)\pi_{2,im}\psi_k(X_i,X_m)\}\right)^2  \right\|_{L^1(\pr)} \\
    & = \left\| \max_{(j,k)\in[p]^2} \max_{(m,l)\in I_{n,2}} \left(  \sum_{i:i\neq m,l}  P_i\{\pi_{2,im}\psi_j(X_i,X_m)\pi_{2,il}\psi_k(X_i,X_l)\}\right)^2  \right\|_{L^1(\pr)}\\
    & \le  \left\| \max_{(j,k)\in[p]^2} \max_{(m,l)\in I_{n,2}}   \left(  \sum_{i:i\neq m,l}  P_i(\pi_{2,im}\psi_j)^2(X_m)\right)  \left(  \sum_{i:i\neq m,l}  P_i(\pi_{2,il}\psi_k)^2(X_l)\right)  \right\|_{L^1(\pr)} \\
    & \le  \left\| \max_{j\in[p]} \max_{i\in[n]} \sum_{m\neq i} P_m(\pi_{2,im}\psi_j)^2(X_i) \right\|_{L^2(\mathbb P)}^2,
\end{align*}
where the second follows from Schwarz inequality.
\end{proof}

\begin{proof}[Proof of \cref{eq:lem5-16-inid}]
Recall that $\varphi_{m,jk}^{(2,2),\mathtt{diag}}(X_m)$ is defined as $\varphi_{m,jk}^{(2,2),\mathtt{diag}}(X_m) \coloneqq  \varphi^{(2,2)}_{mm,jk}(X_m,X_m)$.
From \cref{eq:recall_def_varphi22} and $c_r$ inequality, we have
\begin{align*}
    & \max_{(j,k)\in[p]^2} \sum_{m=1}^n \int (\varphi^{(2,2), \mathtt{diag}}_{m,jk})^2(x_m)dP_m(x_m) \\ 
    & \le 2 \max_{(j,k)\in[p]^2} \sum_{m=1}^n  \int \left( \sum_{i\neq m}P_i\{\pi_{2,im}\psi_j(X_i,x_m)\pi_{2,im}\psi_k(X_i,x_m)\}\right)^2 dP_m(x_m) \\
    & \quad + 2 \max_{(j,k)\in[p]^2}  \sum_{m=1}^n  \left(\sum_{i\neq m} P_iP_m\{\pi_{2,im}\psi_j(X_i,X_m)\pi_{2,im}\psi_k(X_i,X_m)\} \right)^2 \\
    & \le 4 \max_{(j,k)\in[p]^2}  \sum_{m=1}^n  \int \left( \sum_{i\neq m} P_i\{\pi_{2,im}\psi_j(X_i,x_m)\pi_{2,im}\psi_k(X_i,x_m)\}\right)^2 dP_m(x_m) \\
    & \le 4 \max_{(j,k)\in[p]^2} \sum_{m=1}^n \int \left(\sum_{i\neq m}P_i(\pi_{2,im}\psi_j)^2(x_m) \right)\left( \sum_{i\neq m}P_i(\pi_{2,im}\psi_k)^2(x_m)\right) dP_m(x_m)  \\
    & \le  4 \left( \max_{j\in[p]} \sum_{m=1}^n \left\| \sum_{i\neq m} P_i(\pi_{2,im}\psi_j)^2(X_m)\right\|_{L^2(P_m)}^2 \right)^{1/2}  \left( \max_{k\in[p]} \sum_{m=1}^n \left\| \sum_{i\neq m} P_i(\pi_{2,im}\psi_k)^2(X_m)\right\|_{L^2(P_m)}^2 \right)^{1/2}   \\
    & \le 4  \max_{j\in[p]} \sum_{i=1}^n \left\| \sum_{m\neq i} P_m(\pi_{2,im}\psi_j)^2(X_i)\right\|_{L^2(P_i)}^2 ,
\end{align*}
where the second inequality follows from Jensen's inequality, the third and the final inequality follows from Schwarz inequality.
\end{proof}

\begin{proof}[Proof of \cref{eq:lem5-17-inid}]
From \cref{eq:recall_def_varphi22} and $c_r$ inequality, we have
\begin{align*}
    & \left\| \max_{(j,k)\in[p]^2} \max_{m\in[n]}  (\varphi^{(2,2), \mathtt{diag}}_{m,jk})^2(X_m)\right\|_{L^1(\pr)}  \\
    & \le 2  \left\| \max_{(j,k)\in[p]^2} \max_{m\in[n]} \left(\sum_{i\neq m} P_i\{\pi_{2,im}\psi_j(X_i,X_m)\pi_{2,im}\psi_k(X_i,X_m)\} \right)^2\right\|_{L^1(\pr)} \\
    & \quad + 2  \max_{(j,k)\in[p]^2} \max_{m\in[n]} \left(\sum_{i\neq m} P_iP_m\{\pi_{2,im}\psi_j(X_i,X_m)\pi_{2,im}\psi_k(X_i,X_m)\} \right)^2,
\end{align*}
where the second inequality follows from Jensen's inequality and Schwarz inequality and the final inequality follows from Schwarz inequality.
In terms of the first term, Schwarz inequality gives
\begin{align*}
    &  2  \left\| \max_{(j,k)\in[p]^2} \max_{m\in[n]} \left(\sum_{i\neq m} P_i\{\pi_{2,im}\psi_j(X_i,X_m)\pi_{2,im}\psi_k(X_i,X_m)\} \right)^2\right\|_{L^1(\pr)} \\
    & \le 2\left\| \max_{(j,k)\in[p]^2} \max_{m\in[n]} \left(\sum_{i\neq m} P_i(\pi_{2,im}\psi_j)^2(X_m) \right)\left(\sum_{i\neq m} P_i(\pi_{2,im}\psi_k)^2(X_m) \right)\right\|_{L^1(\pr)} \\
    & \le  2\left\| \max_{j\in[p]} \max_{m\in[n]} \left(\sum_{i\neq m} P_i(\pi_{2,im}\psi_j)^2(X_m) \right)^2\right\|_{L^1(\pr)} \\
    & =  2\left\| \max_{j\in[p]} \max_{m\in[n]} \sum_{i\neq m} P_i(\pi_{2,im}\psi_j)^2(X_m)\right\|_{L^2(\pr)}^2.
\end{align*}
In terms of the second term, Jensen's inequality gives
\begin{align*}
    & 2  \max_{(j,k)\in[p]^2} \max_{m\in[n]} \left(\sum_{i\neq m} P_iP_m\{\pi_{2,im}\psi_j(X_i,X_m)\pi_{2,im}\psi_k(X_i,X_m)\} \right)^2 \\
    & =  2  \max_{(j,k)\in[p]^2} \max_{m\in[n]} \left(P_m\sum_{i\neq m} P_i\{\pi_{2,im}\psi_j(X_i,X_m)\pi_{2,im}\psi_k(X_i,X_m)\} \right)^2 \\
    & \le 2  \max_{(j,k)\in[p]^2} \max_{m\in[n]} P_m \left(\sum_{i\neq m} P_i\{\pi_{2,im}\psi_j(X_i,X_m)\pi_{2,im}\psi_k(X_i,X_m)\} \right)^2 \\
    & \le 2\max_{j\in[p]} \max_{m\in[n]} P_m \left(\sum_{i\neq m} P_i(\pi_{2,im}\psi_j)^2(X_m)\right)^2 \\
    & \le 2\left\| \max_{j\in[p]} \max_{m\in[n]} \sum_{i\neq m} P_i(\pi_{2,im}\psi_j)^2(X_m)\right\|_{L^2(\pr)}^2.
\end{align*}
\end{proof}

\clearpage
\setcounter{page}{1}
\renewcommand{\thepage}{SII-\arabic{page}}

\setcounter{section}{0}
\renewcommand{\thesection}{SII-\arabic{section}}
\vspace{1cm}
\begin{center}
{\LARGE
{\bf Supplementary Materials II : Proof for Applications}
}
\end{center}

\etocdepthtag.toc{suppII}
\supplementtoc{suppII}

\section{Proofs for \cref{subsubsec:application-chao}} \label{subsec:proof-chao}

\subsection{Proof of \cref{lemma:chao-Gamma-convergence}}
\begin{proof}[Proof of \cref{lemma:chao-Gamma-convergence}]
First observe that a simple algebra gives
\begin{align*}
    \hat{\Gamma}_{n,j} & = \mu_{n,j}^{-2}\sum_{i=1}^{n}\sum_{m\neq i}^n X_{ij}\Pi_{im,j}X_{mj}^\top \\
    & = \mu_{n,j}^{-2}\sum_{i=1}^{n}\sum_{m\neq i}^n  \mathbb{E}[X_{ij}\mid Z] \Pi_{im,j}\mathbb{E}[X_{mj} \mid Z]^\top  \\
    & \quad +  \mu_{n,j}^{-2}\sum_{i=1}^{n}\sum_{m\neq i}^n \{X_{ij}-\mathbb{E}[X_{ij}\mid Z]\} \Pi_{im,j}\{X_{mj} - \mathbb{E}[X_{mj} \mid Z]\}^\top \\
    & \quad + \mu_{n,j}^{-2}\sum_{i=1}^{n}\sum_{m\neq i}^n\mathbb{E}[X_{ij}\mid Z] \Pi_{im,j}\{X_{mj} - \mathbb{E}[X_{mj} \mid Z]\}^\top \\
    & \quad + \mu_{n,j}^{-2}\sum_{i=1}^{n}\sum_{m\neq i}^n \{X_{ij}-\mathbb{E}[X_{ij}\mid Z]\} \Pi_{im,j}\mathbb{E}[X_{mj} \mid Z]^\top \\
    & \eqqcolon (\Gamma.I)_j + (\Gamma.II)_j + (\Gamma.III)_j + (\Gamma.IV)_j.
\end{align*}

In terms of $(\Gamma.I)_j$, since $\mathbb{E}[X_{ij} \mid Z] = \gamma_j^\top Z_j = n^{-1/2}\mu_{n,j}(\pi_j^\top Z_{ij})$, it can be seen that
\begin{align*}
    (\Gamma.I)_j &= \frac{1}{n} \sum_{i=1}^{n}\sum_{m\neq i}^n \Pi_{im,j} (\pi_{j}^\top Z_{ij})  (\pi_{j}^\top Z_{mj})^\top  \\
    & = \frac{1}{n} \sum_{i=1}^{n}\sum_{m=1}^n \Pi_{im,j} (\pi_{j}^\top Z_{ij})  (\pi_{j}^\top Z_{mj})^\top - \frac{1}{n} \sum_{i=1}^n \Pi_{ii,j} (\pi_{j}^\top Z_{ij})  (\pi_{j}^\top Z_{ij})^\top \\
    & =  \frac{1}{n} \sum_{i=1}^{n}\sum_{m=1}^n \Pi_{im,j} (\pi_{j}^\top Z_{ij})  (\pi_{j}^\top Z_{mj})^\top - \frac{1}{n} \sum_{i=1}^n \Pi_{ii,j} (\pi_{j}^\top Z_{ij})  (\pi_{j}^\top Z_{ij})^\top \\
    & \quad + \frac{1}{n} \sum_{i=1}^n \Pi_{ii,j} (\pi_{j}^\top Z_{ij})  (\pi_{j}^\top Z_{ij})^\top - \frac{1}{n} \sum_{i=1}^n \Pi_{ii,j} (\pi_{j}^\top Z_{ij})  (\pi_{j}^\top Z_{ij})^\top \\
    & = \frac{1}{n} \sum_{i=1}^n (1-\Pi_{ii,j})(\pi_{j}^\top Z_{ij})  (\pi_{j}^\top Z_{ij})^\top  + \frac{1}{n} \sum_{i=1}^n (\pi_{j}^\top Z_{ij}) \left\{ \sum_{m=1}^n \Pi_{im,j} (\pi^\top_j Z_{mj})^\top  -  (\pi^\top_j Z_{ij})^\top \right\} \\
    & = \frac{1}{n} \sum_{i=1}^n (1-\Pi_{ii,j})(\pi_{j}^\top Z_{ij})  (\pi_{j}^\top Z_{ij})^\top ,
\end{align*}
where the last inequality follows from
\begin{align*}
    & \sum_{m=1}^n \Pi_{im,j} (\pi^\top_j Z_{mj})^\top  -  (\pi^\top_j Z_{ij})^\top  \\
    & = \left( \sum_{m=1}^n \Pi_{im,j} Z_{mj}^\top - Z_{ij}^\top\right) \pi_j = \left( \sum_{m=1}^n Z_{ij}^\top\left(\sum_{l=1}^n Z_{lj}Z_{lj}^\top \right)^{-1} Z_{mj} Z_{mj}^\top - Z_{ij}^\top\right) \pi_j = 0.
\end{align*}
Therefore, we have
\begin{align*}
    (\Gamma.I)_j  = \frac{1}{n} \sum_{i=1}^n (1-\Pi_{ii,j})(\pi_{j}^\top Z_{ij})  (\pi_{j}^\top Z_{ij})^\top  = \Gamma_{n,j}.
\end{align*}

In terms of the $(\Gamma.II)_j$, since $X_{ij} = \gamma_{n,j}^\top Z_{ij} + \varepsilon_{ij}$ and $\mathbb{E}[X_{ij} \mid Z] = \gamma_{n,j}^\top Z_{ij}$, 
\begin{align*}
     (\Gamma.II)_j 
     & = \mu_{n,j}^{-2}\sum_{i=1}^{n}\sum_{m\neq i}^n \{X_{ij}-\mathbb{E}[X_{ij}\mid Z]\} \Pi_{im,j}\{X_{mj} - \mathbb{E}[X_{mj} \mid Z]\}^\top \\
     & = \mu_{n,j}^{-2}\sum_{i=1}^{n}\sum_{m\neq i}^n \Pi_{im,j} \varepsilon_{ij}\varepsilon_{mj}^\top. 
\end{align*}
We denote the $(k_1,k_2)$-th element in $(\Gamma.II)_j$ as $(\Gamma.II)_{j,k_1k_2}$. Then, since $(\Gamma.II)_j$ is a second-order degenerate $U$-statistics, \cref{thm:max-is-inid} gives
\begin{align}
    & \mathbb{E}\left[ \max_{j\in[p]}\max_{(k_1,k_2)\in[d_j]^2} |(\Gamma.II)_{j,k_1k_2}| \mid Z\right] \nonumber\\
    &\lesssim \log(npd) \left( \max_{j\in[p]} \max_{(k_1,k_2)\in[d_j]^2} \sum_{i=1}^n\sum_{m\neq i}^n \mathbb{E}\left[ \left| \mu_{n,j}^{-2}\Pi_{im,j} \varepsilon_{ij,k_1}\varepsilon_{mj,k_2} \right|^2 \mid Z\right] \right)^{1/2}\nonumber \\
    & \quad + \log^{3/2}(npd) \left( \max_{j\in[p]} \max_{(k_1,k_2)\in[d_j]^2} \max_{i\in[n]} \sum_{m\neq i}^n \mathbb{E}\left[\left| \mu_{n,j}^{-2}\Pi_{im,j} \varepsilon_{ij,k_1}\varepsilon_{mj,k_2} \right|^2\mid X_i, \varepsilon_i, Z\right] \right)^{1/2} \nonumber\\
    & \quad + \log^{2}(npd) \left( \max_{j\in[p]} \max_{(k_1,k_2)\in[d_j]^2} \max_{(i,m)\in I_{n,2}} \left| \mu_{n,j}^{-2}\Pi_{im,j} \varepsilon_{ij,k_1}\varepsilon_{mj,k_2} \right|^2 \right)^{1/2}. \label{eq:max-chao-gammaII}
\end{align}
Observe that
\begin{align*}
    &  \max_{j\in[p]} \max_{(k_1,k_2)\in[d_j]^2} \sum_{i=1}^n\sum_{m\neq i}^n\mathbb{E}\left[ \left| \mu_{n,j}^{-2}\Pi_{im,j} \varepsilon_{ij,k_1}\varepsilon_{mj,k_2} \right|^2 \mid Z\right] \\
    & \le \max_{j\in[p]} \left( \max_{k\in[d_j]}\max_{i\in[n]}  \mathbb{E}\left[ \varepsilon_{ij,k}^2  \mid Z\right]\right) \mu_{n,j}^{-4} \sum_{i=1}^n\sum_{m\neq i}^n \Pi_{im,j}^2 \\
    & \le \max_{j\in[p]} \left( \max_{k\in[d_j]}\max_{i\in[n]}  \mathbb{E}\left[ \varepsilon_{ij,k}^2  \mid Z\right]\right) \mu_{n,j}^{-4}K_{n,j},
\end{align*}
where the final inequality follows from $\sum_{i=1}^n\sum_{m=1}^n \Pi_{im,j}^2 = \text{trace}(\Pi_j^2) = \text{trace}(\Pi_j) = K_n$.
Also, it holds that
\begin{align*}
     & \max_{j\in[p]} \max_{(k_1,k_2)\in[d_j]^2} \max_{i\in[n]} \sum_{m\neq i}^n \mathbb{E}\left[\left| \mu_{n,j}^{-2}\Pi_{im,j} \varepsilon_{ij,k_1}\varepsilon_{mj,k_2} \right|^2\mid X_i, \varepsilon_i, Z\right] \\
     & = \max_{j\in[p]} \max_{(k_1,k_2)\in[d_j]^2} \max_{i\in[n]} \mu_{n,j}^{-4} \sum_{m\neq i}^n   \Pi_{im,j}^2 \varepsilon_{ij,k_1}^2 \mathbb{E}\left[  \varepsilon_{mj,k_2}^2\mid Z \right] \\
     & \le  \max_{j\in[p]} \left( \max_{(k_1,k_2)\in[d_j]^2} \max_{(i,m)\in I_{n,2}} \varepsilon_{ij,k_1}^2 \mathbb{E}\left[  \varepsilon_{mj,k_2}^2 \mid Z \right]\right)\mu_{n,j}^{-4}  \max_{i\in[n]} \sum_{m\neq i}^n   \Pi_{im,j}^2 \\
     & \le \max_{j\in[p]} \left( \max_{(k_1,k_2)\in[d_j]^2} \max_{(i,m)\in I_{n,2}} \varepsilon_{ij,k_1}^2 \mathbb{E}\left[  \varepsilon_{mj,k_2}^2 \mid Z \right]\right)\mu_{n,j}^{-4} \max_{i\in[n]}  \Pi_{ii,j}
\end{align*}
where the final inequality follows from $\sum_{m=1}^n \Pi_{im,j}^2 = \Pi_{ii,j}^2 = \Pi_{ii,j}$.
Finally, it holds that
\begin{align*}
    & \max_{j\in[p]} \max_{(k_1,k_2)\in[d_j]^2} \max_{(i,m)\in I_{n,2}} \left| \mu_{n,j}^{-2}\Pi_{im,j} \varepsilon_{ij,k_1}\varepsilon_{mj,k_2} \right|^2  \\
    & \le \max_{j\in[p]} \left( \max_{k\in[d_j]} \max_{i\in[n]} \varepsilon_{ij,k}^2 \right)^2 \mu_{n,j}^{-4} \max_{(i,m)\in I_{n,2}} \Pi_{im,j}^2. 
\end{align*}
Inserting these results into \cref{eq:max-chao-gammaII}, in conjunction with \cref{as:many-weak-iv}, we have 
\begin{align*}
    & \mathbb{E}\left[ \max_{j\in[p]}\max_{(k_1,k_2)\in[d_j]^2} |(\Gamma.II)_{j,k_1k_2}| \mid Z\right] \\
    & \lesssim  \max_{j\in[p]}\frac{\sqrt{K_{n,j}}}{\mu_{n,j}^2} \log(npd) +  \max_{j\in[p]} \frac{1}{\mu_{n,j}^2} \log^{2}(ndp) \lesssim \max_{j\in[p]}\frac{\sqrt{K_{n,j}}}{\mu_{n,j}^2} \log^{2}(npd).
\end{align*}

In terms of $(\Gamma.III)_j$, $X_{ij} = \gamma_{n,j}^\top Z_{ij} + \varepsilon_{ij}$ and $\mathbb{E}[X_{ij} \mid Z_{ij}] = \gamma_{n,j}^\top Z_{ij}$
\begin{align*}
    (\Gamma.III)_j 
    &= \mu_{n,j}^{-2}\sum_{i=1}^{n}\sum_{m\neq i}^n\mathbb{E}[X_{ij}\mid Z] \Pi_{im,j}\{X_{mj} - \mathbb{E}[X_{mj} \mid Z]\}^\top \\
    &= \mu_{n,j}^{-2}\sum_{i=1}^{n}\sum_{m\neq i}^n  \Pi_{im,j}(\gamma_{n,j}^\top Z_{ij}) \varepsilon_{mj}^\top \\
    & = \frac{1}{\sqrt{n}\mu_{n,j}}\sum_{i=1}^{n}\sum_{m\neq i}^n\Pi_{im,j}(\pi_j^\top Z_{ij}) \varepsilon_{mj}^\top \\
    & =  \frac{1}{\sqrt{n}\mu_{n,j}}\sum_{m=1}^{n} (1-\Pi_{mm,j}) (\pi_j^\top Z_{mj}) \varepsilon_{mj}^\top,
\end{align*}
where the final equality follows from the identity of the projection matrix.
We denote the $(k_1,k_2)$-th element in $(\Gamma.II)_j$ as $(\Gamma.II)_{j,k_1k_2}$. Then, since $(\Gamma.III)_j$ is a first-order degenerate $U$-statistics, \cref{thm:max-is-inid} gives
\begin{align}
    & \mathbb{E}\left[ \max_{j\in[p]}\max_{(k_1,k_2)\in[d_j]^2} |(\Gamma.III)_{j,k_1k_2}| \mid Z\right] \nonumber\\
    &\lesssim \log^{1/2}(npd) \left( \max_{j\in[p]} \max_{(k_1,k_2)\in[d_j]^2} \sum_{m=1}^n \mathbb{E}\left[ \left| n^{-1/2}\mu_{n,j}^{-1}(1-\Pi_{mm,j}) (\pi_j^\top Z_{mj})_{k_1}\varepsilon_{mj,k_2} \right|^2 \mid Z\right] \right)^{1/2}\nonumber \\
    & \quad + \log(npd) \left( \max_{j\in[p]} \max_{(k_1,k_2)\in[d_j]^2} \max_{m\in[n]}  \left| n^{-1/2}\mu_{n,j}^{-1}(1-\Pi_{mm,j}) (\pi_j^\top Z_{mj})_{k_1}\varepsilon_{mj,k_2} \right|^2 \right)^{1/2}. \label{eq:max-chao-gammaIII}
\end{align}
Similarly to the evaluation of $(\Gamma.II)_j$, we can see that
\begin{align*}
    & \max_{j\in[p]} \max_{(k_1,k_2)\in[d_j]^2} \sum_{m=1}^n \mathbb{E}\left[ \left| n^{-1/2}\mu_{n,j}^{-1}(1-\Pi_{mm,j}) (\pi_j^\top Z_{mj})_{k_1}\varepsilon_{mj,k_2}\right|^2 \mid Z\right] \\
    & \quad \le \max_{j\in[p]} \left( \max_{k\in[d_j]} \max_{i\in[n]} (\pi_{j}^\top Z_{ij})_k^2 \right) \left( \max_{k\in[d_j]} \max_{m\in[n]} \mathbb{E}[\varepsilon_{mj,k}^2 \mid Z]\right) \mu_{n,j}^{-2} \frac{1}{n}\sum_{m=1}^n (1-\Pi_{mm,j})^2, \\
    & \max_{j\in[p]} \max_{(k_1,k_2)\in[d_j]^2} \max_{(i,m)\in I_{n,2}} \left| n^{-1/2}\mu_{n,j}^{-1}(1-\Pi_{mm,j}) (\pi_j^\top Z_{mj})_{k_1}\varepsilon_{mj,k_2}\right|^2 \\
    & \quad \le \max_{j\in[p]}\left( \max_{k\in[d_j]} \max_{i\in[n]} (\pi_{j}^\top Z_{ij})_k^2 \right) \left( \max_{k\in[d_j]} \max_{m\in[n]} \varepsilon_{mj,k}^2 \right) n^{-1}\mu_{n,j}^{-2} \max_{m\in [n]}  (1-\Pi_{mm,j})^2.
\end{align*}
Inserting these results into \cref{eq:max-chao-gammaIII}, in conjunction with \cref{as:many-weak-iv}, we have 
\begin{align*}
    & \mathbb{E}\left[ \max_{j\in[p]}\max_{(k_1,k_2)\in[d_j]^2} |(\Gamma.III)_{j,k_1k_2}| \mid Z\right] \\
    & \lesssim  \max_{j\in[p]}\frac{1}{\mu_{n,j}} \log^{1/2}(npd) +  \max_{j\in[p]} \frac{1}{\sqrt{n}\mu_{n,j}} \log(ndp) \lesssim  \max_{j\in[p]}\frac{1}{\mu_{n,j}} \log(npd).
\end{align*}
In the same way as the evaluation on $(\Gamma.III)_j$, we can see that 
\begin{align*}
    & \mathbb{E}\left[ \max_{j\in[p]}\max_{(k_1,k_2)\in[d_j]^2} |(\Gamma.IV)_{j,k_1k_2}| \mid Z\right] \lesssim  \max_{j\in[p]}\frac{1}{\mu_{n,j}} \log(npd).
\end{align*}

Summing up
\begin{align*}
    & \mathbb{E}\left[ \max_{j\in[p]} \max_{(k_1,k_2)\in[d_j]^2} \left| e_{k_1}^\top (\hat{\Gamma}_{n,j} - \Gamma_{n,j}) e_{k_2} \right|\right] \\
    & = O_p\left(  \max_{j\in[p]}\frac{\sqrt{K_{n,j}}}{\mu_{n,j}^2} \log^{3/2}(npd) +  \max_{j\in[p]}\frac{1}{\mu_{n,j}} \log(npd)\right).
\end{align*}
\end{proof}

\subsection{Proof of \cref{lemma:chao-hoeffding}}

\begin{proof}[Proof of \cref{lemma:chao-hoeffding}]
First, observe that
    \begin{align*}
     \hat{S}_{n,j} 
     &= \frac{1}{\mu_{n,j}^2} \sum_{i=1}^{n}\sum_{m\neq i}^n X_{ij}\Pi_{im,j}u_{mj} \\
     & = \frac{1}{\mu_{n,j}^2} \sum_{i=1}^{n}\sum_{m\neq i}^n \frac{1}{2} \{X_{ij}\Pi_{im,j}u_{mj} +  X_{mj}\Pi_{mi,j}u_{ij} \} 
     = \frac{1}{\mu_n^2} \frac{1}{2} \sum_{i=1}^{n}\sum_{m\neq i}^n  \Pi_{im,j}\{X_{ij}u_{mj} +  X_{mj}u_{ij} \},
\end{align*}
where the second equality symmetrize the summand, the third equality uses the symmetric property of the element of projection matrix $\Pi_{im,j}$ and the final equality follows from the symmetric property of the summand. 
Letting $S_{im,j} \coloneqq \Pi_{im,j}\{X_{ij}u_{mj} +  X_{mj}u_{ij} \}$, we give the following conditional Hoeffding-type decomposition:
\begin{align*}
    2\mu_{n,j}^2 \hat{S}_{n,j}& = \sum_{i=1}^{n}\sum_{m\neq i}^n S_{im,j}  \\
    &= \sum_{i=1}^{n}\sum_{m\neq i}^n\mathbb{E}[S_{im,j} \mid Z] + 2\sum_{i=1}^n \left( \sum_{m\neq i}^n \Big\{ \mathbb{E}[S_{im,j} \mid Z, X_{ij}, u_{ij}] - \mathbb{E}[S_{im,j} \mid Z] \Big\} \right) \\
    & \quad + 2\sum_{i=1}^{n}\sum_{m>i}^n \Big( S_{im,j} - \mathbb{E}[S_{im,j} \mid Z, X_{ij}, u_{ij}] - \mathbb{E}[S_{im,j} \mid Z, X_{mj}, u_{mj}] + \mathbb{E}[S_{im,j} \mid Z]\Big) \\
    &\coloneqq (S.I)_j + (S.II)_j + (S.III)_j.
\end{align*}

In order to scrutinize properties of these terms, we calculate $\mathbb{E}[S_{im,j} \mid Z]  $ and $\mathbb{E}\left[ S_{im,j} \mid Z, X_{ij}, u_{ij} \right] $.
We can see that the conditional moment $\mathbb{E}[S_{im,j} \mid Z]  $ satisfies
\begin{align*}
    \mathbb{E}[S_{im,j} \mid Z]  
    &= \Pi_{im,j} \left(  \mathbb{E}[X_{ij} u_{mj} \mid Z_{ij}, Z_{mj}] + \mathbb{E}[X_{mj}u_{ij} \mid Z_{ij}, Z_{mj}] \right) \\
    &= \Pi_{im,j} \left( \mathbb{E}[X_{ij}\mid Z_{ij}] \mathbb{E}[ u_{mj}\mid Z_{mj}] + \mathbb{E}[ X_{mj}\mid Z_{mj}] \mathbb{E}[u_{ij}\mid Z_{ij}]\right) = 0
\end{align*}
where the final equality follows from $\mathbb{E}[u_{ij}\mid Z_{ij}]=0$, 
and $\mathbb{E}\left[ S_{im,j} \mid Z, X_{ij}, u_{ij} \right] $ satisfies
\begin{align*}
    \mathbb{E}\left[ S_{im,j} \mid Z, X_{ij}, u_{ij} \right] 
    &= \Pi_{im,j} \left( \mathbb{E}[X_{ij} u_{mj} \mid Z_{ij}, Z_{mj}, X_{ij}, u_{ij}] + \mathbb{E}[ X_{mj}u_{ij} \mid Z_{ij}, Z_{mj}, X_{ij}, u_{ij}] \right) \\
    &= \Pi_{im,j} \left( X_{ij} \mathbb{E}[ u_{mj} \mid Z_{mj}] + u_{ij}  \mathbb{E}[ X_{mj}\mid  Z_{mj}] \right) \\
    &= \Pi_{im,j} u_{ij}  \mathbb{E}[ X_{mj}\mid  Z_{mj}]  \\
    &= \Pi_{im,j} u_{ij} (\gamma_{n,j}^\top Z_{mj} ) \\
    &= \Pi_{im,j} u_{ij} \left(\frac{\mu_{n,j}}{\sqrt{n}} \pi_j^\top Z_{ij}  \right)
\end{align*} 
where the third equality follows from $\mathbb{E}[u_{ij} \mid Z_{ij}]=0$, the fourth equality follows from $X_{ij} = \gamma_{n,j}^\top Z_{ij} + \varepsilon_{ij}$ and $\mathbb{E}[\varepsilon_{ij} \mid Z_{ij}] = 0$, and the final equality follows from $\gamma_{n,j} = n^{-1/2}\mu_{n,j}(\pi_j^\top Z_{ij})$.

Using these results, we can see that  
\begin{align*}
    (S.I)_j = 0,
\end{align*}
and $(S.II)_j$ is a sum of the independent random variable as follows,
\begin{align*}
    & (S.II)_j \\
    &= \frac{\mu_{n,j}}{\sqrt{n}}\sum_{i=1}^n\sum_{m\neq i}^n \Pi_{im,j}u_{ij} (\pi_j^\top Z_{mj}) \\
    &= \frac{\mu_{n,j}}{\sqrt{n}}\sum_{i=1}^n\sum_{m=1}^n \Pi_{im,j}u_{ij} (\pi_j^\top Z_{mj}) - \frac{\mu_{n,j}}{\sqrt{n}} \sum_{i=1}^n \Pi_{ii,j}u_i (\pi_j^\top Z_{ij}) \\
    &= \frac{\mu_{n,j}}{\sqrt{n}}\sum_{i=1}^n\sum_{m=1}^n \Pi_{im,j}u_{ij} (\pi_j^\top Z_{mj}) - \frac{\mu_{n,j}}{\sqrt{n}} \sum_{i=1}^n \Pi_{ii,j}u_{ij} (\pi_j^\top Z_{ij}) + \frac{\mu_{n,j}}{\sqrt{n}} \sum_{i=1}^n u_{ij}(\pi_j^\top Z_{ij})  -\frac{\mu_{n,j}}{\sqrt{n}} \sum_{i=1}^n u_{ij}(\pi_j^\top Z_{ij}) \\
    &= \frac{\mu_{n,j}}{\sqrt{n}}\sum_{i=1}^n(1 - \Pi_{ii,j}) (\pi_j^\top Z_{ij})u_{ij} - \frac{\mu_{n,j}}{\sqrt{n}}\sum_{i=1}^n u_{ij} \left\{  (\pi_j^\top Z_{ij}) - \sum_{m=1}^n \Pi_{im,j} (\pi_j^\top Z_{mj})\right\} \\
    & =  \frac{\mu_{n,j}}{\sqrt{n}}\sum_{i=1}^n(1 - \Pi_{ii,j}) (\pi_j^\top Z_{ij})u_{ij},
\end{align*}
where the last equality follows from
\begin{align*}
    & \sum_{m=1}^n  P_{im,j}  (\pi_j^\top Z_{mj})^\top - (\pi_j^\top Z_{ij})^\top\\
    &=  \left( \sum_{j=1}^n \Pi_{im,j} Z_{mj}^\top - Z_{ij}^\top \right) \pi_j =  \left( \sum_{m=1}^n Z_{ij}^\top (\sum_{l=1}^n Z_{lj} Z_{lj}^\top)^{-1} Z_{mj} Z_{mj}^\top - Z_{ij}^\top \right) \pi_j  =  0.
\end{align*}
Finally, it holds that
\begin{align*}
    (S.III)_j
    &= \sum_{i=1}^n\sum_{m>i}^n \Pi_{im,j} \Big( \{X_{ij}u_{mj} + X_{mj}u_{ij}\} \\
    & \qquad\qquad- \mathbb{E}[X_{ij}u_{mj} + X_{mj}u_{ij} \mid Z, X_{ij}, u_{ij}]- \mathbb{E}[X_{ij}u_{mj} + X_{mj}u_{ij} \mid Z, X_{mj}, u_{mj}] \Big)\\
    &= \sum_{i=1}^n\sum_{m>i}^n \Pi_{im,j} \Big( \{(\gamma_{n,j}^\top Z_{ij}+\varepsilon_{ij})u_{ij}+(\gamma_{n,j}^\top Z_{mj}+\varepsilon_{mj})u_{ij}\} - u_{ij}(\gamma_{n,j}^\top Z_{mj}) - u_{mj}(\gamma_{n,j}^\top Z_{ij}) \Big) \\
    &= \frac{1}{2}\sum_{i=1}^n\sum_{m\neq i}^n \Pi_{im,j}  (\varepsilon_{ij}u_{mj} + \varepsilon_{mj}u_{ij}) .
\end{align*}

Summing up, we can see that
\begin{align*}
     \hat{S}_{n,j} = \frac{1}{\mu_{n,j}\sqrt{n}}\sum_{i=1}^n(1 - \Pi_{ii,j}) (\pi_j^\top Z_{ij})u_{ij} + \frac{1}{2\mu_{n,j}^2} \sum_{i=1}^n\sum_{m\neq i}^n \Pi_{im,j} (\varepsilon_{ij}u_{mj} + \varepsilon_{mj}u_{ij})  .
\end{align*}
\end{proof}

\subsection{Proof of \cref{thm:gauss-chao}}

\begin{proof}[Proof of \cref{thm:gauss-chao}]
First define $T^\circ_{n,jk} \coloneqq \sigma_{jk}^{-1}(Z)e_k^\top \Gamma_{n,j}^{-1}(L_{n,j} + Q_{n,j})$. 
Since $T^\circ_{jk}$ is a weighted $U$-statistics with size-dependent kernels conditional on $Z$, it is sufficient to show the convergence of the conditional on $Z$ version of quantities in the right hand side of \cref{thm:main-inid}.

Before the application of \cref{thm:main-inid}, observe that, under \cref{as:many-weak-iv}, it holds that
\begin{align*}
    &  \Var[e_k^\top \Gamma_{n,j}^{-1} L_{n,j}\mid Z] \\
    & =  \frac{1}{\mu_{n,j}^2} \frac{1}{n} \sum_{i=1}^n \left(e_k^\top\Gamma_{n,j}^{-1}(1 - \Pi_{ii,j}) (\pi_j^\top Z_{ij}) \right)^2 \mathbb{E}[u_{ij}^2 \mid Z]  \asymp  \frac{1}{\mu_{n,j}^2},
\end{align*}
for all $(j,k)\in[p]\times[d_j]$,
and, since $\sum_{i=1}^n\sum_{m=1}^n \Pi_{im,j}^2 = K_{n,j}$ and $\sum_{i=1}^n \Pi_{ii,j} = K_{n,j}$, it holds that
\begin{align*}
    \sum_{i=1}^n\sum_{m\neq i}^n \Pi_{im,j}^2 = K_{n,j} - \sum_{i=1}^n \Pi_{ii,j}^2 \ge \left( 1 - \max_{i\in[n]} \Pi_{ii,j}\right) K_{n,j} \asymp K_{n,j}
\end{align*}
and this implies
\begin{align*}
    &  \Var[e_k^\top \Gamma_{n,j}^{-1} Q_{n,j}\mid Z] \\
    & =\frac{1}{\mu_{n,j}^4}  \sum_{i=1}^n\sum_{m\neq i}^n \Pi_{im,j}^2  \mathbb{E}[( e_k^\top\Gamma_{n,j}^{-1}\{\varepsilon_{ij}u_{mj} + \varepsilon_{mj}u_{ij}\})^2 \mid Z]  \ge  \frac{C}{\mu_{n,j}^4}  \sum_{i=1}^n\sum_{m\neq i}^n \Pi_{im,j}^2 \ge C \frac{K_n}{\mu_{n,j}^4},
\end{align*}
for some constant $C$ and for all $(j,k)\in[p]\times[d_j]$.
Therefore, for all $(j,k)\in[p]\times[d_j]$, it holds that
\begin{align}
    & \sigma_{jk}^{-1}(Z) \nonumber\\
    & \lesssim   \min\{\Var[e_k^\top \Gamma_{n,j}^{-1} L_{n,j}\mid Z]^{-1/2},\Var[e_k^\top \Gamma_{n,j}^{-1} Q_{n,j}\mid Z]^{-1/2} \} \nonumber\\
    & \asymp \min\left\{\mu_{n,j}, \mu_{n,j}^{2}/\sqrt{K_{n,j}}\right\}. \label{eq:chao-eval-se}
\end{align}
In what follows, we often use this fact.

From \cref{lemma:chao-hoeffding}, we can define
\begin{align*}
    \pi_{1,i}\psi_{jk} \coloneqq  \frac{1}{\mu_{n,j}\sqrt{n}}e_k^\top\Gamma_{n,j}^{-1}(1 - \Pi_{ii,j}) (\pi_j^\top Z_{ij})u_{ij}, \quad \pi_{2,im}\psi_{jk} \coloneqq \frac{1}{\mu_{n,j}^2}e_k^\top\Gamma_{n,j}^{-1}\Pi_{im,j} (\varepsilon_{ij}u_{mj} + \varepsilon_{mj}u_{ij})  .
\end{align*}
Also, under \cref{as:many-weak-iv}, it holds that
\begin{align*}
    & \|\pi_{1,i}\psi_{jk}\|_{L^4(P_i\mid Z)}^4 \le \frac{1}{\mu_{n,j}^4 n^2} |e_k^\top\Gamma_{n,j}^{-1} (\pi_j^\top Z_{ij})|^4 \mathbb{E}[u_{ij}^4\mid Z] \le C\frac{ \|\Gamma^{-1}_{n,j}\|_{\text{op}}^4|\pi_j^\top Z_{ij}|^4 }{\mu_{n,j}^4 n^2}\le C \frac{|\pi_j^\top Z_{ij}|^4}{\mu_{n,j}^4 n^2} , \\
    & |\pi_{1,i}\psi_{jk}|\le \frac{1}{\mu_{n,j}\sqrt{n}}|e_k^\top\Gamma_{n,j}^{-1}\pi_j^\top Z_{ij}||u_{ij}| \le C \frac{\|\Gamma^{-1}_{n,j}\|_{\text{op}}|\pi_j^\top Z_{ij}|}{\mu_{n,j}\sqrt{n}} \le  C \frac{|\pi_j^\top Z_{ij}|}{\mu_{n,j}\sqrt{n}}.
\end{align*}
Therefore, from \cref{eq:chao-eval-se} and \cref{as:many-weak-iv}, we have
\begin{align}
    \Delta_{2,*}^{(1)}(1) \log^5D_n
    & =  \max_{j\in[p]} \max_{k\in[d_j]} \frac{1}{\sigma^4_{jk}(Z)} \sum_{i=1}^n \|\pi_{1,i}\psi_{jk}\|_{L^4(P_i \mid Z)}^4 \log^5D_n \nonumber\\
    & \lesssim  \max_{j\in[p]} \ \frac{1}{n^2}  \sum_{i=1}^n |\pi_j^\top Z_{ij}|^4 \log^5 D_n \to 0,  \label{eq:chao-cond-1}
\end{align}
and
\begin{align}
    \Delta_{2,q}^{(2)}(1) \log^5 D_n & \le  \left\|  \max_{j\in[p]} \max_{k\in[d_j]}\max_{i\in[n]} \frac{1}{\sigma_{jk}(Z)}|\pi_{1,i}\psi_{jk}|\right\|_{L^\infty(\pr\mid Z)}^4 \log^6 D_n\nonumber  \\
    &\lesssim \max_{j\in[p]}\max_{i\in[n]}  \frac{1}{n^2} |\pi_j^\top Z_{ij}|^4 \log^6D_n   \nonumber\\
    & \le \max_{j\in[p]}  \frac{1}{n^2}\sum_{i=1}^n  |\pi_j^\top Z_{ij}|^4 \log^6D_n \to 0.\label{eq:chao-cond-2}
\end{align} 
Also, under \cref{as:many-weak-iv}, it holds that
\begin{align*}
    & |\pi_{2,im}\psi_{jk}| \le \frac{C}{\mu_{n,j}^2} |\Pi_{im,j}|, \quad  \|\pi_{2,im}\psi_{jk}\|_{L^4(P_i\otimes P_m \mid Z)}^4 \le  \frac{C}{\mu_{n,j}^8} \Pi_{im,j}^4, \\
    & P_m(\pi_{2,im}\psi_{jk})^2 \le \frac{C}{\mu_{n,j}^4} \Pi_{im,j}^2 , \quad  P_m(\pi_{2,im}\psi_{jk})^4 \le \frac{C}{\mu_{n,j}^8} \Pi_{im,j}^4,
\end{align*}
and Cauchy-Schwarz inequality gives $\Pi_{im,j}^2 \le \Pi_{ii,j} \Pi_{mm,j} \le \max_{i\in[n]} \Pi_{ii}$ and this implies 
\begin{align*}
     \sum_{i=1}^n\sum_{m\neq i}\Pi_{im,j}^4 \le \left(\max_{i\in[n]} \Pi_{ii,j}\right)  \sum_{i=1}^n\sum_{m\neq i}\Pi_{im,j}^2 \le  \left(\max_{i\in[n]} \Pi_{ii,j}\right) K_{n,j}.
\end{align*}
Therefore,  from \cref{eq:chao-eval-se} ad \cref{as:many-weak-iv}, we have
\begin{align}
    \Delta_{2,*}^{(1)}(2)  \log^5 D_n
    & =   \max_{j\in[p]} \max_{k\in[d_j]} \frac{1}{\sigma_{jk}^4(Z)} \sum_{i=1}^n\sum_{m\neq i}\|\pi_{2,im}\psi_{jk}\|_{L^4(P_i\otimes P_m \mid Z)}^4 \log^8 D_n \nonumber\\
    & \lesssim \max_{j\in[p]}  \frac{K_{n,j}}{\mu_{n,j}^4} \left(\max_{i\in[n]} \Pi_{ii,j}\right) \log^8 D_n \lesssim \max_{j\in[p]}  \frac{K_{n,j}}{\mu_{n,j}^4} \log^8 D_n\to 0, \label{eq:chao-cond-3}
\end{align}
and
\begin{align}
    \Delta_{2,*}^{(2)}(2)  \log^5D_n 
    &=  \max_{j\in[p] }\max_{k\in[d_j]} \frac{1}{\sigma_{jk}^4(Z)} \sum_{i=1}^n \left\|\sum_{m\neq i}P_m(|\pi_{2,im}\psi_{jk}|^2)\right\|^2_{L^2(P_i\mid Z)}\log^7 D_n \nonumber\\
    &\lesssim   \max_{j\in[p] } \frac{1}{\mu_{n,j}^4}  \sum_{i=1}^n \left( \sum_{m\neq i} \Pi_{im,j}^2 \right)^2 \log^7 D_n \nonumber\\
    & \le  \max_{j\in[p] } \frac{1}{\mu_{n,j}^4}  \sum_{i=1}^n \Pi_{ii,j}^2  \log^7  D_n \nonumber \\
    & \le \max_{j\in[p] } \frac{K_{n,j}}{\mu_{n,j}^4} \left(\max_{i\in[n]} \Pi_{ii,j} \right)  \log^7  D_n  \lesssim \max_{j\in[p] }  \frac{K_{n,j}}{\mu_{n,j}^4}  \log^7  D_n \to 0, \label{eq:chao-cond-4}
\end{align}
 and
\begin{align}
    \Delta_{2,*}^{(3)}(2)  \log^5 D_n & = \E\left[  \max_{j\in[p] }\max_{k\in[d_j]}\max_{i\in[n]} \frac{1}{\sigma_{jk}^4(Z)} \sum_{m\neq i} P_m(\pi_{2,im}\psi_j)^4  \mid Z\right] \log^9 D_n \nonumber\\
    & \lesssim  \max_{j\in[p] }\max_{i\in[n]} \frac{1}{\mu_{n,j}^4} \sum_{m\neq i} \Pi_{im,j}^4 \log^9 D_n \nonumber\\
    & \le \max_{j\in[p] }\max_{i\in[n]} \frac{1}{\mu_{n,j}^4}\left( \sum_{m=1}^n \Pi_{im,j}^2 \right)^2 \log^9 D_n \nonumber\\
    & \le \max_{j\in[p] }   \frac{1}{\mu_{n,j}^4} \left( \max_{i\in[n]} \Pi_{ii,j} \right)^2 \log^9 D_n \lesssim  \max_{j\in[p] }   \frac{1}{\mu_{n,j}^4}\log^9 D_n \to 0,\label{eq:chao-cond-5}
\end{align}
and
\begin{align}
    \Delta_{2,q}^{(4)}(2)  \log^5D_n
    &\le  \left\| \max_{j\in[p]} \max_{k\in[d_j]} \max_{(i,m)\in I_{n,2}} \frac{1}{\sigma_{jk}(Z)} |\pi_{2,im}\psi_{jk}|\right\|_{L^\infty(\pr \mid Z)}^4 \log^{10}D_n \nonumber\\
    & \lesssim \max_{j\in[p]} \max_{(i,m)\in I_{n,2}} \frac{1}{\mu_{n,j}^4} \Pi_{im,j}^4 \log^{10}D_n \nonumber \\
    & \le \max_{j\in[p]}  \frac{1}{\mu_{n,j}^4} \left( \max_{i\in[n]} \Pi_{ii,j} \right)^2\log^{10}D_n \lesssim  \max_{j\in[p]}  \frac{1}{\mu_{n,j}^4} \log^{10}D_n \to 0, \label{eq:chao-cond-6}
\end{align}
and
 \begin{align}
    \Delta_{2,q}^{(5)}(2) \log^5 D_n & \le \left\| \max_{j\in[p]}\max_{k\in[d_j]} \max_{i\in[n]} \frac{1}{\sigma_{jk}^2(Z)} \sum_{m\neq i} P_m(|\pi_{2,im}\psi_{jk}|^2) \right\|_{L^{\infty}(\pr \mid Z)}^2 \log^8 D_n \nonumber\\
    & \lesssim  \max_{j\in[p]} \max_{i\in[n]} \frac{1}{\mu_{n,j}^4} \left(\sum_{m\neq  i} \Pi_{im,j}^2\right)^2 \log^8 D_n \nonumber\\
    & \le \max_{j\in[p]}   \frac{1}{\mu_{n,j}^4} \left( \max_{i\in[n]} \Pi_{ii,j}\right)^2 \log^8 D_n \lesssim \max_{j\in[p]}   \frac{1}{\mu_{n,j}^4} \log^8 D_n \to 0. \label{eq:chao-cond-7}
 \end{align}
 
Next, we evaluate $\Delta_{1}^{(0)}\log^6 D_n$.
First, recall that the definition of $\Delta_{1}^{(0)}\log^6 D_n$ is given as follows
\begin{align*}
    & \Delta_{1}^{(0)}\log^6 D_n \nonumber \\
     &= \max_{(j,j')\in[p]^2} \max_{k\in[d_j]}\max_{k'\in[d_{j'}]} \frac{1}{\sigma_{jk}^2(Z)\sigma_{j'k'}^2(Z)} \sum_{m=1}^n\sum_{l\neq m} \left\| \sum_{i\neq m,l} \pi_{2,im}\psi_{jk} \star_i^1 \pi_{2,il}\psi_{j'k'}\right\|_{L^2(P_m\otimes P_l\mid Z)}^2 \log^6D_n.
\end{align*}
Since
\begin{align*}
   &  \pi_{2,im}\psi_{jk} \star_i^1 \pi_{2,il}\psi_{j'k'} \\
   & = \frac{\Pi_{im,j}\Pi_{il,j'}}{\mu_{n,j}^2\mu_{n,j'}^2} \mathbb{E}[ e_k^\top\Gamma_{n,j}^{-1} (\varepsilon_{ij}u_{mj} + \varepsilon_{mj}u_{ij})  e_{k'}^\top\Gamma_{n,j'}^{-1} (\varepsilon_{ij'}u_{lj'} + \varepsilon_{lj'}u_{ij'})  \mid Z, u_{mj}, \varepsilon_{mj}, u_{lj'}, \varepsilon_{lj'}],
\end{align*}
we can see that
\begin{align*}
    & \left\| \sum_{i\neq m,l} \pi_{2,im}\psi_{jk} \star_i^1 \pi_{2,il}\psi_{j'k'}\right\|_{L^2(P_m\otimes P_l\mid Z)}^2 \\
    & \lesssim \frac{1}{\mu_{n,j}^4\mu_{n,j'}^4} \left( \sum_{i\neq m,l} \Pi_{im,j}\Pi_{il,j'} \mathbb{E}[(e_k^\top\Gamma_{n,j}^{-1}\varepsilon_{ij})(e_{k'}^\top\Gamma_{n,j'}^{-1}\varepsilon_{ij'}) \mid Z] \right)^2 \\
& \qquad + \frac{1}{\mu_{n,j}^4\mu_{n,j'}^4} \left( \sum_{i\neq m,l} \Pi_{im,j}\Pi_{il,j'} \mathbb{E}[(e_k^\top\Gamma_{n,j}^{-1}\varepsilon_{ij})u_{ij'} \mid Z] \right)^2 \\
& \qquad + \frac{1}{\mu_{n,j}^4\mu_{n,j'}^4} \left( \sum_{i\neq m,l} \Pi_{im,j}\Pi_{il,j'} \mathbb{E}[u_{ij}(e_{k'}^\top\Gamma_{n,j'}^{-1}\varepsilon_{ij'}) \mid Z] \right)^2 \\
& \qquad + \frac{1}{\mu_{n,j}^4\mu_{n,j'}^4} \left( \sum_{i\neq m,l} \Pi_{im,j}\Pi_{il,j'} \mathbb{E}[u_{ij}u_{ij'} \mid Z] \right)^2.
\end{align*}
For fixed $j,j'\in[p]$, $k\in[d_j]$ and $k'\in[d_{j'}]$, let $c_i$ denote generically any of the following four conditional moments:
\begin{align*}
    & \mathbb{E}[(e_k^\top\Gamma_{n,j}^{-1}\varepsilon_{ij})(e_{k'}^\top\Gamma_{n,j'}^{-1}\varepsilon_{ij'}) \mid Z], ~~
    \mathbb{E}[(e_k^\top\Gamma_{n,j}^{-1}\varepsilon_{ij})u_{ij'} \mid Z], ~~\mathbb{E}[u_{ij}(e_{k'}^\top\Gamma_{n,j'}^{-1}\varepsilon_{ij'}) \mid Z], ~~
    \mathbb{E}[u_{ij}u_{ij'} \mid Z].
\end{align*}
Under \cref{as:many-weak-iv}, all four choices satisfy $\max_{i\in[n]}|c_i|\le C$ uniformly over $j,j',k,k'$. Then,
\begin{align*}
    & \sum_{m=1}^n\sum_{l\neq m} \left( \sum_{i\neq m,l} c_i\Pi_{im,j}\Pi_{il,j'} \right)^2 \\
    & = \sum_{l=1}^n\sum_{m\neq l} \left( \sum_{i\neq l} c_i\Pi_{im,j}\Pi_{il,j'} - c_m\Pi_{mm,j}\Pi_{ml,j'} \right)^2 \\
    & \le 2\sum_{l=1}^n\sum_{m=1}^n \left( \sum_{i\neq l} c_i\Pi_{im,j}\Pi_{il,j'} \right)^2
    + 2\sum_{l=1}^n\sum_{m\neq l} c_m^2\Pi_{mm,j}^2\Pi_{ml,j'}^2 \\
    & = 2\sum_{l=1}^n\sum_{i\neq l}\sum_{i'\neq l} c_ic_{i'}\Pi_{il,j'}\Pi_{i'l,j'}\Pi_{ii',j}
    + 2\sum_{l=1}^n\sum_{m\neq l} c_m^2\Pi_{mm,j}^2\Pi_{ml,j'}^2 \\
    & \le 2\sum_{l=1}^n\sum_{i\neq l} c_i^2\Pi_{il,j'}^2
    + 2\sum_{l=1}^n\sum_{m\neq l} c_m^2\Pi_{mm,j}^2\Pi_{ml,j'}^2 \\
    & \lesssim K_{n,j'},
\end{align*}
where we used the symmetry and idempotency of $\Pi_j$ and $\sum_{l=1}^n\sum_{i\neq l}\Pi_{il,j'}^2\le K_{n,j'}$.
Also, changing the roles of $j,j'$ and $m,l$ gives the same bound with $K_{n,j}$ in place of $K_{n,j'}$ as follows
\begin{align*}
    & \sum_{m=1}^n\sum_{l\neq m} \left( \sum_{i\neq m,l} c_i\Pi_{im,j}\Pi_{il,j'} \right)^2 = \sum_{m=1}^n\sum_{l\neq m} \left( \sum_{i\neq m,l} c_i\Pi_{im,j'}\Pi_{il,j} \right)^2
    \lesssim K_{n,j}.
\end{align*}
Hence, 
\begin{align*}
    \sum_{m=1}^n\sum_{l\neq m} \left( \sum_{i\neq m,l} c_i\Pi_{im,j}\Pi_{il,j'} \right)^2
    \lesssim 4\min\{K_{n,j},K_{n,j'}\}
    \lesssim  4K_{n,j}^{1/2}K_{n,j'}^{1/2}.
\end{align*}
therefore, from \cref{eq:chao-eval-se} and \cref{as:many-weak-iv}, we have
\begin{align}
     & \Delta_{1}^{(0)}\log^6 D_n \nonumber \\
     &= \max_{(j,j')\in[p]^2} \max_{k\in[d_j]}\max_{k'\in[d_{j'}]} \frac{1}{\sigma_{jk}^2(Z)\sigma_{j'k'}^2(Z)} \sum_{m=1}^n\sum_{l\neq m} \left\| \sum_{i\neq m,l} \pi_{2,im}\psi_{jk} \star_i^1 \pi_{2,il}\psi_{j'k'}\right\|_{L^2(P_m\otimes P_l\mid Z)}^2 \log^6D_n  \nonumber\\
    & \lesssim \max_{(j,j')\in[p]^2}\max_{k\in[d_j]}\max_{k'\in[d_{j'}]}
    \frac{K_{n,j}^{1/2}K_{n,j'}^{1/2}}{\sigma_{jk}^2(Z)\sigma_{j'k'}^2(Z)\mu_{n,j}^4\mu_{n,j'}^4}\log^6D_n \nonumber\\
    & \lesssim \max_{(j,j')\in[p]^2}
    \frac{K_{n,j}^{1/2}K_{n,j'}^{1/2}}{\mu_{n,j}^2\mu_{n,j'}^2}\log^6D_n \nonumber\\
    & = \max_{j\in[p]}\frac{K_{n,j}}{\mu_{n,j}^4}\log^6D_n
    \to 0. \label{eq:chao-cond-8}
\end{align}

Next, we evaluate $\Delta_{1}^{(1)}\log^5D_n$.
First, recall that the definition of $\Delta_{1}^{(1)}\log^5D_n$ is given as follows
\begin{align*}
    & \Delta_{1}^{(1)}\log^5D_n \nonumber\\
    & = \max_{(j,j')\in[p]^2}\max_{k\in[d_j]}\max_{k'\in[d_{j'}]}
    \frac{1}{\sigma_{jk}^2(Z)\sigma_{j'k'}^2(Z)}
    \sum_{m=1}^n \left\|\sum_{i\neq m}\pi_{1,i}\psi_{jk} \star_i^1 \pi_{2,im}\psi_{j'k'} \right\|_{L^2(P_m\mid Z)}^2\log^5D_n.
\end{align*}
For notational simplicity, let $A_{i,jk} \coloneqq e_k^\top\Gamma_{n,j}^{-1}(1 - \Pi_{ii,j})(\pi_j^\top Z_{ij})$.
Then, expanding the contraction and using \cref{as:many-weak-iv}, we obtain
\begin{align*}
    & \left\|\sum_{i\neq m}\pi_{1,i}\psi_{jk} \star_i^1 \pi_{2,im}\psi_{j'k'} \right\|_{L^2(P_m\mid Z)}^2 \\
    & \lesssim \frac{1}{n\mu_{n,j}^2\mu_{n,j'}^4} \left( \sum_{i\neq m} A_{i,jk}\Pi_{im,j'} \mathbb{E}[u_{ij}(e_{k'}^\top\Gamma_{n,j'}^{-1}\varepsilon_{ij'}) \mid Z] \right)^2 + \frac{1}{n\mu_{n,j}^2\mu_{n,j'}^4} \left( \sum_{i\neq m} A_{i,jk}\Pi_{im,j'} \mathbb{E}[u_{ij}u_{ij'} \mid Z] \right)^2.
\end{align*}
For fixed $j,j'\in[p]$, $k\in[d_j]$ and $k'\in[d_{j'}]$, let $c_i$ denote either
$\mathbb{E}[u_{ij}(e_{k'}^\top\Gamma_{n,j'}^{-1}\varepsilon_{ij'}) \mid Z]$
or $\mathbb{E}[u_{ij}u_{ij'} \mid Z]$.
Under \cref{as:many-weak-iv}, both choices satisfy $\max_{i\in[n]}|c_i|\le C$ uniformly over $j,j',k,k'$.
Then,
\begin{align*}
    & \sum_{m=1}^n \left( \sum_{i\neq m} c_i\Pi_{im,j'}A_{i,jk} \right)^2 \\
    & = \sum_{m=1}^n \left( \sum_{i=1}^n c_i\Pi_{im,j'}A_{i,jk} - c_m\Pi_{mm,j'}A_{m,jk} \right)^2 \\
    & \le 2\sum_{m=1}^n \left( \sum_{i=1}^n c_i\Pi_{im,j'}A_{i,jk} \right)^2
    + 2\sum_{m=1}^n c_m^2\Pi_{mm,j'}^2A_{m,jk}^2 \\
    & = 2\sum_{i=1}^n\sum_{i'=1}^n c_ic_{i'}A_{i,jk}A_{i',jk}\Pi_{ii',j'}
    + 2\sum_{m=1}^n c_m^2\Pi_{mm,j'}^2A_{m,jk}^2 \\
    & \le 2\sum_{i=1}^n c_i^2A_{i,jk}^2
    + 2\sum_{m=1}^n c_m^2\Pi_{mm,j'}^2A_{m,jk}^2 \\
    & \lesssim \sum_{i=1}^n A_{i,jk}^2
    \lesssim \sum_{i=1}^n |\pi_j^\top Z_{ij}|^2
    \lesssim n,
\end{align*}
where we used the symmetry and idempotency of $\Pi_{j'}$, the fact that its eigenvalues are $0$ or $1$, and \cref{as:many-weak-iv}.
Applying this bound to each of the two terms above gives
\begin{align*}
    \sum_{m=1}^n \left\|\sum_{i\neq m}\pi_{1,i}\psi_{jk} \star_i^1 \pi_{2,im}\psi_{j'k'} \right\|_{L^2(P_m\mid Z)}^2
    \lesssim \frac{1}{\mu_{n,j}^2\mu_{n,j'}^4}.
\end{align*}
Therefore, from \cref{eq:chao-eval-se} and \cref{as:many-weak-iv}, we have
\begin{align}
    & \Delta_{1}^{(1)}\log^5D_n \nonumber\\
    & = \max_{(j,j')\in[p]^2}\max_{k\in[d_j]}\max_{k'\in[d_{j'}]}
    \frac{1}{\sigma_{jk}^2(Z)\sigma_{j'k'}^2(Z)}
    \sum_{m=1}^n \left\|\sum_{i\neq m}\pi_{1,i}\psi_{jk} \star_i^1 \pi_{2,im}\psi_{j'k'} \right\|_{L^2(P_m\mid Z)}^2\log^5D_n \nonumber\\
    & \lesssim \max_{(j,j')\in[p]^2}\max_{k\in[d_j]}\max_{k'\in[d_{j'}]}
    \frac{1}{\sigma_{jk}^2(Z)\sigma_{j'k'}^2(Z)\mu_{n,j}^2\mu_{n,j'}^4}\log^5D_n \nonumber\\
    & \lesssim \max_{j'\in[p]}\frac{1}{\mu_{n,j'}^2}\log^5D_n
    \to 0. \label{eq:chao-cond-9}
\end{align}

Finally, given the evaluation presented above, we can immediately deduce that
\begin{align}
    & \left(\max_{j\in[p]}  \frac{1}{\sigma_j^2} \sum_{i=1}^n \|\pi_{1,i}\psi_j\|_{L^2(P_i)}^2 \right)^{1/2} \left( \Delta_{2,q}^{(5)}(2) \log^7 (npd) \right)^{1/4} \lesssim \max_{j\in[p]}\frac{1}{\mu_{n,j}} \log^{5/2}(npd) \to 0. \label{eq:chao-cond-10}
\end{align}
From \cref{eq:chao-cond-1}--\cref{eq:chao-cond-10} and \cref{thm:main-inid}, we can see that
\begin{align}
     \sup_{t\in\mathbb{R}}\left|\mathbb{P}\left( \max_{j\in[p]}\max_{k\in[d_j]}|T^\circ_{jk}|\le t \mid Z\right) - \mathbb{P}\left( \max_{j\in[p]}\max_{k\in[d_j]}|G_{jk}| \le t \mid Z\right) \right| \to 0. \label{eq:chao-gauss}
\end{align}

Finally, we control the estimation error of $\hat{\Gamma}_{n,j}^{-1}$.
From \cref{lemma:chao-hoeffding}, we have
\begin{align*}
    & \max_{j\in[p]}\max_{k\in[d_j]}\left| \frac{\hat{\theta}_{n,jk} - \theta_{jk}}{\sigma_{jk}(Z)}  - \frac{e_k^\top\Gamma_{n,j}^{-1}\left( L_{n,j} + Q_{n,j} \right)}{\sigma_{jk}(Z)} \right| \\
    & = \max_{j\in[p]}\max_{k\in[d_j]}\left| \frac{e_k^\top(\hat{\Gamma}_{n,j}^{-1} - \Gamma_{n,j}^{-1})\left( L_{n,j} + Q_{n,j} \right)}{\sigma_{jk}(Z)}\right| \\
    & =  \max_{j\in[p]}\max_{k\in[d_j]} \left| \sum_{r=1}^{d_j} e_k^\top \hat{\Gamma}_{n,j}^{-1}(\Gamma_{n,j} - \hat{\Gamma}_{n,j})e_l \frac{\sigma_{jl}(Z)}{\sigma_{jk}(Z)} T^\circ_{jl}\right|\\
    & = \max_{j\in[p]}\max_{k\in[d_j]} \left| \sum_{r=1}^{d_j}\sum_{l=1}^{d_j} (e_k^\top \hat{\Gamma}_{n,j}^{-1}e_r)\{e_r^\top(\Gamma_{n,j} - \hat{\Gamma}_{n,j})e_l\} \frac{\sigma_{jl}(Z)}{\sigma_{jk}(Z)} T^\circ_{jl}\right|,
\end{align*}
where the second equality uses $\hat{\Gamma}_{n,j}^{-1} - \Gamma_{n,j}^{-1} = \hat{\Gamma}_{n,j}^{-1}(\Gamma_{n,j} - \hat{\Gamma}_{n,j})\Gamma_{n,j}^{-1}$ and $e_l^\top \Gamma_{n,j}^{-1}(L_{n,j} + Q_{n,j}) = \sigma_{jl}(Z)T_{jl}^\circ$, also the third equality uses $I_{d_j} = \sum_{r=1}^{d_j} e_r e_r^\top$.
Therefore
\begin{align*}
    & \max_{j\in[p]}\max_{k\in[d_j]}\left| \frac{\hat{\theta}_{n,jk} - \theta_{jk}}{\sigma_{jk}(Z)}  - \frac{e_k^\top\Gamma_{n,j}^{-1}\left( L_{n,j} + Q_{n,j} \right)}{\sigma_{jk}(Z)} \right| \\ 
    & \le \max_{j\in[p]} \left( \max_{(k,r)\in [d_j]^2}|e_k^\top \hat{\Gamma}_{n,j}^{-1} e_r| \max_{(k,l)\in[d_j]^2} \frac{\sigma_{jl}(Z)}{\sigma_{jk}(Z)} \right) \max_{j\in[p]} \max_{(r,l)\in[d_j]^2} |e_r^\top (\hat{\Gamma}_{n,j} - \Gamma_{n,j}) e_l|  \max_{j\in[p]} \max_{k\in[d_j]} |T_{jk}^\circ|.
\end{align*}
First, from the uniform convergence of $\hat{\Gamma}_{n,j}$ to $\Gamma_{n,j}$ (\cref{lemma:chao-Gamma-convergence}) and  \cref{as:many-weak-iv}(iii), we have $\max_{j\in[p]}\max_{(k,r)\in [d_j]^2}|e_k^\top \hat{\Gamma}_{n,j}^{-1} e_r| = O_p(1)$.
Second, from \cref{eq:chao-eval-se},  $\max_{j\in[p]} \max_{(k,l)\in[d_j]^2} \sigma_{jl}(Z)/\sigma_{jk}(Z) \lesssim 1$. 
Third, from \cref{eq:chao-gauss} and Lemma 2.2.10 in \cite{van1996weak}, we have $\max_{j\in[p]} \max_{k\in[d_j]} |T_{jk}^\circ| = O(\sqrt{\log D_n})$. 
Finally, note that \cref{lemma:chao-Gamma-convergence} gives the uniform convergence rate of $\max_{j\in[p]} \max_{(r,l)\in[d_j]^2} |e_r^\top (\hat{\Gamma}_{n,j} - \Gamma_{n,j}) e_l|$.
\begin{align*}
    & \max_{j\in[p]}\max_{k\in[d_j]}\left| \frac{e_k^\top(\hat{\Gamma}_{n,j}^{-1} - \Gamma_{n,j}^{-1})\left( L_{n,j} + Q_{n,j} \right)}{\sigma_{jk}}\right| \\
    & = O_p\left\{\left( \max_{j\in[p]}\frac{\sqrt{K_{n,j}}}{\mu_{n,j}^2} \log^{3/2} D_n +  \max_{j\in[p]}\frac{1}{\mu_{n,j}} \log D_n\right) \log^{1/2} D_n\right\}  \\
    & = O_p\left\{ \left( \max_{j\in[p]}\frac{\sqrt{K_{n,j}}}{\mu_{n,j}^2} \log^{5/2} D_n +  \max_{j\in[p]}\frac{1}{\mu_{n,j}} \log^{2}D_n \right) \frac{1}{\log^{1/2}D_n} \right\} = o_p\left(  \frac{1}{\log^{1/2} D_n}\right).
\end{align*}
Therefore \cref{eq:chao-gauss} and Lemma 1 in \cite{chernozhukov2023high} imply 
\begin{align*}
    \sup_{t\in\mathbb{R}}\left|\mathbb{P}\left( \max_{j\in[p]}\max_{k\in[d_j]}|T_{jk}|\le t \mid Z\right) - \mathbb{P}\left( \max_{j\in[p]}\max_{k\in[d_j]}|G_{jk}| \le t \mid Z\right) \right| \to 0.
\end{align*}
\end{proof}

\section{Proofs for \cref{subsubsec:application-ms}} \label{subsec:proof-ms}

\subsection{Proof of \cref{thm:mikusheva-sun-variance}(i)}
\begin{proof}[Proof of \cref{thm:mikusheva-sun-variance}(i)]
First observe that
\begin{align*}
    \hat{\Phi}_j - \Phi_j = (\hat{\Phi}_j - \mathbb{E}[\hat{\Phi}_j \mid Z]) + ( \mathbb{E}[\hat{\Phi}_j \mid Z] - \Phi_j) \eqqcolon (\Phi.I)_j + (\Phi.II)_j.
\end{align*}

In terms of $(\Phi.II)_j$, we can see that
\begin{align*}
   & \mathbb{E}[\hat{\Phi}_j \mid Z] \\
   & = \frac{2}{K_{n,j}} \sum_{i=1}^n\sum_{m\neq i} \Pi_{im,j}^2 (M_{ii,j}M_{mm,j} + M_{im,j}^2)^{-1} \mathbb{E}[(u_{i,j} M_{i,j}u_j)(u_{m,j} M_{m,j} u_j)\mid Z] \\
   & = \frac{2}{K_{n,j}} \sum_{i=1}^n\sum_{m\neq i} \Pi_{im,j}^2 (M_{ii,j}M_{mm,j} + M_{im,j}^2)^{-1} \mathbb{E}\left[\left(u_{i,j} \sum_{l=1}^n M_{il,j}u_{lj}\right)\left(u_{mj} \sum_{l=1}^n M_{ml,j}u_{lj}\right)\mid Z\right] \\
   & = \frac{2}{K_{n,j}} \sum_{i=1}^n\sum_{m\neq i} \Pi_{im,j}^2 (M_{ii,j}M_{mm,j} + M_{im,j}^2)^{-1} (M_{ii,j}M_{mm,j} + M_{im,j}^2) \mathbb{E}[ u_{ij}^2\mid Z]\mathbb{E}[ u_{mj}^2\mid Z]\\
   & = \frac{2}{K_{n,j}} \sum_{i=1}^n\sum_{m\neq i} \Pi_{im,j}^2 \mathbb{E}[ u_{ij}^2\mid Z]\mathbb{E}[ u_{mj}^2\mid Z]\\
   & = \Phi_j,
\end{align*}
where the second line follows from $ M_{i,j}u_j = \sum_{l=1}^n M_{ml,j}u_{lj}$ and the third line follows from $\mathbb{E}[u_{ij} \mid Z] = 0$.
This implies that $(\Phi.II)_j = 0$.

In terms of $(\Phi.I)_j$, since $ M_{i,j}u_j = \sum_{l=1}^n M_{ml,j}u_{l,j}$, we can see that
\begin{align*}
    (\Phi.I)_j 
    &= \frac{2}{K_{n,j}} \sum_{i=1}^n\sum_{m\neq i}  \frac{\Pi_{im,j}^2}{M_{ii,j}M_{mm,j} + M_{im,j}^2} \\
    & \qquad \times \left\{ (u_{ij} M_{i,j}u_j)(u_{mj} M_{m,j} u_j)  - (M_{ii,j}M_{mm,j} + M_{im,j}^2) \mathbb{E}[ u_{ij}^2\mid Z]\mathbb{E}[ u_{mj}^2\mid Z]\right\}  \\
    & = \frac{2}{K_{n,j}} \sum_{i=1}^n\sum_{m\neq i} \Pi_{im,j}^2 \{u_{ij}^2u_{mj}^2 - \mathbb{E}[ u_{ij}^2\mid Z]\mathbb{E}[ u_{mj}^2\mid Z] \} \\
    & \quad + \frac{2}{K_{n,j}} \sum_{i=1}^n\sum_{m\neq i}  \frac{\Pi_{im,j}^2}{M_{ii,j}M_{mm,j} + M_{im,j}^2}  M_{ii,j}u_{i,j}^2\sum_{l\neq m} M_{ml,j}u_{mj}u_{lj} \\
    & \quad + \frac{2}{K_{n,j}} \sum_{i=1}^n\sum_{m\neq i}  \frac{\Pi_{im,j}^2}{M_{ii,j}M_{mm,j} + M_{im,j}^2}  M_{mm,j}u_{m,j}^2\sum_{l\neq i} M_{il,j}u_{ij}u_{lj} \\
    & \quad + \frac{2}{K_{n,j}} \sum_{i=1}^n\sum_{m\neq i}  \frac{\Pi_{im,j}^2}{M_{ii,j}M_{mm,j} + M_{im,j}^2} \underset{\substack{l\neq i, r\neq m \\ (l,r) \neq (m,i)}}{\sum\sum} M_{il,j}M_{mr,j}u_{ij}u_{lj}u_{mj}u_{rj} \\
    & = \sum_{(i,m) \in I_{n,2}} \psi^{[2]}_{j,(i,m)} + \sum_{(i,m,l) \in I_{n,3}}  \psi^{[3]}_{j,(i,m,l)} +   \sum_{(i,m,l,r) \in I_{n,4}}  \psi^{[4]}_{j,(i,m,l,r)},
\end{align*}
with
\begin{align*}
    \psi^{[2]}_{j,(i,m)} &=  \frac{2}{K_{n,j}} \Pi_{im,j}^2 \{u_{ij}^2u_{mj}^2 - \mathbb{E}[ u_{ij}^2\mid Z]\mathbb{E}[ u_{mj}^2\mid Z] \} \\
    & \quad +  \frac{2}{K_{n,j}} \frac{\Pi_{im,j}^2}{M_{ii,j}M_{mm,j} + M_{im,j}^2} \{M_{ii,j}M_{mi,j}u_{ij}^3 u_{mj} + M_{im,j}M_{mm,j}u_{ij}u_{mj}^3\}, \\
    \psi^{[3]}_{j,(i,m,l)} &=   \frac{2}{K_{n,j}} \frac{\Pi_{im,j}^2}{M_{ii,j}M_{mm,j} + M_{im,j}^2} \big[ \{M_{ii,j}M_{ml,j} + M_{il,j}M_{mi,j}\} u_{ij}^2u_{mj}u_{lj} \\
    & \qquad\qquad +\{M_{il,j}M_{mm,j} + M_{im,j}M_{ml,j}\} u_{ij}u_{mj}^2u_{lj} + M_{il,j}M_{ml,j} u_{ij}u_{mj}u_{lj}^2  \big], \\
    \psi^{[4]}_{j,(i,m,l,r)} &=  \frac{2}{K_{n,j}} \frac{\Pi_{im,j}^2}{M_{ii,j}M_{mm,j} + M_{im,j}^2} M_{il,j}M_{mr,j}u_{ij}u_{lj}u_{mj}u_{rj}.
\end{align*}
Therefore, symmetrization and Hoeffding decomposition yield
\begin{align*}
    (\Phi.I)_j 
    & = J_2(\tilde\psi^{[2]}_{j}) + J_3(\tilde\psi^{[3]}_{j}) + J_4(\tilde\psi^{[4]}_{j}) = \sum_{b=2}^4\sum_{q=1}^b J_q(\pi_q\tilde\psi^{[b]}_j).
\end{align*}
Since for all $b=2,3,4$ and $q=1,\dots, b$, $J_q(\pi_q\psi^{[b]}_j)$ are degenerate $U$-statistics, \cref{thm:max-is-inid} gives
\begin{align*}
    & \mathbb{E}\left[ \max_{j\in[p]} |J_q(\pi_q\tilde\psi^{[b]}_j)| \mid Z \right] \\
    & \lesssim \max_{0 \le h \le q} \log^{(q+h)/2}(np) \mathbb{E}\left[ \max_{j\in[p]} \max_{(i_1\dots, i_h)\in I_{n,h}} \sum_{(k_1,\dots, k_{q-h})} \mathbb{E}\left[ (\pi_q\tilde\psi^{[b]}_{l,j})^2 \mid u_{i_1j}, \dots, u_{i_hj}, Z\right] \mid Z\right]^{1/2}.
\end{align*}
Observe that, since $M_{ii,j} = 1 - \Pi_{ii,j} \ge 1- C$, $(M_{ii,j}M_{mm,j} + M_{im,j}^2)^{-1} \lesssim 1$. 
Also, 
\begin{align*}
    \sum_{m=1}^n \Pi_{im,j}^2 = \Pi_{ii,j} \le C, ~ \sum_{i=1}^n\sum_{m=1}^n \Pi_{im,j}^2 = K_{n,j}, ~ \sum_{i=1}^n\sum_{m=1}^n \Pi_{im,j}^4 \lesssim K_{n,j}, ~ \sum_{l=1}^n M_{il,j}^2 = M_{ii,j}\le 1.
\end{align*}
In what follows, we often use these facts. 
Regarding the terms associated with $\psi^{[2]}_{j,(i,m)}$, 
since coefficients appearing in $\pi_{1}\psi^{[2]}_{j,(i,m)}$ are bounded in absolute value by a constant multiple of $K_{n,j}^{-1}\sum_{m\neq i}\Pi_{im,j}^2$  and coefficients appearing in $\pi_{2}\psi^{[2]}_{j,(i,m)}$ are bounded by a constant multiple of $K_{n,j}^{-1}\Pi_{im,j}^2$,
we can see 
\begin{align*}
   &  \mathbb{E}\left[ \max_{j\in[p]} |J_1(\pi_1\tilde\psi^{[2]}_j)| \mid Z \right] +  \mathbb{E}\left[ \max_{j\in[p]} |J_2(\pi_2\tilde\psi^{[2]}_j)| \mid Z \right] \\
   &\lesssim \log(np) \max_{j\in[p]} \left\{K_{n,j}^{-2} \sum_{i=1}^n \left( \sum_{m\neq i} \Pi_{im,j}^2\right)^2 + K_{n,j}^{-2} \max_{i\in[n]}\left( \sum_{m\neq i} \Pi_{im,j}^2\right)^2\right\}^{1/2} \\
   & \quad + \log^2(np) \max_{j\in[p]} \left\{  K_{n,j}^{-2} \left( \sum_{(i,m)\in I_{n,2}}\Pi_{im,j}^4 + \max_{i\in[n]} \sum_{m\neq i} \Pi_{im,j}^4 + \max_{(i,m)\in I_{n,2}} \Pi_{im,j}^4 \right)\right\}^{1/2} \\
   & \lesssim \log^2(np) \max_{j\in[p]}K_{n,j}^{-1/2}.
\end{align*}

Regarding the terms associated with $\psi^{[3]}_{j,(i,m,l)}$, 
we first note that $\pi_1\psi^{[3]}_j=0$. 
The coefficients appearing in $\pi_2\psi^{[3]}_j$ are bounded by a constant multiple of 
\begin{align*}
    & (\Phi.I.III.II)_{ml,j} \\
    & \coloneqq K_{n,j}^{-1}\sum_{i\neq m,l}\Pi_{im,j}^2 |M_{ml,j}| \vee K_{n,j}^{-1}\sum_{i\neq m,l}\Pi_{im,j}^2 |M_{il,j}M_{mi,j}| \vee K_{n,j}^{-1}\Pi_{im,j}^2 \sum_{l\neq {i,m}}\left|M_{il,j}M_{ml,j}\right|.
\end{align*}
Similarly, the coefficients appearing in $\pi_3\psi^{[3]}_j$ are bounded  by a constant multiple of 
\begin{align*}
    & (\Phi.I.III.III)_{iml,j}  \coloneqq K_{n,j}^{-1}\Pi_{im,j}^2 |M_{ml,j}| \vee
    K_{n,j}^{-1}\Pi_{im,j}^2 |M_{il,j}M_{mi,j}| \vee 
    K_{n,j}^{-1}\Pi_{im,j}^2 |M_{il,j}M_{ml,j}|.
\end{align*}
Therefore, from \cref{thm:max-is-inid}, we have
\begin{align*}
    &\mathbb E\left[\max_{j\in[p]}|J_2(\pi_2\tilde \psi^{[3]}_j)|\mid Z\right]
    +
    \mathbb E\left[\max_{j\in[p]}|J_3(\pi_3\tilde \psi^{[3]}_j)|\mid Z\right]\\
    &\lesssim 
    \log^2(np)  \max_{j\in[p]} \Bigg\{\sum_{(m,l)\in I_{n,2}}\widetilde{(\Phi.I.III.II)}_{ml,j}^2  \\
    & \qquad\qquad\qquad+ \max_{m\in[n]}\sum_{l\neq m}\widetilde{(\Phi.I.III.II)}_{ml,j}^2  + \max_{(m,l)\in I_{n,2}}\widetilde{(\Phi.I.III.II)}_{ml,j}^2 \Bigg\}^{1/2}  \\
    & \quad + \log^3(np) \max_{j\in[p]} \Bigg\{ \sum_{(i,m,l)\in I_{n,3}} \widetilde{(\Phi.I.III.III)}_{iml,j}^2  + \max_{i\in[n]} \sum_{\substack{1\le m\neq l\le n\\ m,l\neq i}}\widetilde{(\Phi.I.III.III)}_{iml,j}^2  \\
    & \qquad\qquad\qquad +  \max_{(i,m)\in I_{n,2}}\sum_{l\neq i,m} \widetilde{(\Phi.I.III.III)}_{iml,j}^2 + \max_{(i,m,l)\in I_{n,3}}\widetilde{(\Phi.I.III.III)}_{iml,j}^2  \Bigg\}^{1/2} \\
    & \lesssim \log^3(np) \max_{j\in[p]} K_{n,j}^{-1/2}.
\end{align*}

Regarding the terms associated with $\psi^{[4]}_{j,(i,m,l,r)}$, 
we have $\pi_q\psi^{[4]}_j=0$ for $q=1,2,3$ since $\mathbb E[u_{ij}\mid Z]=0$.
Thus $\psi^{[4]}_j$ is already a fourth-order degenerate kernel.
The coefficients appearing in $\psi^{[4]}_{j,(i,m,l,r)}$ are bounded in absolute value by a constant multiple of
\begin{align*}
    (\Phi.I.IV)_{imlr,j} \coloneqq K_{n,j}^{-1}\Pi_{im,j}^2 |M_{il,j}M_{mr,j}|.
\end{align*}
Hence, from \cref{thm:max-is-inid}, we have
\begin{align*}
    & \mathbb E\left[\max_{j\in[p]}|J_4(\tilde\psi^{[4]}_j)|\mid Z\right] \\
    &\lesssim
    \log^4(np) \Bigg\{ \max_{j\in[p]} \sum_{(i,m,l,r)\in I_{n,4}} \widetilde{(\Phi.I.IV)}_{imlr,j}^2  + \max_{j\in[p]} \max_{i\in[n]}\sum_{\substack{(m,l,r)\in I_{n,3}\\m,l,r\neq i}} \widetilde{(\Phi.I.IV)}_{imlr,j}^2 \\
    & ~~ + \max_{j\in[p]} \max_{(i,m)\in I_{n,2}}\sum_{\substack{(l,r)\in I_{n,2}\\l,r\neq i, m}} \widetilde{(\Phi.I.IV)}_{imlr,j}^2  + \max_{j\in[p]} \max_{(i,m,l)\in I_{n,3}}\sum_{r\neq i,m,l} \widetilde{(\Phi.I.IV)}_{imlr,j}^2  + \max_{(i,m,l,r)\in I_{n,4}}\widetilde{(\Phi.I.IV)}_{imlr,j}^2\Bigg\}^{1/2}\\
    & \lesssim \log^4(np) \max_{j\in[p]} K_{n,j}^{-1/2}.
\end{align*}

Summing up, it holds that
\begin{align*}
    \mathbb{E}\left[ \max_{j\in[p]} |(\Phi.I)_j| \mid Z \right] \lesssim  \max_{j\in[p]} \frac{\log^{4}(np)}{\sqrt{K_{n,j}}}.
\end{align*}
In conjunction with $(\Phi.II)_j = 0$ and \cref{as:mikusheva-sun}(ii), which implies $ \min_{j\in[p]}\Phi_j \asymp 1$, we have
\begin{align*}
    \mathbb{E}\left[ \max_{j\in[p]} |\hat{\Phi}_j/ \Phi_j -1 | \mid Z \right] \lesssim  \max_{j\in[p]} \frac{\log^{4}(np)}{\sqrt{K_{n,j}}}.
\end{align*}
This result and Markov's inequality complete the proof.
\end{proof}

\subsection{Proof of \cref{thm:mikusheva-sun-variance}(ii)}
\begin{proof}[Proof of \cref{thm:mikusheva-sun-variance}(ii)]
First define $\tilde\Phi_j$ as 
\begin{align*}
    \tilde\Phi_j \coloneqq \frac{2}{K_{n,j}} \sum_{i=1}^n \sum_{m\neq i}^n \frac{ \Pi_{im,j}^2}{M_{ii,j}M_{mm,j} + M_{im,j}^2} (\eta_{i,j} M_{i,j}\eta_j)(\eta_{m,j} M_{m,j}\eta_j),
\end{align*}
with $\eta_j \coloneqq (\eta_{1,j}, \dots, \eta_{n,j})^\top$.
Then, we can see that
\begin{align*}
    \hat{\Phi}_j - \Phi_j = (\hat{\Phi}_j - \tilde{\Phi}_j) + (\tilde{\Phi}_j - \Phi_j).
\end{align*}
We can derive the convergence of rate of the second term in the exactly same way as that of $\hat{\Phi}_j - \Phi_j  $ in the proof of \cref{thm:mikusheva-sun-variance}(i) and
\begin{align*}
    \mathbb{E}\left[ \max_{j\in[p]} |\tilde{\Phi}_j - \Phi_j |  \mid Z \right] \lesssim  \max_{j\in[p]} \frac{\log^{4}(np)}{\sqrt{K_{n,j}}}.
\end{align*}

In terms of the first term, observe that
\begin{align*}
    & \hat{\Phi}_j - \tilde{\Phi}_j \\
    &=\frac{2}{K_{n,j}} \sum_{i=1}^n\sum_{m\neq i}  \frac{\Pi_{im,j}^2}{M_{ii,j}M_{mm,j} + M_{im,j}^2} \{(u_{ij} M_{i,j}u_j)(u_{mj} M_{m,j} u_j) - (\eta_{i,j} M_{i,j}\eta_j)(\eta_{m,j} M_{m,j}\eta_j)\}.
\end{align*}
Under the alternative hypothesis $\theta_j = \theta_{0,j} + \Delta_j$, it holds that
\begin{align*}
    u_{ij} &= u_{ij}(\theta_{0,j}) \\
    & = Y_{ij} - X_{ij}\theta_{0,j} \\
    & = Y_{ij} - X_{ij}(\theta_j - \Delta_j) = u_{ij}(\theta_j) + X_{ij}\Delta_j  = u_{ij}(\theta_j) + (\gamma_{n,j}^\top Z_{ij} + \varepsilon_{ij})\Delta_j  = \eta_{i,j} + \Delta_{j}(\gamma_{n,j}^\top Z_{ij}),
\end{align*}
so
\begin{align*}
    & \hat{\Phi}_j - \tilde{\Phi}_j \\
    & = \frac{2}{K_{n,j}} \sum_{i=1}^n\sum_{m\neq i}  \frac{\Pi_{im,j}^2}{M_{ii,j}M_{mm,j} + M_{im,j}^2}\Big\{  \Delta_j(\gamma_j^\top Z_{ij}) (M_{i,j}\eta_j)(\eta_{m,j} M_{m,j}\eta_j) \\
    & \qquad + \Delta_j(\gamma_j^\top Z_{mj})(\eta_{i,j} M_{i,j}\eta_j)(M_{m,j}\eta_j) +  \Delta_j^2(\gamma_j^\top Z_{ij})(\gamma_j^\top Z_{mj})(M_{i,j}\eta_j)(M_{m,j}\eta_j)\Big\} \\
    & = \frac{2}{K_{n,j}} \sum_{i=1}^n\sum_{m\neq i} \sum_{l=1}^n\sum_{r=1}^n  \frac{\Pi_{im,j}^2 M_{il,j}M_{mr,j}}{M_{ii,j}M_{mm,j} + M_{im,j}^2}  \Delta_j(\gamma_j^\top Z_{mj}) \eta_{i,j}\eta_{l,j}\eta_{r,j}  \\
    & \quad + \frac{2}{K_{n,j}} \sum_{i=1}^n\sum_{m\neq i} \sum_{l=1}^n\sum_{r=1}^n  \frac{\Pi_{im,j}^2 M_{il,j}M_{mr,j}}{M_{ii,j}M_{mm,j} + M_{im,j}^2}  \Delta_j(\gamma_j^\top Z_{ij}) \eta_{m,j}\eta_{l,j}\eta_{r,j}\\
    & \quad +  \frac{2}{K_{n,j}} \sum_{i=1}^n\sum_{m\neq i} \sum_{l=1}^n\sum_{r=1}^n\frac{\Pi_{im,j}^2 M_{il,j}M_{mr,j}}{M_{ii,j}M_{mm,j} + M_{im,j}^2} \Delta_j^2(\gamma_j^\top Z_{ij})(\gamma_j^\top Z_{mj})\eta_{l,j}\eta_{r,j} \\
    & = \sum_{l=1}^n \psi_{j,l}^{[1],\Delta} + \sum_{(l,r)\in I_{n,2}} \psi_{j,(l,r)}^{[2],\Delta} + \sum_{(i,l,r)\in I_{n,3}} \psi_{j,(i,l,r)}^{[3],\Delta},
\end{align*}
with
\begin{align*}
    &\psi_{j,l}^{[1],\Delta} \coloneqq  \frac{2}{K_{n,j}}\sum_{m\neq l} \frac{\Pi_{lm,j}^2M_{ll,j}M_{ml,j}}{M_{ll,j}M_{mm,j} + M_{lm,j}^2} \Delta_j(\gamma_j^\top Z_{mj}) \eta_{l,j}^3 \\
    & \qquad +  \frac{2}{K_{n,j}}\sum_{i\neq l} \frac{\Pi_{il,j}^2 M_{il,j}M_{ll,j}}{M_{ii,j}M_{ll,j} + M_{il,j}^2}  \Delta_j(\gamma_j^\top Z_{ij}) \eta_{l,j}^3 \\
    & \qquad  +  \frac{2}{K_{n,j}}\sum_{i=1}^n\sum_{m\neq i} \frac{\Pi_{im,j}^2 M_{il,j}M_{ml,j}}{M_{ii,j}M_{mm,j} + M_{im,j}^2} \Delta_j^2(\gamma_j^\top Z_{ij})(\gamma_j^\top Z_{mj})\eta_{l,j}^2, \\
    & \psi_{j,(l,r)}^{[2],\Delta} \coloneqq \frac{2}{K_{n,j}}\sum_{m\neq l} \Delta_j(\gamma_j^\top Z_{mj}) \Big[ \frac{\Pi_{lm,j}^2 (M_{ll,j}M_{mr,j} + M_{lr,j}M_{ml,j})}{M_{ll,j}M_{mm,j} + M_{lm,j}^2} \eta_{l,j}^2\eta_{r,j} + \frac{\Pi_{lm,j}^2M_{lr,j}M_{mr,j}}{M_{ll,j}M_{mm,j} + M_{lm,j}^2}\eta_{l,j}\eta_{r,j}^2 \Big] \\
    & \qquad + \frac{2}{K_{n,j}}\sum_{i\neq l} \Delta_j(\gamma_j^\top Z_{ij}) \Big[ \frac{\Pi_{il,j}^2 (M_{il,j}M_{lr,j} + M_{ir,j}M_{ll,j})}{M_{ii,j}M_{ll,j} + M_{il,j}^2}  \eta_{l,j}^2\eta_{r,j} + \frac{\Pi_{il,j}^2 M_{ir,j}M_{lr,j}}{M_{ii,j}M_{ll,j} + M_{il,j}^2}  \eta_{l,j}\eta_{r,j}^2 \Big]  \\
    & \qquad + \frac{2}{K_{n,j}} \sum_{i=1}^n\sum_{m\neq i} \frac{\Pi_{im,j}^2 M_{il,j}M_{mr,j}}{M_{ii,j}M_{mm,j} + M_{im,j}^2} \Delta_j^2(\gamma_j^\top Z_{ij}) (\gamma_j^\top Z_{mj}) \eta_{l,j}\eta_{r,j}, \\
    & \psi_{j,(i,l,r)}^{[3],\Delta} \coloneqq \frac{2}{K_{n,j}}\sum_{m\neq i}\frac{\Pi_{im,j}^2 M_{il,j}M_{mr,j}}{M_{ii,j}M_{mm,j} + M_{im,j}^2} \Delta_j(\gamma_j^\top Z_{mj}) \eta_{i,j}\eta_{l,j}\eta_{r,j} \\
    & \qquad + \frac{2}{K_{n,j}}\sum_{h\neq i} \frac{\Pi_{hi,j}^2 M_{hl,j}M_{ir,j}}{M_{hh,j}M_{ii,j} + M_{hi,j}^2} \Delta_j(\gamma_j^\top Z_{hj}) \eta_{i,j}\eta_{l,j}\eta_{r,j} .
\end{align*}
Then, symmetrization and Hoeffding decomposition yield
\begin{align*}
     \hat{\Phi}_j - \tilde{\Phi}_j \asymp \sum_{b=1}^3 \mathbb{E}[ J_b(\tilde\psi_{l,j}^{[b],\Delta}) \mid Z] + \sum_{b=1}^3\sum_{q=1}^b J_q(\pi_q\tilde\psi_j^{[b],\Delta}).
\end{align*}
Note that $\mathbb{E}[ J_2(\tilde\psi_{j}^{[2],\Delta}) \mid Z] = 0$ and $\mathbb{E}[ J_3(\tilde\psi_{j}^{[3],\Delta}) \mid Z] = 0$.
Also,  it holds that
\begin{align*}
    & \mathbb{E}[ J_1(\tilde\psi_{j}^{[1],\Delta}) \mid Z]\\
    &=  \frac{2}{K_{n,j}}\sum_{i=1}^n \sum_{m\neq i}\Bigg[  \frac{\Pi_{im,j}^2M_{ii,j}M_{mi,j}}{M_{ii,j}M_{mm,j} + M_{im,j}^2} \Delta_j(\gamma_j^\top Z_{mj}) \mathbb{E}[\eta_{i,j}^3 \mid Z]\\
    & \qquad +  \frac{\Pi_{im,j}^2M_{im,j}M_{mm,j}}{M_{ii,j}M_{mm,j} + M_{im,j}^2} \Delta_j(\gamma_j^\top Z_{ij}) \mathbb{E}[\eta_{m,j}^3 \mid Z] \\
    & \qquad  + \sum_{l=1}^n\frac{\Pi_{im,j}^2 M_{il,j}M_{ml,j}}{M_{ii,j}M_{mm,j} + M_{im,j}^2} \Delta_j^2(\gamma_j^\top Z_{ij})(\gamma_j^\top Z_{mj}) \mathbb{E}[\eta_{l,j}^2 \mid Z] \Bigg].
\end{align*}
Since $(M_{ii,j}M_{mm,j} + M_{im,j}^2)^{-1}\lesssim 1$, $\sum_{m=1}^n \Pi_{im,j}^2 \le C$, $\sum_{i=1}^n\sum_{m=1}^n \Pi_{im,j}^4 \lesssim K_{n,j}$, $M_{ii,j} \le 1$, $M_{im,j} = - \Pi_{im,j}$ and $\sum_{l=1}^n|M_{il,j}M_{ml,j}| \le (\sum_{l=1}^n M_{il,j}^2)^{1/2}(\sum_{l=1}^n M_{ml,j}^2)^{1/2} \le 1$, then it holds that
\begin{align*}
    & |\mathbb{E}[ J_1(\tilde\psi_{j}^{[1],\Delta}) \mid Z]| \\
    & \lesssim \frac{1}{K_{n,j}}\sum_{i=1}^n \sum_{m\neq i} |\Pi_{im,j}|^3 |\Delta_j(\gamma_j^\top Z_{mj})|\\
    &\quad + \frac{1}{K_{n,j}}\sum_{i=1}^n \sum_{m\neq i} |\Pi_{im,j}|^3 |\Delta_j(\gamma_j^\top Z_{ij})| + \frac{1}{K_{n,j}} \sum_{i=1}^n \sum_{m\neq i} \Pi_{im,j}^2| \Delta_j^2(\gamma_j^\top Z_{ij})(\gamma_j^\top Z_{mj})|\\
    &\lesssim \frac{1}{K_{n,j}} \left( \sum_{i=1}^n \sum_{m\neq i} \Pi_{im,j}^4 \right)^{1/2}\left( \sum_{i=1}^n \sum_{m\neq i}\Pi_{im,j}^2 |\Delta_j(\gamma_j^\top Z_{mj})|^2 \right)^{1/2} \\
    & \quad + \frac{1}{K_{n,j}} \left( \sum_{i=1}^n \sum_{m\neq i} \Pi_{im,j}^4 \right)^{1/2}\left( \sum_{i=1}^n \sum_{m\neq i} \Pi_{im,j}^2 |\Delta_j(\gamma_j^\top Z_{ij})|^2 \right)^{1/2} \\
    & \quad+ \frac{1}{K_{n,j}} \left( \sum_{i=1}^n \sum_{m\neq i} \Pi_{im,j}^2 |\Delta_j(\gamma_j^\top Z_{ij})|^2\right)^{1/2}\left( \sum_{i=1}^n \sum_{m\neq i}\Pi_{im,j}^2 |\Delta_j(\gamma_j^\top Z_{mj})|^2 \right)^{1/2} \\
    & \lesssim \left\{ {K_{n,j}^{-1}}\Delta_j^2 |Z_j\gamma_{n,j}|^2 \right\}^{1/2} + K_{n,j}^{-1}\Delta_j^2|Z_j\gamma_{n,j}|^2,
\end{align*}
where the final inequality follows from
\begin{align*}
    & \sum_{i=1}^n \sum_{m\neq i}\Pi_{im,j}^2 |\Delta_j(\gamma_j^\top Z_{mj})|^2 \\
    & = \sum_{m=1}^n \Delta_j^2|(\gamma_j^\top Z_{mj})|^2 \sum_{i\neq m} \Pi_{im,j}^2 \lesssim \sum_{m=1}^n \Delta_j^2|(\gamma_j^\top Z_{mj})|^2 = \Delta_j^2|Z_j\gamma_{n,j}|^2.
\end{align*}
Since ${K_{n,j}^{-1}} \Delta_j^2|Z_j\gamma_{n,j}|^2 \to 0$, for all $j\in[p]$, we have
\begin{align}
    \sum_{b=1}^3 \mathbb{E}[ J_b(\tilde\psi_{l,j}^{[b],\Delta}) \mid Z] \lesssim  K_{n,j}^{-1/2}\Delta_j|Z_j\gamma_{n,j}| . \label{eq:mikusheva-sun-proof-alt-1}
\end{align}
Since for all $b=1,2,3$ and $q=1,\dots, b$, $J_q(\pi_q\tilde\psi^{[b], \Delta}_j)$ are degenerate $U$-statistics, \cref{thm:max-is-inid} gives
\begin{align*}
    & \mathbb{E}\left[ \max_{j\in[p]} |J_q(\pi_q\tilde\psi^{[b], \Delta}_j)| \mid Z \right] \\
    & \lesssim \max_{0 \le h \le q} \log^{(q+h)/2}(np) \mathbb{E}\left[ \max_{j\in[p]} \max_{(i_1\dots, i_h)\in I_{n,h}} \sum_{(k_1,\dots, k_{q-h})} \mathbb{E}\left[ (\pi_q\tilde\psi^{[b], \Delta}_{j})^2 \mid u_{i_1j}, \dots, u_{i_hj}, Z\right] \mid Z\right].
\end{align*}
The coefficients appearing in $\pi_1\psi^{[1], \Delta}_j$ are bounded by a constant multiple of
\begin{align*}
    (\Phi^{\Delta}.I)_{l,j} & \coloneqq K_{n,j}^{-1} \sum_{m\neq l} |\Pi_{lm,j}|^3 |\Delta_j(\gamma_j^\top Z_{mj})| \vee K_{n,j}^{-1} \sum_{i\neq l} |\Pi_{il,j}|^3 |\Delta_j(\gamma_j^\top Z_{ij})| \\
    & \qquad\vee K_{n,j}^{-1} \sum_{i=1}^n \sum_{m\neq i} \Pi_{im,j}^2| |M_{il,j} M_{ml,j}|\Delta_j^2(\gamma_j^\top Z_{ij})(\gamma_j^\top Z_{mj})|.
\end{align*}
Therefore, from \cref{thm:max-is-inid}, we have
\begin{align}
    & \mathbb{E}\left[ \max_{j\in[p]} |J_1(\pi_1\tilde\psi^{[1], \Delta}_j)| \mid Z \right] \nonumber\\
    & \lesssim \log(np) \left( \max_{j\in[p]} \sum_{l=1}^n (\Phi^{\Delta}.I)_{l,j}^2 + \max_{j\in[p]} \max_{l\in[n]} (\Phi^{\Delta}.I)_{l,j}^2 \right)^{1/2} \nonumber\\
    & \lesssim \log(np) \max_{j\in[p]} \frac{\Delta_j|Z_j\gamma_{n,j}|}{\sqrt{K_{n,j}}} , \label{eq:mikusheva-sun-proof-alt-2}
\end{align}
where the final inequality uses $\sum_{m=1}^n \Pi_{im,j}^2 \lesssim 1$, $\sum_{i=1}^n \sum_{m=1}^n \Pi_{im,j}^4 \lesssim K_{n,j}$ and $\sum_{l=1}^n |M_{il,j} M_{ml,j}| \le 1$ with Cauchy-Schwarz inequality.

Regarding the terms associated with $\tilde\psi_j^{[2],\Delta}$, the coefficients appearing in $\pi_1\psi_j^{[2], \Delta}$ are bounded by a constant multiple of
\begin{align*}
    (\Phi^{\Delta}.II.I)_{l,j} \coloneqq K_{n,j}^{-1}\sum_{m \neq l} |\Pi_{lm,j}|^3|\Delta_j (\gamma_j^\top Z_{mj})| \vee  K_{n,j}^{-1} \sum_{i\neq l} |\Pi_{il,j}|^3|\Delta_j(\gamma_j^\top Z_{ij})|.
\end{align*}
Also, for $(l,r)\in I_{n,2}$, define
\begin{align*}
    (\Phi^{\Delta}.II.II)_{(l,r),j} 
    & \coloneqq K_{n,j}^{-1} \sum_{m\neq l} \Pi_{lm,j}^2|\Delta_j(\gamma_j^\top Z_{mj})| (|M_{ll,j}M_{mr,j}| + |M_{lr,j}M_{ml,j}| + |M_{lr,j}M_{mr,j}|) \\
    & \quad \vee K_{n,j}^{-1}\sum_{i\neq l} \Pi_{il,j}^2|\Delta_j(\gamma_j^\top Z_{ij})| (|M_{il,j}M_{lr,j}| + |M_{ir,j}M_{ll,j}| + |M_{ir,j}M_{lr,j}|) \\
    & \quad \vee K_{n,j}^{-1}\sum_{i=1}^n\sum_{m\neq i} \Pi_{im,j}^2 |M_{il,j}M_{mr,j}||\Delta_j^2(\gamma_j^\top Z_{ij})(\gamma_j^\top Z_{mj})|.
\end{align*}
Then, the coefficients appearing in $\pi_2\psi_j^{[2], \Delta}$ are bouunded by a constant multiple of $(\Phi^{\Delta}.II)_{(l,r),j}$.
Therefore, from \cref{thm:max-is-inid}, we have
\begin{align}
    & \mathbb{E}\left[ \max_{j\in[p]} |J_1(\pi_1\tilde\psi^{[2], \Delta}_j)| \mid Z \right] + \mathbb{E}\left[ \max_{j\in[p]} |J_2(\pi_2\tilde\psi^{[2], \Delta}_j)| \mid Z \right] \nonumber\\
    & \lesssim \log(np) \max_{j\in[p]} \left( \sum_{l=1}^n   (\Phi^{\Delta}.II.I)_{l,j}^2 +  \max_{l\in[n]}    {(\Phi^{\Delta}.II.I)}_{l,j}^2 \right)^{1/2} \nonumber\\
    & ~~ +  \log^2(np) \max_{j\in[p]} \left( \sum_{(l,r)\in I_{n,2}}  \widetilde{(\Phi^{\Delta}.II.II)}_{(l,r),j}^2 + \max_{l\in[n]} \sum_{r\neq l} \widetilde{(\Phi^{\Delta}.II.II)}_{(l,r),j}^2  +  \max_{(l,r)\in I_{n,2}} \widetilde{(\Phi^{\Delta}.II.II)}_{(l,r),j}^2 \right)^{1/2} \nonumber\\
    & \lesssim \log^2(np) \max_{j\in[p]} \frac{\Delta_j|Z_j\gamma_{n,j}|}{\sqrt{K_{n,j}}} , \label{eq:mikusheva-sun-proof-alt-3}
\end{align}
where the final inequality follows from $\sum_{m=1}^n \Pi_{im,j}^2 \lesssim 1$, $\sum_{i=1}^n \sum_{m=1}^n \Pi_{im,j}^4 \lesssim K_{n,j}$ and $\sum_{l=1}^n |M_{il,j} M_{ml,j}| \le 1$ with Cauchy-Schwarz inequality. 

Regarding the terms associated with $\tilde\psi_j^{[3],\Delta}$, we have $\pi_q\psi_j^{[3],\Delta} = 0$ for $q=1,2$ since $\mathbb{E}[\eta_{i,j} \mid Z] =0$. 
Also, for $(i,l,r)\in I_{n,3}$, define
\begin{align*}
    & (\Phi^\Delta.III)_{(i,l,r),a} \\
    & \coloneqq K_{n,j}^{-1} \sum_{m\neq i}\Pi_{im,j}^2 |M_{il,j} M_{mr,j}| |\Delta_j(\gamma_j^\top Z_{m,j})|  \vee  K_{n,j}^{-1} \sum_{h\neq i}\Pi_{hi,a}^2 |M_{hl,a} M_{ir,j}| |\Delta_j(\gamma_j^\top Z_{h,a})|.
\end{align*}
Them the coefficients appearing in  $\tilde\psi_j^{[3],\Delta}$ are bounded by a constant multiple of $(\Phi^\Delta.III)_{(i,l,r),a}$.
Therefore, from \cref{thm:max-is-inid}, we have
\begin{align}
    &  \mathbb{E}\left[ \max_{j\in[p]} |J_3(\tilde\psi^{[3], \Delta}_j)| \mid Z \right]  \nonumber\\
    & \lesssim \log^3(np) \max_{j\in[p]} \Bigg( \sum_{(i,l,r)\in I_{n,3}}\widetilde{(\Phi^\Delta.III)}_{(i,l,r),a}^2 + \max_{i\in[n]} \sum_{\substack{(l,r)\in I_{n,2} \\ l,r\neq i}} \widetilde{(\Phi^\Delta.III)}_{(i,l,r),a}^2 \nonumber\\
    & \quad  + \max_{(i,l)\in I_{n,2}} \sum_{r\neq i,l} \widetilde{(\Phi^\Delta.III)}_{(i,l,r),a}^2 + \max_{(i,l,r)\in I_{n,3}} \widetilde{(\Phi^\Delta.III)}_{(i,l,r),a}^2 \Bigg)^{1/2} \nonumber\\
    & \lesssim \log^3(np) \max_{j\in[p]} \frac{\Delta_j|Z_j\gamma_{n,j}|}{\sqrt{K_{n,j}}} , \label{eq:mikusheva-sun-proof-alt-4}
\end{align}
where the final inequality uses $\sum_{m=1}^n \Pi_{im,j}^2 \lesssim 1$, $\sum_{l=1}^nM_{il,j}^2 \le 1$ with Cauchy-Schwarz inequality.

From \cref{eq:mikusheva-sun-proof-alt-1}-\cref{eq:mikusheva-sun-proof-alt-4}, we have
\begin{align*}
    \mathbb{E}\left[ \max_{j\in[p]} |\hat{\Phi}_j - \tilde{\Phi}_j|  \mid Z\right] \lesssim \log^3(np) \max_{j\in[p]} \frac{\Delta_j|Z_j\gamma_{n,j}|}{\sqrt{K_{n,j}}}.
\end{align*}
Combined with $\mathbb{E}\left[ \max_{j\in[p]} |\tilde{\Phi}_j - \Phi_j |  \mid Z \right] \lesssim  \max_{j\in[p]} K_{n,j}^{-1/2}\log^{4}(np)$ and \cref{as:mikusheva-sun}(ii), which implies $ \min_{j\in[p]}\Phi_j \asymp 1$, we have
\begin{align*}
    \mathbb{E}\left[ \max_{j\in[p]} |\hat{\Phi}_j /\Phi_j -1|  \mid Z \right] \lesssim \max_{j\in[p]} \frac{\Delta_j|Z_j\gamma_{n,j}|}{\sqrt{K_{n,j}}}\log^3(np)  + \max_{j\in[p]}  \frac{1}{\sqrt{K_{n,j}}}\log^4(np).
\end{align*}
\end{proof}

\subsection{Proof of \cref{thm:mikusheva-sun-null}}
\begin{proof}[Proof of \cref{thm:mikusheva-sun-null}]
Observe that
\begin{align*}
    & \sup_{R\in\mathcal{R}_p} \left| \mathbb{P}\left( \hat{AR} \in R \mid Z\right) - \mathbb{P}\left( G \in R \mid Z\right)  \right| \\
    & \le\sup_{R\in\mathcal{R}_p}  \left| \mathbb{P}\left({AR} \in R \mid Z \right) - \mathbb{P}\left(  G \in R \mid Z \right)  \right| \\
    & \qquad\qquad + \sup_{R\in\mathcal{R}_p} \left| \mathbb{P}\left( \hat{AR} \in R\mid Z \right) - \mathbb{P}\left(  {AR} \in R \mid Z\right)  \right|   \eqqcolon (\rho.I) + (\rho.II).
\end{align*}

First, we show the convergence of $(\rho.I)$.
Observe that $AR_j$ is a second-order $U$-statistic with kernel $ \psi_{j,(i,m)} \coloneqq 2(\sqrt{K_{n,j}}\sqrt{\Phi_j})^{-1}  \Pi_{im,j} u_{i,j}u_{m,j}$.
Since $\mathbb{E}[u_{i,j}\mid Z] = 0$ and $\mathbb{E}[u_{m,j}\mid Z] = 0$ under the null hypothesis, $AR_j$ is degenerate for all $j\in[p]$. 
Therefore, in view of \cref{rmk:varification-U-main}, the Gaussian approximation follows if $\Delta_1^{(0)}\log^6(np)\to0$ and $\Delta_{2,q}(2)\log^5(np)\to0$ hold.

In what follows we use the facts that  $\Phi_j \asymp 1$ and
\begin{align*}
    \Var(AR_j \mid Z) & = \frac{1}{K_{n,j}\Phi_j} \sum_{i=1}^n \sum_{m\neq i} \sum_{l=1}^n \sum_{r\neq l}\Pi_{im,j} \Pi_{lr,j} \mathbb{E}[u_{i,j}u_{m,j}u_{l,j}u_{r,j} \mid Z] \\
    & = \frac{2}{K_{n,j}\Phi_j} \sum_{i=1}^n \sum_{m\neq i} \Pi_{im,j}^2 \mathbb{E}[u_{i,j}^2 \mid Z] \mathbb{E}[u_{m,j}^2 \mid Z] = \Phi_j^{-1} \Phi_j = 1.
\end{align*}

First, we show $\Delta_1^{(0)}\log^6(np)\to0$. 
Since
\begin{align*}
     \psi_{j_1,(i,m)}\star_i^1 \psi_{j_2,(i,l)} = \mathbb{E}\left[  4\frac{ \Pi_{im,j_1} u_{i,j_1}u_{m,j_1}}{\sqrt{K_{n,j_1}}\sqrt{\Phi_{j_1}}} \frac{ \Pi_{il,j_2} u_{i,j_2}u_{l,j_2}}{\sqrt{K_{n,j_2}}\sqrt{\Phi_{j_2}}}   \mid Z, u_m, u_l \right],
\end{align*}
we have
\begin{align*}
    & \Delta_1^{(0)}\log^6(np) \\
    & = {\max_{(j_1,j_2)\in[p]^2}} \sum_{m=1}^n\sum_{l\neq m}  \left\| \sum_{i\neq m,l}\psi_{j_1,(i,m)}\star_i^1 \psi_{j_2,(i,l)}\right\|_{L^2(P_m\otimes P_l\mid Z)}^2\\
    & = {\max_{(j_1,j_2)\in[p]^2}} \sum_{m=1}^n\sum_{l\neq m} \mathbb{E}\left[\left(\sum_{i\neq m,l} \mathbb{E}\left[ 4 \frac{ \Pi_{im,j_1} u_{i,j_1}u_{m,j_1}}{\sqrt{K_{n,j_1}}\sqrt{\Phi_{j_1}}} \frac{ \Pi_{il,j_2} u_{i,j_2}u_{l,j_2}}{\sqrt{K_{n,j_2}}\sqrt{\Phi_{j_2}}}   \mid Z, u_m, u_l \right] \right)^2 \mid Z\right] \log^6(np) \\
    & \lesssim {\max_{(j_1,j_2)\in[p]^2}} \sum_{m=1}^n\sum_{l=1}^n \frac{1}{K_{n,j_1} K_{n,j_2}} \left( \sum_{i=1}^n \Pi_{im,j_1} \Pi_{il,j_2} \mathbb{E}[u_{i,j_1} u_{i,j_2} \mid Z]\right)^2\log^6(np)  \\
    & ~~  + {\max_{(j_1,j_2)\in[p]^2}} \sum_{m=1}^n\sum_{l=1}^n \frac{(\mathbb{E}[u_{m,j_1}u_{m,j_2} \mid Z])^2 }{K_{n,j_1} K_{n,j_2}} \Pi_{mm,j_1}^2\Pi_{ml,j_2}^2 \log^6(np)  \\
    & ~~  + {\max_{(j_1,j_2)\in[p]^2}} \sum_{m=1}^n\sum_{l=1}^n \frac{(\mathbb{E}[u_{l,j_1}u_{l,j_2} \mid Z])^2 }{K_{n,j_1} K_{n,j_2}} \Pi_{lm,j_1}^2\Pi_{ll,j_2}^2 \log^6(np).
\end{align*}
In term of the first term, it follows that
\begin{align*}
    & {\max_{(j_1,j_2)\in[p]^2}} \sum_{m=1}^n\sum_{l=1}^n \frac{1}{K_{n,j_1} K_{n,j_2}} \left( \sum_{i=1}^n \Pi_{im,j_1} \Pi_{il,j_2} \mathbb{E}[u_{i,j_1} u_{i,j_2} \mid Z]\right)^2\log^6(np) \\
    & =  {\max_{(j_1,j_2)\in[p]^2}} \sum_{i=1}^n\sum_{i'=1}^n \frac{ \mathbb{E}[u_{i,j_1} u_{i,j_2} \mid Z] \mathbb{E}[u_{i',j_1} u_{i',j_2} \mid Z]}{K_{n,j_1} K_{n,j_2}} \left(\sum_{m=1}^n\Pi_{im,j_1} \Pi_{i'm,j_1} \right)\left( \sum_{l=1}^n\Pi_{il,j_2} \Pi_{i'l,j_2} \right) \log^6(np) \\
    & = {\max_{(j_1,j_2)\in[p]^2}} \sum_{i=1}^n\sum_{i'=1}^n \frac{ \mathbb{E}[u_{i,j_1} u_{i,j_2} \mid Z] \mathbb{E}[u_{i',j_1} u_{i',j_2} \mid Z]}{K_{n,j_1} K_{n,j_2}} \Pi_{ii',j_1}\Pi_{ii',j_2}\log^6(np) \\
    & \lesssim  {\max_{(j_1,j_2)\in[p]^2}} \frac{1}{\sqrt{K_{n,j_1} K_{n,j_2}}} \log^6(np),
\end{align*}
where the final line follows from $|\mathbb{E}[u_{i',j_1} u_{i',j_2} \mid Z]|\lesssim 1$ and Cauchy-Schwarz inequality. In terms of the second term, from $|\mathbb{E}[u_{i',j_1} u_{i',j_2} \mid Z]|\lesssim 1$, $\sum_{l=1}^n \Pi_{ml,j}^2 = \Pi_{mm,j}$ and $0 \le \Pi_{mm,j}\le 1$, it follows that
\begin{align*}
     & {\max_{(j_1,j_2)\in[p]^2}} \sum_{m=1}^n\sum_{l=1}^n \frac{(\mathbb{E}[u_{m,j_1}u_{m,j_2} \mid Z])^2 }{K_{n,j_1} K_{n,j_2}} \Pi_{mm,j_1}^2\Pi_{ml,j_2}^2 \log^6(np)   \\
     & \lesssim {\max_{(j_1,j_2)\in[p]^2}} \sum_{m=1}^n \frac{1}{K_{n,j_1} K_{n,j_2}}  \Pi_{mm,j_1}^2\Pi_{mm,j_2} \log^6(np)  \\
     & \le {\max_{(j_1,j_2)\in[p]^2}}  \frac{1}{\sqrt{K_{n,j_1} K_{n,j_2}}} \log^6(np) .
\end{align*}
Similar evaluation also holds for the third term.
Therefore
\begin{align*}
    \tilde \Delta_1^{(0)}\log^6(np) \lesssim {\max_{(j_1,j_2)\in[p]^2}}  \frac{1}{\sqrt{K_{n,j_1} K_{n,j_2}}} \log^6(np) \lesssim \max_{j\in[p]} \frac{1}{K_{n,j}}\log^6(np).
\end{align*}

Also, since 
\begin{align*}
    & \|\psi_{j,(i,m)}\|_{L^4(P_i\otimes P_m \mid Z)}^4 \lesssim  \frac{\Pi_{im,j}^4}{K_{n,j}^2}, \quad  P_m\psi_{j,(i,m)}^4 \lesssim \frac{\Pi_{im,j}^4}{K_{n,j}^2}, \quad  P_m\psi_{j,(i,m)}^2 \lesssim \frac{\Pi_{im,j}^2}{K_{n,j}}, \quad |\psi_{j,(i,m)}| \lesssim \frac{|\Pi_{im,j}|}{\sqrt{K_{n,j}}},
\end{align*}
and
\begin{align*}
    & \sum_{i=1}^n\sum_{m\neq i} \Pi_{im,j}^4 \le K_{n,j}, ~~ \sum_{i=1}^n \left(\sum_{m\neq i} \Pi_{im,j}^2 \right)^2 \le K_{n,j}, \\
    & \max_{i\in[n]} \sum_{m\neq i}\Pi_{im,j}^4 \le \max_{i\in[n]} \left( \sum_{m=1}^n \Pi_{im,j}^2 \right)^2  = \max_{i\in[n]} \Pi_{ii,j}^2 \lesssim 1.
\end{align*}
we have
\begin{align*}
    \tilde\Delta_{2,*}^{(1)}(2) \log^5(np)
    &\coloneqq  \max_{j\in[p]} \sum_{i=1}^n\sum_{m\neq i} \|\psi_{j,(i,m)}\|_{L^4(P_i\otimes P_m \mid Z)}^4 \log^3 (np) \\
    & \le  \max_{j\in[p]} \sum_{i=1}^n\sum_{m\neq i} \frac{\Pi_{im,j}^4}{K_{n,j}^2} \log^8 (np) \le \max_{j\in[p]} \frac{1}{K_{n,j}} \log^8 (np),
\end{align*}
and
\begin{align*}
    \tilde\Delta_{2,*}^{(2)}(2) \log^5(np)
    &\coloneqq \max_{j\in [p]} \sum_{i=1}^n \left\|\sum_{m\neq i}P_m\psi_{j,(i,m)}^2\right\|^2_{L^2(P_i \mid Z)}\log^7 (np) \\
    & \lesssim  \max_{j\in [p] } \sum_{i=1}^n \left( \sum_{m\neq i} \frac{\Pi_{im,j}^2}{K_{n,j}} \right)^2 \log^7 (np) \le \max_{j\in[p]} \frac{1}{K_{n,j}} \log^7 (np),
\end{align*}
and
\begin{align*}
    \tilde\Delta_{2,*}^{(3)}(2) \log^5(np)
    & \coloneqq \E\left[  \max_{j\in[p]}\max_{i\in[n]}  \sum_{m\neq i}P_m\psi_{j,(i,m)}^4 \mid Z \right] \log^9 (np) \\
    & \lesssim   \max_{j\in[p] }\max_{i\in[n]}  \sum_{m\neq i}\frac{\Pi_{im,j}^4}{K_{n,j}^2} \log^9 (np) \le  \max_{j\in [p] } \frac{1}{K_{n,j}^2} \log^9 (np), 
\end{align*}
and
\begin{align*}
     \tilde\Delta_{2,\infty}^{(4)}(2) \log^5(np)
    & = \left\| \max_{j\in[p]} \max_{(i,m)\in I_{n,2}}  |\psi_{a,(i,m)}|\right\|_{L^\infty(\pr\mid Z)}^4 \log^{10}(np) \\
    & \lesssim \max_{j\in[p]} \max_{(i,m)\in I_{n,2}}  \frac{|\Pi_{im,j}|^4}{K_{n,j}^2} \log^{10}(np) \\
    & \le \max_{j\in[p]} \max_{(i,m)\in I_{n,2}}  \frac{|\Pi_{ii,j}\Pi_{mm,j}|^2}{K_{n,j}^2}\log^{10}(np) \\
    & \lesssim \max_{j\in[p]} \max_{(i,m)\in I_{n,2}}  \frac{1}{K_{n,j}^2}\log^{10}(np),
\end{align*}
and 
\begin{align*}
      \tilde\Delta_{2,\infty}^{(5)}(2)\log^5(np) 
    & =  \left\| \max_{j\in[p]} \max_{i\in[n]}  \sum_{m\neq i}P_m\psi_{j,(i,m)}^2  \right\|_{L^{\infty}(\pr \mid Z)}^2 \log^8 (np) \\
    & \lesssim \max_{j\in[p]} \max_{i\in[n]} \left( \sum_{m\neq i} \frac{\Pi_{im,j}^2}{K_{n,j}} \right)^2 \log^8 (np) \lesssim \max_{j\in[p]} \frac{1}{K_{n,j}^2}\log^8 (np).
 \end{align*}
Therefore, under \cref{as:mikusheva-sun}, 
\begin{align*}
    \tilde \Delta_1^{(0)}\log^6(np)\to0, \quad \tilde \Delta_{2,q}(2)\log^5(np)\to0,
\end{align*}
and, in view of \cref{coro:main-kernel-inid}, this implies that $(\rho.I)\to0$.

Next, we show the convergence of $(\rho.II)$.
Observe that
\begin{align*}
    & \max_{j\in[p]}\left|\left( \frac{1}{\hat{\Phi}_j^{1/2}} - \frac{1}{\Phi_j^{1/2}} \right)\frac{1}{\sqrt{K_{n,j}}} \sum_{i=1}^n \sum_{m\neq i} \Pi_{im,j}^2 u_{i,j} u_{m,j} \right| \\
    & \le \max_{j\in[p]} |AR_j| \max_{j\in[p]} \left| \left(\frac{{\Phi}_j}{\hat \Phi_j}\right)^{1/2} - 1\right| \lesssim \max_{j\in[p]} |AR_j| \max_{j\in[p]} \left| \frac{{\Phi}_j}{\hat \Phi_j} - 1\right| .
\end{align*}
Since we have shown $(\rho.I) \to 0$, we have the Gaussian approximation of $\max_{j\in[p]} |AR_j|$ to $\max_{j\in[p]} |G_j|$.
Also, since we can see that $\max_{j\in[p]} |G_j| = O_p(\sqrt{\log p})$ from Lemma 2.2.10 in \citealp{van1996weak}, we have $\max_{j\in[p]} |AR_j| = O_p(\sqrt{\log p})$. 
In conjunction with the evaluation of $\max_{j\in [p]}  \left| \hat{\Phi}_j / \Phi_j -1\right| = O_p\left( \log^4(np) \max_{j\in [p]} K_{n,j}^{-1/2} \right)$ and $\Phi_j \asymp 1$ for all $j\in[p]$, \cref{as:mikusheva-sun} implies that 
\begin{align*}
    \max_{j\in[p]}\left|\left( \frac{1}{\hat{\Phi}_j^{1/2}} - \frac{1}{\Phi_j^{1/2}} \right)\frac{1}{\sqrt{K_{n,j}}} \sum_{i=1}^n \sum_{m\neq i} \Pi_{im,j}^2 u_{i,j} u_{m,j} \right| \lesssim \max_{j\in[p]} \frac{\log^4(np)\sqrt{\log p}}{\sqrt{K_{n,j}}}  = o_p\left( \frac{1}{\sqrt{\log(np)}} \right),
\end{align*}
so Lemma 1 in \cite{chernozhukov2023high} completes the proof.
\end{proof}

\subsection{Proof of \cref{thm:mikusheva-sun-alt}}
\begin{proof}[Proof of \cref{thm:mikusheva-sun-alt}]
As in the case of proof of \cref{thm:mikusheva-sun-null}, it holds that
\begin{align*}
    &  \sup_{\substack{ (\Delta_1, \dots, \Delta_p) \in \mathbb{R}^{p} \\ \max_{j\in[p]} |\Delta_j|^2 \le c}}\sup_{A\in\mathcal{A}_p} \left| \mathbb{P}\left( \hat{AR} \in A \mid Z\right) - \mathbb{P}\left( G^\Delta + \lambda^\Delta \in A \mid Z\right)  \right| \\
    & =  \sup_{\substack{ (\Delta_1, \dots, \Delta_p) \in \mathbb{R}^{p} \\ \max_{j\in[p]} |\Delta_j|^2 \le c}}\sup_{A\in\mathcal{A}_p} \left| \mathbb{P}\left( \hat{AR}- \lambda^\Delta \in A \mid Z\right) - \mathbb{P}\left( G^\Delta  \in A \mid Z\right)  \right|\\
    & \le \sup_{\substack{ (\Delta_1, \dots, \Delta_p) \in \mathbb{R}^{p} \\ \max_{j\in[p]} |\Delta_j|^2 \le c}}\sup_{A\in\mathcal{A}_p}  \left| \mathbb{P}\left({AR} - \lambda^\Delta \in A \mid Z \right) - \mathbb{P}\left(  G^\Delta \in A \mid Z \right)  \right| \\
    & \qquad\qquad +  \sup_{\substack{ (\Delta_1, \dots, \Delta_p) \in \mathbb{R}^{p} \\ \max_{j\in[p]} |\Delta_j|^2 \le c}}\sup_{A\in\mathcal{A}_p} \left| \mathbb{P}\left( \hat{AR} - \lambda^\Delta \in A \mid Z \right) - \mathbb{P}\left(  {AR} - \lambda^\Delta \in A \mid Z\right)  \right| \\
    & \eqqcolon (\rho^\Delta.I) + (\rho^\Delta.II).
\end{align*}

We first confirm the convergence of $(\rho^\Delta.I) \to 0$.
As seen in the proof of \cref{thm:mikusheva-sun-variance}(ii), under the alternative hypothesis $\theta_j = \theta_{0,j} + \Delta_j$, it holds that $u_{ij} = \eta_{i,j} + \Delta_{j}(\gamma_{n,j}^\top Z_{ij})$,
so
\begin{align*}
    AR_j 
    &= \frac{1}{\sqrt{K_{n,j}}\sqrt{\Phi_j}} \sum_{i=1}^n\sum_{m\neq i} \Pi_{im,j} u_{ij}u_{mj} \\
    & = \frac{\Delta_j^2}{\sqrt{K_{n,j}}\sqrt{\Phi_j}} \sum_{i=1}^n\sum_{m\neq i} \Pi_{im,j} (\gamma_{n,j}^\top Z_{ij})(\gamma_{n,j}^\top Z_{mj})\\
    & \quad + \frac{\Delta_j}{\sqrt{K_{n,j}}\sqrt{\Phi_j}} \sum_{i=1}^n\sum_{m\neq i} \Pi_{im,j} \{\eta_{i,j}(\gamma_{n,j}^\top Z_{mj}) + \eta_{m,j}(\gamma_{n,j}^\top Z_{ij})\} \\
    & \quad  +  \frac{1}{\sqrt{K_{n,j}}\sqrt{\Phi_j}} \sum_{i=1}^n\sum_{m\neq i} \Pi_{im,j} \eta_{i,j} \eta_{m,j} \\
    & \coloneqq \lambda_j^\Delta + L_j^\Delta + Q_j^\Delta,
\end{align*}
where the second equality defines
\begin{align*}
    & \lambda_j^\Delta \coloneqq \frac{\Delta_j^2}{\sqrt{K_{n,j}}\sqrt{\Phi_j}} \sum_{i=1}^n\sum_{m\neq i} \Pi_{im,j} (\gamma_{n,j}^\top Z_{ij})(\gamma_{n,j}^\top Z_{mj}) =  \frac{\Delta_j^2\varsigma^2_j}{\sqrt{K_{n,j}}\sqrt{\Phi_j}}, \\
    & L_j^\Delta \coloneqq \frac{2\Delta_j}{\sqrt{K_{n,j}}\sqrt{\Phi_j}} \sum_{i=1}^n\sum_{m\neq i} \Pi_{im,j} \eta_{i,j}(\gamma_{n,j}^\top Z_{mj}), \quad  Q_j^\Delta \coloneqq \frac{1}{\sqrt{K_{n,j}}\sqrt{\Phi_j}} \sum_{i=1}^n\sum_{m\neq i} \Pi_{im,j} \eta_{i,j} \eta_{m,j}.
\end{align*}
We first show that the linear term is ignorable. 
In terms of  $ L_j^\Delta$, note that
\begin{align*}
     L_j^\Delta & = \frac{2\Delta_j}{\sqrt{K_{n,j}}\sqrt{\Phi_j}} \sum_{i=1}^n \eta_{i,j}\left( \sum_{m\neq i} \Pi_{im,j}(\gamma_{n,j}^\top Z_{mj})\right)\\
     & = \frac{2\Delta_j}{\sqrt{K_{n,j}}\sqrt{\Phi_j}} \sum_{i=1}^n \eta_{i,j}\left( \sum_{m=1}^n  \Pi_{im,j}(\gamma_{n,j}^\top Z_{mj}) - \Pi_{ii,j}(\gamma^\top Z_{ij})\right) \\
     & = \frac{2\Delta_j}{\sqrt{K_{n,j}}\sqrt{\Phi_j}} \sum_{i=1}^n  (1 - \Pi_{ii,j})(\gamma_{n,j}^\top Z_{ij})\eta_{i,j},
\end{align*}
and letting $L_{ij} \coloneqq 2\Delta_j\{\sqrt{K_{n,j}}\sqrt{\Phi_j}\}^{-1}  (1 - \Pi_{ii,j})(\gamma_{n,j}^\top Z_{ij})\eta_{i,j}$, \cref{thm:max-is-inid} gives
\begin{align*}
    \mathbb{E}\left[ \max_{j\in[p]} |L_j^\Delta| \mid Z\right] 
    \lesssim \log^{1/2}(np) \left( \max_{j\in[p]} \sum_{i=1}^n \mathbb{E}[L_{ij}^2 \mid Z] \right)^{1/2} + \log(np) \mathbb{E} \left[\max_{j\in[p]} \max_{i\in[n]} L_{ij}^2 \mid Z\right]^{1/2}.
\end{align*}
In terms of the first term, we can see that
\begin{align*}
    \sum_{i=1}^n \mathbb{E}[L_{ij}^2 \mid Z]  = \frac{4\Delta_j^2}{K_{n,j} \Phi_j} \sum_{i=1}^n (1-\Pi_{ii,j})^2(\gamma_{n,j}^\top Z_{ij})^2 \mathbb{E}[\eta_{ij}^2 \mid Z] \lesssim \frac{\Delta_j^2 |Z_j \gamma_{n,j}|^2}{K_{n,j}},
\end{align*}
and, in terms of the second term, from the boundedness of $\eta_{ij}$, we have
\begin{align*}
    \mathbb{E} \left[\max_{j\in[p]} \max_{i\in[n]} L_{ij}^2 \mid Z\right] \lesssim \max_{j\in[p]} \max_{i\in[n]}\frac{\Delta_j^2(\gamma_{n,j}^\top Z_{ij})^2}{K_{n,j}} \le \max_{j\in[p]}\frac{\Delta_j^2 |Z_j \gamma_{n,j}|^2}{K_{n,j}} .
\end{align*}
Therefore
\begin{align*}
    \mathbb{E}\left[ \max_{j\in[p]} |L_j^\Delta| \mid Z\right] 
    \lesssim \log(np) \left(\max_{j\in[p]}\frac{\Delta_j^2 |Z_j \gamma_{n,j}|^2}{K_{n,j}}\right)^{1/2}.
\end{align*}
Since $\max_{j\in[p]}\Delta_j^2K_{n,j}^{-1} |Z_j \gamma_{n,j}|^2 \log^3(np) \to 0$, it holds that
\begin{align}
    \mathbb{E}\left[ \max_{j\in[p]} |L_j^\Delta| \mid Z\right]  = o\left( \frac{1}{\sqrt{\log (np)}} \right) \implies \max_{j\in[p]} |L_j^\Delta| = o_p\left( \frac{1}{\sqrt{\log (np)}} \right). \label{eq:ms-alt-linear-bound}
\end{align}
Therefore, in conjunction with Lemma 1 in \cite{chernozhukov2023high}, we have
\begin{align*}
    (\rho^\Delta. I) & = \sup_{\substack{ (\Delta_1, \dots, \Delta_p) \in \mathbb{R}^{p} \\ \max_{j\in[p]} |\Delta_j|^2 \le c}}\sup_{R\in\mathcal{R}_p}  \left| \mathbb{P}\left({AR} - \lambda^\Delta \in R \mid Z \right) - \mathbb{P}\left(  G^\Delta \in R \mid Z \right)  \right|  \\
    & \le \sup_{\substack{ (\Delta_1, \dots, \Delta_p) \in \mathbb{R}^{p} \\ \max_{j\in[p]} |\Delta_j|^2 \le c}}\sup_{R\in\mathcal{R}_p}  \left| \mathbb{P}\left( Q^\Delta \in R \mid Z \right) - \mathbb{P}\left(  G^\Delta \in R \mid Z \right)  \right| + o_p(1).
\end{align*}
The right-hand side converges to zero by the same argument as in the proof of the Gaussian approximation under the null hypothesis (proof of \cref{thm:mikusheva-sun-null}), with $u_{ij}$ replaced by $\eta_{ij}$, so we have $(\rho^\Delta. I) \to 0$.

Finally, we evaluate $(\rho^\Delta.II)$.
Observe that,
\begin{align*}
\max_{j\in[p]}|\widehat{AR}_j-AR_j|
    &= \max_{j\in[p]} \left| \left( \sqrt{\frac{\Phi_j}{\widehat{\Phi}_j}}-1 \right) \bigl(AR_j-\lambda_j^\Delta+\lambda_j^\Delta\bigr) \right|\\
    &\lesssim \max_{j\in[p]}|AR_j-\lambda_j^\Delta| \max_{j\in[p]} \left|\frac{\widehat{\Phi}_j}{\Phi_j}-1\right|
    + \max_{j\in[p]}  |\lambda_j^\Delta| \max_{j\in[p]}  \left|\frac{\widehat{\Phi}_j}{\Phi_j}-1\right|.
\end{align*}
Since we have shown $(\rho^\Delta.I) \to 0$, we have the Gaussian approximation of $\max_{j\in[p]} |AR_j - \lambda_j^\Delta|$ to $\max_{j\in[p]} |G_j^\Delta|$.
Also, since we can see that $\max_{j\in[p]} |G_j^\Delta| = O_p(\sqrt{\log p})$ from Lemma 2.2.10 in \cite{van1996weak}, we have $\max_{j\in[p]} |AR_j - \lambda^\Delta| = O_p(\sqrt{\log p})$. 
In conjunction with the evaluation of 
\begin{align*}
     \max_{j\in[p]}  \left|\frac{\hat{\Phi}_j}{\Phi_j} - 1 \right| = O_p\left(\max_{j\in[p]} \frac{\Delta_j|Z_j\gamma_{n,j}|}{\sqrt{K_{n,j}}}\log^3(np)  + \max_{j\in[p]}  \frac{1}{\sqrt{K_{n,j}}}\log^4(np)\right)
\end{align*}
and assumptions for \cref{thm:mikusheva-sun-alt}  imply that 
\begin{align*}
    & \max_{j\in[p]}|AR_j-\lambda_j^\Delta| \max_{j\in[p]} \left|\frac{\widehat{\Phi}_j}{\Phi_j}-1\right| \\
    & = O_p\left(\max_{j\in[p]} \frac{\Delta_j|Z_j\gamma_{n,j}|}{\sqrt{K_{n,j}}}\log^{7/2}(np)  + \max_{j\in[p]}  \frac{1}{\sqrt{K_{n,j}}}\log^{9/2}(np)\right) = o_p\left( \frac{1}{\sqrt{\log(np)}} \right),
\end{align*}
and, since $\lambda_j^\Delta = O_p(\Delta_j^2 K_{n,j}^{-1/2} |Z_j \gamma_{n,j}|^2)$ for each $j\in[p]$, assumptions for \cref{thm:mikusheva-sun-alt}  imply that 
\begin{align*}
    & \max_{j\in[p]} |\lambda_j^\Delta|\max_{j\in[p]} \left|\frac{\widehat{\Phi}_j}{\Phi_j}-1\right| \\ 
    & =
    O_p\left(\max_{j\in[p]} \frac{\Delta_j^3|Z_j\gamma_{n,j}|^3}{K_{n,j}}\log^3(np)  + \max_{j\in[p]}  \frac{\Delta_j^2|Z_j\gamma_{n,j}|^2}{K_{n,j}}\log^4(np)\right) = o_p\left( \frac{1}{\sqrt{\log(np)}} \right).
\end{align*}
Therefore $(\rho^\Delta.II) = o_p(1/\sqrt{\log(np)})$ and Lemma 1 in \cite{chernozhukov2023high} completes the proof.

\end{proof}

\subsection{Proof of \cref{thm:mikusheva-sun-bootstrap}}

\begin{proof}
First, observe that
\begin{align*}
    & \sup_{\substack{(\Delta_1, \dots, \Delta_p) \in \mathbb{R}^p \\ \max_{j\in[p]} |\Delta_j|^2 \le c }}\sup_{R\in \mathcal{R}_p}\left| \mathbb{P}\left(\hat{AR} - \lambda^\Delta\in R \mid Z\right) - \mathbb{P}^*\left( \hat{AR}^* \in R  \right)\right| \\
    & \le \sup_{\substack{(\Delta_1, \dots, \Delta_p) \in \mathbb{R}^p \\ \max_{j\in[p]} |\Delta_j|^2 \le c }}\sup_{R\in \mathcal{R}_p}\left| \mathbb{P}\left({AR} - \lambda^\Delta \in R \mid Z\right) - \mathbb{P}\left( G^\Delta\in R  \mid Z\right)\right| \\
    & \quad + \sup_{\substack{(\Delta_1, \dots, \Delta_p) \in \mathbb{R}^p \\ \max_{j\in[p]} |\Delta_j|^2 \le c }}\sup_{R\in \mathcal{R}_p}\left| \mathbb{P}\left(\hat{AR} - \lambda^\Delta \in R \mid Z\right) - \mathbb{P}\left( AR - \lambda^\Delta \in R  \mid Z\right)\right| \\
    & \quad + \sup_{\substack{(\Delta_1, \dots, \Delta_p) \in \mathbb{R}^p \\ \max_{j\in[p]} |\Delta_j|^2 \le c }}\sup_{R\in \mathcal{R}_p}\left|\mathbb{P}^*\left( \hat{AR}^* \in R  \right) - \mathbb{P}\left( G^\Delta\in R  \mid Z\right) \right|.
\end{align*}
Since we have already shown the convergence of the first two terms, it is sufficient to show the convergence of the final term. 
We define
\begin{align*}
    (\rho^\Delta.III) \coloneqq \sup_{\substack{(\Delta_1, \dots, \Delta_p) \in \mathbb{R}^p \\ \max_{j\in[p]} |\Delta_j|^2 \le c }}\sup_{R\in \mathcal{R}_p}\left|\mathbb{P}^*\left( \hat{AR}^* \in R  \right) - \mathbb{P}\left( G^\Delta\in R  \mid Z\right) \right|.
\end{align*}
Also, define $G^* \sim N(0, \Cov^*(\hat{AR}^*))$, then from the triangle inequality and Proposition 2.1 in \cite{CCKK22}, we have
\begin{align*}
     (\rho^\Delta.III) & \le \sup_{\substack{(\Delta_1, \dots, \Delta_p) \in \mathbb{R}^p \\ \max_{j\in[p]} |\Delta_j|^2 \le c }}\sup_{R\in \mathcal{R}_p}\left|\mathbb{P}^*\left( \hat{AR}^* \in R  \right) - \mathbb{P}^*\left( G^*\in R  \right) \right| \\
     & \quad + C (\log^2p)^{1/2}  \sup_{\substack{(\Delta_1, \dots, \Delta_p) \in \mathbb{R}^p \\ \max_{j\in[p]} |\Delta_j|^2 \le c }}\sup_{(j,k)\in[p]^2} \left| \Cov^*(\hat{AR}^*_j, \hat{AR}^*_k) - \Cov(G^\Delta_j, G^\Delta_k \mid Z)\right|^{1/2} \\
     & \eqqcolon (\rho^\Delta.III.I) +  (\rho^\Delta.III.II).
\end{align*}

We first evaluate $ (\rho^\Delta.III.I)$.
At first, we do not specify the threshold level $\tau_n$, and later show that taking $\tau_n = \sqrt{log(np)}$ is sufficient.
Since $\hat{AR}_j^\star$ is defined as 
\begin{align*}
    \hat{AR}_j^\star \coloneqq \frac{1}{\sqrt{K_{n,j}}\sqrt{\check{\Phi}_j}} \sum_{i=1}^n\sum_{m\neq i} \Pi_{im,j} \bar u_{ij}\bar u_{mj} \xi_{i} \xi_m,
\end{align*}
where $\check{\Phi}_j$ and $\Pi_{im,j}\bar u_{ij}\bar u_{mj}$ are deterministic under $\mathbb{P}^*$, $ \hat{AR}_j^\star$ is a centered Gaussian quadratic form with unit conditional variance.
Therefore, High-dimensional CLT for Gaussian quadratic form (Lemma 10 in \citealp{imai2025gaussian}) gives
\begin{align*}
    &\sup_{R\in\mathcal R_p}\left|\mathbb P^*(\hat{AR}^*\in R)-\mathbb P^*(G^*\in R)\right|
    \lesssim\left\{\max_{j\in[p]}\kappa_4^*(\hat{AR}_j^*)\log^6p\right\}^{1/4},
\end{align*}
where $\kappa_4^*(\hat{AR}_j^*) =\mathbb E^*[(\hat{AR}_j^*)^4]-3$ denotes the conditional fourth cumulant.
From a identity for the fourth cumulant of Gaussian quadratic form (for example, Eq.(11) in \cite{dalalyan2011second}), we can see that
\begin{align*}
    \kappa_4^*(\hat{AR}_j^*) = 48\text{tr}\left[ \left\{ \frac{\diag(\bar u_j) \{\Pi_j - \diag(\Pi_j)\} \diag(\bar u_j)  }{\sqrt{K_{n,j}}\sqrt{\check{\Phi}_j}} \right\}^4 \right].
\end{align*}
Note that the inequality $\text{tr}(A^4)=\|A^2\|_F^2\le \|A\|_{\text{op}}^2 \|A\|_F^2 =\|A\|_{\text{op}}^2 \text{tr}(A^2)$ holds for a symmetric matrix $A$.
Also,
\begin{align*}
    &\left\|
    \frac{\diag(\bar u_j)
    \{\Pi_j-\diag(\Pi_j)\}\diag(\bar u_j)}
    {\sqrt{K_{n,j}}\sqrt{\check{\Phi}_j}}
    \right\|_{\mathrm{op}}^2\\
    &\le
    \frac{
    \|\diag(\bar u_j)\|_{\mathrm{op}}^4
    \|\Pi_j-\diag(\Pi_j)\|_{\mathrm{op}}^2
    }{K_{n,j}\check{\Phi}_j}\le
    \frac{4\tau_n^4}{K_{n,j}\check{\Phi}_j},
\end{align*}
where the first inequality follows from the standard property of the operator norm $(\|AB\|_{\text{op}} \le \|A\|_{\text{op}}\|B\|_{\text{op}})$,
and the second inequality uses $\|\diag(\bar u_j)\|_{\mathrm{op}}\le\tau_n$ and $\|\Pi_j-\diag(\Pi_j)\|_{\mathrm{op}}\le2$.
Also,
\begin{align*}
    &\operatorname{tr}\left[\left\{\frac{\diag(\bar u_j)\{\Pi_j-\diag(\Pi_j)\}\diag(\bar u_j)}{\sqrt{K_{n,j}}\sqrt{\check{\Phi}_j}}\right\}^2\right]=\frac{1}{K_{n,j}\check{\Phi}_j}\sum_{i=1}^n\sum_{m\neq i} \Pi_{im,j}^2\bar u_{ij}^2\bar u_{mj}^2=\frac{1}{2},
\end{align*}
where the first equality follows from the symmetry of $\Pi_j$ and the zero diagonal of $\Pi_j-\diag(\Pi_j)$, and the second equality follows from the definition of $\check{\Phi}_j$.
Therefore, 
\begin{align*}
    (\rho^\Delta.III.I) \lesssim \left(  \sup_{\substack{(\Delta_1, \dots, \Delta_p) \in \mathbb{R}^p \\ \max_{j\in[p]} |\Delta_j|^2 \le c }}\max_{j\in[p]} \frac{\tau_n^4 \log^6 p}{K_{n,j}\check{\Phi}_j} \right)^{1/4}.
\end{align*}

Next, we evaluate $(\rho^\Delta.III.II)$.
First, observe that
\begin{align*}
    & \Cov^*(\hat{AR}^*_j, \hat{AR}^*_k) = \frac{2}{\sqrt{K_{n,j} K_{n,k}  \check{\Phi}_j \check{\Phi}_k }} \sum_{i=1}^n\sum_{m\neq i} \Pi_{im,j}\Pi_{im,k} \bar u_{ij} \bar u_{mj} \bar u_{ik} \bar u_{mk}, \\
    & \Cov(G^\Delta_j, G^\Delta_k \mid Z) = \frac{2}{\sqrt{K_{n,j} K_{n,k} \Phi_j \Phi_k}}\sum_{i=1}^n\sum_{m\neq i} \Pi_{im,j}\Pi_{im,k} \mathbb{E}[\eta_{i,j}\eta_{i,k}\mid Z]\mathbb{E}[\eta_{m,j}\eta_{m,k}\mid Z].
\end{align*}

Since $\Phi_j\asymp 1$ uniformly in $j\in[p]$,
whenever $\max_{j\in[p]}|\check{\Phi}_j-\Phi_j|$
is sufficiently small, we also have
$\check{\Phi}_j\asymp 1$ uniformly in $j\in[p]$.
In particular,
\begin{align*}
    \left|
    \sqrt{\frac{\Phi_j}{\check{\Phi}_j}}-1
    \right|
    &=
    \frac{|\Phi_j-\check{\Phi}_j|}
    {\sqrt{\check{\Phi}_j}
    (\sqrt{\Phi_j}+\sqrt{\check{\Phi}_j})}\lesssim |\check{\Phi}_j-\Phi_j|.
\end{align*}
Therefore,
\begin{align*}
    &\left|\sqrt{\frac{\Phi_j\Phi_k}{\check{\Phi}_j\check{\Phi}_k}}-1\right|
    \le
    \left|\sqrt{\frac{\Phi_j}{\check{\Phi}_j}}-1\right|
    \sqrt{\frac{\Phi_k}{\check{\Phi}_k}}+\left|\sqrt{\frac{\Phi_k}{\check{\Phi}_k}}-1\right|
    \lesssim
    |\check{\Phi}_j-\Phi_j|+|\check{\Phi}_k-\Phi_k|.
\end{align*}
Consequently,
\begin{align*}
    &\left|
    \Cov^*(\hat{AR}_j^*,\hat{AR}_k^*)
    -\Cov(G_j^\Delta,G_k^\Delta\mid Z)
    \right|\\
    &\le
    \frac{1}{\sqrt{\check{\Phi}_j\check{\Phi}_k}}
    \Bigg|
    \frac{2}{\sqrt{K_{n,j}K_{n,k}}}
    \sum_{i=1}^n\sum_{m\neq i}
    \Pi_{im,j}\Pi_{im,k}
    \Big\{
    \bar u_{ij}\bar u_{mj}\bar u_{ik}\bar u_{mk}
    -
    \mathbb E[\eta_{i,j}\eta_{i,k}\mid Z]
    \mathbb E[\eta_{m,j}\eta_{m,k}\mid Z]
    \Big\}
    \Bigg|\\
    &\quad+
    \left|\Cov(G_j^\Delta,G_k^\Delta\mid Z)\right|
    \left|
    \sqrt{\frac{\Phi_j\Phi_k}
    {\check{\Phi}_j\check{\Phi}_k}}-1
    \right|\\
    &\lesssim
    \Bigg|
    \frac{2}{\sqrt{K_{n,j}K_{n,k}}}
    \sum_{i=1}^n\sum_{m\neq i}
    \Pi_{im,j}\Pi_{im,k}
    \Big\{
    \bar u_{ij}\bar u_{mj}\bar u_{ik}\bar u_{mk}
    -
    \mathbb E[\eta_{i,j}\eta_{i,k}\mid Z]
    \mathbb E[\eta_{m,j}\eta_{m,k}\mid Z]
    \Big\}
    \Bigg|\\
    &\quad+
    |\check{\Phi}_j-\Phi_j|
    +|\check{\Phi}_k-\Phi_k|,
\end{align*}
where the first inequality follows from the triangle
inequality, and the second inequality uses the
preceding bound and
$|\Cov(G_j^\Delta,G_k^\Delta\mid Z)|\le 1$.
Since the differences $\check{\Phi}_j-\Phi_j$ and $\check{\Phi}_k-\Phi_k$
correspond to the diagonal cases of the numerator differences, we obtain
\begin{align*}
    &\max_{(j,k)\in[p]^2}
    \left|
    \Cov^*(\hat{AR}_j^*,\hat{AR}_k^*)
    -\Cov(G_j^\Delta,G_k^\Delta\mid Z)
    \right|\\
    &\lesssim
    \max_{(j,k)\in[p]^2}
    \Bigg|
    \frac{2}{\sqrt{K_{n,j}K_{n,k}}}
    \sum_{i=1}^n\sum_{m\neq i}
    \Pi_{im,j}\Pi_{im,k}
    \Big\{
    \bar u_{ij}\bar u_{mj}\bar u_{ik}\bar u_{mk}
    -
    \mathbb E[\eta_{i,j}\eta_{i,k}\mid Z]
    \mathbb E[\eta_{m,j}\eta_{m,k}\mid Z]
    \Big\}
    \Bigg|.
\end{align*}
Therefore, it suffices to evaluate the difference between the numerators.
The triangle inequality gives
\begin{align}
    & \Bigg|\frac{2}{\sqrt{K_{n,j}K_{n,k}}}\sum_{i=1}^n\sum_{m\neq i}\Pi_{im,j}\Pi_{im,k}\Big\{\bar u_{ij}\bar u_{mj}\bar u_{ik}\bar u_{mk}-\mathbb E[\eta_{i,j}\eta_{i,k}\mid Z]\mathbb E[\eta_{m,j}\eta_{m,k}\mid Z]\Big\}\Bigg|\nonumber \\
    & \le \Bigg|\frac{2}{\sqrt{K_{n,j}K_{n,k}}}\sum_{i=1}^n\sum_{m\neq i} \Pi_{im,j}\Pi_{im,k}\Big\{\bar u_{ij}\bar u_{mj}\bar u_{ik}\bar u_{mk}-\eta_{i,j}\eta_{i,k}\eta_{m,j}\eta_{m,k}\Big\}\Bigg| \nonumber\\
    & \quad + \Bigg|\frac{2}{\sqrt{K_{n,j}K_{n,k}}}\sum_{i=1}^n\sum_{m\neq i}\Pi_{im,j}\Pi_{im,k}\Big\{\eta_{i,j}\eta_{i,k}\eta_{m,j}\eta_{m,k}-\mathbb E[\eta_{i,j}\eta_{i,k}\mid Z]\mathbb E[\eta_{m,j}\eta_{m,k}\mid Z]\Big\}\Bigg|. \label{eq:ms-boot-3-2-1}
\end{align}

We first evaluate the first term of \cref{eq:ms-boot-3-2-1}. 
A simple algebra and Cauchy-Schwarz inequality and the triangle inequality give
\begin{align*}
    & \Bigg|\frac{2}{\sqrt{K_{n,j}K_{n,k}}}\sum_{i=1}^n\sum_{m\neq i} \Pi_{im,j}\Pi_{im,k}\Big\{\bar u_{ij}\bar u_{mj}\bar u_{ik}\bar u_{mk}-\eta_{i,j}\eta_{i,k}\eta_{m,j}\eta_{m,k}\Big\}\Bigg| \\
    & = \Bigg|\frac{2}{\sqrt{K_{n,j}K_{n,k}}}\sum_{i=1}^n\sum_{m\neq i} \Pi_{im,j}\Pi_{im,k}\Big\{(\bar u_{ij}\bar u_{mj} -\eta_{i,j}\eta_{m,j} )\bar u_{ik}\bar u_{mk} + \eta_{i,j}\eta_{m,j}(\bar u_{ik}\bar u_{mk}-\eta_{i,k}\eta_{m,k})\Big\}\Bigg| \\
    & \le 2\left( \frac{1}{K_{n,j}} \sum_{i=1}^n\sum_{m\neq i} \Pi_{im,j}^2  (\bar u_{ij}\bar u_{mj} -\eta_{i,j}\eta_{m,j})^2 \right)^{1/2} \left( \frac{1}{K_{n,k}} \sum_{i=1}^n\sum_{m\neq i} \Pi_{im,k}^2\bar u_{ik}^2\bar u_{mk}^2 \right)^{1/2} \\
    & \quad + 2 \left( \frac{1}{K_{n,j}} \sum_{i=1}^n\sum_{m\neq i} \Pi_{im,j}^2  \eta_{i,j}^2\eta_{m,j}^2\right)^{1/2}\left( \frac{1}{K_{n,k}} \sum_{i=1}^n\sum_{m\neq i} \Pi_{im,k}^2  (\bar u_{ik}\bar u_{mk} -\eta_{i,k}\eta_{m,k})^2 \right)^{1/2}.
\end{align*}
Also
\begin{align*}
    & \frac{1}{K_{n,j}} \sum_{i=1}^n\sum_{m\neq i} \Pi_{im,j}^2  (\bar u_{ij}\bar u_{mj} -\eta_{i,j}\eta_{m,j})^2 \\
    & = \frac{1}{K_{n,j}} \sum_{i=1}^n\sum_{m\neq i} \Pi_{im,j}^2  \{(\bar u_{ij} - \eta_{i,j})\bar u_{mj} +\eta_{i,j}(\bar u_{mj} - \eta_{m,j})\}^2 \\
    & \le \frac{2\Delta_j^2}{K_{n,j}}\sum_{i=1}^n\sum_{m\neq i} \Pi_{im,j}^2 \{\tau_n^2|\gamma_{n,j}^\top Z_{ij}|^2 + C|\gamma_{n,j}^\top Z_{mj}|^2\} \le \frac{2\Delta_j^2}{K_{n,j}} (\tau_n^2 + C) |Z_j \gamma_{n,j}|^2,
\end{align*}
and
\begin{align*}
    \frac{1}{K_{n,k}} \sum_{i=1}^n\sum_{m\neq i} \Pi_{im,k}^2\bar u_{ik}^2\bar u_{mk}^2 \le \frac{\tau_n^4}{K_{n,k}} \sum_{i=1}^n\sum_{m\neq i} \Pi_{im,k}^2 \le  \frac{\tau_n^4}{K_{n,k}} \sum_{i=1}^n\sum_{m =1}^n \Pi_{im,k}^2  = \tau^4_n,
\end{align*}
and
\begin{align*}
    \frac{1}{K_{n,j}} \sum_{i=1}^n\sum_{m\neq i} \Pi_{im,j}^2  \eta_{i,j}^2\eta_{m,j}^2 \le  \frac{C}{K_{n,j}} \sum_{i=1}^n\sum_{m\neq i} \Pi_{im,j}^2  \le \frac{C}{K_{n,j}} \sum_{i=1}^n \sum_{m=1}^n \Pi_{im,j}^2 \le C.
\end{align*}
Therefore, in terms of the first term of
\cref{eq:ms-boot-3-2-1}, we have
\begin{align*}
    &\sup_{\max_{j\in[p]}|\Delta_j|^2\le c}\max_{(j,k)\in[p]^2}\Bigg|
    \frac{2}{\sqrt{K_{n,j}K_{n,k}}}
    \sum_{i=1}^n\sum_{m\neq i}
    \Pi_{im,j}\Pi_{im,k}
    \Big\{
    \bar u_{ij}\bar u_{mj}\bar u_{ik}\bar u_{mk}
    -\eta_{i,j}\eta_{i,k}\eta_{m,j}\eta_{m,k}
    \Big\}
    \Bigg| \\
    & \lesssim
    \sup_{\max_{j\in[p]}|\Delta_j|^2\le c}\max_{j\in[p]}\tau_n^3
    \frac{|\Delta_j||Z_j\gamma_{n,j}|}
    {\sqrt{K_{n,j}}}.
\end{align*}

Next, we evaluate the second term of
\cref{eq:ms-boot-3-2-1}.
Since $\Pi_{im,j}\Pi_{im,k}=\Pi_{mi,j}\Pi_{mi,k}$,
a simple algebra yields
\begin{align*}
    &\frac{2}{\sqrt{K_{n,j}K_{n,k}}}
    \sum_{i=1}^n\sum_{m\neq i}
    \Pi_{im,j}\Pi_{im,k}
    \Big\{
    \eta_{i,j}\eta_{i,k}\eta_{m,j}\eta_{m,k}
    -
    \mathbb E[\eta_{i,j}\eta_{i,k}\mid Z]
    \mathbb E[\eta_{m,j}\eta_{m,k}\mid Z]
    \Big\}\\
    &=
    \frac{4}{\sqrt{K_{n,j}K_{n,k}}}
    \sum_{1\le i<m\le n}
    \Pi_{im,j}\Pi_{im,k}
    \Big\{
    \eta_{i,j}\eta_{i,k}
    -\mathbb E[\eta_{i,j}\eta_{i,k}\mid Z]
    \Big\}
    \Big\{
    \eta_{m,j}\eta_{m,k}
    -\mathbb E[\eta_{m,j}\eta_{m,k}\mid Z]
    \Big\}\\
    &\quad+
    \frac{4}{\sqrt{K_{n,j}K_{n,k}}}
    \sum_{i=1}^n
    \Big\{
    \eta_{i,j}\eta_{i,k}
    -\mathbb E[\eta_{i,j}\eta_{i,k}\mid Z]
    \Big\}
    \sum_{m\neq i}
    \Pi_{im,j}\Pi_{im,k}
    \mathbb E[\eta_{m,j}\eta_{m,k}\mid Z].
\end{align*}
Note that $\mathbb E\left[\eta_{i,j}\eta_{i,k}-\mathbb E[\eta_{i,j}\eta_{i,k}\mid Z]\mid Z\right]=0$.
Since $(\eta_{i,j})_{j\in[p]}$ are independent
across $i\in[n]$ conditional on $Z$ by
\cref{as:mikusheva-sun}(ii), the first term
on the right-hand side is a second-order
degenerate $U$-statistic, whereas the second
term is a first-order degenerate $U$-statistic.
Therefore, we can apply \cref{thm:max-is-inid}
with $r=2$ and $r=1$, respectively.
Since $|\eta_{i,j}|\le C$ uniformly in $i\in[n]$
and $j\in[p]$, applying \cref{thm:max-is-inid}
with $r=1,2$ and $q=1$ gives
\begin{align*}
    &\mathbb E\Bigg[
    \sup_{\max_{j\in[p]}|\Delta_j|^2\le c}\max_{(j,k)\in[p]^2}
    \Bigg|
    \frac{2}{\sqrt{K_{n,j}K_{n,k}}}
    \sum_{i=1}^n\sum_{m\neq i}
    \Pi_{im,j}\Pi_{im,k}\\
    & \qquad\qquad\qquad\qquad\qquad\Big\{
    \eta_{i,j}\eta_{i,k}\eta_{m,j}\eta_{m,k}
    -
    \mathbb E[\eta_{i,j}\eta_{i,k}\mid Z]
    \mathbb E[\eta_{m,j}\eta_{m,k}\mid Z]
    \Big\}
    \Bigg|\Biggm|Z\Bigg]\\
    &\lesssim
    \sup_{\max_{j\in[p]}|\Delta_j|^2\le c}\max_{(j,k)\in[p]^2}
    \frac{1}{\sqrt{K_{n,j}K_{n,k}}}
    \Bigg\{
    \log(np)
    \left(
    \sum_{i=1}^n\sum_{m\neq i}
    \Pi_{im,j}^2\Pi_{im,k}^2
    \right)^{1/2}\\
    &\qquad + 
    \log^{3/2}(np)
    \left(
    \max_{i\in[n]}\sum_{m\neq i}
    \Pi_{im,j}^2\Pi_{im,k}^2
    \right)^{1/2} +
    \log^2(np)
    \max_{i\neq m}|\Pi_{im,j}\Pi_{im,k}|\\
    &\qquad + 
    \sqrt{\log(np)}
    \left[
    \sum_{i=1}^n
    \left(
    \sum_{m\neq i}|\Pi_{im,j}\Pi_{im,k}|
    \right)^2
    \right]^{1/2}+
    \log(np)
    \max_{i\in[n]}
    \sum_{m\neq i}|\Pi_{im,j}\Pi_{im,k}|
    \Bigg\}\\
    &\lesssim
    \sup_{\max_{j\in[p]}|\Delta_j|^2\le c}\max_{(j,k)\in[p]^2}
    \left\{
    \frac{\log(np)}{(K_{n,j}K_{n,k})^{1/4}}
    +
    \frac{\log^2(np)}{\sqrt{K_{n,j}K_{n,k}}}
    \right\} \\
    & \lesssim
    \sup_{\max_{j\in[p]}|\Delta_j|^2\le c}\max_{j\in[p]}
    \left\{
    \frac{\log(np)}{\sqrt{K_{n,j}}}
    +
    \frac{\log^2(np)}{K_{n,j}}
    \right\}
    =o\left(\log^{-2}(np)\right),
\end{align*}
where the second inequality follows from
\begin{align*}
    &\sum_{i=1}^n\sum_{m\neq i}
    \Pi_{im,j}^2\Pi_{im,k}^2
    \le\sqrt{K_{n,j}K_{n,k}}, ~~\max_{i\in[n]}\sum_{m\neq i}
    \Pi_{im,j}^2\Pi_{im,k}^2\le1,\\
    &\sum_{i=1}^n
    \left(\sum_{m\neq i}
    |\Pi_{im,j}\Pi_{im,k}|\right)^2
    \le\sqrt{K_{n,j}K_{n,k}}, ~~\max_{i\in[n]}
    \sum_{m\neq i}|\Pi_{im,j}\Pi_{im,k}|
    \le1,
\end{align*}
and the final equality follows from  \cref{as:mikusheva-sun}(i).

Summing up, by \cref{eq:ms-boot-3-2-1} and Markov's inequality, we obtain
\begin{align*}
    &\sup_{\max_{j\in[p]}|\Delta_j|^2\le c}
    \max_{(j,k)\in[p]^2}
    \Bigg|
    \frac{2}{\sqrt{K_{n,j}K_{n,k}}}
    \sum_{i=1}^n\sum_{m\neq i}
    \Pi_{im,j}\Pi_{im,k}\Big\{\bar u_{ij}\bar u_{mj}\bar u_{ik}\bar u_{mk}-\mathbb E[\eta_{i,j}\eta_{i,k}\mid Z]\mathbb E[\eta_{m,j}\eta_{m,k}\mid Z]\Big\}\Bigg|\\
    &\lesssim
    \tau_n^3
    \sup_{\max_{j\in[p]}|\Delta_j|^2\le c}
    \max_{j\in[p]}
    \frac{|\Delta_j||Z_j\gamma_{n,j}|}
    {\sqrt{K_{n,j}}}
    +
    O_p\left(\max_{j\in[p]}\left\{\frac{\log(np)}{\sqrt{K_{n,j}}}+\frac{\log^2(np)}{K_{n,j}}\right\}\right).
\end{align*}
Now, take $\tau_n=\sqrt{\log(np)}$.
By the assumptions of \cref{thm:mikusheva-sun-alt},
\begin{align*}
    &\tau_n^3\log^2(np)
    \sup_{\max_{j\in[p]}|\Delta_j|^2\le c}
    \max_{j\in[p]}
    \frac{|\Delta_j||Z_j\gamma_{n,j}|}
    {\sqrt{K_{n,j}}}
    =
    \log^{7/2}(np)
    \sup_{\max_{j\in[p]}|\Delta_j|^2\le c}
    \max_{j\in[p]}
    \frac{|\Delta_j||Z_j\gamma_{n,j}|}
    {\sqrt{K_{n,j}}}
    =o(1).
\end{align*}
Also, \cref{as:mikusheva-sun}(i) gives
\begin{align*}
    \sup_{\max_{j\in[p]}|\Delta_j|^2\le c}\max_{j\in[p]}
    \left\{
    \frac{\log^3(np)}{\sqrt{K_{n,j}}}
    +
    \frac{\log^4(np)}{K_{n,j}}
    \right\}
    =o(1).
\end{align*}
Consequently, the difference between the
numerators is uniformly $o_p(\log^{-2}(np))$.
Therefore, 
\begin{align*}
    &\sup_{\max_j|\Delta_j|^2\le c}
    \max_{(j,k)\in[p]^2}
    \left|
    \Cov^*(\hat{AR}_j^*,\hat{AR}_k^*)
    -
    \Cov(G_j^\Delta,G_k^\Delta\mid Z)
    \right|=o_p(\log^{-2}(np)).
\end{align*}
Hence, $(\rho^\Delta.III.II)=o_p(1)$.

Moreover, since $\tau_n^4=\log^2(np)$,
the preceding bound for $(\rho^\Delta.III.I)$
gives
\begin{align*}
    (\rho^\Delta.III.I)
    &\lesssim
    \left\{
    \sup_{\max_j|\Delta_j|^2\le c}
    \max_{j\in[p]}
    \frac{\tau_n^4\log^6p}
    {K_{n,j}\check{\Phi}_j}
    \right\}^{1/4}=O_p\left(
    \left\{
    \max_{j\in[p]}
    \frac{\log^8(np)}{K_{n,j}}
    \right\}^{1/4}
    \right)
    =o_p(1),
\end{align*}
where the final equality follows from \cref{as:mikusheva-sun}(i).
Thus, $(\rho^\Delta.III)=o_p(1)$, which completes the proof.
\end{proof}

\subsection{Proof of \cref{coro:mikusheva-sun-multitest}}
\begin{proof}[Proof of \cref{coro:mikusheva-sun-multitest}(i)]
Let $q_{1-\alpha,\mathcal H_0}^*$ and $q_{1-\alpha-\eta,\mathcal H_0}^{G}(Z)$ denote the conditional quantiles of $\max_{j\in\mathcal H_0}\hat{AR}_j^*$ and $\max_{j\in\mathcal H_0}G_j$, respectively, where $\eta\in(0,1-\alpha)$ is arbitrary and $G$ is the Gaussian vector in \cref{thm:mikusheva-sun-null}.
First, since $ q_{1-\alpha,\mathcal H_0}^* \le q_{1-\alpha}^*$ by definition, we have 
\begin{align*}
    & \mathbb{P}\left( \max_{j\in \mathcal{H}_0} \hat{AR}_j > q_{1-\alpha}^* \mid Z\right) \le  \mathbb{P}\left( \max_{j\in \mathcal{H}_0} \hat{AR}_j > q_{1-\alpha,\mathcal{H}_0}^* \mid Z\right).
\end{align*}
Next, for any fixed \(\eta\in(0,1-\alpha)\), we have
\begin{align*}
&\mathbb P\left(
\max_{j\in\mathcal H_0}\hat{AR}_j
>q_{1-\alpha,\mathcal H_0}^*
\mid Z\right)\\
&\le
\mathbb P\left(
\max_{j\in\mathcal H_0}\hat{AR}_j
>q_{1-\alpha-\eta,\mathcal H_0}^{G}(Z)
\mid Z\right)+
\mathbb P\left(
q_{1-\alpha,\mathcal H_0}^*
<q_{1-\alpha-\eta,\mathcal H_0}^{G}(Z)
\mid Z\right)\\
&=
\alpha+\eta+o_p(1)
+
\mathbb P\left(
q_{1-\alpha,\mathcal H_0}^*
<q_{1-\alpha-\eta,\mathcal H_0}^{G}(Z)
\mid Z\right),
\end{align*}
where the equality follows from \cref{thm:mikusheva-sun-null}.
Thus, for each fixed \(\eta\in(0,1-\alpha)\), it remains to show that
\begin{align*}
    \mathbb P\left(
q_{1-\alpha,\mathcal H_0}^*
<q_{1-\alpha-\eta,\mathcal H_0}^{G}(Z)
\mid Z\right)=o_p(1).
\end{align*}
By the definition of the conditional quantile, we have
\[
\begin{aligned}
&\left\{
q_{1-\alpha,\mathcal H_0}^*
<q_{1-\alpha-\eta,\mathcal H_0}^{G}(Z)
\right\}\subseteq
\left\{
\mathbb P^*\left(
\max_{j\in\mathcal H_0}\hat{AR}_j^*
\le q_{1-\alpha-\eta,\mathcal H_0}^{G}(Z)
\right)\ge1-\alpha
\right\}.
\end{aligned}
\]
Therefore,
\[
\begin{aligned}
&\mathbb P\left(
q_{1-\alpha,\mathcal H_0}^*
<q_{1-\alpha-\eta,\mathcal H_0}^{G}(Z)
\mid Z\right)\\
&\le
\mathbb P\left(
\mathbb P^*\left(
\max_{j\in\mathcal H_0}\hat{AR}_j^*
\le q_{1-\alpha-\eta,\mathcal H_0}^{G}(Z)
\right)\ge1-\alpha
\mid Z\right).
\end{aligned}
\]
From \cref{thm:mikusheva-sun-bootstrap}(i) and \cref{thm:mikusheva-sun-null}, we have
\begin{align*}
&\mathbb P^*\left(
\max_{j\in\mathcal H_0}\hat{AR}_j^*
\le q_{1-\alpha-\eta,\mathcal H_0}^{G}(Z)
\right)\\
&=
\mathbb P\left(
\max_{j\in\mathcal H_0}G_j
\le q_{1-\alpha-\eta,\mathcal H_0}^{G}(Z)
\mid Z\right)+o_p(1)=1-\alpha-\eta+o_p(1).
\end{align*}
Since \(\eta>0\) is fixed, the definition of conditional quantiles implies that
\[
\mathbb P\left(
q_{1-\alpha,\mathcal H_0}^*
<q_{1-\alpha-\eta,\mathcal H_0}^{G}(Z)
\right) \le \mathbb{P}(1-\alpha - \eta + o_p(1) \ge 1-\alpha) = \mathbb{P}(o_p(1) \ge \eta)  \to0.
\]
Summing up, we have
\[
\mathbb P\left(
\max_{j\in\mathcal H_0}\hat{AR}_j
>q_{1-\alpha}^*
\mid Z\right)
\le\alpha+\eta+o_p(1).
\]
Letting \(n\to\infty\) and then \(\eta\downarrow0\) completes the proof.
\end{proof}

\begin{proof}[Proof of \cref{coro:mikusheva-sun-multitest}(ii)]
Fix a coordinate $j\in[p]$ satisfying $\lambda_j^\Delta-q_{1-\alpha}^{\Delta,G}(Z)\xrightarrow{p}\infty$.
For any fixed $a>0$, we have
\begin{align*}
    &\mathbb P\left(\hat{AR}_j>q_{1-\alpha}^*\mid Z\right)\\
    &=\mathbb P\left(\hat{AR}_j>q_{1-\alpha}^*,\,q_{1-\alpha}^*\le q_{1-\alpha}^{\Delta,G}(Z)+a\mid Z\right)
    +\mathbb P\left(\hat{AR}_j>q_{1-\alpha}^*,\,q_{1-\alpha}^*>q_{1-\alpha}^{\Delta,G}(Z)+a\mid Z\right)\\
    &\ge \mathbb P\left(\hat{AR}_j>q_{1-\alpha}^{\Delta,G}(Z)+a,\,q_{1-\alpha}^*\le q_{1-\alpha}^{\Delta,G}(Z)+a\mid Z\right)\\
    &=\mathbb P\left(\hat{AR}_j>q_{1-\alpha}^{\Delta,G}(Z)+a\mid Z\right)
    -\mathbb P\left(\hat{AR}_j>q_{1-\alpha}^{\Delta,G}(Z)+a,\,q_{1-\alpha}^*>q_{1-\alpha}^{\Delta,G}(Z)+a\mid Z\right)\\
    &\ge \mathbb P\left(\hat{AR}_j>q_{1-\alpha}^{\Delta,G}(Z)+a\mid Z\right)
    -\mathbb P\left(q_{1-\alpha}^*>q_{1-\alpha}^{\Delta,G}(Z)+a\mid Z\right)\\
    &=\left\{1-\Phi\left(q_{1-\alpha}^{\Delta,G}(Z)+a-\lambda_j^\Delta\right) \right\}+o_p(1)
    -\mathbb P\left(q_{1-\alpha}^*>q_{1-\alpha}^{\Delta,G}(Z)+a\mid Z\right),
\end{align*}
where the final equality follows from \cref{thm:mikusheva-sun-alt}.
By assumption, the first term in the last expression
converges to one in probability for every fixed $a>0$.
Thus, it suffices to show that there exists a fixed
$a>0$ such that
\begin{align}
    \mathbb P\left(
        q_{1-\alpha}^*>
        q_{1-\alpha}^{\Delta,G}(Z)+a
        \mid Z
    \right)=o_p(1).
    \label{eq:ms-consistency-critical}
\end{align}
By the definition of the bootstrap critical value, we have
\begin{align*}
    &\left\{
        q_{1-\alpha}^*>
        q_{1-\alpha}^{\Delta,G}(Z)+a
    \right\}
    \subseteq
    \left\{
        \mathbb P^*\left(
            \max_{k\in[p]}\hat{AR}_k^*>
            q_{1-\alpha}^{\Delta,G}(Z)+a
        \right)>\alpha
    \right\}.
\end{align*}
Also, by \cref{thm:mikusheva-sun-bootstrap}(ii)
and \cref{thm:mikusheva-sun-alt},
\begin{align*}
    &\mathbb P^*\left(
        \max_{k\in[p]}\hat{AR}_k^*>
        q_{1-\alpha}^{\Delta,G}(Z)+a
    \right)
    =
    \mathbb P\left(
        \max_{k\in[p]}G_k^\Delta>
        q_{1-\alpha}^{\Delta,G}(Z)+a
        \mid Z
    \right)+o_p(1).
\end{align*}
Therefore, it remains to show that, for some fixed
$a>0$, the Gaussian tail probability on the
right-hand side is bounded by a constant strictly
smaller than $\alpha$.
To this end, observe that the Gaussian Poincaré inequality
\citep[Theorem 3.20]{boucheron2013concentration}
gives
\begin{align*}
    \Var\left(
        \max_{k\in[p]}G_k^\Delta\mid Z
    \right)
    \le
    \max_{k\in[p]}\Var(G_k^\Delta\mid Z)
    =1,
\end{align*}
where the equality follows from the definition of
$\Phi_k$.
Consequently, Chebyshev's inequality yields
\begin{align*}
    &\mathbb P\left(
        \max_{k\in[p]}G_k^\Delta
        \le
        \mathbb E\left[
            \max_{k\in[p]}G_k^\Delta\mid Z
        \right]-\frac a2
        \mid  Z
    \right)\le \frac{4}{a^2}.
\end{align*}
If $4/a^2<1-\alpha$, the definition of the
conditional $(1-\alpha)$-quantile implies that
\begin{align*}
    q_{1-\alpha}^{\Delta,G}(Z)
    \ge
    \mathbb E\left[
        \max_{k\in[p]}G_k^\Delta\mid Z
    \right]-\frac a2.
\end{align*}
It follows that
\begin{align*}
    &\mathbb P\left(
        \max_{k\in[p]}G_k^\Delta>
        q_{1-\alpha}^{\Delta,G}(Z)+a
        \mid Z
    \right)
    \le
    \mathbb P\left(
        \max_{k\in[p]}G_k^\Delta>
        \mathbb E\left[
            \max_{k\in[p]}G_k^\Delta\mid Z
        \right]+\frac a2
        \mid Z
    \right)
    \le \frac{4}{a^2},
\end{align*}
where the last inequality follows again from
Chebyshev's inequality.
Now choose a fixed $a>0$ sufficiently large so that $4a^{-2}<\min\{\alpha,1-\alpha\}$.
Combining the preceding bounds with the bootstrap approximation gives
\begin{align*}
    &\mathbb P\left(
        q_{1-\alpha}^*>
        q_{1-\alpha}^{\Delta,G}(Z)+a
    \right)\\
    &\le
    \mathbb P\left(
        \mathbb P^*\left(
            \max_{k\in[p]}\hat{AR}_k^*>
            q_{1-\alpha}^{\Delta,G}(Z)+a
        \right)>\alpha
    \right)\le
    \mathbb P\left(
        \frac{4}{a^2}+o_p(1)>\alpha
    \right)
    \longrightarrow0,
\end{align*}
where the convergence follows from $4/a^2<\alpha$.
Also, by the law of iterated expectations,
\begin{align*}
    &\mathbb E\left[
        \mathbb P\left(
            q_{1-\alpha}^*>
            q_{1-\alpha}^{\Delta,G}(Z)+a
            \mid Z
        \right)
    \right]=
    \mathbb P\left(
        q_{1-\alpha}^*>
        q_{1-\alpha}^{\Delta,G}(Z)+a
    \right)\to0.
\end{align*}
Hence, Markov's inequality yields \cref{eq:ms-consistency-critical}.
Substituting this result into the first inequality of the proof, we obtain
\begin{align*}
    &\mathbb P\left(
        \hat{AR}_j>q_{1-\alpha}^*\mid Z
    \right)
    \ge
    1-\Phi\left(
        q_{1-\alpha}^{\Delta,G}(Z)
        +a-\lambda_j^\Delta
    \right)+o_p(1)\xrightarrow{p}1,
\end{align*}
where the convergence follows from
$\lambda_j^\Delta-q_{1-\alpha}^{\Delta,G}(Z)
\xrightarrow{p}\infty$ and the fact that $a$ is fixed.
\end{proof}

\section{Proof of \cref{lem:gluing_separable_array}} \label{subsec:proof:lem:gluing_separable_array}
\begin{proof}[Proof of \cref{lem:gluing_separable_array}]
Fix $A\in\mathcal{R}_p$. 
Since $II$ is $\mathcal{F}$-measurable and $\mathcal{R}_p$ is closed under coordinatewise shifts,
$A-II\in\mathcal{R}_p$ almost surely, and
\[
\mathbb{P}(I+II\in A \mid \mathcal{F})=\mathbb{P}(I\in A-II\mid \mathcal{F}),
~~
\mathbb{P}(Z_I(\mathcal{F})+II\in A \mid \mathcal{F})=\mathbb{P}(Z_I(\mathcal{F})\in A-II\mid \mathcal{F}).
\]
Therefore,
\begin{align*}
& \Bigl| \mathbb{P}(I+II\in A)-\mathbb{P}\bigl(Z_I(\mathcal{F})+II\in A\bigr)\Bigr| \\
&=\Bigl| \mathbb{E}\bigl[\mathbb{P}(I\in A-II\mid\mathcal{F})-\mathbb{P}(Z_I(\mathcal{F})\in A-II\mid\mathcal{F})\bigr]\Bigr| \\
&\le \mathbb{E}\Bigl[\sup_{B\in\mathcal{R}_p}\Bigl|\mathbb{P}(I\in B\mid\mathcal{F})-\mathbb{P}(Z_I(\mathcal{F})\in B\mid\mathcal{F})\Bigr|\Bigr] \le \mathbb{E}[\delta_{1,n}(\mathcal{F})].
\end{align*}
Then, similarly using shift-invariance of rectangles, we can see that
\begin{align*}
& \Bigl| \mathbb{P}\bigl(Z_I(\mathcal{F})+II\in A\bigr)-\mathbb{P}(Z_I+II\in A)\Bigr| \\
&=\Bigl| \mathbb{E}\bigl[\mathbb{P}(Z_I(\mathcal{F})\in A-II \mid \mathcal{F})-\mathbb{P}(Z_I\in A-II)\bigr]\Bigr| \\
& \le \mathbb{E}\Bigl[ \sup_{B\in\mathcal{R}_p}\Bigl|\mathbb{P}(Z_I(\mathcal{F})\in B \mid\mathcal{F})-\mathbb{P}(Z_I\in B)\Bigr|\Bigr] \le \mathbb{E}[\delta_{2,n}(\mathcal{F})].
\end{align*}
Finally, since $Z_I$ is independent of $(II,Z_{II})$, it holds that
\[
\mathbb{P}(Z_I+II\in A)=\mathbb{E}\bigl[\mathbb{P}(II\in A-Z_I)\bigr],
\qquad
\mathbb{P}(Z_I+Z_{II}\in A)=\mathbb{E}\bigl[\mathbb{P}(Z_{II}\in A-Z_I)\bigr].
\]
Since $A-z\in\mathcal{R}_p$ for every fixed $z\in\mathbb{R}^p$, it follows that
\begin{align*}
& \Bigl| \mathbb{P}(Z_I+II\in A)-\mathbb{P}(Z_I+Z_{II}\in A)\Bigr| \\
&=\Bigl| \mathbb{E}\bigl[\mathbb{P}(II\in A-Z_I)-\mathbb{P}(Z_{II}\in A-Z_I)\bigr]\Bigr| 
\le \sup_{B\in\mathcal{R}_p}\Bigl|\mathbb{P}(II\in B)-\mathbb{P}(Z_{II}\in B)\Bigr| \le \delta_{3,n}.
\end{align*}
Combining the three bounds by the triangle inequality and taking the supremum over $A\in\mathcal{R}_p$ yields
\[
\sup_{A\in\mathcal{R}_p}\Bigl| \mathbb{P}(I+II\in A)-\mathbb{P}(Z_I+Z_{II}\in A)\Bigr|
\le \mathbb{E}[\delta_{1,n}(\mathcal{F})]+\mathbb{E}[\delta_{2,n}(\mathcal{F})]+\delta_{3,n}.
\]
Since $Z_I$ and $Z_{II}$ are independent, $Z_I+Z_{II}\sim N(0,\Sigma_I+\Sigma_{II})$. This proves the claim.
\end{proof}

\clearpage
\setcounter{page}{1}
\renewcommand{\thepage}{R.\arabic{page}}
\bibliography{ref}
\bibliographystyle{apalike}
\end{document}